\documentclass[msc,american]{ThesisPUC_uk}

\usepackage[top=3cm, bottom=3cm, left=3cm, right=3cm]{geometry}
\usepackage{amsmath,amsfonts,a4,amssymb,enumitem,psfrag}
\usepackage{amsthm}
\usepackage[latin1]{inputenc}  
\usepackage{mathtools}
\usepackage{graphicx}
\usepackage{float}
\usepackage{subcaption}
\usepackage{scalerel}
\usepackage{rotating}
\usepackage[cal=cm, scr=dutchcal]{mathalfa}
\usepackage{bbm}
\usepackage{url}
\usepackage{color}
\usepackage[toc, nonumberlist, nogroupskip]{glossaries}

\newcommand{\R}{\mathbb{R}}
\newcommand{\Sp}{\mathbb{S}}
\newcommand{\C}{\mathbb{C}}
\newcommand{\T}{\mathbb{T}}
\newcommand{\N}{\mathbb{N}}
\newcommand{\Z}{\mathbb{Z}}

\newcommand{\bi}{\mathbf{i}}

\newcommand{\Mod}[1]{\ (\text{mod}\ #1)}

\newcommand{\quotient}[2]{{\left.\raisebox{.2em}{$#1$}\middle/\raisebox{-.2em}{$#2$}\right.}}
\DeclareRobustCommand{\qquotient}[2]{{\left.\raisebox{.2em}{$#1$}\middle/\raisebox{-.2em}{$#2$}\right.}} 

\DeclareRobustCommand{\rrestr}{\text{$f$}\raisebox{-.2em}{$\big|_{A}$}}

\DeclareMathOperator{\sgn}{sgn}
\DeclareMathOperator{\card}{card}

\DeclareMathOperator*{\Bigcdot}{\scalerel*{\cdot}{\bigodot}}

\DeclarePairedDelimiter\floor{\lfloor}{\rfloor}

\theoremstyle{plain}
\newtheorem*{theorem*}{Theorem}
\newtheorem{teorema}{Theorem} 
\newtheorem{theo}{Theorem}
\newtheorem{prop}{Proposition}[section]
\newtheorem{corolario}[prop]{Corollary}
\newtheorem{lema}[prop]{Lemma}

\newcommand\EatDot[1]{}

\AfterBegin{itemize}{\addtolength{\itemsep}{-0.5em}}
\BeforeBegin{equation}{\vspace{-0.5em}}
\BeforeEnd{equation}{\vspace{-0.5em}}
\BeforeBegin{theo}{\vspace{-0.4em}}
\BeforeEnd{theo}{\vspace{-0.2em}}
\BeforeBegin{prop}{\vspace{-0.4em}}
\BeforeEnd{prop}{\vspace{-0.2em}}
\author{Fillipo de Souza Lima Impellizieri}
\authorR{Impellizieri, Fillipo de Souza Lima}
\resume{The author obtained the degree of Bacharel em Matem\'atica from PUC-Rio in July 2013}
\advisor{Nicolau Cor\c{c}\~ao Saldanha}
\advisorR{Saldanha, Nicolau Cor\c{c}\~ao}
\title{Domino Tilings of the Torus}
\runninghead{Domino Tilings of the Torus}

\titlebr{Coberturas do Toro por Domin\'os}
\day{11} \month{September} \monthN{09} \year{2015}

\city{Rio de Janeiro}
\CDD{510}
\department{Matem\'atica}
\program{Matem\'atica}
\programbr{Matem\'atica}
\school{Centro T\'ecnico Cient\'ifico}
\university{Pontif\'icia Universidade Cat\'olica do Rio de Janeiro}
\uni{PUC-Rio}

\jury{
\jurymember{Carlos Tomei}{Departamento de Matem\'atica -- PUC-Rio}
\jurymember{Juliana Abrantes Freire}{Departamento de Matem\'atica -- PUC-Rio}
\jurymember{Marcio da Silva Passos Telles}{Instituto de Matem\'atica e Estat\'istica -- UERJ}
\jurymember{Robert David Morris}{Instituto de Matem\'atica Pura e Aplicada -- IMPA}
\schoolhead{Jos\'e Eugenio Leal}
}

\acknowledgment{ 
To CNPq, FAPERJ and PUC-Rio, for making this research possible.

To my advisor, Prof. Nicolau Cor\c{c}\~ao Saldanha, for his boundless patience, eagerness to help and valuable insight.

To the Department of Mathematics, especially Creuza, Fred and Renata, for helping me get here and putting up with my fickleness.

To my jury, for their appreciation, understanding and precious feedback. In particular, I would like to thank Carlos Tomei for his spirit and enthusiasm, and Marcio Telles for his dedication.

To my professor Ricardo Earp, for his guidance and caring.

To my friends, for their positivity and reassurance. In particular, to my graduation friends Felipe, Leonardo and Gregory for holding onto this `academic friendship'.

To my family, for their patience, unwavering support and cherished love.
}
\keywordsbr
{
\key{domin\'o}
\key{cobertura} 
\key{toro}
\key{reticulado}
\key{fluxo}
\key{flip}
\key{fun\c{c}\~ao altura}
\key{matriz de Kasteleyn}
}

\abstractbr
{
	Consideramos o problema de contar e classificar coberturas por domin\'{o}s de toros quadriculados. O problema de contagem para ret\^{a}ngulos foi estudado por Kasteleyn e usamos muitas de suas ideias. Coberturas por domin\'{o}s de regi\~oes planares podem ser representadas por fun\c{c}\~oes altura; para um toro dado por um reticulado $L$, estas fun\c{c}\~oes exibem $L$-quasiperiodicidade aritm\'{e}tica. As constantes aditivas determinam o fluxo da cobertura, que pode ser interpretado como um vetor no reticulado dual $(2L)^*$. Damos uma caracteriza\c{c}\~ao dos valores de fluxo efetivamente realizados e de como coberturas correspondentes se comportam. Tamb\'{e}m consideramos coberturas por domin\'{o}s do reticulado quadrado infinito; coberturas de toros podem ser vistas como um caso particular destas. Descrevemos a constru\c{c}\~ao e uso de matrizes de Kasteleyn no problema de contagem, e como elas podem ser aplicadas para contar coberturas com valores de fluxo prescritos. Finalmente, estudamos a distribui\c{c}\~ao limite do n\'{u}mero de coberturas com um dado valor de fluxo quando o reticulado $L$ sofre uma dilata\c{c}\~ao uniforme.
}
\abstract
{
	We consider the problem of counting and classifying domino tilings of a quadriculated torus. The counting problem for rectangles was studied by Kasteleyn and we use many of his ideas. Domino tilings of planar regions can be represented by height functions; for a torus given by a lattice $L$, these functions exhibit arithmetic $L$-quasiperiodicity. The additive constants determine the flux of the tiling, which can be interpreted as a vector in the dual lattice $(2L)^*$. We give a characterization of the actual flux values, and of how corresponding tilings behave. We also consider domino tilings of the infinite square lattice; tilings of tori can be seen as a particular case of those. We describe the construction and usage of Kasteleyn matrices in the counting problem, and how they can be applied to count tilings with prescribed flux values. Finally, we study the limit distribution of the number of tilings with a given flux value as a uniform scaling dilates the lattice $L$.
}

\tablesmode{noshow}
\makenoidxglossaries
\newglossaryentry{dual graph}
{
name={dual graph},
sort={dual graph},
description={The graph $G$ obtained from a quadriculated region $R$ by substituting each of its squares by a vertex and joining neighboring vertices by an edge; see page \pageref{def:dualgraph}}
}

\newglossaryentry{matching}
{
name={(perfect) matching},
sort={matching},
description={A perfect matching is a set of edges on a graph $G$ in which each vertex features exactly once. Corresponds to a tiling. See page \pageref{def:perfectmatching}}
}
\newglossaryentry{bipartite}
{
name={bipartite},
sort={bipartite},
description={A graph whose vertices can be separated into two disjoint sets $U, V$ such that every edge joins a vertex from $U$ to $V$. See page \pageref{def:bipartite}}
}

\newglossaryentry{black and white condition}
{
name={black and white (condition)},
sort={black and white},
description={The requirement that a region's dual graph be bipartite. Alternatively, the requirement that every domino on a tiling consist of one square of each color. Implies an equal number of black squares and white squares. See pages \pageref{def:bwcondition1}, \pageref{def:bwcondition2} and \pageref{def:bwcondition3}}
}

\newglossaryentry{flip}
{
name={flip},
sort={flip},
description={A move on a tiling that exchanges two dominoes tiling a $2 \times 2$ square by two dominoes in the only other possible configuration}
}

\newglossaryentry{height function}
{
name={height function},
sort={height function},
description={An integer function that encodes a tiling. For a tiling of a quadriculated planar region $R$, it is defined on the vertices of $R$; see Section \ref{sec:flipplano}. For a tiling of a quadriculated torus $\qquotient{\R^2}{L}$, it is $L$-quasiperiodic and defined on $\Z^2$ (sometimes called its \emph{toroidal} height function); see Section \ref{sec:hfuntor}}
}

\newglossaryentry{Kasteleyn matrix}
{
name={Kasteleyn matrix},
sort={Kasteleyn matrix},
description={A modified adjacency matrix whose determinant counts tilings of a region. For Kasteleyn matrices of planar regions, see Section \ref{sec:kastmatplano}. For an overview of Kasteleyn matrices on the square torus, see page \pageref{def:kasttorussimplified}; for a detailed exposition on more general tori, refer to Chapter \ref{chap:kastmattorus}}
}

\newglossaryentry{Pfaffian}
{
name={Pfaffian},
sort={Pfaffian},
description={The determinant of a skew-symmetric matrix $A$ can be written as the square of a polynomial in $A$'s entries. This polynomial is $A$'s Pfaffian}
}

\newglossaryentry{i}
{
name={$\bi$},
sort={i},
description={The imaginary unit}
}

\newglossaryentry{Dn}
{
name={$D_n$},
sort={dn},
description={A $2n \times 2n$ square fundamental domain for the $2n \times 2n$ square torus $\T_n$. See page \pageref{def:dn}}
}
\newglossaryentry{Tn}
{
name={$\T_n$},
sort={Tn},
description={The $2n \times 2n$ square torus $\T_n$. See pages \pageref{def:tn}}
}

\newglossaryentry{cross-flip}
{
name={cross-flip},
sort={cross-flip},
description={A flip involving a cross-over domino. See page \pageref{def:crossflip}}
}

\newglossaryentry{cross-over domino}
{
name={cross-over domino},
sort={cross-over domino},
description={On a fundamental domain of a torus $\T_L$, a domino that crosses a side that's been identified with another side. See page \pageref{def:crossoverdomino} }
}

\newglossaryentry{flux}
{
name={flux},
sort={flux},
description={A flux of a tiling counts cross-over dominoes with a sign; see page \pageref{def:fluxsimple} for an overview, and Chapter \ref{chap:flux} for a detailed exposition. For a torus $\T_L$, the flux may also be thought as an element of the affine lattice $L^{\#}$; see Section \ref{sec:lsharp}}
}

\newglossaryentry{Laurent polynomial}
{
name={Laurent polynomial},
sort={Laurent polynomial},
description={A Laurent series with finitely many nonzero coefficients}
}

\newglossaryentry{Phi}
{
name={$\Phi$},
sort={zPhi},
description={The mod 4 prescription function on the infinite square lattice $\Z^2$. See page \pageref{def:phiprescrip}}
}

\newglossaryentry{Quasiperiodicity}
{
name={(arithmetic) quasiperiodicity},
sort={quasiperiodicity},
description={A function satisfying $f(u+v) = f(u)+C$ for some $v, C$ and all $u$}
}

\newglossaryentry{L}
{
name={$L$},
sort={Lattice},
description={A lattice. See page \pageref{def:lattice}}
}

\newglossaryentry{Lstar}
{
name={$L^*$},
sort={Latticedual},
description={The dual lattice of $L$, given by $\text{Hom}(L, \Z)$. See page \pageref{def:duallattice} }
}

\newglossaryentry{Lsharp}
{
name={$L^{\#}$},
sort={Latticesharp},
description={The translate of $L^*$ in $(2L)^*$ that contains all flux values of tilings of $\T_L$. See Section \ref{sec:lsharp}}
}

\newglossaryentry{disjointUnion}
{
name={$A \sqcup B$},
sort={zunion},
description={The union of two disjoint sets $A$ and $B$}
}

\newglossaryentry{valid lattice}
{
name={valid lattice},
sort={valid lattice},
description={A lattice whose vectors have integral coordinates that are the same parity. See Section \ref{sec:vallat}}
}

\newglossaryentry{E}
{
name={$\mathscr{E}, \mathscr{O}$},
sort={even},
description={The lattices $2\Z^2$ and $2\Z^2 + (1,1)$ of vectors whose coordinates are respectively both even and both odd}
}

\newglossaryentry{TL}
{
name={$\T_L$},
sort={TL},
description={The torus $\qquotient{\R^2}{L}$, where $L$ is a lattice. See page \pageref{def:tl}}
}

\newglossaryentry{t}
{
name={$t$},
sort={t},
description={A tiling}
}
\newglossaryentry{phi}
{
name={$\varphi$},
sort={zphi},
description={A flux value}
}
\newglossaryentry{edge}
{
name={edge},
sort={edge},
description={An edge on a graph, or an edge on the boundary of a square of a quadriculated region}
}
\newglossaryentry{edge-path}
{
name={edge-path},
sort={edge-path},
description={A sequence of neighboring vertices (either on a graph or on a quadriculated region)}
}
\newglossaryentry{FL}
{
name={$\mathscr{F}(L)$},
sort={FL},
description={The set of all flux values of tilings of $\T_L$. See Section \ref{sec:fluxchar}}
}
\newglossaryentry{v1}
{
name={$\lVert v \rVert_1$},
sort={zvnorm1},
description={For $v=(x,y)$, the $1$-norm $\lvert x \rvert + \lvert y \rvert$}
}
\newglossaryentry{vinfinity}
{
name={$\lVert v \rVert_{\infty}$},
sort={zvnorminfinity},
description={For $v=(x,y)$, the infinity norm $\max\{\lvert x \rvert, \lvert y \rvert$\}}
}
\newglossaryentry{hmax}
{
name={$h_{\max}$},
sort={hmax},
description={The height function that is maximal over height functions on $\Z^2$ that are 0 at the origin. See page \pageref{hmax}}
}
\newglossaryentry{hmin}
{
name={$h_{\min}$},
sort={hmin},
description={The height function that is minimal over height functions on $\Z^2$ that are 0 at the origin. See page \pageref{hminplano}}
}
\newglossaryentry{doublyinfinite}
{
name={doubly-infinite},
sort={doublyinfinite},
description={A staircase edge-path or domino staircase that is infinite in both directions}
}
\newglossaryentry{H0R}
{
name={$\mathbcal{H}_0(R)$},
sort={H0R},
description={The set of height functions on $R$ that are 0 at the origin. See page \pageref{def:h0r}}
}
\newglossaryentry{gammavw}
{
name={$\Gamma (v,w), \Psi (v,w)$},
sort={zgammavw},
description={The set of edge-paths in $\Z^2$ joining $v$ to $w$ and that respectively respect and reverse color-induced edge orientation. See pages \pageref{def:gammavw} and \pageref{def:psivw}}
}
\newglossaryentry{lgamma}
{
name={$l(\gamma)$},
sort={lgamma},
description={The length of an edge-path $\gamma$}
}
\newglossaryentry{edge-profile}
{
name={edge-profile (type)},
sort={edge-profile},
description={The color-induced orientation of edges round a vertex, one of two types. See page \pageref{def:edgeprofile}}
}
\newglossaryentry{ordered sum}
{
name={ordered sum (representation)},
sort={ordered sum},
description={A unique representation of edge-paths in $\Gamma$ or $\Psi$ sets. See \pageref{def:orderedsumrepresentation}}
}
\newglossaryentry{hvphimax}
{
name={$h_{\max}^{v, \varphi}$},
sort={hmaxvphi},
description={The height function that is maximal over height functions on $\Z^2$ that take the value $4\cdot \langle \varphi, v \rangle$ at $v$. See \pageref{def:hvphimax}}
}
\newglossaryentry{hLphimax}
{
name={$h_{\max}^{L,\varphi}$},
sort={hmaxLphi},
description={The height function that is maximal over toroidal height functions of $\T_L$ with flux $\varphi$. See page \pageref{def:hLphimax}}
}
\newglossaryentry{flip-connected}
{
name={flip-connected},
sort={flipcon},
description={A quadriculated region or set of tilings such that any two distinct tilings can be joined by a sequence of flips. For flip-connectedness on simply-connected planar regions, see Corollary \ref{flipconec} in Section \ref{sec:flipplano}. For a discussion of flip-connectedness on tori, refer to Chapter \ref{chap:fliptorus}}
}
\newglossaryentry{flip-isolated}
{
name={flip-isolated},
sort={flipiso},
description={A tiling which admits no flips, or a set of tilings in which no two distinct tilings can be joined by a sequence of flips}
}
\newglossaryentry{staircase}
{
name={staircase},
sort={staircase},
description={For domino staircases, see page \pageref{def:dominostaircase}. For staircase edge-paths, see page \pageref{def:staircaseedgepath}. For types of staircases, see pages \pageref{def:staircaseedgepathtype} and \pageref{def:stairtype}}
}
\newglossaryentry{windmill}
{
name={windmill (tiling)},
sort={windmill},
description={A tiling of the infinite square lattice that admits no flips and consists entirely of infinite domino staircases that are never doubly-infinite. See page \pageref{def:windmill}}
}
\newglossaryentry{brick wall}
{
name={brick wall},
sort={brick wall},
description={A tiling that uses only one type of domino (vertical or horizontal) and consists entirely of doubly-infinite domino staircases. See page \pageref{def:brickwall}}
}
\newglossaryentry{rugged rectangle}
{
name={rugged rectangle},
sort={rugged rectangle},
description={A type of quadriculated region. See page \pageref{def:rugrec}}
}
\newglossaryentry{rugged quadrant}
{
name={(cardinal) rugged quadrant},
sort={rugged quadrant},
description={A type of quadriculated region. See page \pageref{def:rugquad}}
}
\newglossaryentry{vLeL}
{
name={$[e]_L, [v]_L,[S]_L$},
sort={zvLeL},
description={An $L$-equivalence class of a vertex $v$, edge $e$ or doubly-infinite domino staircase $S$. See pages \pageref{def:vbracketl}, \pageref{def:ebracketl} and \pageref{def:Lclassstaircase}}
}
\newglossaryentry{LKasteleynsigning}
{
name={$L$-Kasteleyn signing},
sort={LKasteleynsigning},
description={An assignment of plus and minus signs to $L$-equivalence classes of edges. See page \pageref{def:lkastsign}}
}
\newglossaryentry{LKasteleynsigninguniversal}
{
name={Universal Kasteleyn signing},
sort={universalKasteleynsigning},
description={A Kasteleyn signing that applies to any valid lattice. See page \pageref{def:lkastsignuni}}
}
\newglossaryentry{gz2}
{
name={$G(\Z^2)$},
sort={gz2},
description={The dual graph of the infinite square lattice. See page \pageref{def:gz2}}
}
\newglossaryentry{Lflip}
{
name={$L$-flip (round $v$)},
sort={Lflip},
description={A flip round each vertex in $[v]_L$. See page \pageref{def:lflip}}
}
\newglossaryentry{stairflip}
{
name={stairflip},
sort={stairflip},
description={A move on a tiling that exchanges a doubly-infinite domino staircase by a doubly-infinite domino staircase in the only other possible configuration. See page \pageref{def:stairflip}}
}
\newglossaryentry{Lstairflip}
{
name={$L$-stairflip (on $S$)},
sort={Lstairflip},
description={A stairflip on each doubly-infinite domino staircase in $[S]_L$. See page \pageref{def:Lstairflip}}
}
\newglossaryentry{stairL}
{
name={$\text{Stair}(L)$},
sort={stairL},
description={The set of $L$-equivalence classes of doubly-infinite domino staircases in $\Z^2$. See page \pageref{def:stairL}}
}
\newglossaryentry{stairLk}
{
name={$\text{Stair}(L;k)$},
sort={stairLk},
description={The set of $L$-equivalence classes of type-$k$ doubly-infinite domino staircases in $\Z^2$. See page \pageref{def:stairLk}}
}
\newglossaryentry{stairLkverthor}
{
name={$\text{Stair}(L;k;\text{vert}), \text{Stair}(L;k;\text{hor})$},
sort={stairLkverthor},
description={The sets of $L$-equivalence classes of type-$k$ doubly-infinite domino staircases in $\Z^2$ whose dominoes are all vertical, or all horizontal. See page \pageref{def:stairLkverthor}}
}
\newglossaryentry{xiL}
{
name={$\xi_L$},
sort={zxiL},
description={The $L$-stairflip operator. See page \pageref{def:xiL}}
}
\newglossaryentry{xiLkexclusive}
{
name={$\xi_L$-$k$-exclusive},
sort={zxiLkexclusive},
description={A type of subset of $\text{Stair}(L;k)$. See page \pageref{def:xiLkexclusive}}
}
\newglossaryentry{Qk}
{
name={$Q_k$},
sort={Qk},
description={A component of $\partial Q$. See pages \pageref{def:qk} and \pageref{def:qkprop}}
}
\newglossaryentry{cycleflip}
{
name={cycle, cyle flip},
sort={cycle flip},
description={Simultaneously representing two tilings of $\T_L$ on the same fundamental domain induces a decomposition of the domain into disjoint domino cycles. A cycle flip uses these cycles to go from one tiling to the other. For details, see page \pageref{def:cycflip}}
}
\newglossaryentry{cycle interior}
{
name={interior of a cycle},
sort={interior of a cycle},
description={The union of the interior of domino paths of a closed cycle; see page \pageref{def:intc}}
}
\newglossaryentry{cycle exterior}
{
name={exterior of a cycle},
sort={exterior of a cycle},
description={The intersection of the exterior of domino paths of a closed cycle; see page \pageref{def:extc}}
}
\newglossaryentry{ctot1}
{
name={$C(t_0,t_1)$},
sort={ct0t1},
description={The set of cycles formed by the superposition of tilings $t_0$ and $t_1$. See page \pageref{def:ct0t1}}
}
\newglossaryentry{short}
{
name={short},
sort={short},
description={A vector $v$ of a lattice $L$ such that $s \cdot v \notin L$ for all $s \in [0,1)$}
}
\newglossaryentry{parameter}
{
name={(cycle) parameter},
sort={parameter},
description={A short vector in $L$, unique up to sign, associated to an open cycle. See page \pageref{def:cycparam}}
}
\newglossaryentry{quasicycle}
{
name={($v$-)quasicycle},
sort={quasicycle},
description={A type of ($v$-)quasiperiodic function. See page \pageref{def:quasicyc}}
}
\newglossaryentry{pseudoflux}
{
name={$\phi(\gamma)$},
sort={zphi de gamma},
description={The pseudo-flux of the quasicycle $\gamma$; see page \pageref{def:pseudoflux}}
}
\newglossaryentry{compatible}
{
name={compatible (with $\gamma$)},
sort={compatible},
description={A tiling of a torus is compatible with a quasicycle $\gamma$ if it contains every other domino of $\gamma$; see page \pageref{def:compatible}}
}
\newglossaryentry{argument}
{
name={argument (function)},
sort={argument},
description={A function that can be defined for certain edge-paths. See page \pageref{def:argfunction}}
}
\newglossaryentry{flux connecting}
{
name={flux-connecting (basis)},
sort={flux-connecting},
description={A basis $\{v_0^*,v_1^*\}$ of $L^*$ such that the moves $\pm v_i^*$ connect $\mathscr{F}(L)$. See page \pageref{def:fluxconnecbasis}}
}
\newglossaryentry{Kast}
{
name={$\text{Kast}(e)$},
sort={Kaste},
description={The Kasteleyn weight of the edge $e$. For Kasteleyn weights, refer to the construction of Kasteleyn matrices (Section \ref{sec:kastmatplano} and Chapter \ref{chap:kastmattorus})}
}
\newglossaryentry{DL}
{
name={$D_L$},
sort={DL},
description={A specific choice of rectangular fundamental domain for a lattice $L$. See page \pageref{def:dl}}
}
\newglossaryentry{Dab}
{
name={$D(a,b)$},
sort={Dab},
description={A set of points in a translate of $D_L$. See page \pageref{def:dab}}
}
\newglossaryentry{op0p1}
{
name={$o(p_0p_1)$},
sort={op0p1},
description={The oriented weight of the edge $p_0p_1$. See page \pageref{def:orientedweight}}
}
\newglossaryentry{qquotient}
{
name={$\qquotient{A}{B}$},
sort={zquotient},
description={The quotient of $A$ by $B$}
}
\newglossaryentry{z2modLstar}
{
name={$\left(\qquotient{\Z^2}{L}\right)^*$},
sort={zquotientstar},
description={The group $\text{Hom}\left(\qquotient{\Z^2}{L},\Sp^1\right)$}
}
\newglossaryentry{FLq}
{
name={$\mathcal{F}(L,q)$},
sort={FLq},
description={The space of complex functions on $\big(\Z+\frac12\big)^2$ that are $L$-periodic with parameter $q$. See page \pageref{def:flq}}
}
\newglossaryentry{psi}
{
name={$\psi$},
sort={zpsi},
description={An isomorphism between a space of $L$-periodic functions and $\mathcal{F}(L,q)$. See page \pageref{def:psi}}
}
\newglossaryentry{l0}
{
name={$L_0$},
sort={Latticez},
description={The lattice spanned by $\{(2,0),(1,1)\}$; the lattice that contains every valid lattice. See page \pageref{def:l0}}
}
\newglossaryentry{l0bw}
{
name={$L_b, L_w$},
sort={Latticezbw},
description={The affine lattices $L_0 + \big(\frac12, \frac12 \big)$ and $L_0 + \big(\frac12, -\frac12\big)$, respectively. Equivalently, respectively the set of black vertices and the set of white vertices of $G\left(\Z^2\right)$. See page \pageref{def:l0b}}
}
\newglossaryentry{BLq}
{
name={$\mathcal{B}(L,q)$},
sort={BLq},
description={The space of complex functions on $L_b$ that are $L$-periodic with parameter $q$. See page \pageref{def:blq}}
}
\newglossaryentry{WLq}
{
name={$\mathcal{W}(L,q)$},
sort={WLq},
description={The space of complex functions on $L_w$ that are $L$-periodic with parameter $q$. See page \pageref{def:wlq}}
}
\newglossaryentry{DbDw}
{
name={$D_b,D_w$},
sort={DbDw},
description={Ordered bases of $\mathcal{B}(L,q)$ and $\mathcal{W}(L,q)$, associated to the construction of a Kasteleyn matrix. See page \pageref{def:dbdw}}
}
\newglossaryentry{KKd}
{
name={$K, K_D$},
sort={KKD},
description={A Kasteleyn matrix. The subindex in $K_D$ refers to the matrix constructed from Kasteleyn weights in a fundamental domain $D_L$; see page \pageref{def:kd}}
}
\newglossaryentry{KE}
{
name={$K_E$},
sort={KKE},
description={A diagonal Kasteleyn matrix, obtained from representing $K_D$ in exponential bases. See page \pageref{def:ke}}
}
\newglossaryentry{frestra}
{
name={$\rrestr$},
sort={zfrestra},
description={The restriction of $f$ to $A$}
}
\newglossaryentry{EbEw}
{
name={$E_b,E_w$},
sort={EbEw},
description={Ordered bases of $\mathcal{B}(L,q)$ and $\mathcal{B}(L,q)$, called exponential bases, associated to a diagonalization of $K_D$. See page \pageref{def:ebew}}
}
\newglossaryentry{XAB}
{
name={$X(A,B)$},
sort={XAB},
description={The change of basis from $A$ to $B$. See page \pageref{def:XAB}}
}
\newglossaryentry{rho}
{
name={$\rho$},
sort={zrho},
description={A complex phase related to the determinants of $K_D$ and $K_E$. See page \pageref{def:rho}}
}
\newglossaryentry{rho12}
{
name={$\rho_1,\rho_2$},
sort={zrho12},
description={Complex phases associated to $\rho$. See page \pageref{def:rho12}}
}
\newglossaryentry{even lattice}
{
name={even lattice},
sort={even lattice},
description={A valid lattice that does not contain points with odd coordinates. See page \pageref{def:latticeoddeven}}
}
\newglossaryentry{odd lattice}
{
name={odd lattice},
sort={odd lattice},
description={A valid lattice that contains points with odd coordinates. See page \pageref{def:latticeoddeven}}
}
\begin{document}

\chapter{Introduction}

Tilings of planar regions by dominoes (and also lozenges) can be thought of as perfect matchings of a corresponding graph.
In this sense, the enumeration of matchings was studied as early as 1915 by MacMahon \cite{macmahon1984combinatory}, whose focus was on plane partitions.
Also around the time, chemists and physicists were interested in aromatic hydrocarbons and the behavior of liquids.
Hereafter, I will refer to perfect matchings simply by `matchings'.

Research on dimers in statistical mechanics had a major breakthrough in 1961, when Kasteleyn \cite{kasteleyn1961statistics} (and, independently, Temperley and Fisher \cite{temperley1961dimer}) discovered a technique to count the matchings of a subgraph $G$ of the infinite square lattice.
He proved that this number is equal to the Pfaffian of a certain $0$,$1$-matrix $M$ associated with $G$.
Not much later, Percus \cite{percus1969one} showed that when $G$ is bipartite, one can modify $M$ so as to obtain the number from its determinant (rather than from its Pfaffian).
James Propp \cite{propp1999enumeration} provides an interesting overview of the topic on his `Problems and Progress in Enumeration of Matchings'.

In the early 90s, more advances were made and gave new impetus to research. Conway \cite{conway1990tiling} devised a group-theoretic argument that, in many interesting cases, may be used to show that a given region cannot be tessellated by a given set of tiles.
In a related work, Thurston \cite{thurston1990} introduced the concept of height functions: integer-valued functions that encode a tiling of a region.
With them, he presented a simple algorithm that verifies the domino-tileability of simply-connected planar regions.

In 1992, Aztec diamonds were examined by Elkies, Kuperberg, Larsen and Propp \cite{elkies1992alternating}, who gave four proofs of a very simple formula for the number of domino tilings of these regions.
Later, probability gained importance with the study of random tilings, and Jockush, Propp and Shor \cite{jockusch1998random, cohn1996local} proved the Arctic Circle Theorem.
This framework was further generalized in the early 2000s by Kenyon, Okounkov and Sheffield \cite{kenyonokounkov2006dimers,kenyonokounkov2006planar}, whose work relates random tilings to Harnack curves and describes the variational problem in terms of the complex Burgers equation.

While now much is known about tilings for planar regions, higher dimensions have proven less tractable.
Randall and Yngve \cite{randall2000random} examined analogues of Aztec diamonds in three dimensions for which many of the two-dimensional results can be adapted.
Hammersley \cite{hammersley1966limit} makes asymptotic estimates on the number of brick tilings of a $d$-dimensional box as all dimensions go to infinity.
In his thesis, Milet \cite{milet2015domino} studied certain three-dimensional regions for which he defines an invariant that can be interpreted under knot theory.

This dissertation was motivated by the observation of a certain asymptotic behavior in the statistics of domino tilings of square tori.
We elaborate: consider a quadriculated torus, represented by a square with sides of even length and whose opposite sides are identified.
A \emph{domino} is a $2 \times 1$ rectangle.
Below, we have a tiling of the $4 \times 4$ torus which also happens to be a tiling of the $4 \times 4$ square.
\begin{figure}[H]
		\centering
    \includegraphics[width=0.22\textwidth]{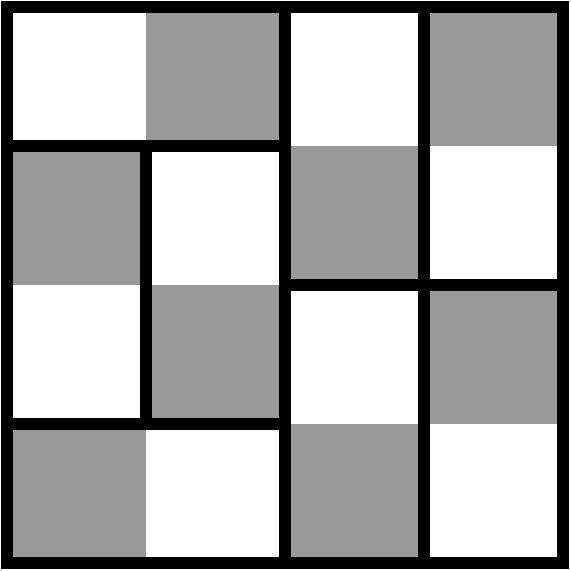}
		\caption*{A domino tiling of the $4 \times 4$ torus}
\end{figure}

Because in the torus opposite sides are identified, we may also consider tilings with dominoes that `cross over' to the opposing side.
\begin{figure}[H]
		\centering
		\includegraphics[width=\textwidth]{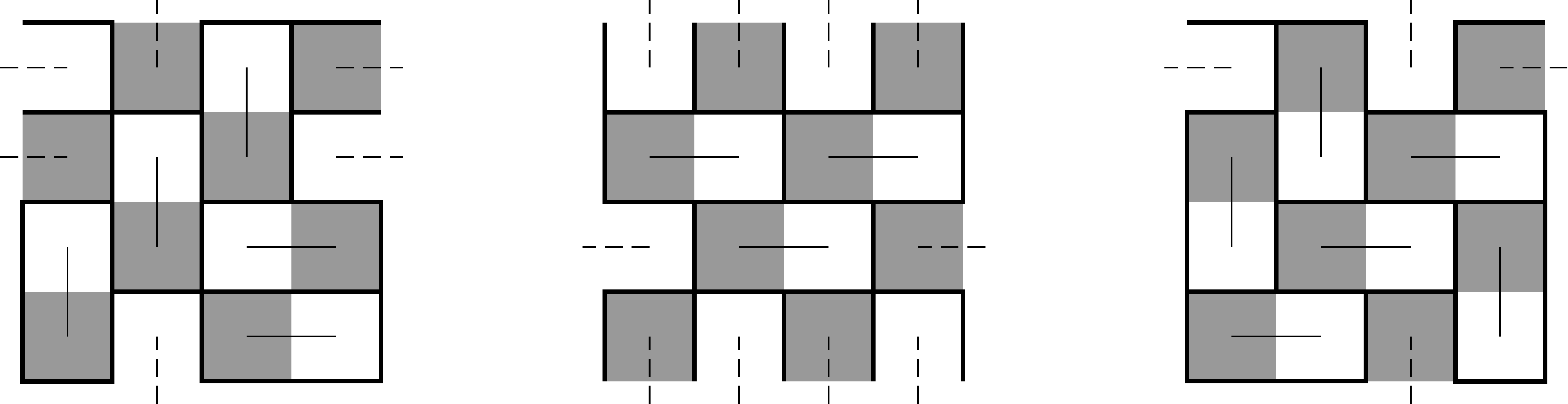}
		\caption*{Tilings of the $4 \times 4$ torus featuring cross-over dominoes}
\end{figure}

The \emph{flux} of a tiling is an algebraic construct that counts these cross-over dominoes, with a sign; one may think of it as a pair of integers.
In the next figure, we assign the positive sign when a white square is to the right of the blue curve or when a black square is above the red curve (and the negative sign otherwise).
Hence, their fluxes are $(0,-1)$, $(1,0)$ and $(1,1)$, where the first integer counts horizontal dominoes crossing the blue curve and the second integer counts vertical dominoes crossing the red curve.

\begin{figure}[H]
		\vspace{0.4cm}
		\centering
		\includegraphics[width=\textwidth]{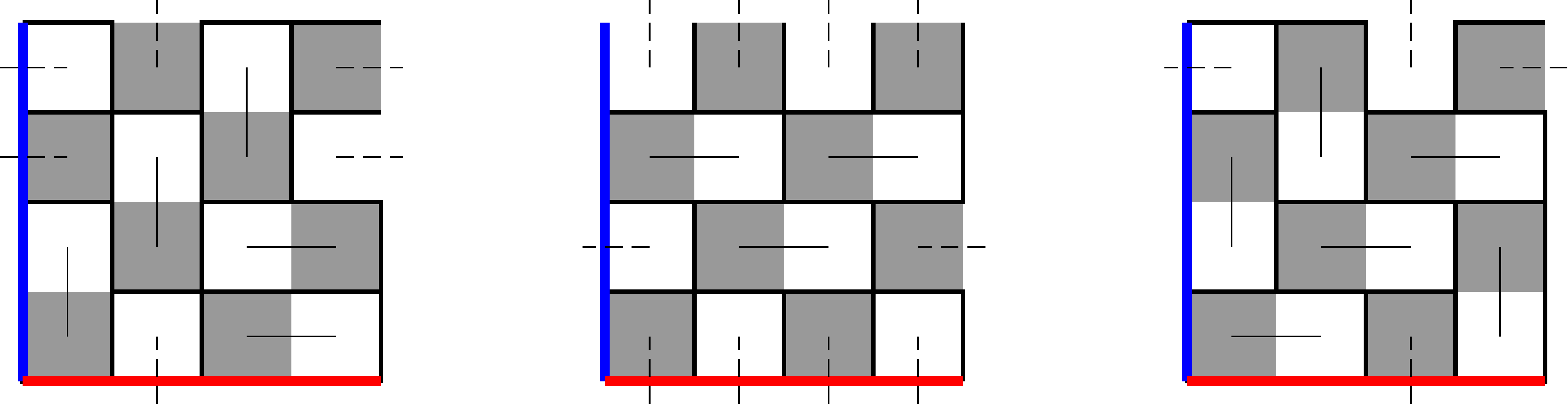}
		\caption*{From left to right, tilings with flux $(0,-1)$, $(1,0)$ and $(1,1)$}
\end{figure}

We may thus count tilings of tori by flux. In the $4 \times 4$ model, we have:
\begin{equation*}
\hspace{-15pt}\boxed{\left.
\begin{array}{c}
\text{Flux}\\
\hphantom{\text{Proportion}}\\
\text{Tilings}
\end{array} \middle|
\begin{array}{c}
(0,0)\\
\hphantom{0.48529}\\
132
\end{array} \middle|
\begin{array}{c}
(0, \pm 1), (\pm 1,0)\\
\hphantom{0.48529}\\
32
\end{array} \middle|
\begin{array}{c}
(\pm 1, \pm 1)\\
\hphantom{0.48529}\\
2
\end{array} \middle|
\begin{array}{c}
(0, \pm 2), (\pm 2,0) \\
\hphantom{0.48529}\\
 1
\end{array}\right.}
\end{equation*}

For a total of 272 tilings. Observe the proportion of total tilings by flux:
\begin{equation*}
\hspace{-15pt}\boxed{\left.
\begin{array}{c}
\text{Flux}\\
\\
\text{Proportion}
\end{array} \middle|
\begin{array}{c}
(0,0)\\
\\
0.48529
\end{array} \middle|
\begin{array}{c}
(0, \pm 1), (\pm 1,0)\\
\\
0.11765
\end{array} \middle|
\begin{array}{c}
(\pm 1, \pm 1)\\
\\
0.00735
\end{array} \middle|
\begin{array}{c}
(0, \pm 2), (\pm 2,0) \\
\\
 0.00368
\end{array}\right.}
\end{equation*}

Now we repeat the process for different square tori:
\begin{equation*}
\hspace{-15pt}\boxed{\left.\begin{array}{c}
\text{Flux}\\
\\
4 \times 4\\
\\
6 \times 6\\
\\
10 \times 10\\
\\
16 \times 16\\
\end{array} \middle|
\begin{array}{c}
(0,0)\\
\\
0.48529\\
\\
0.48989\\
\\
0.49436\\
\\
0.49564
\end{array} \middle|
\begin{array}{c}
(0, \pm 1), (\pm 1,0)\\
\\
0.11765\\
\\
0.11082\\
\\
0.10575\\
\\
0.10411
\end{array} \middle|
\begin{array}{c}
(\pm 1, \pm 1)\\
\\
0.00735\\
\\
0.01416\\
\\
0.01820\\
\\
0.02053
\end{array} \middle|
\begin{array}{c}
(0, \pm 2), (\pm 2,0) \\
\\
0.00368\\
\\
0.00253\\
\\
0.00141\\
\\
0.00109
\end{array}\right.}
\end{equation*}

In each case, tilings with flux $(0,0)$ comprise almost half of all tilings of the $2n \times 2n$ square torus.
For other values of flux in the table it may not be as apparent, but as $n$ increases the proportions stabilize.

\begin{theorem*}As $n$ goes to infinity, the proportions converge to a discrete gaussian distribution.
More specifically, for each $i,j \in \Z$, as $n$ goes to infinity the proportion relative to flux $(i,j)$ tends to
$$\frac{2 \cdot {\Gamma\left(\frac34\right)}^2}{\sqrt{\left(6 + 4 \sqrt{2}\right) \cdot \pi}} \cdot \exp\left(-\frac12 \left(i^2+j^2\right)\right)$$
\end{theorem*}

The formula for the rather curious constant can be derived from theta-function identities; see Yi \cite{yi2004theta}.
For comparison, we provide the previous table, together with the limit value given by the formula above:
\begin{equation*}
\hspace{-15pt}\boxed{\left.\begin{array}{c}
\text{Flux}\\
\\
4 \times 4\\
\\
6 \times 6\\
\\
10 \times 10\\
\\
16 \times 16\\
\\
\text{Limit}
\end{array} \middle|
\begin{array}{c}
(0,0)\\
\\
0.48529\\
\\
0.48989\\
\\
0.49436\\
\\
0.49564\\
\\
0.49629
\end{array} \middle|
\begin{array}{c}
(0, \pm 1), (\pm 1,0)\\
\\
0.11765\\
\\
0.11082\\
\\
0.10575\\
\\
0.10411\\
\\
0.10317
\end{array} \middle|
\begin{array}{c}
(\pm 1, \pm 1)\\
\\
0.00735\\
\\
0.01416\\
\\
0.01820\\
\\
0.02053\\
\\
0.02145
\end{array} \middle|
\begin{array}{c}
(0, \pm 2), (\pm 2,0) \\
\\
0.00368\\
\\
0.00253\\
\\
0.00141\\
\\
0.00109\\
\\
0.00093
\end{array}\right.}
\end{equation*}

We will \textbf{not} prove this theorem in this dissertation.

Nevertheless, it motivated us to study domino tilings of the torus and the underlying combinatorial and algebraic structures involved.
We expect the content of this text lays the groundwork for writing a proof of this theorem in the future.

That said, results of this kind are not new to physicists, and in fact neither to mathematicians.
Boutillier and de Tili\`{e}re \cite{boutillier2009loop} derived explicit formulas for the limit proportions in the honeycomb model of the torus (in this model, a matching may be thought of as a lozenge tiling of the torus).
They interpret matchings as loops (see Cycles and cycle flips, Section \ref{sec:cyc}) and study the asymptotic behavior of corresponding winding numbers.
Although their methods differ from ours, parallels can be drawn.

In Chapter \ref{chap:domplano}, we examine domino tilings of quadriculated planar, simply-connected regions.
We discuss how the study of domino tilings is related to the problem of determining perfect matchings of a graph, and present the idea of black-and-white colorings (so our equivalent graphs are bipartite).
In Section \ref{sec:flipplano} we explore two concepts, as well as their relations.
A \emph{flip} is a move on a tiling that exchanges two dominoes tiling a $2 \times 2$ square by two dominoes in the only other possible configuration.
A \emph{height function} is an integer valued function on the vertices of the squares of a tiling $t$ that encodes $t$.
Later, these concepts will be generalized to the torus case, and many results of this section (like a characterization of height functions, or the flip-connectedness of these regions) admit adaptation.

Section \ref{sec:kastmatplano} details the construction of \emph{Kasteleyn matrices} and explains how their determinants can be used to count domino tilings of a region.
Finally, Section \ref{sec:exret} contains a worked, classical example: the problem of enumerating domino tilings of the rectangle.

In Chapter \ref{chap:tor} begins our study of the torus; we initially consider the square torus $\T_n$ with side length $2n$.
The notion of \emph{flux} is introduced here, and an overview of how Kasteleyn matrices can be adapted is provided.
We also supply a figure with all tilings of the $4 \times 4$ square torus.

Section \ref{sec:hfuntor} extends height functions to this scenario by interpreting $\T_n$ as a quotient $\quotient{\R^2}{L}$, where $L$ is the lattice generated by $\{(2n,0),(0,2n)\}$.
Moreover, we show the flux manifests in the \emph{arithmetic quasiperiodicity} of height functions: they satisfy $h(u+v) = h(u)+k$ for some $v, k$ and all $u$.

In Section \ref{sec:vallat}, we consider more general tori $\T_L$ by allowing other lattices $L$ in the quotient.
These are called \emph{valid lattices}: their vectors have integral coordinates that are the same parity.
This condition is necessary for the resulting graph to be bipartite.

Chapter \ref{chap:flux} further investigates the flux.
In Section \ref{sec:lsharp}, we describe how the flux can be thought of as an element of the \emph{dual lattice} $(2L)^*$.
More precisely, we show there is a translate of $L^*$ in $(2L)^*$ that contains all flux values; we call this affine lattice $L^{\#}$.

Section \ref{sec:fluxchar} provides our first theorem.
For a valid lattice $L$, let $\mathscr{F}(L)$ be the set of all flux values of tilings of $\T_L$; the inner product identification allows us to regard $\mathscr{F}(L)$ as a subset of $\R^2$.
Consider also the (filled) square $Q \subset \R^2$ with vertices $\left(\pm \frac12, 0\right), \left(0,\pm \frac12\right)$.

\begin{teorema}[Characterization of flux values] $\mathscr{F}(L) = L^\# \cap Q$.
\end{teorema}

The proof is given by two separate propositions, each showing one inclusion.
Much of the technical work here relates to the description of \emph{maximal height functions} (given a base value at a base point).

In Chapter \ref{chap:fliptorus}, we discuss how flip-connectedness extends to the torus.
Flips preserve flux values, so of course the situation must be unlike that of Section \ref{sec:flipplano}.
It turns out that for flux values in the interior of $Q$, its tilings are flip-connected, but for flux values in the boundary, none of its tilings admit any flips: they are flip-isolated!

In order to show that, Section \ref{sec:tilplancarac} is devoted to understanding tilings that do not admit flips, and contains our second theorem.
It is a fairly independent section, requiring only that the reader be familiar with (maximal) height functions and flips; see Sections \ref{sec:flipplano} and \ref{sec:fluxchar}.

\begin{teorema}[Characterization of tilings of the infinite square lattice]\label{introteo2}
Let $t$ be a tiling of $\Z^2$. Then exactly one of the following applies:
\begin{enumerate}
\item $t$ admits a flip;
\item $t$ consists entirely of parallel, doubly-infinite domino staircases;
\begin{figure}[H]
		\centering
		\includegraphics[width=0.95\textwidth]{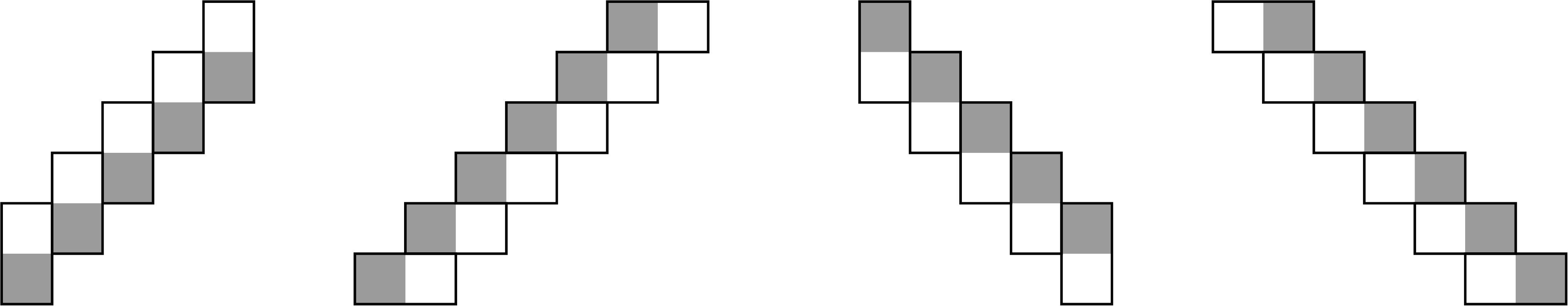}
		\caption*{Examples of domino staircases}
\end{figure}
\item $t$ is a windmill tiling.
\begin{figure}[H]
		\centering
		\includegraphics[width=0.95\textwidth]{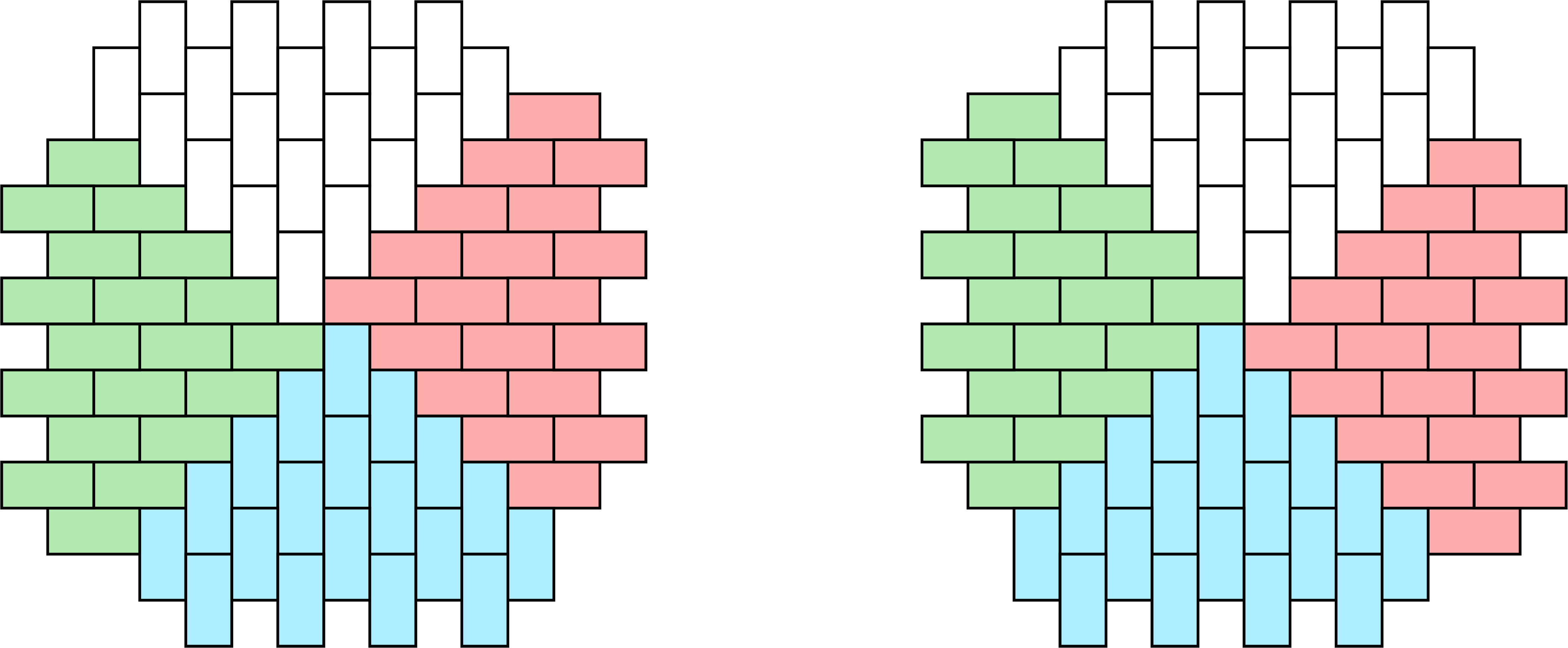}
		\caption*{Windmill tilings of $\Z^2$}
\end{figure}
\end{enumerate}
\end{teorema}

The proof (and theory leading up to it) delves into properties of domino staircases and staircase edge-paths.

Tilings of the torus can be seen as periodic tilings of $\Z^2$.
Section \ref{sec:plantor} combines this observation with Theorem \ref{introteo2} to obtain relations between the shape of a tiling and its flux.
The final result is the above description of flip-connectedness on the torus.
For a survey of flip-connectedness on more general surfaces, see Saldanha, Tomei, Casarin and Romualdo \cite{saldanha1995}.

In Chapter \ref{chap:kastmattorus}, we go into detail about the construction of a Kasteleyn matrix for the torus.
Some of its entries are monomials in $q_0, q_1, q_0^{-1},q_1^{-1}$, or Laurent monomials in $q_0,q_1$, so its determinant is a \emph{Laurent polynomial} $p_K$ in $q_0,q_1$.
We show that each monomial in $p_K$ counts tilings with a flux given by the exponents of $q_0, q_1$.
Later, Chapter \ref{chap:detkast} will consider these variables as complex numbers on the unit circle.
Moreover, Section \ref{sec:strucbound} examines the structure of tilings with flux in the boundary of $Q$, primarily through a move called \emph{stairflip}, that exchanges a doubly-infinite domino staircase by the only other one.

Chapter \ref{chap:signflux} elaborates on how the signs of monomials in $p_K$ are assigned.
The main tool here are \emph{cycles} and \emph{cycle flips}.
Cycles are obtained by representing two tilings simultaneously, and cycle flips use them to go from one tiling to the other.

\begin{figure}[H]
		\centering
		\includegraphics[width=0.75\textwidth]{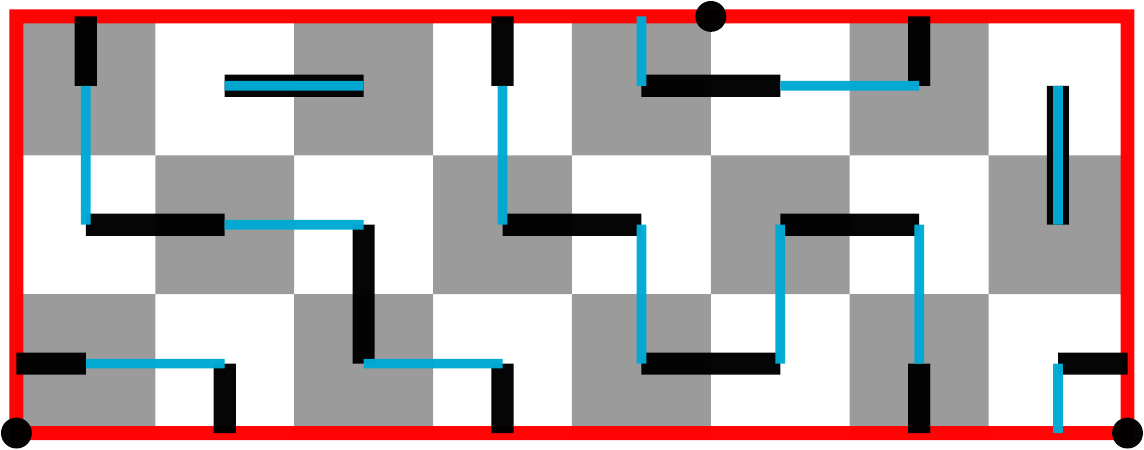}
		\caption*{Cycles: one tiling has black dominoes, the other has blue dominoes}
\end{figure}

In the end of Section \ref{sec:cycfluxsgn}, we exhibit an odd-one-out pattern for signs over $\mathscr{F}(L)$, and use it to show that the total number of tilings can be given as a linear combination of $p_K(\pm1,\pm1)$ where each coefficient is either $\frac12$ or $-\frac12$.

\begin{figure}[ht]
		\centering
		\def\svgwidth{0.9\columnwidth}
    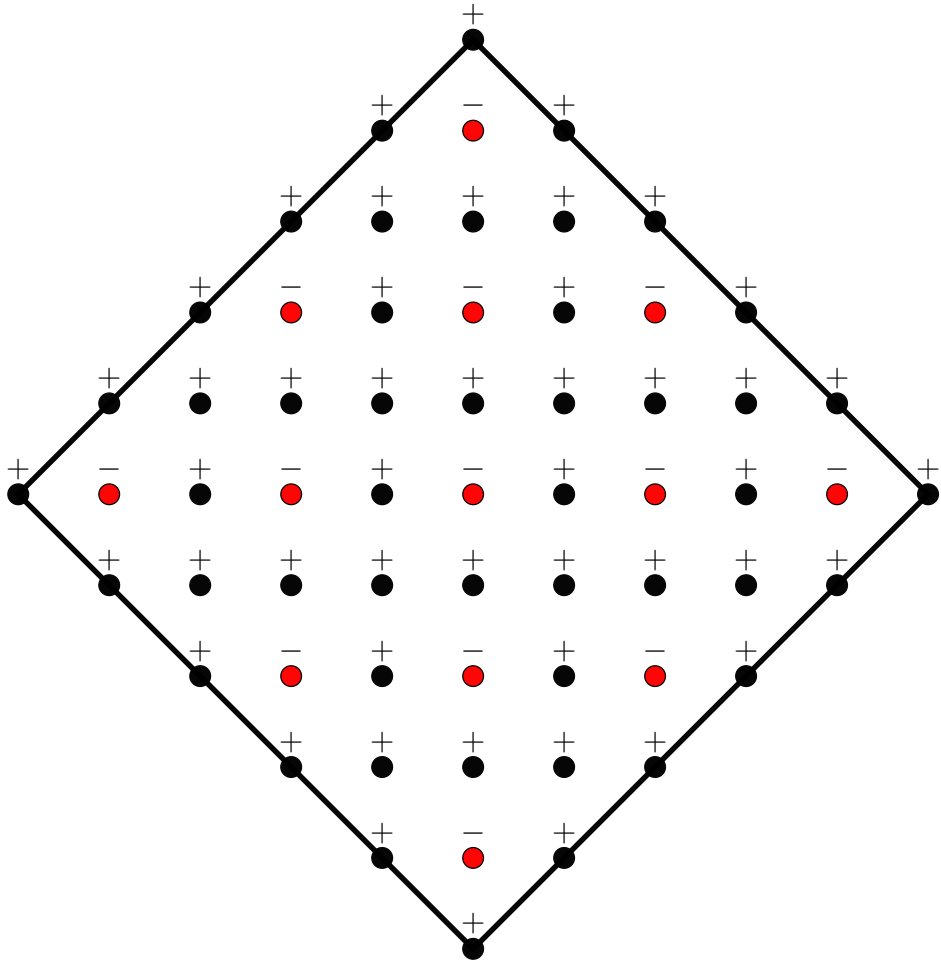
		\caption*{The odd-one-out sign pattern}
\end{figure}

Chapter \ref{chap:detkast} revisits techniques used in Section \ref{sec:exret} and refines them for the calculation of Kasteleyn determinants of the torus.
In Section \ref{sec:casom}, we examine the case of $M = KK^* \oplus K^*K$, for which we can compute all eigenvalues.
Section \ref{sec:espfunc} interprets $K, K^*$ as linear maps on spaces of $L$-quasiperiodic functions, allowing us to exhibit bases for which they are diagonal.
Studying the change of basis, we are able to relate the determinant of the original matrix to that of its diagonal version.

Finally, Section \ref{sec:contas} makes explicit calculations on these determinants and investigates the effects of scaling $L$ uniformly.
This leads to our third and last theorem, which relates the Laurent polynomials (from Kasteleyn determinants) for $L$ and $nL$ by a simple product formula.

Let $p_{[L,E]}: \R^2 \longrightarrow \C$ be defined by $p_{[L,E]}(u_0,u_1) = \det\big(K_E(q_0,q_1)\big)$, where $K_E$ is the diagonal Kasteleyn matrix for $L$ and $q_m = \exp(2 \pi \bi \cdot u_m)$.

\begin{teorema}For any positive integer $n$ and reals $u_0, u_1$ $$p_{[nL,E]}(n\cdot u_0, n\cdot u_1) = \prod\limits_{0 \leq i,j < n} p_{[L,E]}\left(u_0+\frac{i}{n},u_1 - \frac{j}{n}\right).$$
\end{teorema}

Intuitively, Theorem \ref{pkn} says $\det\big({K[nL]}_E(q_0,q_1)\big)$ can be obtained from determinants of ${K[L]}_E$ by considering all $n$-th roots of $q_0$ and of $q_1$.
Product formulas of this kind have been encountered by Saldanha and Tomei \cite{saldanha2003tilings} in their study of quadriculated annuli.
\chapter{Definitions and Notation}
\label{chap:notation}

This will be a short chapter detailing definitions and conventions used in the dissertation.

The imaginary unit will be denoted by the boldface $\bi$.

A \emph{lattice}\label{def:lattice} is a subgroup of $\R^2$ that is isomorphic to $\Z^2$ and spans $\R^2$ (as a real vector space).
An equivalent description is that a lattice $L$ is the (additive) group of all integer linear combinations of a basis $\beta$ of $\R^2$; in this case, we say $L$ is generated by $\beta$.
Notice different bases may generate the same lattice.

The \emph{dual lattice}\label{def:duallattice} of $L$ is $L^* = \text{Hom}(L, \Z)$, the set of homomorphisms from $L$ to $\Z$.
Observe that, under addition, $L^*$ is a group.
Moreover, we may identify an element $f \in L^*$ with a unique $\tilde{f} \in \R^2$ via $f(v) = \langle \tilde{f}, v \rangle$ (for all $v \in L$).
This allows us to see $L^*$ as an additive subgroup of $\R^2$, so that $L^*$ is itself a lattice.
Under this representation, it is easy to see that ${(L^*)}^* = L$. We will generally not make a distinction between $f$ and $\tilde{f}$.

Given a basis $\beta = \{v_0,v_1\}$ of $L$, its \emph{dual basis} is $\beta^* = \{{v_0}^*,{v_1}^*\}$, where $\langle {v_i}^*,v_j \rangle = \delta_{ij}$ ($0 \leq i, j \leq 1$).
Geometrically, this means ${v_i}^*$ is perpendicular to $v_j$ and its length is determined by the equality $\langle {v_i}^*,v_i\rangle = 1$.
It is a straightforward exercise to check that $\beta^*$ generates $L^*$, and that ${(\beta^*)}^* = \beta$. We can also make explicit calculations; let $v_0 = (a,b)$ and $v_1 = (c,d)$.
Then
\begin{alignat*}{1}
{v_0}^* &= \frac{1}{ad-bc}\cdot \left( d , -c \right)\\
{v_1}^* &= \frac{1}{ad-bc}\cdot \left( -b, a \right)
\end{alignat*}

Notice that because $\{v_0,v_1\}$ is a basis of $\R^2$, $ad-bc$ is always nonzero, so the dual basis is well-defined.

A \emph{fundamental domain} for a lattice $L$ is a set $D \subset \R^2$ such that, for all $v \in \R^2$, the affine lattice $L+v$ intersects $D$ exactly once.
Another way to think of this is as follows: $L$ acts on $\R^2$ by translation, so the orbit of any $v \in \R^2$ under $L$ (that is, the set of images of $v$ under $L$) is the affine lattice $L+v$. Hence, $D$ contains exactly one point from each orbit: it is a visual realization of the representatives of each orbit.
It is easily seen that $\R^2$ is partitioned by the sets $D+v, v \in L$.

For a lattice generated by $\{v_0,v_1\}$, the fundamental domain is usually the parallelogram $\{s\cdot v_0 + t\cdot v_1 \text{ }|\text{ } s,t \in [0,1)\}$, but we will generally prefer other kinds of fundamental domain (discussed in Section \ref{sec:vallat}).

Consider the infinite square lattice $\Z^2$.
A \emph{quadriculated region} is a union of (closed, filled) unit squares with vertices in $\Z^2$.
We say two squares are adjacent if they share an edge.
A \emph{domino} is a union of two adjacent unit squares, that is, a $2 \times 1$ rectangle with vertices in $\Z^2$.
A \emph{(domino) tiling} of a quadriculated region is a collection $t$ of dominoes on $R$ with pairwise disjoint interiors and such that every unit square of $R$ belongs to a domino in $t$.

A \emph{torus} is a quotient $\quotient{\R^2}{L}$; we may represent it by a fundamental domain of $L$ whose boundary has appropriate identifications.
If the fundamental domain is chosen to be a quadriculated region, we say the torus is a \emph{quadriculated torus}.
A tiling of a quadriculated torus is much like that of its fundamental domain, except dominoes account for boundary identifications.
Alternatively, a tiling of a quadriculated torus is an $L$-periodic tiling of the infinite square lattice.
\chapter{Domino tilings on the plane}
\label{chap:domplano}

Let $R$ be a finite, simply-connected, quadriculated planar region.
A \textit{domino} is a 2$\times$1 rectangle made of two unit squares.
Is it possible to tile $R$ entirely using only domino pieces? In how many ways can this be done?

For instance, if $R$ has an odd number of squares, then there is no domino tiling of $R$.
If $R$ is the 2$\times$3 rectangle below...
\begin{figure}[ht]
		\centering
    \includegraphics[width=0.245\textwidth]{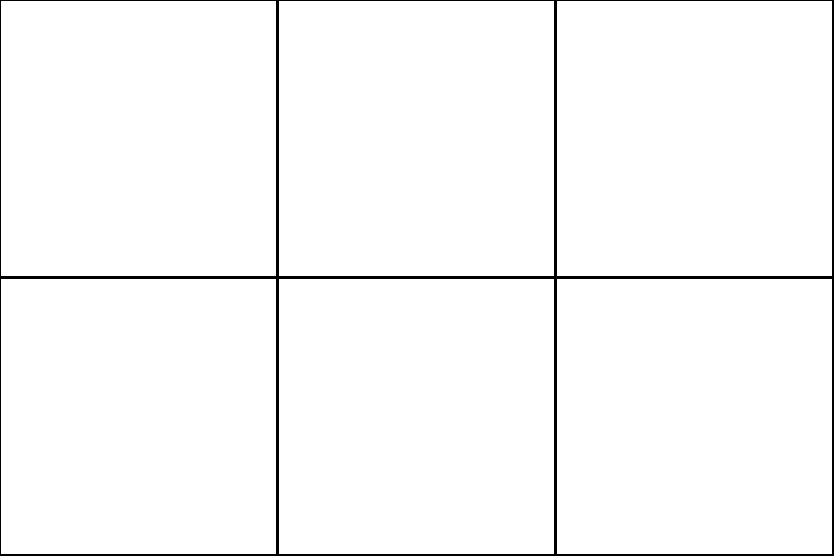}
		\caption{A $2\times 3$ rectangle.}
\end{figure}

\noindent ...then there are exactly three distinct domino tilings of $R$:
\begin{figure}[ht]
		\centering
    \includegraphics[width=0.9\textwidth]{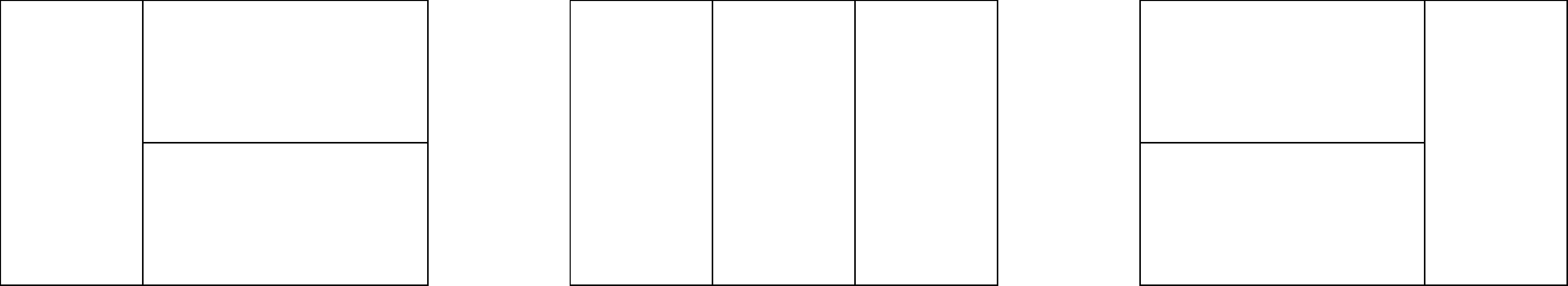}
		\caption{Domino tilings of the $2\times 3$ rectangle.}
\end{figure}

The first observation is this problem can be converted to a dual problem on graph theory.
This conversion associates to the region $R$ a graph $G$ ($R$'s \textit{dual graph})\label{def:dualgraph} obtained by substituting each square of $R$ by a vertex and joining neighboring vertices by an edge (horizontally and vertically, but not diagonally).
On a domino tiling level, each domino corresponds to an edge on $G$: the edge joining the two vertices whose associated squares that are tiled by that domino.

For instance, the region $R$ and the graph $G$ in Figure \ref{fig:Rgraph} are dual. Likewise, the domino tiling of $R$ and the $G$ subgraph in Figure \ref{fig:RSubgraph} are dual.
\begin{figure}[ht]
		\centering
		\includegraphics[width=0.625\textwidth]{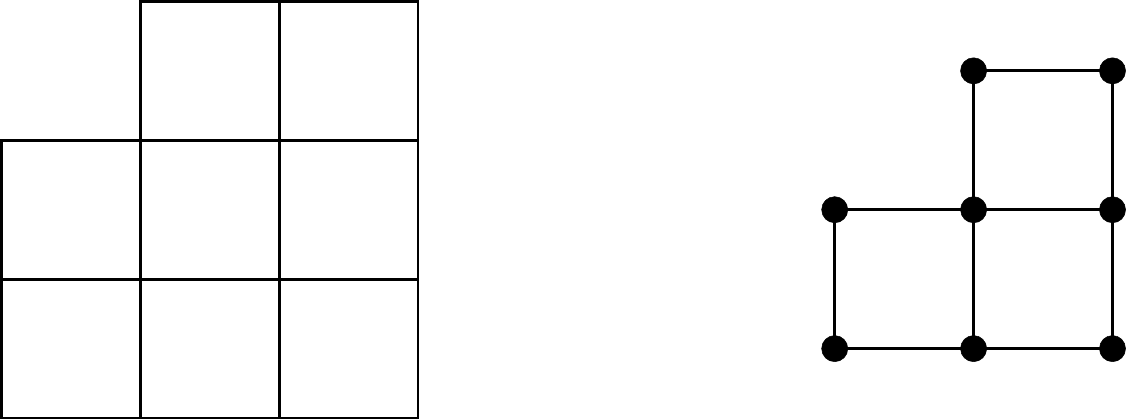}
		\caption{A quadriculated region $R$ and its dual graph $G$.}
		\label{fig:Rgraph}
\end{figure}

\begin{figure}[ht]
		\centering
		\includegraphics[width=0.625\textwidth]{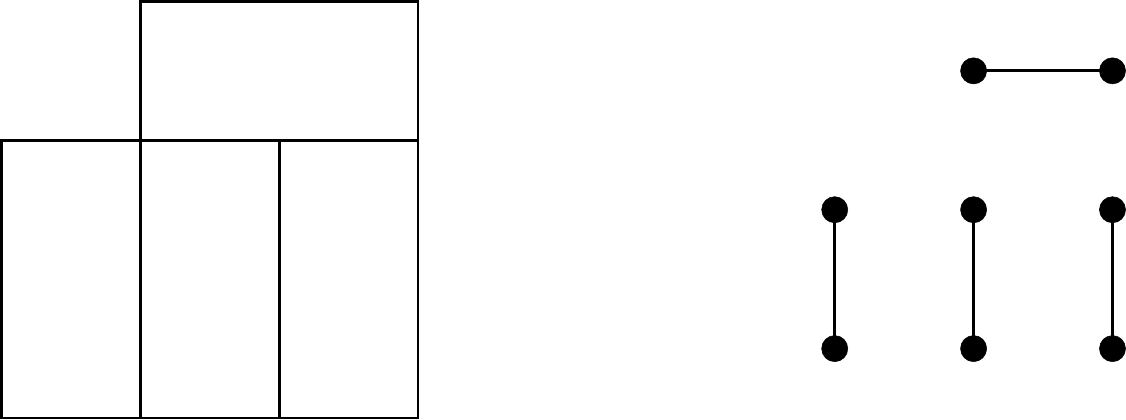}
		\caption{A domino tiling of $R$ and the corresponding $G$ subgraph.}
		\label{fig:RSubgraph}
\end{figure}

In this context, a question on domino tilings of $R$ can be translated naturally into a question on the matchings of $G$.
A \textit{matching} $M$ of a graph $G$ is a set of edges on $G$ with no common vertex.
If two vertices on $G$ are joined by an edge of $M$, we say $M$ \textit{matches} those vertices.
A \textit{perfect matching}\label{def:perfectmatching} of a graph $G$ is a matching of $G$ that matches all vertices on $G$.

Now we may translate the opening questions: `Is it possible to tile $R$ by dominoes?' becomes `Is there a perfect matching of the dual graph $G$?'; and `In how many ways can this be done?' becomes `How many perfect matchings does the dual graph $G$ have?'.
Henceforth, unless explicitly stated, we shall use \textit{matchings} when referring to perfect matchings. Non-perfect matchings do not interest us in this study.

The second observation is that these constructions lend themselves naturally to the concept of bipartite graphs.
A graph $G$ is \textit{bipartite}\label{def:bipartite} if its vertices can be separated into two disjoint sets $U$ and $V$ so that every edge on $G$ joins a vertex in $U$ to a vertex in $V$.
In this case, the sets $U$ and $V$ are called a \textit{bipartition} of $G$.
With this in mind, we may return to our initial problem and consider a prescribed `bipartition' on $R$: we assign the label `black' to an initial square, then assign the label `white' to its neighbors, and so on in alternating fashion.
Naturally, the vertices of the dual graph $G$ inherit the labels.
\begin{figure}[H]
		\centering
		\includegraphics[width=0.625\textwidth]{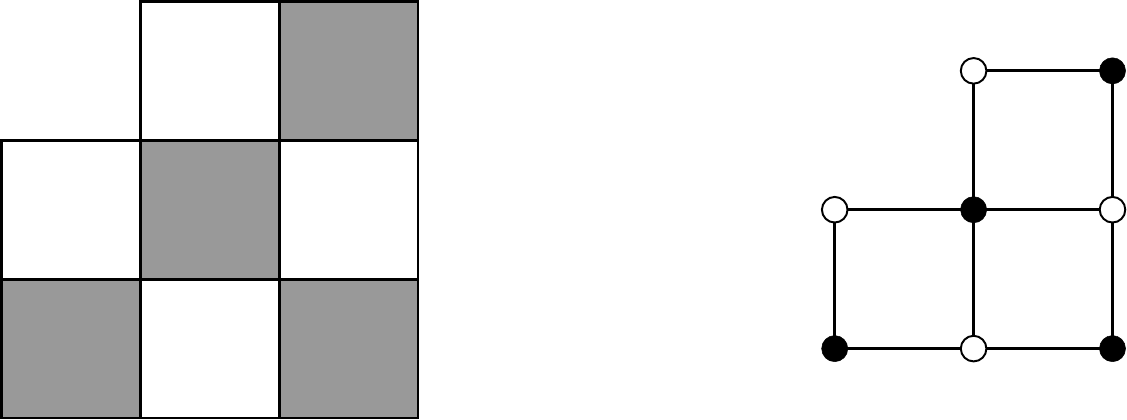}
		\caption{A quadriculated region and its dual graph colored as above.}
\end{figure}

At this point, notice every domino in a tiling of $R$ must be made of a single black square and a single white square.
Hence, a \textbf{necessary condition}\label{def:bwcondition1} for $R$ to admit a domino tiling is that the number of black squares and the number of white squares be equal.
Observe, however, that it is not sufficient.
\begin{figure}[H]
		\centering
		\includegraphics[width=0.35\textwidth]{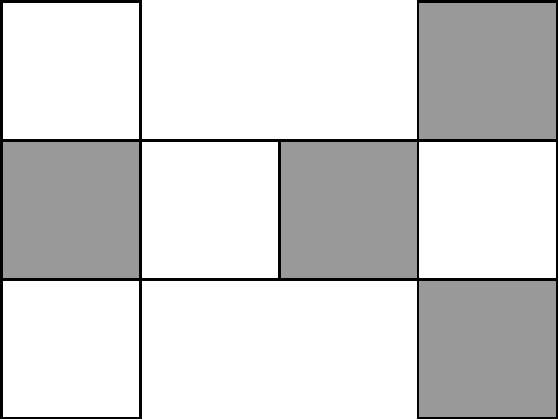}
		\caption{A region that satisfies the black-and-white condition but admits no domino tiling.}
\end{figure}

We point out that we will generally think of $G$ as embedded on the region $R$, with each vertex lying on the center of its corresponding square and each edge a straight line.

\section{Flips and height functions}\label{sec:flipplano}

We now introduce the concept of \textit{flips}.
To that end, notice a 2$\times$2 square can be tiled by two dominoes in exactly two ways: by using both dominoes vertically, or by using both dominoes horizontally.

Consider two adjacent parallel dominoes forming a 2$\times$2 square.
A \textit{flip} of these two dominoes consists in substituting the domino tiling of the square they form by the only other domino tiling of that same square.
Naturally, the concept of flip is transferred to the graph treatment of the problem.
\begin{figure}[H]
		\centering
		\includegraphics[width=0.635\textwidth]{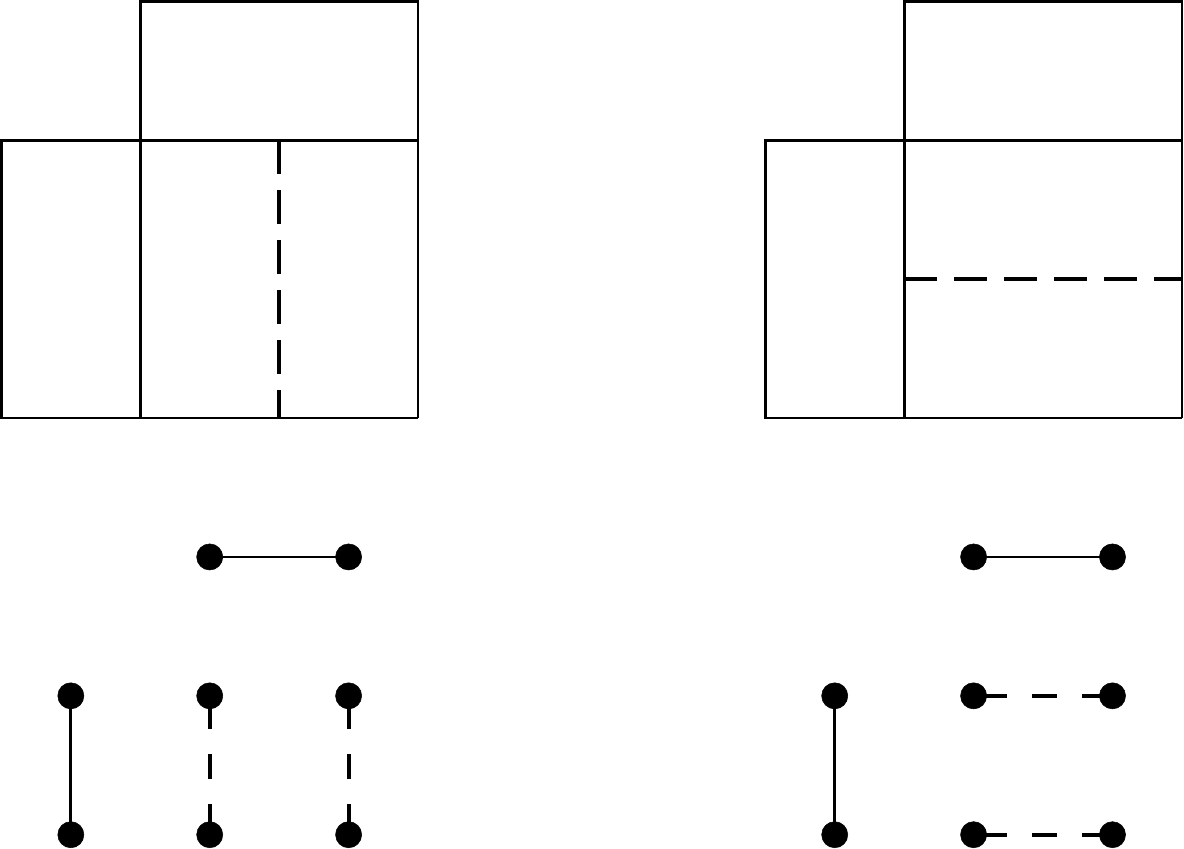}
		\caption{A flip on a region's tiling and on its corresponding graph.}
\end{figure}

Of course, given a domino tiling of a planar region $R$, the execution of a flip takes us to a new domino tiling of $R$.
Following this train of thought, a natural question might be whether two given domino tilings of $R$ can be joined by a sequence of flips.
To answer this question, we will investigate the \textbf{height function} $h$ of a domino tiling $t$ of $R$.

We highlight the distinction between an edge on a graph $G$ and an edge on a quadriculated region $R$: the latter refers to an edge on the boundary of a square on $R$.
Similarly, an edge on a domino tiling $t$ of $R$ is an edge on $R$ (of a square, not of a domino) that does not cross a domino (it has not been `erased' to produce said domino).

We choose once and for all the clockwise orientation for black squares; the other orientation is assigned to white squares.
This choice induces an orientation on each edge on $R$.
Notice it is consistent: along an edge where two squares meet, each square will have a different orientation and thus the orientations induced on the edge will agree.

Now, choose a base vertex $v$ on $R$ and assign an integer value to it; we will always choose a base vertex in the boundary $\partial R$ of the region $R$ and we will always assign the value $0$ to it.
This is the value $h$ takes on $v$.
We now propagate that value across all vertices of $R$ as follows.
For each vertex $w$ joined to $v$ by an edge on $t$, that edge may point from $v$ to $w$ or from $w$ to $v$, depending on its orientation as defined above.
In the first case, $w$ is assigned the integer value $h(v) + 1$; otherwise, it is assigned the integer value $h(v) - 1$.

By connectivity, this process defines the height function on each vertex of $R$, but it may not be clear whether or not the definition is consistent. 
It's easy to verify consistency on a single domino, as the image below shows.
\begin{figure}[H]
    \centering
    \def\svgwidth{0.85\columnwidth}
    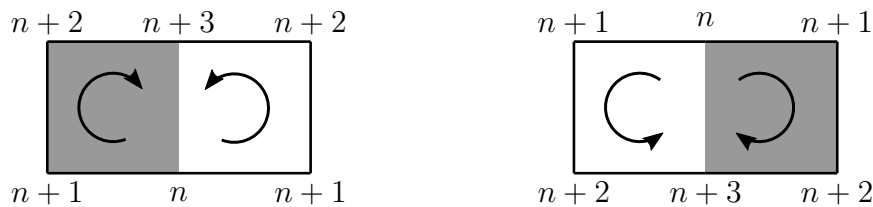
		\caption{Height consistency for horizontal dominoes.}
\end{figure}

Consistency for a general simply-connected planar region $R$ can be proved as follows: starting from a vertex on which $h$ is well-defined (for instance, the base vertex $v$), suppose we wish to check consistency on another vertex, say $w$.
Consider then two different edge-paths $\gamma_0$ and $\gamma_1$ on $t$ joining those vertices; these paths can be seen as the boundary of a region tiled by dominos.
The area of that region is thus well-defined.
Now, incrementally deform $\gamma_0$ onto $\gamma_1$, with each step producing a new region with less area than the previous one through the removal of a domino.
Here, consistency on a single domino ensures each step is consistent with the previous one.
Finally, the simply-connectedness of $R$ guarantees this process can fully deform $\gamma_0$ onto $\gamma_1$.

We provide a simple example of this process below.
\begin{figure}[H]
		\centering
		\includegraphics[width=0.875\textwidth]{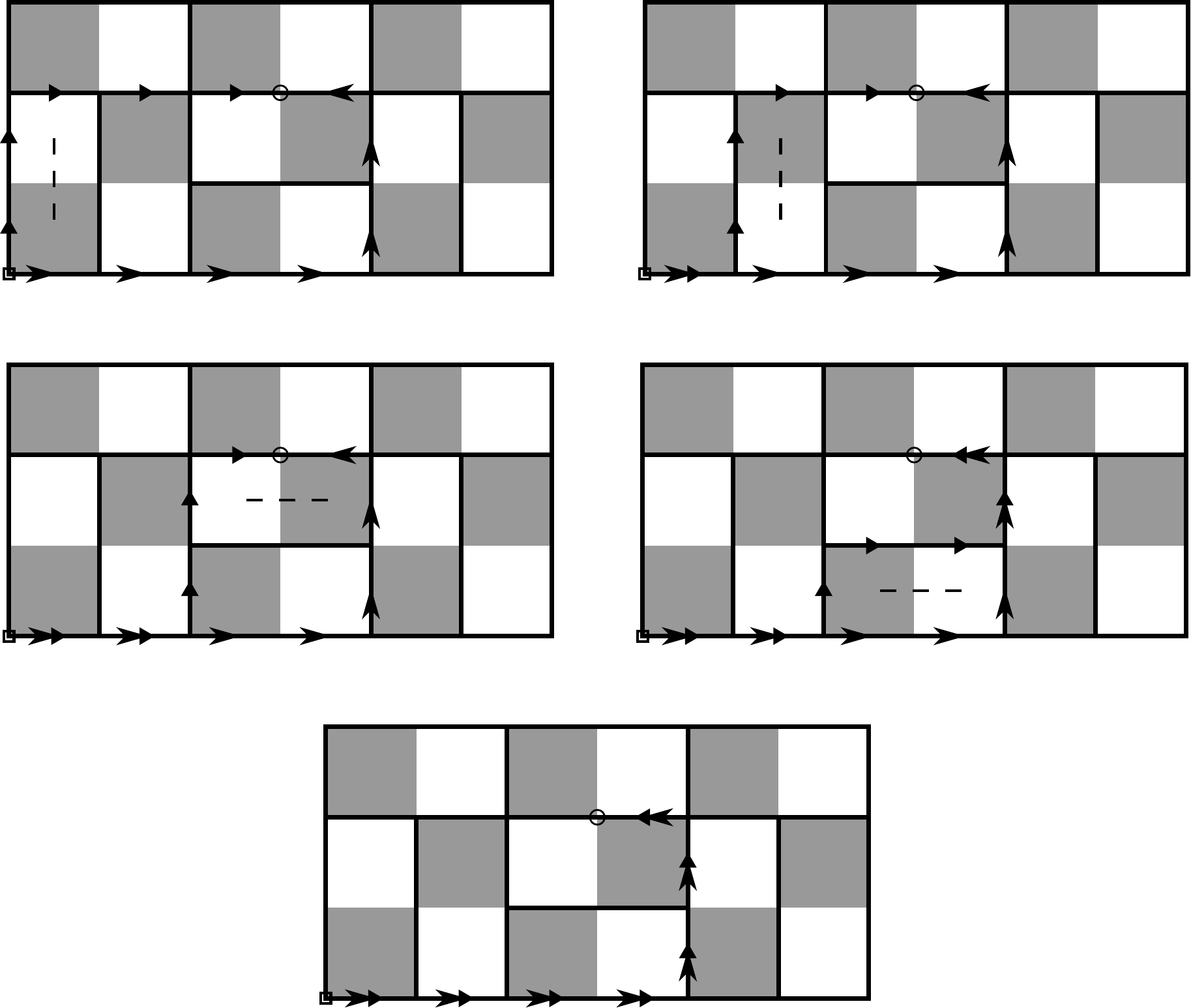}
		\caption{Deforming one edge-path into another.}
\end{figure}

With these conventions, given a domino tiling $t$ and a base vertex $v$ of a black-and-white quadriculated region $R$, the height function $h$ of $t$ is well-defined.
An example of height function $h$ can be seen in the following image; the marked vertex is the base vertex.
\begin{figure}[H]
    \centering
    \def\svgwidth{0.725\columnwidth}
    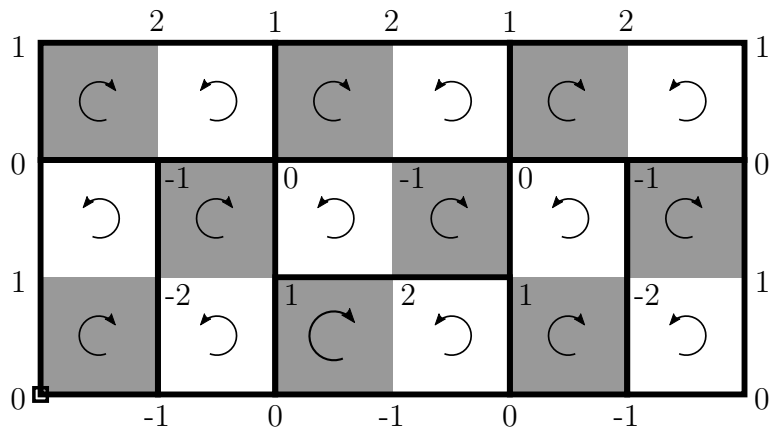
		\caption{An example of height function.}
\end{figure}

This provides a constructive definition of height functions, but we highlight now some of their properties.

\begin{prop}\label{hprescrip}
Let $R$ be a black-and-white quadriculated region.
Fix a base vertex $v \in \partial R$ (independent of choice of tiling).
Then (1) the values a height function takes on $\partial R$ and (2) the mod 4 values a height function takes on all of $R$ are all independent of choice of tiling. 
\end{prop}
\begin{proof}
Remember that, regardless of the choice of tiling $t$, an edge on $\partial R$ is an edge on $t$.
Since we have already proved consistency, (1) is automatic.

For (2), let $u$ and $w$ be vertices on $R$ joined by an edge $e$.
Observe that the orientation of $e$ depends only on the region $R$ and not on choice of tiling; assume then that $e$ is oriented from $u$ to $w$.
The constructive definition implies a change in height function along $e$ occurs in one of the following ways:
\begin{itemize}
	\item If $e$ is on the tiling $t$, then $h(w)=h(u)+1$.
	\item If $e$ is not on the tiling $t$, then $h(w)=h(u)-3$.
\end{itemize}

Notice that in both cases $h(w)$ has the same mod 4 value.
The same occurs when $e$ is oriented from $w$ to $u$. 
By connectivity, we are done.
\end{proof}

Proposition \ref{hprescrip} allows us to fully characterize height functions of tilings of a region $R$.

\begin{prop}[Characterization of height functions]\label{hcarac}
Let $R$ be a black-and-white quadriculated region. Fix a base vertex $v \in \partial R$.
Then an integer function $h$ on the vertices of $R$ is a height function (of a tiling of $R$) if and only if $h$ satisfies the following properties:
\begin{enumerate}
	\item $h$ has the prescribed values on $\partial R$.
	\item $h$ has the prescribed mod 4 values on all of $R$.
	\item $h$ changes by at most 3 along an edge on $R$.
\end{enumerate}
\end{prop}
\begin{proof}
Proposition \ref{hprescrip} and its proof guarantee that any height function satisfies the listed properties.
We will now show that if an integer function on the vertices of $R$ satisfies those properties, it is the height function of a tiling on $R$.
To that end, we will construct a tiling $t$ that realizes one such function $h$.

On $R$, whenever two vertices joined by an edge have $h$-values that differ by 3, erase that edge (thus producing a domino).
We claim the result is a domino tiling $t$ on $R$.
Indeed, properties (2) and (3) ensure each square on $R$ will have exactly one of its sides erased.
Furthermore, by (1) that side will never occur on $\partial R$.
It's easy to see this yields a domino tiling of $R$; furthermore, by construction this tiling's height function is $h$.
\end{proof}

From now on, for any black-and-white quadriculated region $R$, assume the base vertex $v$ is fixed independently of choice of tiling.

Another interesting and perhaps less obvious property of height functions is that the minimum of two height functions is itself a height function.

\begin{prop}\label{hmin}
Let $R$ be a black-and-white quadriculated region and $t_1$, $t_2$ be two domino tilings of $R$ with corresponding height functions $h_1$, $h_2$.
Then $h_m = min \{h_1,h_2\}$ is a height function on $R$.
\end{prop}
\begin{proof}
Indeed, by Proposition \ref{hcarac}, it suffices to show that $h_m$ changes by at most $3$ along an edge on $R$.
This is trivially verified on vertices $v$ and $w$ joined by an edge whenever $h_m = h_1$ or $h_m = h_2$ on both $v$ and $w$.
Suppose this is not the case; furthermore, suppose without loss of generality $h_m(v) = h_1(v)$, $h_m(w) = h_2(w)$ and that the edge joining them points from $v$ to $w$.

The edge's orientation implies $h_i(w)= h_i(v)+1$ if the edge is on $t_i$ and $h_i(w)= h_i(v)- 3$ otherwise ($i=1,2$).
Since $h_1(v) < h_2(v)$, the only possibility that realizes $h_2(w)<h_1(w)$ is the edge being on $t_1$ and not on $t_2$, so that $h_1(w)=h_1(v)+1$ and $h_2(w)=h_2(v) - 3$.
Now, because $h_1(v) - h_2(v) < 0$ and mod 4 values are prescribed, the difference must be $-4k$ for some positive integer $k$, so that $h_1(v) = h_2(v) - 4k$.

Finally, $h_2(w)<h_1(w)$ can now be rewritten as $h_2(v)-3 < (h_2(v)-4k)+1$, or simply $-4<-4k$.
This is a contradiction, implying only the cases when $h_m = h_1$ or $h_m = h_2$ on both $v$ and $w$ can occur.
\end{proof}

\begin{corolario}[Minimal height function]\label{hminimal}
Let $R$ be a black-and-white quadriculated region.
If $R$ can be tiled by dominoes, then there is a minimal height function.
\end{corolario}

Along a 2$\times$2 square tiled by dominoes, it's easy to verify that height function values are distributed so that the center vertex is a local maximum or minimum.
Furthermore, applying a flip changes a local maximum vertex to a local minimum vertex, and vice-versa, leaving other values unchanged.
Figure \ref{fig:heightflip} illustrates this phenomenon.
\begin{figure}[ht]
    \centering
    \def\svgwidth{0.85\columnwidth}
    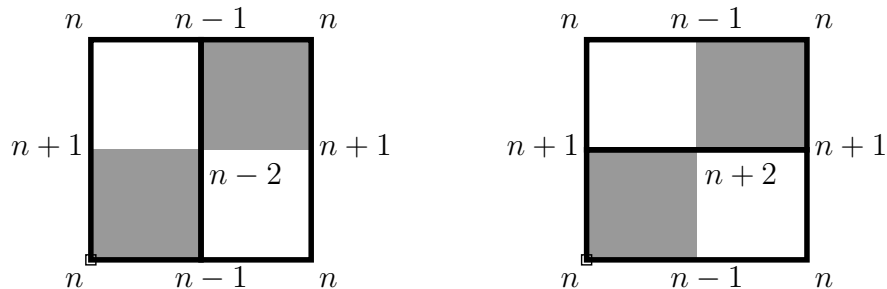
		\caption{The effect of a flip on the height function.}
		\label{fig:heightflip}
\end{figure}

Together with Corollary \ref{hminimal}, an application of this technique provides the following result.

\begin{prop}\label{hredux}
Let $R$ be a black-and-white quadriculated region with minimal height function $h_m$.
Let $h \neq h_m$ be a height function associated to the domino tiling $t$ of $R$.
Then there is a flip on $t$ that produces a height function $\tilde{h} \leq h$ with $\tilde{h}<h$ on one vertex of $R$.
\end{prop}
\begin{proof}
Consider the difference $h - h_m$.
By Proposition \ref{hcarac}, it is 0 along the boundary and takes nonnegative values on $4\Z$.
Let $V$ be the set of vertices of $R$ on which $h-h_m$ is maximum, and choose a vertex $v \in V$ that maximizes $h$.
Notice by hypothesis $V$ is non-empty, and does not intersect $\partial R$.
We assert that $v$ is a local maximum of $h$.

Suppose $v$ were not a local maximum of $h$, that is, suppose there were a vertex $w$ joined to $v$ by an edge $e$ so that $h(w) > h(v)$.
There are two cases:
\begin{enumerate}
	\item $e$ is on $t$ and points from $v$ to $w$, so that $h(w)=h(v)+1$.
	\item $e$ is not on $t$ and points from $w$ to $v$, so that $h(w)=h(v)+3$.
\end{enumerate}

Remember edge orientation does not depend on choice of tiling (and thus does not depend on the height function considered).

In case (1), $h_m(w)=h_m(v)+1$ if $e$ is on the associated minimal tiling $t_m$, and $h_m(w)=h_m(v)-3$ otherwise.
Neither can occur: the first contradicts $v$ maximizing $h$ (since $h(w)>h(v)$), and the latter contradicts $v$ maximizing $h-h_m$ (since $h(w)-h_m(w)> h(v)-h_m(v)$).

Case (2) is similar: $h_m(w)=h_m(v)-1$ if $e$ is on $t_m$, and $h_m(w)=h_m(v)+3$ otherwise.
The first contradicts $v$ maximizing $h-h_m$, and the latter contradicts $v$ maximizing $h$.

Whatever the situation, we derive a contradiction, implying $v$ must indeed be a local maximum.
Since $v$ is not on $\partial R$, we can perform a flip round $v$.
This makes it a local minimum while preserving the values $h$ takes on all other vertices of $R$ and completes the proof.
\end{proof}

Because the situation is finite, Proposition \ref{hredux} essentially tells us any tiling of a region $R$ can be taken by a sequence of flips to the tiling that minimizes height functions over tilings of $R$.
A simple but important corollary follows.

\begin{corolario}[Flip-connectedness]\label{flipconec}
Let $R$ be a black-and-white simply-connected quadriculated region tileable by dominoes.
Then any two distinct tilings of $R$ can be joined by a sequence of flips.
\end{corolario}

\section{Kasteleyn matrices}\label{sec:kastmatplano}

A \textit{Kasteleyn matrix} `encodes' a quadriculated black-and-white region $R$ in matrix form, and its construction is similar to that of adjacency matrices.

Given one such region $R$, we can obtain an adjacency matrix $A$ of $R$ from its dual graph $G$ as follows: enumerate each black vertex (starting from 1), and do the same to white vertices.
Then $A_{ij} = 1$ if the $i$-th black vertex and $j$-th white vertex are joined by an edge, and 0 otherwise.
\begin{figure}[H]
    \centering
    \def\svgwidth{0.55\columnwidth}
    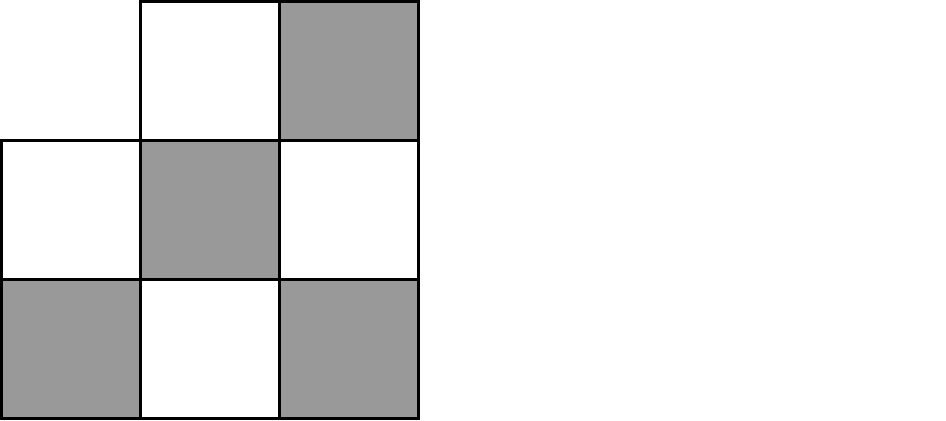 \linebreak \linebreak
		$A = \left( \begin{array}{cccc}
		1&0&1&0\\
		1&1&1&1\\
		0&1&0&1\\
		0&0&1&1
		\end{array}\right)$
		\caption{The construction of an adjacency matrix $A$ for a quadriculated, colored region.}
\end{figure}

Consider now an $n\times n$ adjacency matrix $A$ and the combinatorial expansion of its determinant:
\begin{equation}
\text{det}(A) = \sum \limits_{\sigma \in S_n} {\text{sgn}(\sigma) \prod \limits_{i=1}^n {A_{i,\sigma (i)}}}
\label{deta}
\end{equation}

In the expansion above, each nonzero term of the form $\prod\limits_{i=1}^n A_{i,\sigma (i)}$ can be seen as corresponding to a matching of $G$.
In fact, the term is nonzero if and only if each factor in the product is 1, in which case the $i$-th black vertex is joined by an edge to the $\sigma (i)-th$ white vertex.
Since $\sigma$ is a permutation on $\{1, \dots, n\}$, the collection of these edges is by construction a set of edges on $G$ in which each vertex of $G$ features exactly once.
The observation follows.

Of course, the correspondence goes both ways.
This means that, except for sgn($\sigma$), det($A$) counts the number of matchings of $G$ (and thus also the domino tilings of $R$).
How do we get past the sign?

The obvious way would be to consider ther \textit{permanent} of $A$ $$\text{perm}(A) = \sum \limits_{\sigma \in S_n} {\prod \limits_{i=1}^n {A_{i,\sigma (i)}}},$$ but permanents lack a number of interesting properties when compared to determinants, and are also much more costly to compute.

The answer is precisely the Kasteleyn matrix $K$: an altered adjacency matrix in which some entries are replaced by $-1$.
Its construction is similar to the ordinary adjacency matrix, except some edges on $G$ are assigned the value $-1$ rather than $+1$.
This distribution of minus signs can be done in many ways, but the following observation explains the general principle behind it: a flip on a matching of $G$ always changes the sign of the corresponding permutation in~\eqref{deta}.
This is because, on a permutation level, applying a flip amounts to multiplying the original permutation by a cycle of length 2.

With this in mind, the distribution of minus signs over edges on $G$ is made so that the sign change in a permutation caused by a flip is always counterbalanced by a sign change on the corresponding product of entries of $K$.
Such a distribution ensures that applying a flip does not change the `total' sign of the term $$\text{sgn}(\sigma) \prod \limits_{i=1}^n {K_{i,\sigma (i)}}$$ in~\eqref{deta}.
And since we've shown that any two distinct domino tilings of $R$ (and thus matchings of $G$) can be joined by a sequence flips, this means the sum in~\eqref{deta} is carried over identically signed numbers.
In other words, for a Kasteleyn matrix $K$ of a region $R$, $\lvert \text{det}(K) \rvert$ \textbf{is} the number of domino tilings of $R$.

An easy, convenient way of distributing minus signs over edges on $G$ is assigning them to all horizontal edges in alternating lines (say, all odd lines, or all even lines).
This way, a $2 \times 2$ square in the dual graph will always contain exactly one negative edge (either the topmost or the bottommost horizontal line), so that a flip always will always produce a sign change on the product of entries of $K$.

We highlight that in his original paper \cite{kasteleyn1961statistics}, flip-connectedness (or more generally, flips) was not a part of Kasteleyn's exposition. His methods were combinatorial but he employed Pfaffians.

Below, we show an example of construction of a Kasteleyn matrix.
In the corresponding dual graph, negative edges are red and dashed.
\begin{figure}[H]
    \centering
    \def\svgwidth{0.55\columnwidth}
    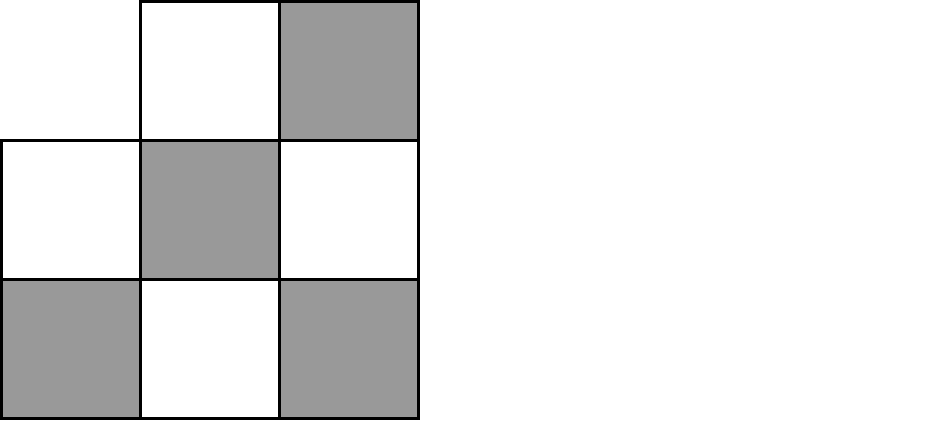 \linebreak \linebreak
		$K = \left( \begin{array}{cccc}
		-1&0&1&0\\
		1&1&1&1\\
		0&1&0&-1\\
		0&0&1&-1
		\end{array}\right)$
		\caption{The construction of a Kasteleyn matrix $K$.}
\end{figure}

\section{A classical result: domino tilings of the rectangle}\label{sec:exret}

We will end this chapter by using our methods to provide a classical result: the counting of domino tilings of an $m \times n$ black-and-white rectangular region, $R_{m,n}$.
Of course, if both $m$ and $n$ are odd, that number is 0; we assume then that $m$ is even.

Let $G$ be $R_{m,n}$'s dual graph with minus signs assigned to all horizontal edges in even lines, from which we obtain the corresponding Kasteleyn matrix $K$.
Rather than compute the determinant of $K$, we will consider the matrix $M = KK^* \oplus K^*K$; it's clear that $\text{det}(M) = \text{det}(K)^4$.

$M$ can be seen as a double adjacency matrix of $G$, acting as a linear map on the space of formal linear combinations of vertices.
It takes a vertex $v$ to the sum of vertices that are joined to $v$ via an edge-path of length two on the graph.
Notice edge sign and vertex multiplicity (when a vertex can be reached from $v$ via two distinct edge-paths) are taken into account.

Because edge-paths considered have length two, $M$ takes white vertices to white vertices and black vertices to black vertices.
Another way of thinking this is as follows: when interpreting the Kasteleyn matrix as a linear map (like $M$ above), our general construction of the Kasteleyn matrix implies it goes from the space of white vertices $W$ to the space of black vertices $B$.
Of course, this also means $K^*: B \longrightarrow W$.
It then becomes clear by the definition of $M$ that it is a color-preserving map.
This essentially means $M$ acts independently on $B$ and $W$.

Consider the grid below, so that each vertex of $G$ is identified by a double index $(i,j)$.
\begin{figure}[H]
    \centering
    \def\svgwidth{0.35\columnwidth}
    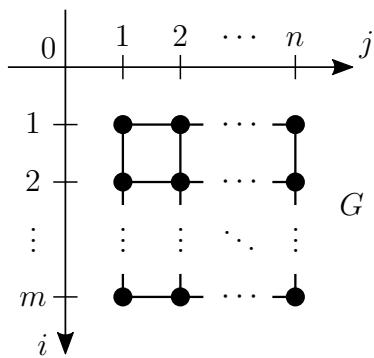
		\caption{Indexation grid.}
\end{figure}

In the obvious notation, vertices $v_{i,j}$ of $G$ at least two units away from the boundary satisfy $Mv_{i,j}= 4v_{i,j} + v_{i+2,j} + v_{i-2,j} + v_{i,j+2} + v_{i,j-2}$.
The coefficient in $v_{i,j}$ comes from moving forward then backwards in each cardinal direction; notice that a negative edge traversed this way will account for two minus signs, so the end result is always positive.
Vertices of the form $v_{i\pm 1, j \pm 1}$ do not feature because each of them can be reached via exactly two distinct edge-paths with necessarily opposite signs.

The formula can be extended to all vertices of $G$ as follows.
Put $v_{i,j} = 0$ if it is immediately outside the boundary of $G$; then, for each line of zero-vertices, reflect $G$ through that line and set a vertex $v_{i,j}$ obtained this way to be \textbf{minus} the vertex from which it was reflected.
We then repeat this process, so that in the end $v_{i,j}$ will be defined for all $i,j \in \Z$.

Figure \ref{fig:gridext} is a visual representation of this extension; in it, each gray square is a zero vertex.
More generally, one such extension can be succinctly represented by the relations:
\begin{figure}[ht]
    \centering
    \def\svgwidth{1.01\columnwidth}
    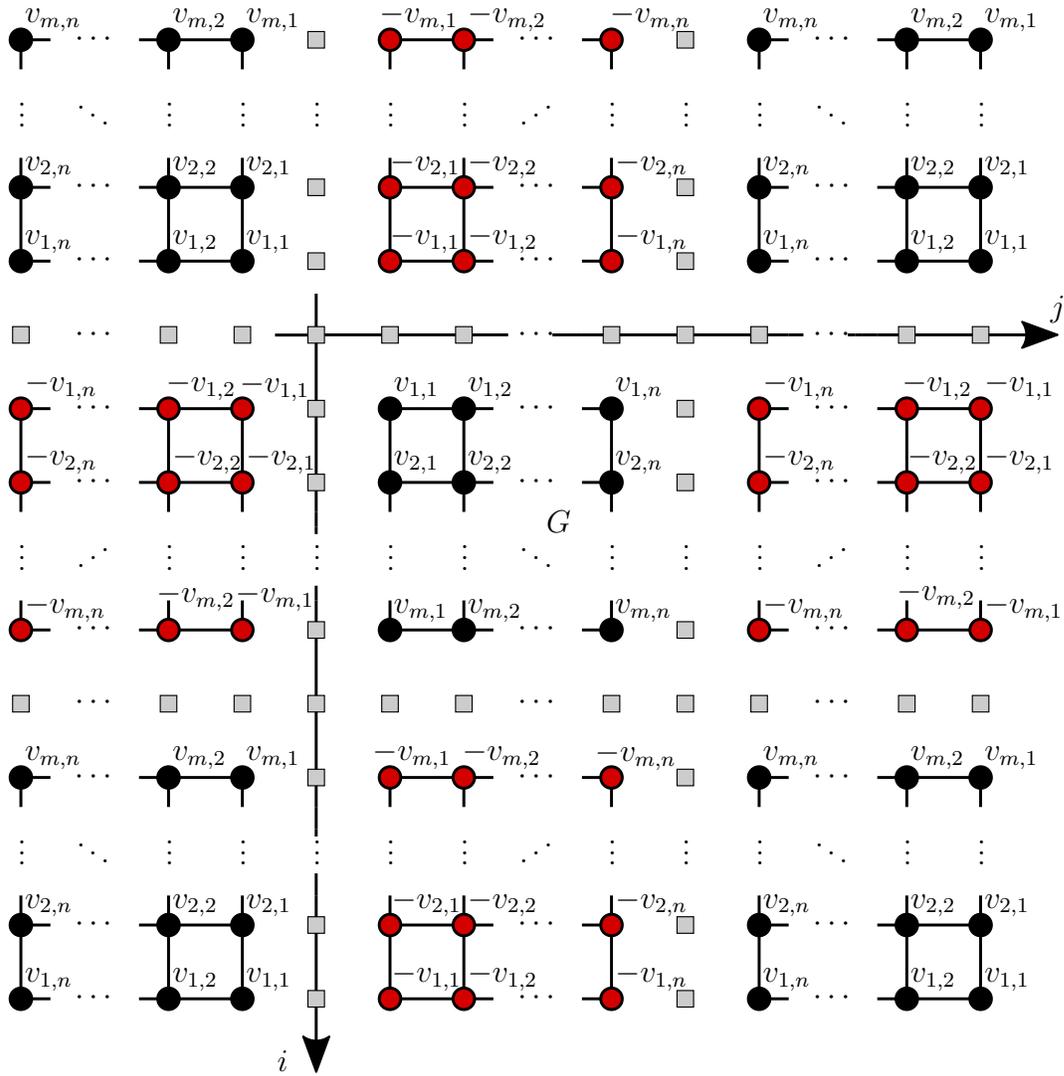
		\caption{Extension through reflections.}
		\label{fig:gridext}
\end{figure}

\begin{equation}\left\{ \begin{array}{l}\label{eqsextensao}
v_{0,j} = v_{m+1,j} = v_{i,0} = v_{i,n+1} = 0 \\
v_{-i,j} = v_{i,-j} = -v_{i,j} \\
v_{i+2m+2,j} = v_{i,j+2n+2} = v_{i,j}
\end{array} \right. \end{equation}

With this, the formula for $M$ now holds not only on all of $G$, but also on all of $\Z^2$.
Notice the space on which $M$ acts is still $mn$-dimensional, since coordinates on the original vertices of $G$ propagate to all vertices of $\Z^2$ via the relations above.
We will now compute $\text{det}(M)$.

We will always refer to the imaginary unit by the boldface $\bi$.
Let $\zeta_1 = \exp\big(\frac{\pi \cdot \bi}{m+1}\big)$ and $\zeta_2 = \exp\big(\frac{\pi \cdot \bi}{n+1}\big)$.
For all $k,l \in \Z$ with $1 \leq k \leq m$ and $1 \leq l \leq n$, let  
\begin{equation*}
\begin{split}
v(k,l) &= \sum_{i,j} 4 \text{ }\sin \left(\frac{i\cdot k \cdot\pi}{m+1}\right) \sin \left(\frac{j\cdot l \cdot \pi}{n+1}\right) \cdot v_{i,j} \\
&= \sum_{i,j} \left( \zeta_1^{ik} - \zeta_1^{-ik} \right) \left( \zeta_2^{jl} - \zeta_2^{-jl} \right) \cdot v_{i,j}
\end{split}
\end{equation*}

Notice $v(k,l)$ is a valid vector in that its coordinates respect the relations in~\eqref{eqsextensao}.
We claim each such $v(k,l)$ is an eigenvector of $M$. Indeed, we have that the $v_{i,j}$-coordinate of $Mv(k,l)$ is
\begin{alignat*}{2}
4 &\left( \zeta_1^{ik} - \zeta_1^{-ik} \right) \left( \zeta_2^{jl} - \zeta_2^{-jl} \right)& \\
+ &\left( \zeta_1^{(i+2)k} - \zeta_1^{-(i+2)k} \right) \left( \zeta_2^{jl} - \zeta_2^{-jl} \right) &+& \left( \zeta_1^{(i-2)k} - \zeta_1^{-(i-2)k} \right) \left( \zeta_2^{jl} - \zeta_2^{-jl} \right) \\
+ &\left( \zeta_1^{ik} - \zeta_1^{-ik} \right) \left( \zeta_2^{(j+2)l} - \zeta_2^{-(j-2)l} \right) &+& \left( \zeta_1^{ik} - \zeta_1^{-ik} \right) \left( \zeta_2^{(j-2)l} - \zeta_2^{-(j-2)l} \right)
\end{alignat*}

The expression above may be rearranged to yield:
\begin{equation*}
\begin{split}
\left[Mv(k,l)\right]_{i,j} = & \left( \zeta_1^{ik} - \zeta_1^{-ik} \right) \left( \zeta_2^{jl} - \zeta_2^{-jl} \right) \left( \left( \zeta_1^{k} + \zeta_1^{-k} \right)^2 + \left( \zeta_2^{l} + \zeta_2^{-l} \right)^2 \right) \\
= & \left[v(k,l)\right]_{i,j} \cdot \left( \left( \zeta_1^{k} + \zeta_1^{-k} \right)^2 + \left( \zeta_2^{l} + \zeta_2^{-l} \right)^2 \right)\text{.}
\end{split}
\end{equation*}

In other words, $v(k,l)$ is an eigenvector of $M$ with associated eigenvalue 
\begin{equation*}
\begin{split}
\lambda(k,l) &= \left( \left( \zeta_1^{k} + \zeta_1^{-k} \right)^2 + \left( \zeta_2^{l} + \zeta_2^{-l} \right)^2 \right) \\
&= 4\cdot \left(\cos \left(\frac{k\pi}{m+1}\right)\right)^2 + 4\cdot \left(\cos \left(\frac{l\pi}{n+1}\right)\right)^2
\end{split}
\end{equation*}

It is well known from Fourier series that the vectors $v(k,l)$ are linearly independent (and in fact orthogonal), so indeed these are all the $mn$ eigenvalues of $M$.
Hence, $\det(M) = \prod_{k=1}^m \prod_{l=1}^n \lambda(k,l)$.

Since $m$ is even, for all $k \in \Z$ with $1 \leq k \leq m/2$ it holds that $\cos\left(\frac{k\pi}{m+1}\right)=-\cos\left(\pi - \frac{k\pi}{m+1}\right)$.
Because in $\lambda(k,l)$ the cosines are squared, this implies $\lambda(k,l) = \lambda(m+1-k,l)$ whenever $1 \leq k \leq m/2$.
Thus, we may write $\det(M) = \prod_{k=1}^{m/2} \prod_{l=1}^n \lambda(k,l)^2$.

Finally, remember that $\det(M) = \det(K)^4$, so the number of tilings of $R_{m,n}$ is
\begin{equation*}
\lvert \det(K) \rvert = \prod_{k=1}^{m/2} \prod_{l=1}^n 2 \left( \left(\cos \left(\frac{k\pi}{m+1}\right)\right)^2 + \left(\cos \left(\frac{l\pi}{n+1}\right)\right)^2 \right)^{1/2}
\end{equation*}
\chapter{Domino tilings on the torus}\label{chap:tor}
Domino tilings on the torus (or more generally planar surfaces) are amenable to much of what's been discussed in the plane case, but modifications are unavoidable.
Let $D_n \subset [0,2n]^2$\label{def:dn} be a $2n \times 2n$ black-and-white quadriculated square region with the square $[0,1]^2$ fixed as black: initially, this will be our model fundamental domain for studying the torus.
The torus arises from identifying opposite sides on $D_n$; we will refer to a torus obtained this way by $\T_n$\label{def:tn}.
Notice $D_n$ respects the black-and-white condition.
Questions on the existence and number of domino tilings for $\T_n$ will receive treatment similar to that of ordinary planar regions.

We first introduce the concept of \textit{flux} of a tiling.
Our torus model is very similar to the ordinary square planar region, except for the identification of opposing sides.
This identification enables us to represent tilings of $\T_n$ on $D_n$ with dominoes that `cross over'\label{def:crossoverdomino} to the opposing side.
For instance, the $4 \times 4$ model of the torus admits tilings such as those represented on $D_2$ below.
\begin{figure}[H]
		\centering
		\includegraphics[width=0.975\textwidth]{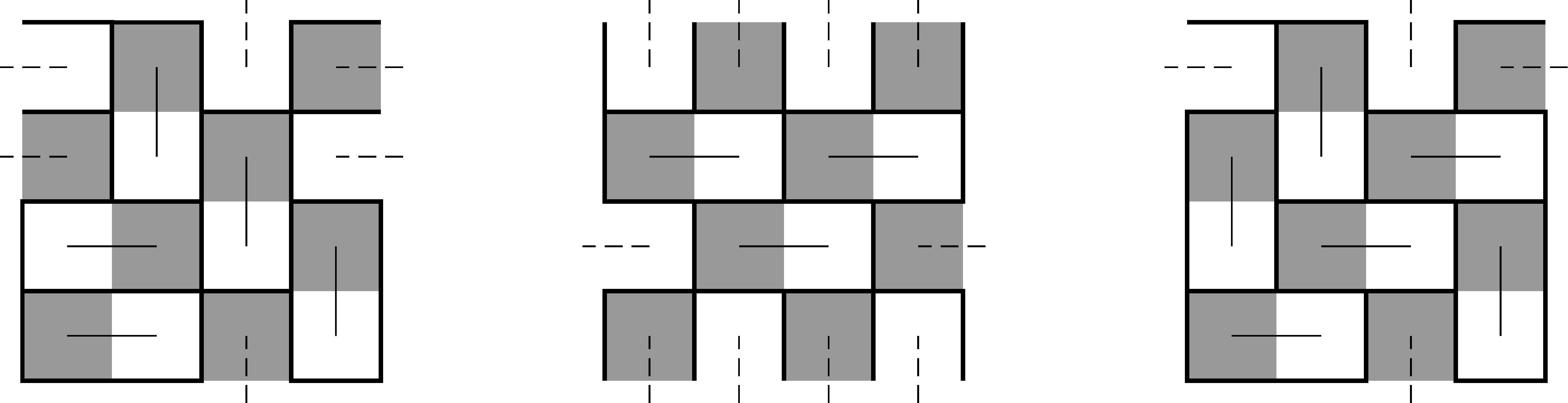}
		\caption{Example tilings of $\T_2$ with `cross over' dominoes.}
\end{figure}
Notice these new tilings also introduce new `cross-flips'\label{def:crossflip}: we can flip `cross over' dominoes just like ordinary dominoes, in the obvious way.
In fact, cross-flips are possible in the first two tilings above.

In a loose sense, the flux of a tiling counts these `cross-over' dominoes. 
We now describe this in detail. 

Remember a domino is always made up of exactly one black square and one white square.
We fix a positive orientation for dominoes: from their black square to their white square.
This orientation applies to cross-over dominoes in the obvious way, and is also inherited by edges on the dual graph in the natural way. 
We also fix positive horizontal and vertical orientations for the fundamental domain $D_n$ itself: from left to right and from top to bottom.

Let $t$ be a tiling of the torus $\T_n$ on $D_n$.
The \textit{horizontal} flux\label{def:fluxsimple} of $t$ is the number of vertical cross-over dominoes on $t$, counted positively if their orientation agrees with $D_n$'s own vertical orientation, and negatively otherwise.
Similarly, the \textit{vertical} flux of $t$ is the number of horizontal cross-over dominoes on $t$, counted positively if their orientation agrees with $D_n$'s own horizontal orientation, and negatively otherwise.
The definitions may appear to be a mismatch, but the reason for them should become clear in time.

This rather algebraic definition interacts well with flips, in a manner similar to how a signed region does: a flip always preserves the flux of a tiling.
Indeed, an ordinary flip does not involve cross-over dominoes at all, while a cross-flip replaces a count of $+1$ and $-1$ by a count of two zeroes.
On the other hand, this also implies our space of tilings is no longer flip-connected: two tilings with different flux values can never be joined by a sequence of flips!
Consequently, we cannot expect to achieve good results from na{\"i}vely applying our Kasteleyn matrix method to the torus model.
Accordingly, we will modify the construction of our Kasteleyn matrix $K$ to account for the flux of a tiling.

$\T_n$'s dual graph $G$ is the same as $D_n$'s dual graph, except it features new `crossing edges'.
Like in the planar case, enumerate each black vertex of $G$ (starting from 1), and do the same to white vertices.
We will assign weights to each edge on $G$.
Let $e_{ij}$ be the edge joining the $i$-th black vertex to the $j$-th white vertex: if no such edge exists, we assign the weight $0$ to it; otherwise, it's assigned the weight $+1$.
Next, distribute minus signs over edges on $G$ like before (taking crossing edges into account); for each edge that's assigned a minus sign this way, multiply its weight by $-1$.\label{def:kasttorussimplified}

Finally, consider vertical crossing edges.
For each of those, if its orientation agrees with the graph's own vertical orientation, multiply its weight by $q_0$; otherwise, multiply its weight by $q_0^{-1}$.
Similarly, consider horizontal crossing edges.
For each of those, if its orientation agrees with the graph's own horizontal orientation, multiply its weight by $q_1$; otherwise, multiply its weight by $q_1^{-1}$.
Notice the effect of minus signs and crossing edges is cumulative!

Now that all edges are assigned their corresponding weights, $K(i,j)$ is simply the weight of $e_{ij}$.
Below, we have the dual graph of $\T_2$, where dashed red edges are assigned the minus sign, and the corresponding Kasteleyn matrix.
\begin{figure}[H]
    \centering
    \def\svgwidth{0.27\columnwidth}
    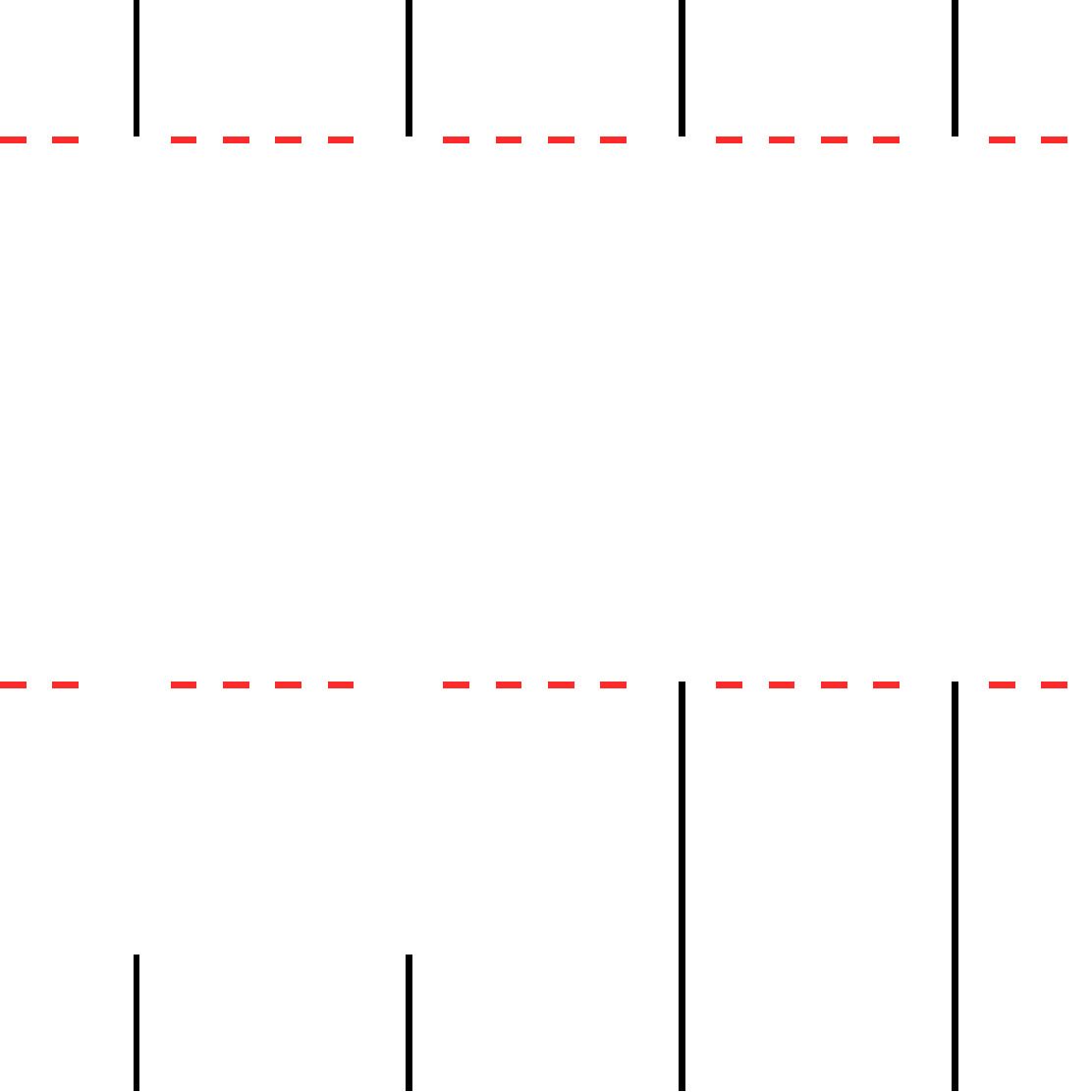\linebreak \linebreak \quad\quad\quad\quad		$K = \left( \begin{array}{cccccccc}
		1&q_1^{-1}&1&0&0&0&q_0&0\\
		1&1&0&1&0&0&0&q_0\\
		1&0&-1&-1&1&0&0&0\\
		0&1&-q_1&-1&0&1&0&0\\
		0&0&1&0&1&q_1^{-1}&1&0\\
		0&0&0&1&1&1&0&1\\
		q_0^{-1}&0&0&0&1&0&-1&-1\\
		0&q_0^{-1}&0&0&0&1&-q_1&-1
		\end{array}\right)$
		\caption{The construction of a Kasteleyn matrix for the dual graph of $\T_2$.}
\end{figure}

It's clear that $\det(K)$ is not a number, but rather a Laurent polynomial $P_K$ in $q_0$ and $q_1$.
For instance, the Kasteleyn matrix above yields the polynomial
\begin{gather*}
132 - 32\cdot \left(q_0 + q_0^{-1} + q_1 + q_1^{-1}\right)\\
- 2 \cdot \left(q_0q_1 + q_0^{-1}q_1 + q_0q_1^{-1} + q_0^{-1}q_1^{-1} \right) + q_0^2 + q_0^{-2} + q_1^{2} + q_1^{-2}
\end{gather*}

Much further ahead, Proposition \ref{signfrontq} will show that each monomial of the form $c_{i,j}\cdot q_0^{i}q_1^{j}$ in $\text{det}(K)$ counts the number of tilings of $\T_n$ with horizontal flux $i$ and vertical flux $j$: that number is precisely the modulus of $c$.
Notice $i$ or $j$ may be negative.
The total number of tilings is thus $\sum \lvert c_{i,j}\rvert$, but Proposition \ref{proplinearcomb} will show it can be computed as a suitable linear combination of $P_K(\pm 1,\pm 1)$, dependent on the assignment of negative edges.
For instance, the linear combination for $P_K$ above is $1/2 \cdot [-P_K(1,1) + P_K(-1,1) + P_K(1,-1) + P_K(-1,-1)]$.

We provide all the possible tilings of $\T_2$, represented as matchings of its dual graph, and grouped according to flux values.
Compare them with the polynomial $P_K$ above.
Notice we omit the vertices at the endpoints of each edge.
\begin{figure}[H]
    \centering
    \def\svgwidth{\columnwidth}
    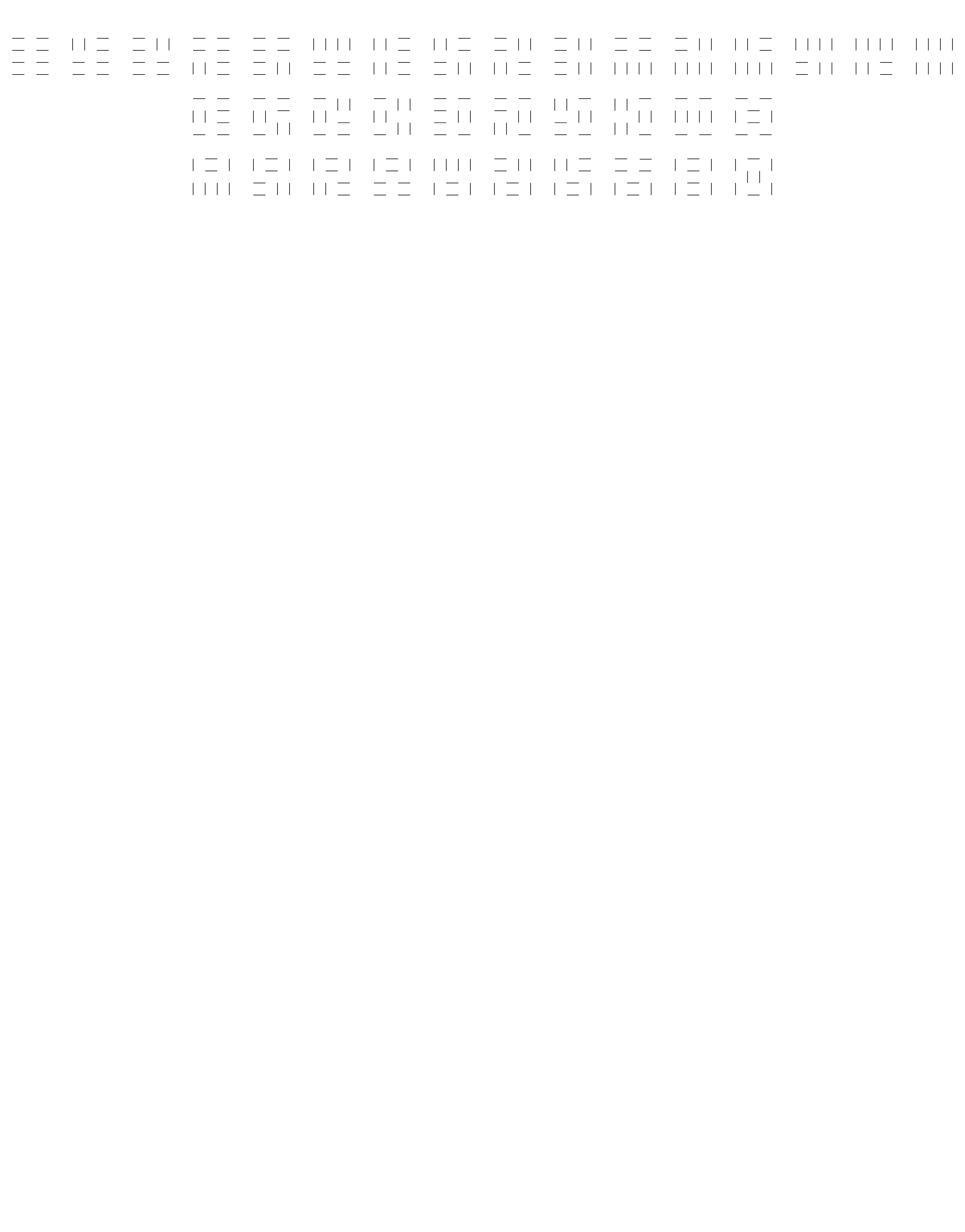
\end{figure}

\begin{figure}[H]
    \centering
    \def\svgwidth{\columnwidth}
    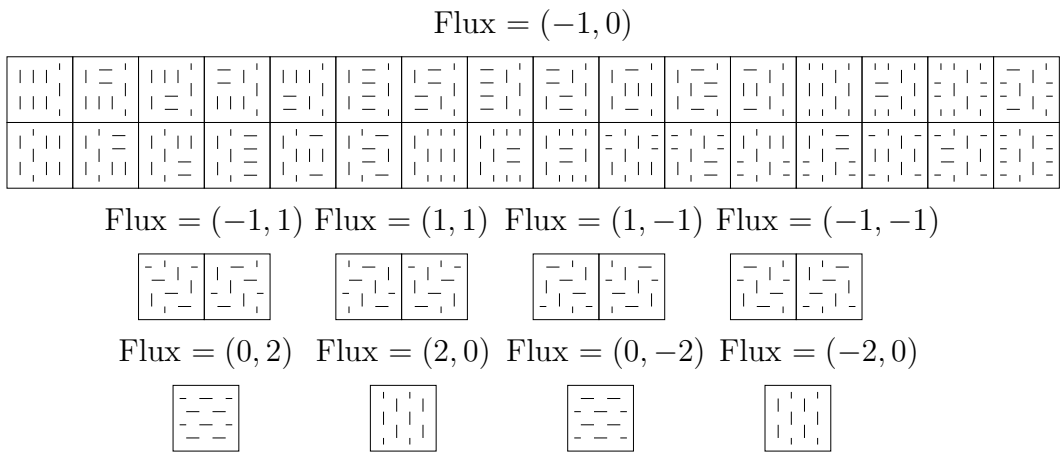
		\caption{Tilings of $\T_2$, 272 in total.}
\end{figure}

\section{Height functions on the torus}\label{sec:hfuntor}

We could attempt to define height functions constructively on $\T_n$ as we did before, but since the torus is not simply-connected we will generally find the process results in incosistencies.
Remember opposite sides are identified, so that corresponding vertices on opposite sides should have the same value.
The following image provides an example of a tiling of $\T_2$ on $D_2$ with an associated `na{\" i}ve' height function; notice its values on corresponding vertices do \textit{not} agree.
\begin{figure}[H]
    \centering
    \def\svgwidth{0.5\columnwidth}
    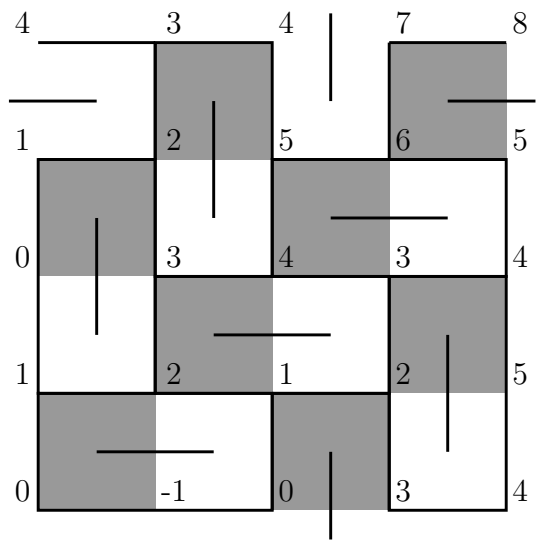
		\caption{A `wrong' height function for a tiling of $\T_2$.}\label{T2HErrado}
\end{figure}

Instead, we will change our methods.
We will interpret $\T_n$ as the quotient $\quotient{\R^2}{L}$, where $L \subset \Z^2$ is the lattice generated by $ \left\{ (2n,0), (0,2n) \right\}$; notice we can still take $D_n$ as its fundamental domain.
Consider the projection map $\Pi : \R^2 \longrightarrow \T_n$.
If base points are provided on each of $\T_n$ and $\R^2$, any tiling of $\T_n$ can be lifted by $\Pi$ to a tiling of the infinite square lattice $\Z^2$ in the obvious way (lift colors too!).
Given a tiling $t$ of $\T_n$ on $D_n$, we will always choose the point $\left (\frac12,\frac12\right)$ as base point for both $\T_n$ and $\R^2$.
Notice this choice guarantees the fiber over every vertex of a square on $D_n$ will consist of points in $\Z^2$, that the square $[0,1]^2 \subset \R^2$ will be colored black, and also that $D_n$ and $t$ will be lifted to an exact copy on the square $[0,2n]^2 \subset \R^2$.

Because identifications and the black-and-white condition are respected, it's easy to see the end result is an $L$-periodic domino tiling $\tilde{t}$ of $\Z^2$.
Since $\Z^2$ is simply-connected, we can define the height function of $\tilde{t}$ as before to be a function $\tilde{h}: \Z^2 \longrightarrow \Z$.
This will not always be the case, but unless stated otherwise, consider the base vertex as lying on the origin with base value 0 assigned to it.

Finally, we define the height function $h$ of $t$ to be the height function $\tilde{h}$ of $\tilde{t}$.
We provide an example of this construction below, using the tiling of Figure \ref{T2HErrado}.
The marked vertex is the origin.
\begin{figure}[H]
		\centering
    \def\svgwidth{\columnwidth}
    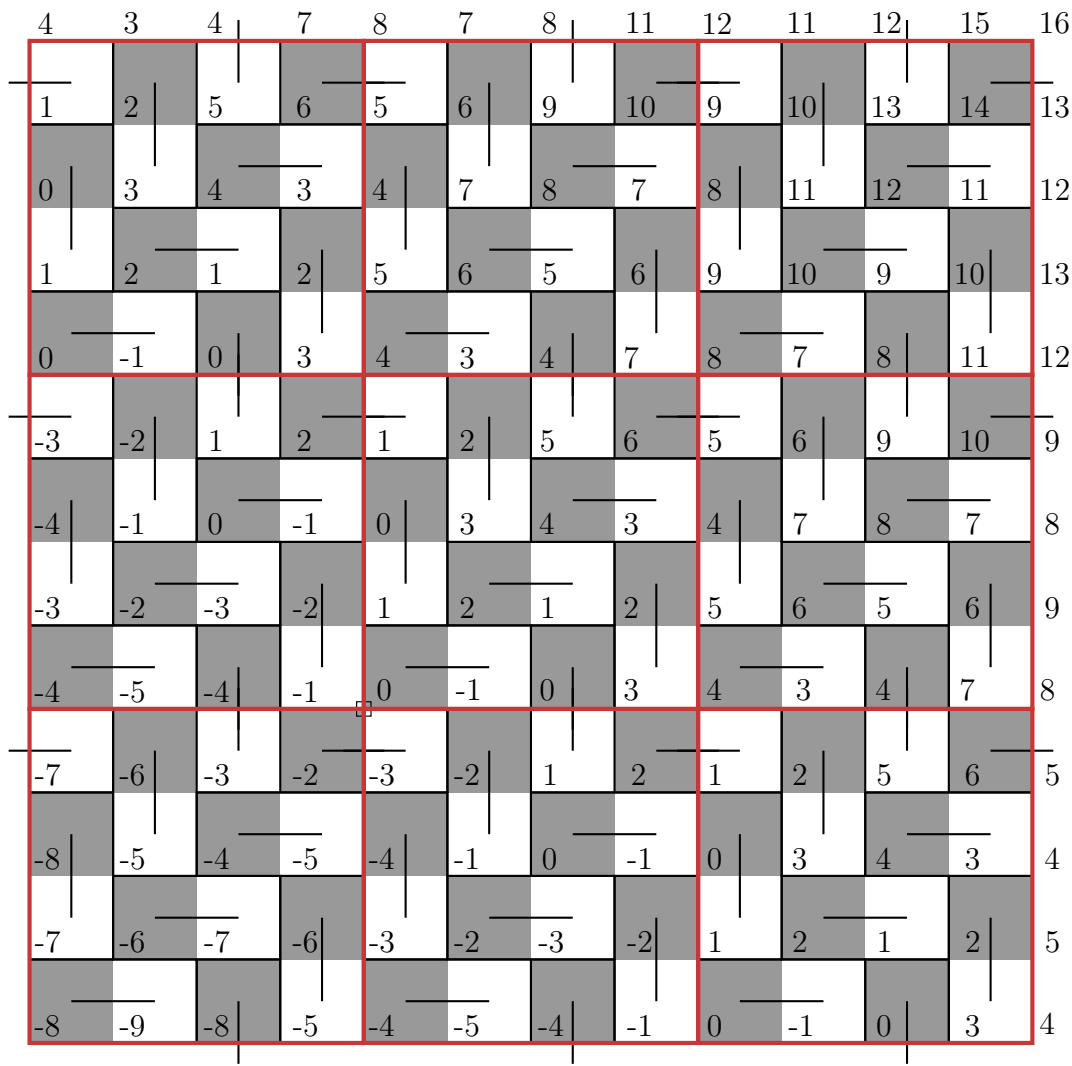
		\caption{A proper height function for a tiling of $\T_2$, defined on $\Z^2$.}
\end{figure}

Inspecting the proofs of Propositions \ref{hprescrip} and \ref{hcarac}, Proposition \ref{hsquare} follows immediately.

\begin{prop}[Height functions on the infinite square lattice]\label{hsquare}
Let $\Z^2$ be the black-and-white infinite square lattice with $[0,1]^2 \subset \R^2$ colored black.
Then an integer function $h$ on $\Z^2$ is a height function if and only if it satisfies the following properties:
\begin{enumerate}
	\item $h$ has the prescribed mod 4 values on $\Z^2$.
	\item $h$ changes by at most 3 along an edge on $\Z^2$.
\end{enumerate}

Furthermore, when $h$ has base vertex lying on the origin with base value 0, then $h(0)=0$.
\end{prop}

The mod 4 prescription function $\Phi: \Z^2 \longrightarrow \{0,1,2,3\}$\label{def:phiprescrip} is easily computed for the usual choice of base vertex $h(0)=0$. It is given by
\begin{equation*}
\Phi(x,y)=\left\{ \begin{array}{ll}
0 & \mbox{ if $x \equiv 0$ and $y \equiv 0 \Mod{2}$;}\\
1 & \mbox{ if $x \equiv 0$ and $y \equiv 1 \Mod{2}$;}\\
2 & \mbox{ if $x \equiv 1$ and $y \equiv 1 \Mod{2}$;}\\
3 & \mbox{ if $x \equiv 1$ and $y \equiv 0 \Mod{2}$.}\\
\end{array} \right.
\end{equation*}

It is important to point out that while Proposition \ref{hsquare} characterizes general height functions on the infinite square lattice $\Z^2$, a function satisfying these properties may not be one obtained from a domino tiling of the torus.
General domino tilings of the infinite square lattice need not have any kind of periodicity.

Consider a domino tiling $t$ of $\T_n$ on $D_n$ and its height function $h$.
The flux of $t$ has a very concrete manifestation in $h$ which we now describe.
Follow the leftmost side of the tiled copy of $D_n \subset \R^2$ on $[0,2n]^2$, edge by edge and starting from the origin, until the opposite horizontal side is reached.
Notice how $h$ changes along this path according to whether or not that edge is on $t$: whenever an edge with a black square on its right is traversed, $h$ changes by $+1$ if that edge is on $t$ and by $-3$ otherwise.
Similarly, whenever an edge with a white square on its right is traversed, $h$ changes by $-1$ if that edge is on $t$ and by $+3$ otherwise.

Now, observe that in this situation an edge being on $t$ means there is no horizontal cross-over domino along it, and an edge \textbf{not} being on $t$ means there \textbf{is} a horizontal cross-over domino along it.
It's then easy to see that whenever we have a $-3$ change on $h$ along that path, the vertical flux changes by $-1$, and whenever we have a $+3$ change, the vertical flux changes by $+1$.
If every edge on that path were on $t$, the total change on $h$ along it would be 0, and it would correspond to a vertical flux value of 0.
This analysis thus makes it clear that in general the total change is a value $4k$ with $k \in \Z$, corresponding to a vertical flux value of $k$.

Let $u_y = (0,2n) \in \Z^2$.
The conclusion of the preceding paragraphs can be succintly expressed as
\begin{equation}\label{hflux}
h(u_y)=4k \Longleftrightarrow t \mbox{ has a vertical flux value of $k$,}
\end{equation} where $k \in \Z$.
This analysis considered only the path along the leftmost side of the original copy of $D_n$, starting from the base vertex lying on the origin, but in fact it can be made more general.
We claim that for \textbf{any} vertex $v$ of $\Z^2$ it holds that
\begin{equation}\label{hflux2}
h(v+u_y) - h(v) =4k \Longleftrightarrow t \mbox{ has a vertical flux value of $k$}
\end{equation}

Indeed, because~\eqref{hflux} means that~\eqref{hflux2} holds when $v=0$, it suffices to show that for any two vertices $v, w$ of $\Z^2$ we have $h(v+u_y) - h(v) = h(w+u_y) - h(w)$, or equivalently $h(v+u_y) - h(w+u_y) = h(v) - h(w)$.
Choose any edge-path $\gamma_0$ in $\tilde{t}$ joining $v$ to $w$.
Because of how $\tilde{t}$ is obtained from $t$ (via lifting), the translated path $\gamma_1 = \gamma_0 + u_y$ \textbf{is} an edge-path in $\tilde{t}$ joining $v+u_y$ to $w+u_y$, and furthermore the constructive definition of height functions implies the total change of $h$ along $\gamma_0$ or along $\gamma_1$ is the same.
This proves~\eqref{hflux2}.

Equation~\eqref{hflux2} means a vertical flux value $k$ of a domino tiling $t$ manifests in its height function $h$ as the (arithmetic) quasiperiodicity relation $h(v+u_y) = h(v) + 4k$.
The same techniques used above show that something very similar holds for the horizontal flux.

Let $u_x = (2n,0) \in \Z^2$.
Then for \textbf{any} vertex $v$ of $\Z^2$ it holds that
\begin{equation}\label{vflux2}
h(v+u_x) - h(v) =4l \Longleftrightarrow t \mbox{ has a horizontal flux value of $l$}
\end{equation}

Of course, equation~\eqref{vflux2} means a horizontal flux value $l$ of a domino tiling $t$ manifests in its height function $h$ as the quasiperiodicity relation $h(v+u_x) = h(v) + 4l$.
These relations allow us to fully characterize \textit{toroidal height functions}, that is, height functions on the infinite square lattice obtained from domino tilings on the torus.

\begin{prop}[Toroidal height functions]\label{hsqtoro}
Let $\Z^2$ be the black-and-white infinite square lattice, as before.
Then a height function $h$ on $\Z^2$ (see Proposition \ref{hsquare}) is a toroidal height function of $\T_n$ if and only if $h$ satisfies the following properties:
\begin{enumerate}
	\item $\exists k \in \Z$, $\forall v \in \Z^2$, $h(v+u_y) = h(v) + 4k$.
	\item $\exists l \in \Z$, $\forall v \in Z^2$, $h(v+u_x) = h(v) + 4l$.
\end{enumerate}
where $u_x = (2n,0)$ and $u_y = (0,2n)$.

Moreover, if $h$ is a toroidal height function of $\T_n$ and $t$ is its associated domino tiling of $\T_n$, then $k$ is $t$'s vertical flux value and $l$ is $t$'s horizontal flux value.
\end{prop}
\begin{proof}
From our previous discussions, it follows immediately that a toroidal height function $h$ of $\T_n$ satisfies those properties and that the integers $k,l$ determine the flux of $h$'s associated domino tiling of $\T_n$.
On the other hand, properties 1 and 2 ensure the domino tiling $\tilde{t}$ of $\Z^2 \subset \R^2$ associated to the height function $h$ is invariant under translation by $u_x$ or by $u_y$.
In other words, letting $L$ be the lattice generated by $\{u_x,u_y\}$, $\tilde{t}$ is $L$-periodic.
This means a quotient by $L$ will result in a tiled torus $\T_n$, and by construction that tiling's height function is $h$.
\end{proof}

We will use this quasiperiodic characterization of toroidal height functions to establish results similar to those we obtained in the planar case.

\begin{prop}\label{htormin}
Let $t_1,t_2$ be two tilings of $\T_n$ with identical flux values $k,l$ and corresponding toroidal height functions $h_1$, $h_2$.
Then $h_m = min \{h_1,h_2\}$ is a toroidal height function of $\T_n$ with flux values $k,l$.
\end{prop}
\begin{proof}
We first use Proposition \ref{hsquare} to check $h_m$ is a height function on $\Z^2$. 
Like before, it suffices to show that $h_m$ changes by at most 3 along an edge on $\Z^2$, and the proof of Proposition \ref{hmin} applies verbatim here.
We then need only check the conditions on Proposition \ref{hsqtoro}.
Notice that
\begin{equation*}
\begin{split}
h_m(v+u_y) & = \min \left\{ h_1(v+u_y),h_2(v+u_y) \right\} = \min \left\{ h_1(v)+4k,h_2(v)+4k \right\} \\
& = \min \left\{ h_1(v),h_2(v) \right\}+4k=h_m(v)+4k\mbox{,}
\end{split}
\end{equation*}
so $h_m(v+u_y)=h_m(v)+4k$.

Similarly, $h_m(v+u_x)=h_m(v)+4l$, and we are done.
\end{proof}

\begin{corolario}[Minimal height functions on the torus]
If there is a tiling of $\T_n$ with flux values $k,l \in \Z$, then there is a tiling of $\T_n$ with flux values $k,l$ and so that its height function is minimal over tilings of $\T_n$ with flux values $k,l$.
\end{corolario}

\section{More general tori: valid lattices}\label{sec:vallat}

Remember the torus $\T_n$ may be seen as a quotient $\R^2/L$, where $L$ is the lattice generated by $\left \{ \left(2n,0\right), \left(0,2n\right) \right \}$. 
We wish to consider other tori --- or equivalently, other lattices.
Of course, for the quotient $\R^2/L$ to be a torus, we still need $L$ to be generated by two linearly independent vectors.
Moreover, when choosing a planar region $D_L$ to be its fundamental domain, we will avoid those whose boundary crosses an edge on the infinite square lattice; this ensures $D_L$ consists of whole squares whenever $L \subset \Z^2$.
See the image below.
\begin{figure}[H]
		\centering
		\includegraphics[width=0.425\textwidth]{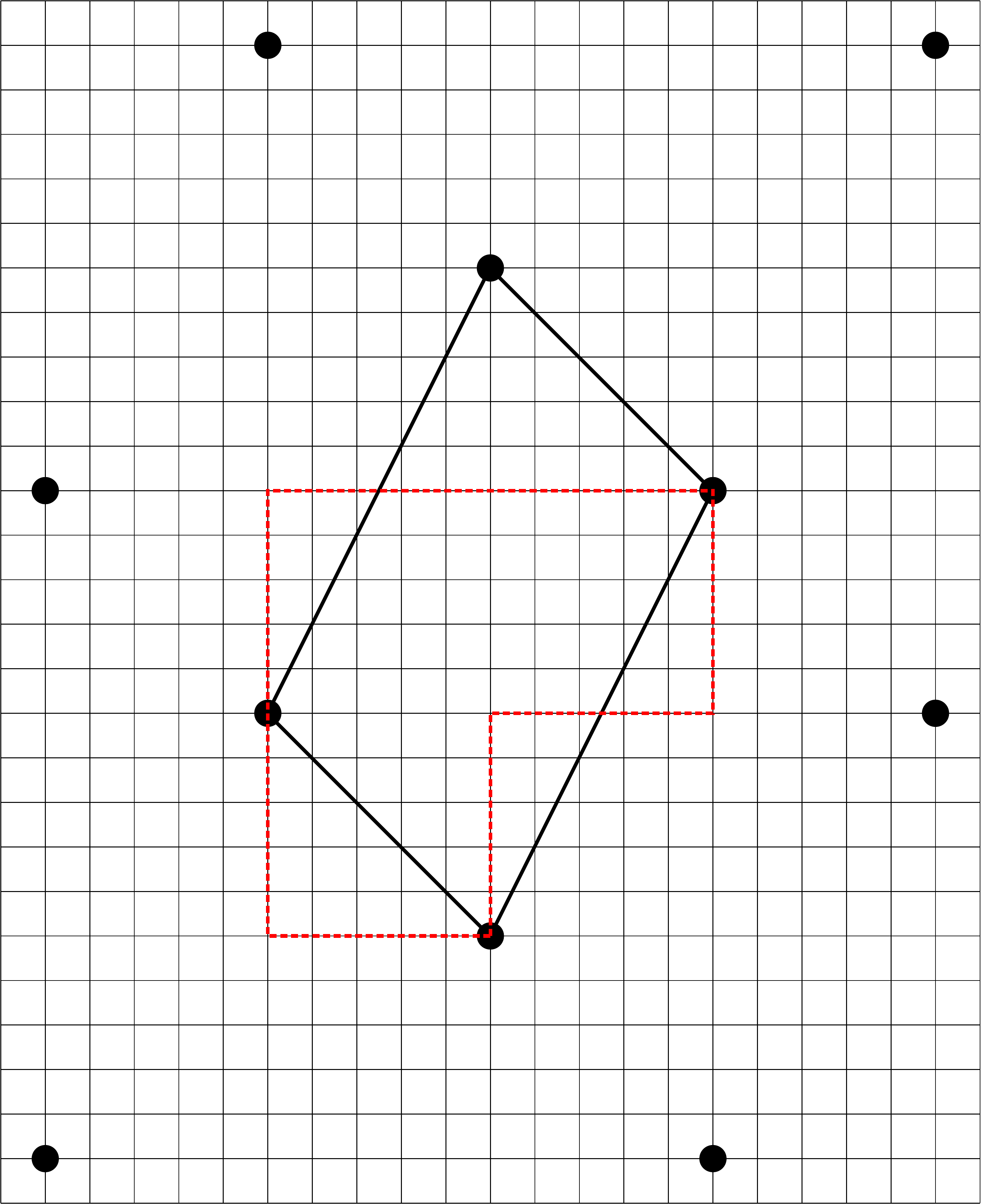}
		\caption{A lattice's usual fundamental domain, and one made up of whole squares.}
\end{figure}

We would like to ensure the black-and-white condition is respected (that is, each domino consists of exactly one white square and one black square).
This is \textit{not}\label{def:bwcondition2} equivalent to $D_L$ having the same number of black squares and white squares.
For instance, the figure below features a lattice $L \subset \Z^2$ whose fundamental domain $D_L$ has the same number of black squares and white squares; however, the identifications allow us to use dominoes consisting of two white squares or two black squares when tiling it.
\begin{figure}[H]
		\centering
		\includegraphics[width=0.465\textwidth]{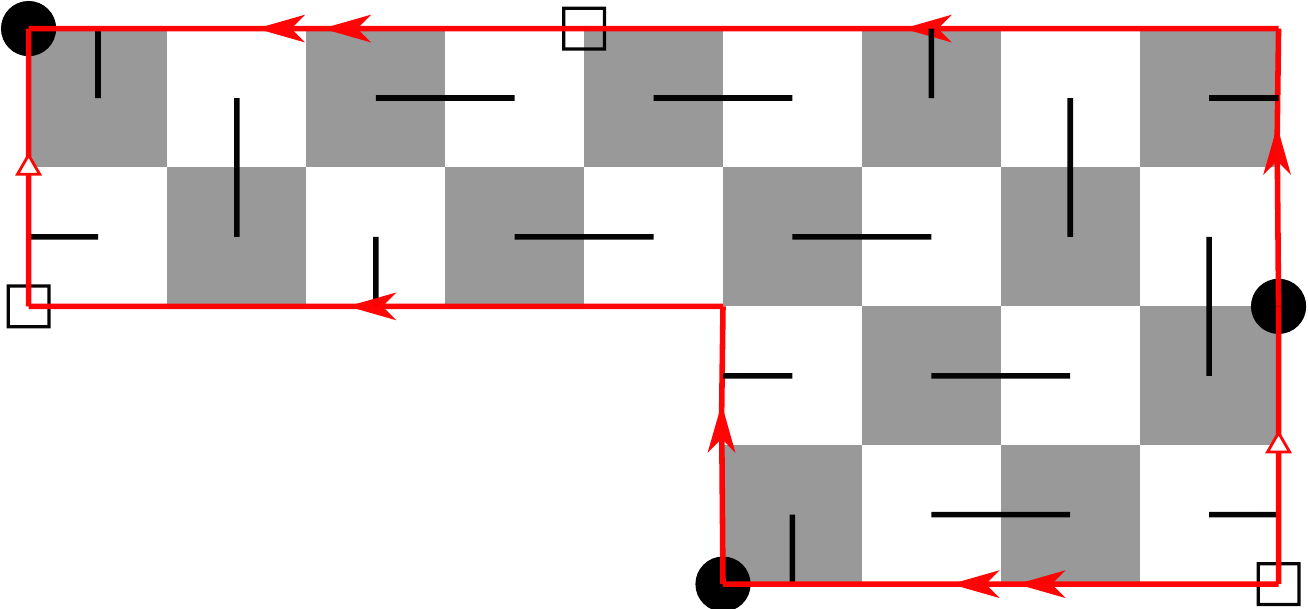}
		\caption{A tiling of a torus which does not respect the black-and-white condition.}
\end{figure}

Perhaps a better interpretation of this situation is that the copies of the fundamental domain $D_L$ that cover $\R^2$ are not all equally colored.
\begin{figure}[H]
		\centering
		\includegraphics[width=0.725\textwidth]{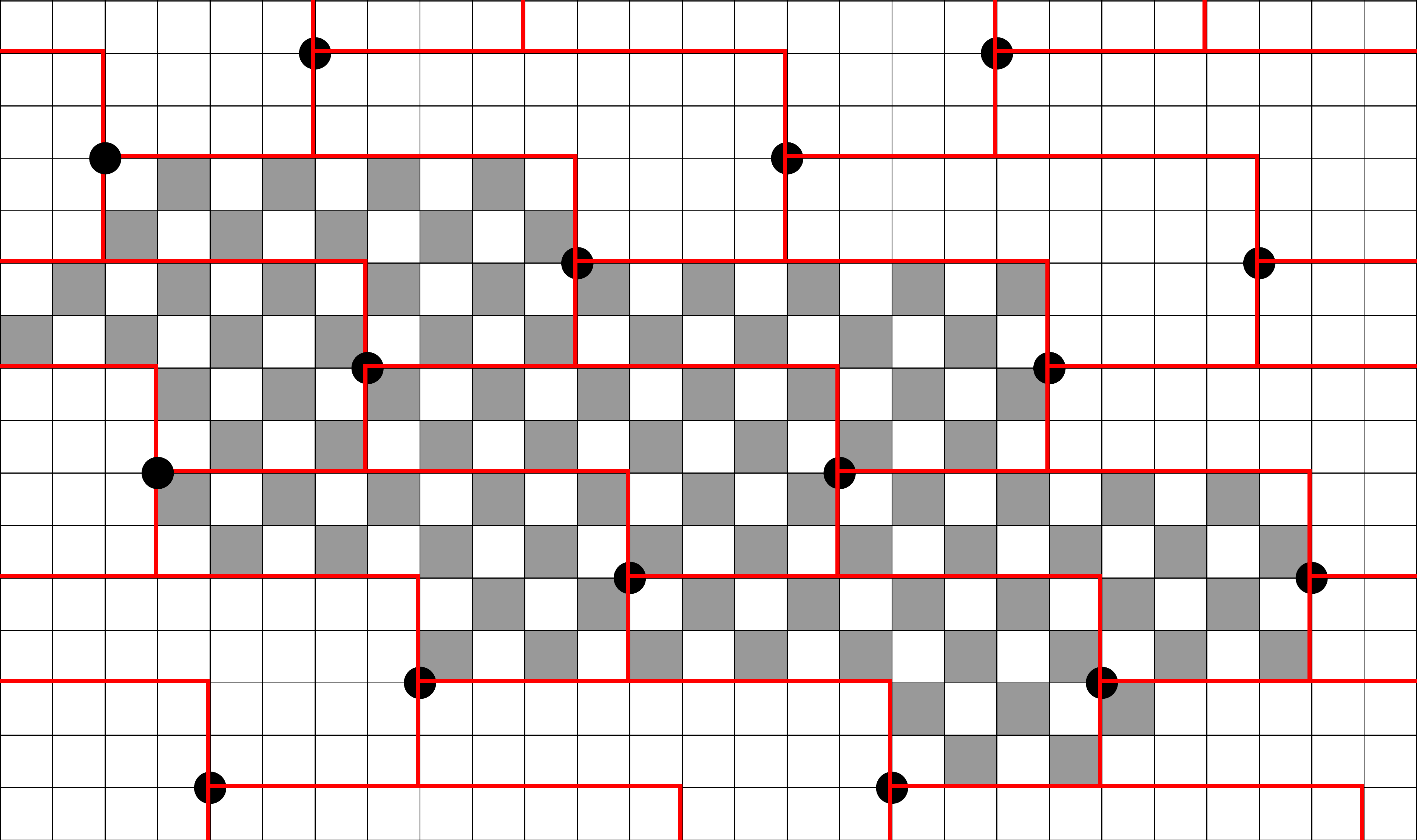}
		\caption{An invalid lattice: fundamental domains are not all equally colored.}
\end{figure}

It's now easy to see that a necessary and sufficient condition for them to be equally colored (and thus for the black-and-white condition to hold)\label{def:bwcondition3} is that the $L$-periodicity in $\R^2$ preserve square color. 
If $L$ is generated by vectors $v_0,v_1 \in \Z^2$, this is equivalent to the sum of $v_i$'s coordinates being even.
Notice this ensures $D_L$ has an even number of squares, because that number is the area of $D_L$ which is given by $\det(v_0,v_1)$.
Moreover, we claim in this situation $D_L$ automatically has an equal number of black squares and white squares.

Indeed, if $D_L$ is a rectangle, the claim clearly holds (since it has an even number of squares).
Otherwise, $D_L$ can be taken to be an L-shaped figure, like in the figure below.

\begin{figure}[H]
    \centering
    \def\svgwidth{0.425\columnwidth}
    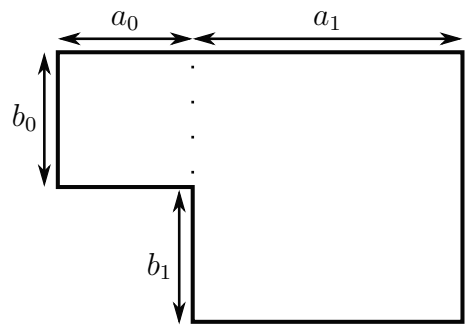
		\caption{$D_L$ can always be taken to be a rectangle or an L-shaped figure.}
\end{figure}

Notice in this case $D_L$ can be decomposed into a union of two rectangles in two different ways:
\begin{itemize}
\item $a_0 \times b_0$ rectangle with area $R_{00}$ and $a_1 \times (b_0 + b_1)$ rectangle with area $R_{01}$
\item $(a_0 + a_1) \times b_0$ rectangle with area $R_{10}$ and $a_1 \times b_1$ rectangle with area $R_{11}$
\end{itemize}

We need only show that for at least one $i = 0,1$ both $R_{i0}$ and $R_{i1}$ are even.
Since $R_{i0} + R_{i1}$ is always even (because that is the area of $D_L$), $R_{i0}$ and $R_{i1}$ always have the same parity.
Thus, the claim would fail to hold only if all $R_{ij}$ were odd. Studying the parity for $a_0, a_1, b_0$ and $b_1$, it's easy to see see this cannot be, so we are done.

Let $\mathscr{E},\mathscr{O} \subset \Z^2$ be the sets of vertices whose coordinates are respectively both even and both odd, that is $\mathscr{E} = 2\Z^2$ and $\mathscr{O} = 2\Z^2 + (1,1)$.
It is clear from this discussion that whenever a lattice $L \subset \Z^2$ is generated by two vertices $v_0,v_1$ with $\det(v_0,v_1) \neq 0$ and $v_0,v_1 \in \mathscr{E} \sqcup \mathscr{O}$, the fundamental domain $D_L$ satisfies the black-and-white condition and can be tiled by dominoes.
We say such a lattice $L$ is a \textit{valid lattice}.
We will refer to a torus obtained from a valid lattice $L$ by $\T_L$\label{def:tl}. Toroidal height functions for these tori are defined much in the same way as before, via lifting.
\chapter{Flux on the torus}
\label{chap:flux}

We now explain how the flux definition that relies on counting cross-over dominoes can be adapted to these more general tori. Let $L$ be a valid lattice generated by $\{v_0, v_1\}$ and $t$ be a tiling of the torus on the fundamental domain $D_L$. As before, consider its lift to a tiling on $\Z^2$.

For any vertex $v \in \Z^2$ there are two L-shaped paths joining $v$ to $v+v_0$; call them $u_0$ and $u_1$. Observe that if one of $v_0$'s coordinates is 0, $u_0$ coincides with $u_1$. Generally, these edge-paths form the boundary of a quadriculated rectangle $R \subset \Z^2$ in which $v$ and $v+v_0$ are opposite vertices.

Remember that whenever an edge-path crosses a domino on a tiling, the height function of that tiling changes by either $+3$ or $-3$ along that edge-path. Were we to define the flux of $t$ through $v_0$ as before, we would like to say it is the number $n_i$ of dominoes (horizontal \textbf{or} vertical) that cross $u_i$, each of which is counted positively if its corresponding height change along $u_i$ is $+3$, and negatively if it is $-3$; notice $u_i$'s orientation matters. However, there is no particular reason why $n_0$ should be used over $n_1$.

When we previously defined the flux via counting dominoes (in Chapter \ref{chap:tor}, for the square torus $\T_n$), $u_0$ and $u_1$ always coincided, so the distinction was irrelevant. If $v_0 \in \mathscr{E}$, the situation is similar. In this case, $R$ is a rectangle with an even number of squares; in particular, the number of black squares and the number of white squares in $R$ are the same. This means $n_0$ and $n_1$ are equal, so the choice of path does not matter.

Real change occurs if $v_0 \in \mathscr{O}$. In this case, $R$ is a rectangle with an odd number of squares, so the number of black squares and the number of white squares in it differ by 1. This means $\lvert n_0 - n_1 \rvert = 1$. Rather than arbitrarily choosing one of $n_0$, $n_1$, we opt for a measured approach: we take their average. Notice that applying this to previous cases yields the same result.

Of course, this means that whenever $v_0 \in \mathscr{O}$, the flux of $t$ through $v_0$ will be some $k$ in $\big(\Z + \frac12\big)$, rather than in $\Z$. This does not contradict our original quasiperiodicity relations, and inspecting the proof of Proposition \ref{hsqtoro}, we need only show that $h(v_0) = 4k$ for this case too.

\begin{lema}\label{fluxgen}
Let $L$ be a valid lattice generated by $\{v_0, v_1\}$.
Let $t$ be a tiling of $\T_L$ and $h$ its toroidal height function.
Then $$h(v_0) = 4k \Longleftrightarrow \text{ $k$ is the flux of $t$ through $v_0$,}$$where $k$ is defined as above.
\end{lema}
\begin{proof}
Let $u_0$ and $u_1$ be the L-shaped edge-paths joining the origin to $v_0$ (and oriented from the origin to $v_0$).
Let $d_i^+$ be the number of dominoes crossing $u_i$ that are counted positively and let $d_i^-$ be the number of dominoes crossing $u_i$ that are counted negatively.
Let $e_i^+$ be the number of edges on $u_i$ whose orientation (as induced by the coloring of $\Z^2$) agrees with $u_i$'s own, and let $e_i^-$ be the number of edges on $u_i$ whose orientation reverses $u_i$'s own.

Suppose $v_0 \in \mathscr{E}$. 
In this case, $e_i^+ = e_i^-$ for each $i=0,1$.
Thus, if no domino crosses $u_i$ (that is, when $d_i^+$ and $d_i^-$ are both 0) the constructive definition of height functions implies $h(v_0) = 0$.
Each domino counted by $d_i^+$ crosses an edge counted by $e_i^-$, and contributes with a height change of $+3$ along that edge (rather than $-	1$); in other words, each domino counted by $d_i^+$ contributes with a total change of $+4$ for $h(v_0)$. 
Similarly, each domino counted by $d_i^-$ contributes with a total change of $-4$ for $h(v_0)$.
All of this implies the following formula\footnote{Equation~\eqref{fluxeven} provides another way to see that when $v_0$'s coordinates are both even, the numbers $n_0$ and $n_1$ are equal.} holds for each $i=0,1$:
\begin{equation}\label{fluxeven}
h(v_0) = 4(d_i^+ - d_i^-)
\end{equation}

The lemma follows from observing that $d_i^+ - d_i^-$ is $t$'s flux through $v_0$.

Now suppose $v_0 \in \mathscr{O}$. 
In this case we no longer have $e_i^+ = e_i^-$; instead, we claim $e_i^+ - e_i^- = \pm 2$, where the sign in $\pm$ is different for each $i=0,1$.
Indeed, let $R$ be the quadriculated rectangle whose boundary is given by $u_0 \cup u_1$.
Each $u_i$ can be divided into three segments as follows: a middle segment of length two fitting a corner square in $R$, and the other two outer segments (each of which possibly has length 0); see Figure \ref{fig:u0u1}.
\begin{figure}[ht]
    \centering
    \def\svgwidth{0.8\columnwidth}
    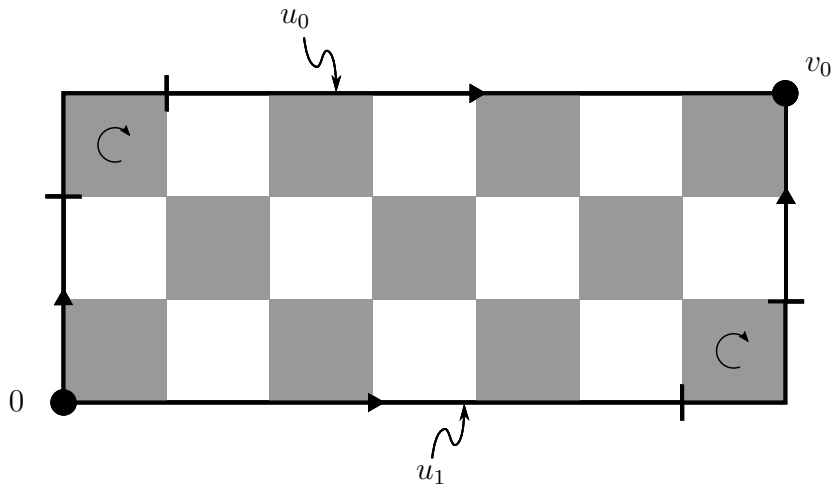
		\caption{The paths $u_0$ and $u_1$, each divided into three segments.}
		\label{fig:u0u1}
\end{figure}

For each $u_i$, each of the outer segments has even length and features edges that are alternatingly counted by $e_i^+$ and by $e_i^-$, so $e_i^+ - e_i^-$ is given entirely by the middle segment.
That segment has two edges that are counted with the same sign, but for each $u_i$ that sign is different, so the claim is proved.

Without loss of generality, say $e_0^+ - e_0^- = 2$ and $e_1^+ - e_1^- = -2$.
If no domino crosses $u_0$, the constructive definition of height functions implies $h(v_0) = 2$.
The same technique used above implies the following formula holds:
\begin{equation}\label{fluxodd0}
h(v_0) = 2 + 4(d_0^+ - d_0^-)
\end{equation}

Applying this process to $u_1$ gives us the formula\footnote{Together, equations~\eqref{fluxodd0} and~\eqref{fluxodd1} provide another way to see that when $v_0$'s coordinates are both odd, the numbers $n_0$ and $n_1$ differ by 1.}:
\begin{equation}\label{fluxodd1}
h(v_0) = -2 + 4(d_1^+ - d_1^-)
\end{equation}

Combining the two yields $h(v_0) = 4 \cdot \frac{1}{2} \left[(d_0^+ - d_0^-) + (d_1^+ - d_1^-)\right]$.
Since $\frac{1}{2} \left[(d_0^+ - d_0^-) + (d_1^+ - d_1^-)\right]$ is the flux of $t$ through $v_0$, the proof is complete.
\end{proof}

The reader might question the choice of $L$-shaped paths for the flux definition.
In this regard, we note the following.
For any edge-path $\gamma$, let $\mathscr{R}(\gamma)$ be the edge-path obtained from $\gamma$ by reflecting it across the middle point between 0 and $v_0$ (in particular, notice $\mathscr{R}(u_0) = u_1$).
Consider the numbers $n_{\gamma}$ and $n_{\mathscr{R}(\gamma)}$ of crossing dominoes, as we defined $n_0,n_1$ for $u_0,u_1$.
Then $n_{\gamma} + n_{\mathscr{R}(\gamma)} = n_0 + n_1$, so that `any measured approach' to choosing an edge-path would yield the same results.

When $v_0 \in \mathscr{O}$, the flux of a tiling through $v_0$ is some $k$ in $(\Z + \frac12)$, so Lemma \ref{fluxgen} implies $h(v_0) \equiv 2 \Mod{4}$, rather than the usual 0.
Observe that this is consistent with the mod 4 prescription function $\Phi$ calculated just after Proposition \ref{hsquare}.

Also, it's clear that the flux of $t$ through $v_1$ is similarly defined, and these properties also hold for $v_1$.

Because of Lemma \ref{fluxgen}, generalizations of Propositions \ref{hsqtoro} and \ref{htormin} to this new scenario are automatic.

\begin{prop}[General toroidal height functions]\label{hsqtorogen}
Let $L$ be a valid lattice generated by $\{v_0, v_1\}$.
Then a height function $h$ on $\Z^2$ (see Proposition \ref{hsquare}) is a toroidal height function of $\T_L$ if and only if $h$ satisfies the following property for each $i=0,1$:
\begin{alignat*}{4}
\text{$v_i \in \mathscr{E}$} \quad &\Rightarrow \quad \exists k_i \in &\Z,& \enspace &\forall v \in \Z^2, \enspace h(v+v_i) = h(v) + 4k_i \\
\text{$v_i \in \mathscr{O}$} \quad &\Rightarrow \quad \exists k_i \in (\Z  &+& \tfrac12), \enspace &\forall v \in \Z^2, \enspace h(v+v_i) = h(v) + 4k_i
\end{alignat*}

Furthermore, if $h$ is a toroidal height function of $\T_L$ and $t$ is its associated domino tiling, then $k_0$ is $t$'s flux through $v_0$ and $k_1$ is $t$'s flux through $v_1$.
\end{prop}

\begin{prop}\label{htormingen}
Let $L$ be a valid lattice and $t_1,t_2$ be two tilings of $\T_L$ with identical flux values $k,l$ and corresponding toroidal height functions $h_1$, $h_2$.
Then $h_m = \min \{h_1,h_2\}$ is a toroidal height function of $\T_L$ with flux values $k,l$.
\end{prop}

\begin{corolario}[Minimal height functions on general tori]\label{htorminimal}
Let $L$ be a valid lattice.
If there is a tiling of $\T_L$ with flux values $k,l$, then there is a tiling of $\T_L$ with flux values $k,l$ and so that its height function is minimal over tilings of $\T_L$ with flux values $k,l$.
\end{corolario}

\section{The affine lattice $L^{\#}$}\label{sec:lsharp}

Let $L$ be a valid lattice generated by $\{v_0, v_1\}$.
Proposition \ref{hsqtorogen} provides a new way to interpret the flux of a tiling of $\T_L$.
Given one such tiling $t$, let $h_t$ be its toroidal height function.
The quantities
\begin{equation*} \begin{split}
\varphi_t(v_0) = \frac{1}{4} \big(h_t(v+v_0) - h_t(v)\big) \\
\varphi_t(v_1) = \frac{1}{4} \big(h_t(v+v_1) - h_t(v)\big)
\end{split} \end{equation*}
do not depend on $v \in \Z^2$.
By the same token, for $i,j \in \Z$, $h_t$'s quasiperiodicity implies
\begin{equation*}\label{fluxdual}
\varphi_t(i \cdot v_0 + j \cdot v_1) = i \cdot \varphi_t(v_0) + j \cdot \varphi_t(v_1)\text{,}
\end{equation*}
so $\varphi_t$ can be seen as a homomorphism on $L$.
Additionally, since $L \subset \Z^2 \subset \R^2$ is generated by two linearly independent vectors, the usual inner product $\langle \cdot,\cdot \rangle$ provides the means to identify $\varphi_t$ with ${\varphi_t}^* \in \R^2$ via $\varphi_t(u) = \langle {\varphi_t}^*,u \rangle$.
From now on, using this identification, we will not distinguish between $\R^2$ and ${(\R^2)}^*$, and similarly we will not distinguish between $\varphi_t$ and ${\varphi_t}^*$.

What can be said about the image of the homomorphism $\varphi_t$?
Of course, it is entirely defined by the values $\varphi_t$ takes on $v_0$ and on $v_1$.
If $v_i \in \mathscr{E}$, $\varphi_t(v_i) \in \Z$.
If $v_i \in \mathscr{O}$, $\varphi_t(v_i) \in \left(\Z + \frac12\right)$.
This allows us to analyze each case separately.
Consider the following sets:
\begin{equation*}
\begin{split}
{L_{00}}^* &= \left\{ \varphi \in \text{Hom}\Big(L; \tfrac12 \Z \Big) \text{ }\Big |\text{ } \varphi(v_0), \varphi(v_1) \in \Z \right\} \\
{L_{01}}^* &= \left\{ \varphi \in \text{Hom}\Big(L; \tfrac12 \Z \Big) \text{ }\Big |\text{ } \varphi(v_0) \in \Z, \text{ } \varphi(v_1) \in \Big(\Z + \tfrac12 \Big) \right\} \\
{L_{10}}^* &= \left\{ \varphi \in \text{Hom}\Big(L; \tfrac12 \Z \Big) \text{ }\Big |\text{ } \varphi(v_0) \in \Big(\Z + \tfrac12 \Big), \text{ } \varphi(v_1) \in \Z \right\} \\
{L_{11}}^* &= \left\{ \varphi \in \text{Hom}\Big(L; \tfrac12 \Z \Big) \text{ }\Big |\text{ } \varphi(v_0), \varphi(v_1) \in \Big(\Z + \tfrac12\Big) \right\}
\end{split}
\end{equation*}

Then it's readily checked that:
\begin{equation*}
\begin{split}
v_0,v_1 \in \mathscr{E} &\Rightarrow \varphi_t \in {L_{00}}^*\\
v_0 \in \mathscr{E}, v_1 \in \mathscr{O} &\Rightarrow \varphi_t \in {L_{01}}^*\\
v_0 \in \mathscr{O}, v_1 \in \mathscr{E} &\Rightarrow \varphi_t \in {L_{10}}^*\\
v_0, v_1 \in \mathscr{O} &\Rightarrow \varphi_t \in {L_{11}}^*
\end{split}
\end{equation*}

Notice that ${L_{00}}^* = \text{Hom}(L; \Z) = L^*$.
Furthermore, the sets ${L_{ij}}^*$ decompose $\text{Hom}\Big(L; \frac{1}{2} \Z \Big)$ into four disjoint and non-empty subsets.
Observe that the parities of $\varphi(2 v_0)$ and of $\varphi(2 v_1)$ provide a way to identify $\text{Hom}\Big(L; \frac{1}{2} \Z \Big)$ with $(2L)^* = \text{Hom}(2L; \Z)$.

Another description of these sets can be given in terms of a basis for $(2L)^*$.
For each $i,j = 0,1$ let $\varphi_i \in (2L)^*$ be defined by $\varphi_i(v_j) = \frac12 \delta_{ij}$.
The set $\{ \varphi_0,\varphi_1 \}$ is a basis for $(2L)^*$, and the following characterizations are immediate:
\begin{equation*}
\begin{split}
{L_{00}}^* &= \left \{ x_0 \cdot \varphi_0 + x_1 \cdot \varphi_1 \in (2L)^* \text{ } |\text{ }x_0, x_1 \in 2\Z \right \} \\
{L_{01}}^* &= \left \{ x_0 \cdot \varphi_0 + x_1 \cdot \varphi_1 \in (2L)^* \text{ } |\text{ }x_0 \in 2\Z, \text{ } x_1 \in (2\Z + 1) \right \} \\
{L_{10}}^* &= \left\{ x_0 \cdot \varphi_0 + x_1 \cdot \varphi_1 \in (2L)^* \text{ } |\text{ }x_0 \in (2\Z + 1), \text{ } x_1 \in 2\Z \right \} \\
{L_{11}}^* &= \left\{ x_0 \cdot \varphi_0 + x_1 \cdot \varphi_1 \in (2L)^* \text{ } |\text{ }x_0, x_1 \in (2\Z +1) \right \}
\end{split}
\end{equation*}

It should now be clear the sets ${L_{ij}}^*$ are related by translations of $\varphi_0$ and/or $\varphi_1$.
Since ${L_{00}}^*$ is itself a lattice, we can generally say the ${L_{ij}}^*$ are affine, or translated, lattices.
The inner product identification (like the one we did with $\varphi_t$) allows us to see this concretely, representing $(2L)^*$, and naturally also the ${L_{ij}}^*$, in $\R^2$.
Under this representation, $(2L)^* = \frac12 L^*$ and we have the chain of inclusions $$L \subset \Z^2 \subset \frac{1}{2}\Z^2 \subset (2L)^*$$

For any given valid lattice $L$, all flux values of tilings of $\T_L$ belong to one same ${L_{ij}}^*$ and no other, depending on the parity of $v_0$ and $v_1$'s coordinates.
We will call this set $L^\#$.

As an example, for the torus $\T_n$ we have $\varphi_{0}=\left(\frac{1}{4n},0\right)$ and $\varphi_{1}=\left(0,\frac{1}{4n}\right)$, so $(2L)^* \subset \R^2$ is the lattice generated by these vectors.
Moreover, in this case $L^\# = {L_{00}}^*$, so $L^\# \subset \R^2$ is the lattice generated by $\left\{\left(\frac{1}{2n},0\right), \left(0,\frac{1}{2n}\right) \right\}$.

\begin{figure}[H]
    \centering
    \def\svgwidth{0.95\columnwidth}
    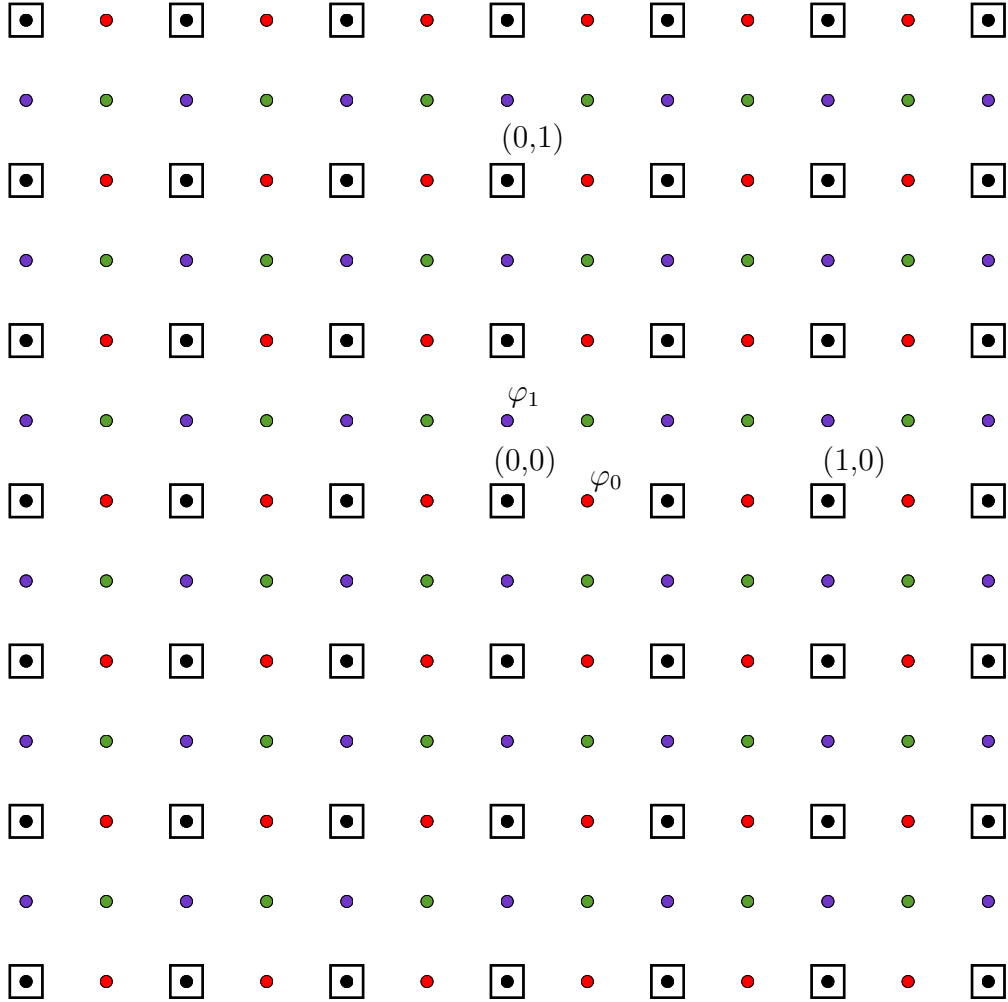
		\caption{The lattice $(2L)^*$ represented in $\R^2$. Each ${L_{ij}}^*$ corresponds to a color: ${L_{00}}^* = L^*$ is black, ${L_{10}}^*$ is red, ${L_{01}}^*$ is purple and ${L_{11}}^*$ is green. The marks round black vertices indicate $L^{\#} = {L_{00}}^*$.}
\end{figure}

\begin{prop}\label{hdelmeio}Let $L$ be a valid lattice.
Under the inner product identification, it holds that $\pm \left(\tfrac12, 0\right)$ and $\pm \left(0, \tfrac12\right)$ are in $L^\#$.
In particular, $L^\# = L^* + \left(\tfrac12,0\right)$.
\end{prop}
\begin{proof}Let $L$ be generated by $v_0 = (a,b)$ and $v_1 = (c, d)$.
Then it's easily checked that:
\begin{equation*}
\begin{split}
\varphi_0 = \frac{1}{2} \cdot \frac{1}{ad-bc} \cdot (d, -c) \\
\varphi_1 = \frac{1}{2} \cdot \frac{1}{ad-bc} \cdot (-b, a)
\end{split}
\end{equation*}

From these, we derive
\begin{equation*}
\begin{split}
&\pm \big(a \cdot \varphi_0 + c \cdot \varphi_1\big) = \pm \big(\tfrac12, 0\big) \\
&\pm \big(b \cdot \varphi_0 + d \cdot \varphi_1\big) = \pm \big(0, \tfrac12\big),
\end{split}
\end{equation*}where choice of signs is the same across a line.

Since $L$ is valid, $a,b$ and $c,d$ have the same parity, so these points are all in the same ${L_{ij}}^*$.
It suffices to see this set is $L^\#$.
\end{proof}

Notice the calculations in Proposition \ref{hdelmeio} also prove that when $v_0$ is multiplied by $k_0$ and $v_1$ is multiplied by $k_1$, $\varphi_{0}$ is multiplied by $k_0^{-1}$ and $\varphi_{1}$ is multiplied by $k_1^{-1}$.
In other words, as the moduli of $v_0$ and $v_1$ increase (but the angle between them is kept constant), the moduli of $\varphi_{0}$ and $\varphi_{1}$ decrease, and vice-versa.
Visually, this means that as $L$ becomes more scattered, $L^\#$ becomes more cluttered.

\section{Characterization of flux values}\label{sec:fluxchar}

For a valid lattice $L$, let $\mathscr{F}(L)$ be the set of all flux values of tilings of $\T_L$.
We know $\mathscr{F}(L) \subset L^\#$, but what more can be said about it? What elements of $L^\#$ are in $\mathscr{F}(L)$?
Surely not all --- $L^\#$ is infinite, and the definition of flux via counting dominoes makes it clear $\mathscr{F}(L)$ must be finite.
This section is devoted to answering these questions, and does so via a full characterization of $\mathscr{F}(L)$.

For $v = (x,y) \in \R^2$, let $\lVert v \rVert_1 = \lvert x\rvert + \lvert y\rvert$ and $\lVert v \rVert_\infty = \max\{\lvert x \rvert, \lvert y \rvert\}$.
Let $Q \subset \R^2$ be the set $\left \{v \in \R^2; \lVert v\rVert_1 \leq \tfrac12 \right \}$. 
\begin{figure}[ht]
    \centering
    \def\svgwidth{0.75\columnwidth}
    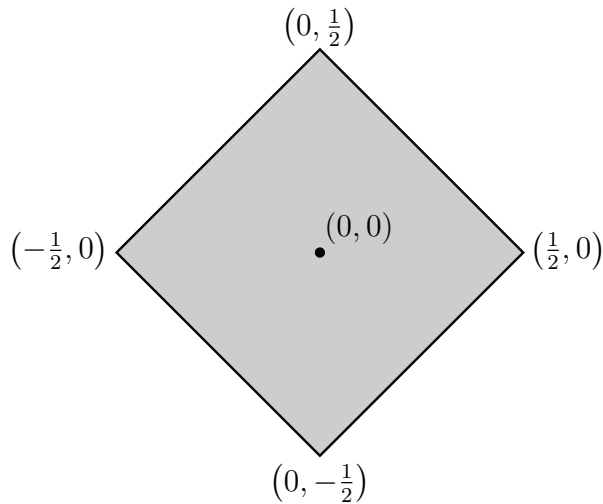
		\caption{The set $Q \subset \R^2$.}
\end{figure}

\begin{theo}[Characterization of flux values]\label{fluxcarac} $\mathscr{F}(L) = L^\# \cap Q$.
\end{theo}

Its proof will be given by Propositions \ref{diamond}  and \ref{conversediamond}, each showing one of the inclusions.

\begin{prop}\label{diamond}
For any valid lattice $L$, $\mathscr{F}(L) \subset L^\# \cap Q$.
\end{prop}

For the proof of Proposition \ref{diamond}, we will need to develop new techniques.

There is a height function $h_{\max}$ on $\Z^2$ that is maximal over height functions $h$ on $\Z^2$ with $h(0)=0$.
Before providing a characterization, recall that a finite \emph{edge-path} in a quadriculated region $R$ is a sequence of vertices $(p_n)_{n=0}^m$ in $R$ such that $p_j$ is neighbor to $p_{j+1}$ for all $j=0,\dots, m-1$;
in this case, it's clear $p_jp_{j+1}$ is an edge in $R$ joining those two vertices.
We say an edge-path $(p_n)_{n=0}^m$ joins $p_0$ (its starting point) to $p_m$ (its endpoint) and has length $m$.
We will also consider infinite edge-paths: those with no starting point, those with no endpoint, and those with neither a starting point nor an endpoint.
In the last case, we say the edge-path is \emph{doubly-infinite}.
Notice the ordering of an edge-path's vertices imbues its edges with a natural orientation, and it need not agree with the natural orientation of $R$'s edges (induced by the coloring).

Given a tiling $t$ of $R$, an edge-path in $t$ is an edge-path in $R$ such that each of its edges are in $t$ (that is, none of its edges cross a domino in $t$).

For $v,w \in \Z^2$, let $\Gamma(v,w)$\label{def:gammavw} be the set of all edge-paths in $\Z^2$ joining $v$ to $w$ that respect edge orientation (as induced by the coloring of $\Z^2$).

Finally, let $\mathbcal{H}_0(R)$\label{def:h0r} be the set of height functions $h$ on $R$ with $h(0)=0$.
We are now ready to state the characterization.

\begin{prop}[Characterization of $h_{\max}$]\label{hmax}
Consider the infinite black-and-white square lattice $\Z^2$ (with $[0,1]^2$ black) and let $h_{\max} \in \mathbcal{H}_0(\Z^2)$ be its maximal height function.
Then $$h_{\max}(v) = \min\limits_{\gamma \in \Gamma(0,v)} l(\gamma),$$ where $l(\gamma)$ is the length of $\gamma$.
\end{prop}
\begin{proof}
Fix $v \in \Z^2$.
We claim that for all $\gamma \in \Gamma(0,v)$ and for all $h \in \mathbcal{H}_0(\Z^2)$ it holds that $h(v) \leq l(\gamma)$.
Indeed, the constructive definition of height functions implies that whenever an edge in $\gamma$ is traversed, $h$ changes by $+1$ if that edge is on $t$ and by $-3$ otherwise, so an induction on the length of $\gamma$ justifies the claim.

Since $\Gamma(0,v)$ is never empty, it follows that any $h \in \mathbcal{H}_0( \Z^2)$ satisfies $h(v) \leq \min_{\gamma \in \Gamma(0,v)} l(\gamma)$.
Letting go of the requirement that $v \in \Z^2$ be fixed, this inequality then holds for all $v \in \Z^2$.

Now define $h_M: \Z^2 \longrightarrow \Z$ to be the function given by $h_M(v) = \min_{\gamma \in \Gamma(0,v)} l(\gamma)$, so that $h(v) \leq h_M(v)$ for all $v \in \Z^2$ and for all $h \in \mathbcal{H}_0(\Z^2)$.
If we show that $h_M \in \mathbcal{H}_0(\Z^2)$, it follows immediately that $h_M = h_{\max}$ and the proposition is proved.

By inspection, $h_M(0) = 0$.
Using Proposition \ref{hsquare}, it's easy to verify $h_M$ is a height function.
Indeed, property 1 follows from the fact that edge-paths in $\Gamma(0,v)$ respect edge orientation.
For property 2, it suffices to check that any two neighboring vertices in $\Z^2$ can always be joined by an edge-path that respects edge orientation and has length at most three: either the edge joining those two vertices, or the edge-path going round a square that contains those two vertices.\end{proof}

There is elegance to the simplicity of this rather abstract proof, but it does little to shed light on the structure and properties of $h_{\max}$; our next proposition addresses this.
In addition, we provide an image of $h_{\max}$ along with its associated tiling $t_{\max}$; see Figure \ref{fig:hmax}.
\begin{figure}[ht]
    \centering
    \def\svgwidth{0.8\columnwidth}
    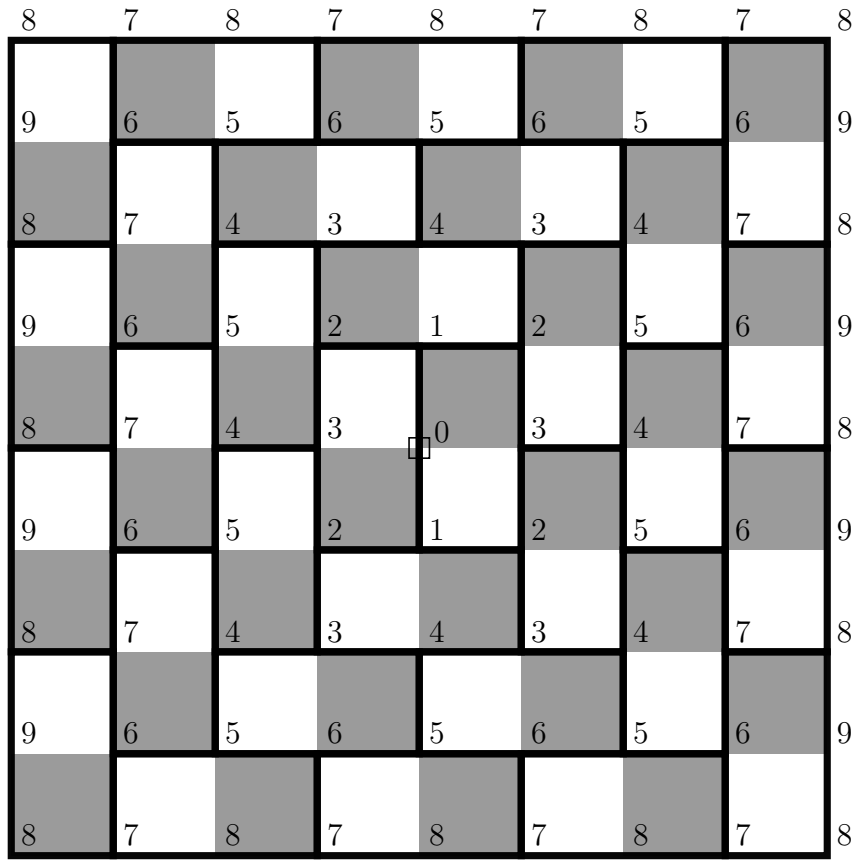
		\caption{The tiling $t_{\max}$ and its associated height function $h_{\max}$. The marked vertex is the origin. Notice its only local extremum is the origin, a minimum, and it is not the height function of any torus (since it is not quasiperiodical).}
		\label{fig:hmax}
\end{figure}

\begin{prop}\label{hmaxfor}
Let $v=(x_1,x_2) \in \Z^2$.
If $x_1 \equiv x_2 \textnormal{ (mod $2$)}$, then
\begin{equation}\label{hmaxmod2}
h_{\max}(v) = 2 \cdot \lVert v \rVert_\infty.
\end{equation}

More generally, it holds that 
\begin{equation}\label{hmaxgen}
\big| h_{\max}(v) - 2 \cdot \lVert v \rVert_\infty \big| \leq 1.
\end{equation} 
\end{prop}
\begin{proof}
The idea is to describe edge-paths $\gamma \in \Gamma(0,v)$ with minimal length.
Because of Proposition \ref{hmax}, the constructive definition of height functions implies any such $\gamma$ is an edge-path not only in $\Z^2$, but also in $t_{\max}$.
The explicit construction of these paths will allow us to derive relations~\eqref{hmaxmod2} and~\eqref{hmaxgen}.

\paragraph{}We introduce the concept of \textit{edge-profiles}\label{def:edgeprofile} round a vertex.
When horizontal edges round a vertex point toward it and vertical edges round that vertex point away from it, we say the edge-profile round that vertex is type-0.
When horizontal edges round a vertex point away from it and vertical edges round that vertex point toward it, we say the edge-profile round that vertex is type-1.
It's clear those are the only possible cases, see the image below.
\begin{figure}[H]
		\centering
		\includegraphics[width=0.55\textwidth]{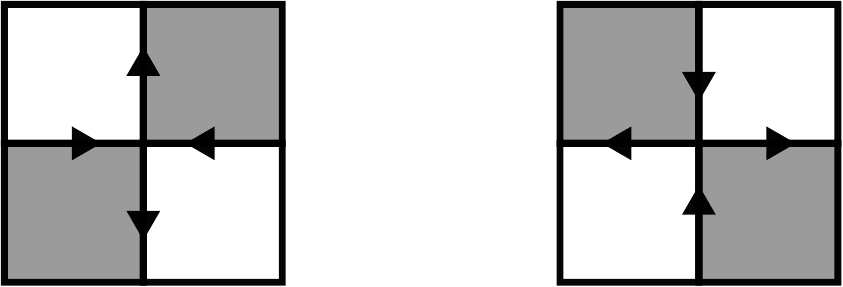}
		\caption{The two edge-profiles; type-0 to the left and type-1 to the right.}
\end{figure}

Notice the edge-profile round a vertex depends only on the region (and not on a tiling of the region).
Moreover, two neighbouring vertices will always have distinct edge-profiles, so that any edge-path on a region will always feature successive vertices with alternating edge-profiles.

This means an edge-path that respects edge orientation will necessarily alternate between vertical and horizontal edges, correspondingly as that edge emanates from a vertex with edge-profile respectively type-0 and type-1.
On the other hand, whenever an edge-path alternates between vertical and horizontal edges, it either always respects orientation (if vertical edges emanate from vertices with edge-profile type-0) or always reverses orientation (if vertical edges emanates from vertices with edge-profile type-1).
This is the content of Corollary \ref{pathorien} below.

We can now characterize edge-paths in $\Gamma(0,v)$.
Since the edge-profile round the origin in $\Z^2$ (as we have colored it) is type-0, any edge-path in $\Gamma(0,v)$ is an alternating sequence of vertical and horizontal edges, starting from the origin with a vertical edge and ending in $v$.

Consider then the vectors $e_1=(1,0)$ and $e_2=(0,1)$.
By the characterization, any edge-path in $\Gamma(0,v)$ can be uniquely represented as an ordered sum of $\pm e_i$ in which the first term is either $e_2$ or $-e_2$ and no two consecutive terms are collinear vectors.
It's that clear the length of an edge-path in this representation is simply the number of terms in the ordered sum.
Furthermore, because an edge-path in $\Gamma(0,v)$ starts at the origin, if we carry out the sum of this unique representation the result \textbf{is} in fact the vector $v \in \Z^2$.

How does the ordered sum representation\label{def:orderedsumrepresentation} of a path $\gamma \in \Gamma(0,v)$ with minimal length look like?
Let $v = x_1\cdot e_1+x_2\cdot e_2$ be a vertex in $\Z^2$ and $i,j \in \{1,2\}$ be different indices with $\lvert x_j \rvert \geq \lvert x_i \rvert$.
The ordered sum representing $\gamma$ will have exactly $\lvert x_j \rvert$ terms of the form $\pm e_j$, all of them with sign given by $\sgn(x_j)$.
Notice they add up to $x_j \cdot e_j$, and no smaller number of $\pm e_j$ terms does so.

Similarly, the ordered sum will feature $\lvert x_i \rvert$ terms of the form $\pm e_i$, all of them with sign given by $\sgn(x_i)$, adding up to $x_i \cdot e_i$.
Because $\lvert x_j \rvert \geq \lvert x_i \rvert$, in order for the ordered sum to fulfill the requirement that it be alternating in $\pm e_1$ and $\pm e_2$, it must have a number $m$ (possibly zero) of additional $\pm e_i$ terms.
Since the ordered sum starts with a $\pm e_2$ term and sums to $v$, $m$ is uniquely defined.

It is clear that whenever $\gamma$ has an ordered sum representation described as above, $\gamma \in \Gamma(0,v)$.
Additionally, no path in $\Gamma(0,v)$ may have smaller length, for the unique ordered sum representation was chosen to have the smallest possible number of terms.
By Proposition \ref{hmax}, $h_{\max}(v) = l(\gamma)$. We provide a example of this construction for $v = (4,-1)$ in Figure \ref{fig:ordsum}.
\begin{figure}[H]
		\vspace{2cm}
    \centering
    \def\svgwidth{\columnwidth}
    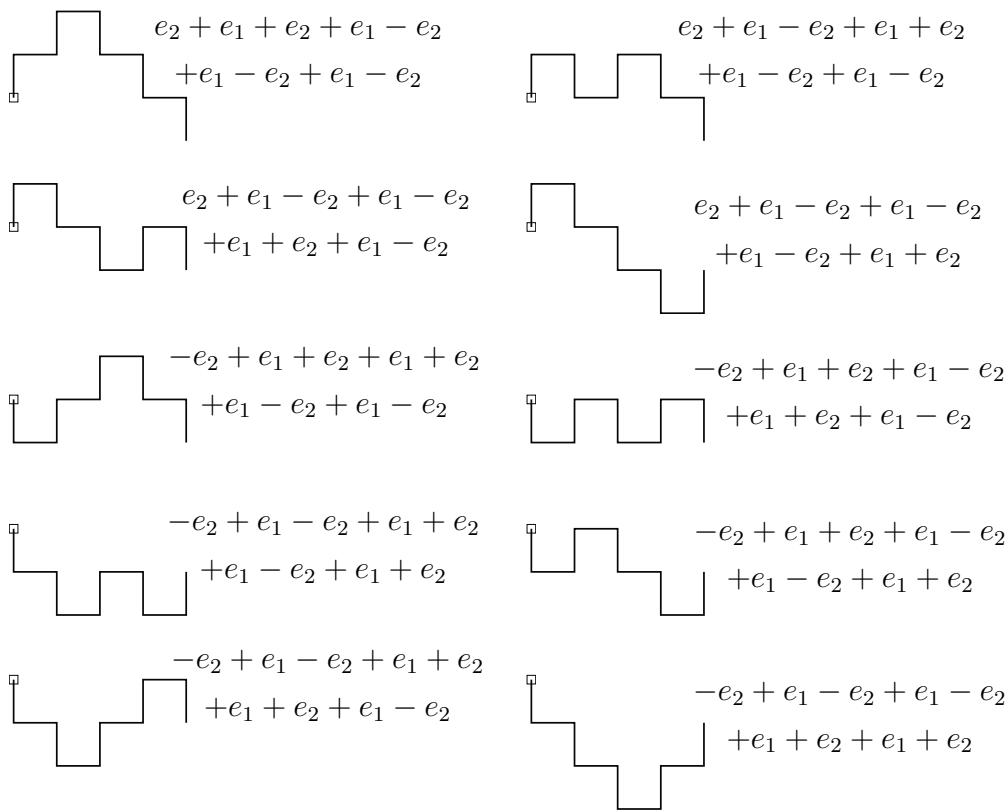
		\caption{The paths in $\Gamma\big(0,(4,-1)\big)$ with minimal length, along with their ordered sum representation. The marked vertex is the origin.}
		\label{fig:ordsum}
\end{figure}

Notice this analysis ensures all of $\gamma$'s horizontal edges or all of $\gamma$'s vertical edges have the same orientation (possibly both); see Figure \ref{fig:sameorient}.
This fact will be used in Lemma \ref{lemapath} later.
\begin{figure}[ht]
\vspace{0.7cm}
    \centering
    \def\svgwidth{0.6\columnwidth}
    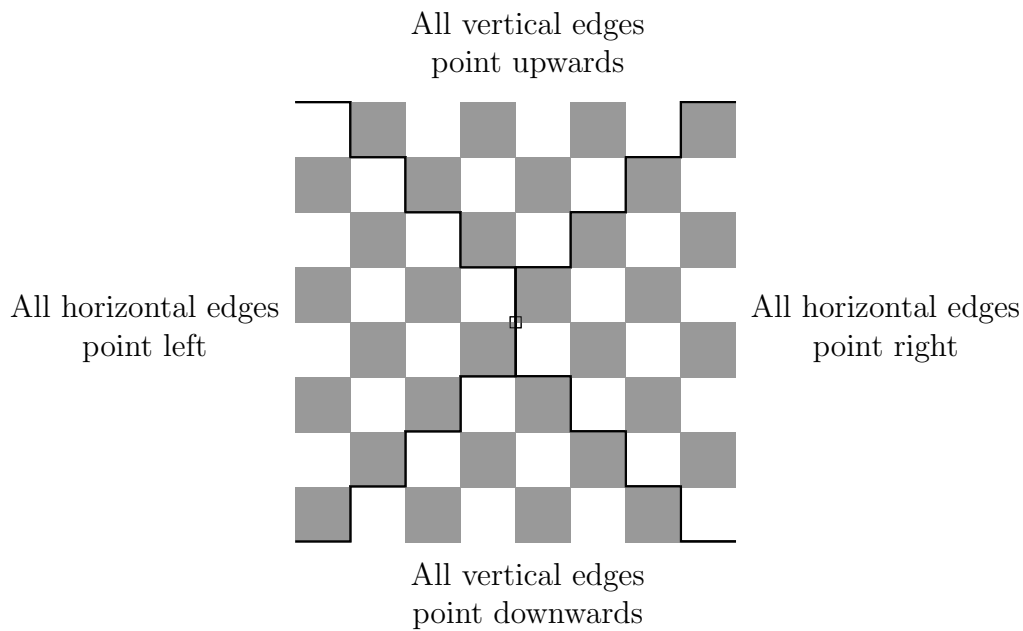
		\vspace{0.7cm}
		\caption{For each of the four regions above, if $v$ belongs to that region, the edges of any $\gamma \in \Gamma(0,v)$ with minimal length satisfy the corresponding property. The marked vertex is the origin.}
		\label{fig:sameorient}
\end{figure}

In the construction above, $m$ is always even.
Indeed, the number of plus signs and the number of minus signs in the additional $m$ terms of the form $\pm e_i$ must be equal, for otherwise they would not add up to 0.
When $x_1 \equiv x_2 \Mod{2}$, $\lvert x_j \rvert - \lvert x_i \rvert$ is even, and in this case it's easy to see we can take $m = \lvert x_j \rvert - \lvert x_i \rvert$.
This implies the ordered sum representation has a total of $2 \cdot \lvert x_j \rvert$ terms, so formula~\eqref{hmaxmod2} is proved.

It remains to prove inequality~\eqref{hmaxgen}.
Formula~\eqref{hmaxmod2} means it trivially holds whenever $x_1 \equiv x_2 \Mod{2}$, so we need only check when $x_1$ and $x_2$ have different mod 2 values.
In particular, we may assume $\lvert x_j \rvert > \lvert x_i \rvert$ (the inequality is strict).

Let $\gamma \in \Gamma(0,v)$ have minimal length.
Consider the edge-path $\tilde{\gamma}$ obtained from $\gamma$ by removing its last edge $e$. It is an edge-path in $\Gamma(0, v-e)$ with minimal length, for otherwise $\gamma \in \Gamma(0,v)$ would not have minimal length. This implies the equality $h_{\max}(v) = h_{\max}(v-e) + 1$.

Write $v-e = (y_1,y_2)$.
Observe that $v-e$ is obtained from $v$ by changing one of its coordinates by $\pm 1$.
Since $x_1$ and $x_2$ have different mod 2 values, it follows that $y_1 \equiv y_2 \Mod{2}$, and formula~\eqref{hmaxmod2} applies: $h_{\max}(v-e) = 2 \cdot \max \{ \lvert y_1 \rvert, \lvert y_2 \rvert \}$.
Combining the two equalities yields
\begin{equation}\label{hmaxineq}
h_{\max}(v) - 2 \cdot \max \{ \lvert y_1 \rvert, \lvert y_2 \rvert \} = 1
\end{equation}

There are two cases: (1) $e$ is of the form $\pm e_i$; and (2) $e$ is of the form $\pm e_j$.

In case (1), $v$'s $x_i$ coordinate is changed by $\pm 1$, so $\lvert y_j \rvert =\lvert x_j \rvert \geq \lvert y_i \rvert$.
Substituting into~\eqref{hmaxineq}, the inequality holds.

In case (2), because $\lvert x_j \rvert > \lvert x_i \rvert$, all of $\gamma$'s edges of the form $\pm e_j$ have the same orientation.
This implies $\lvert y_j \rvert =\lvert x_j \rvert - 1 \geq \lvert x_i \rvert = \lvert y_i \rvert$.
Once again, substituting into~\eqref{hmaxineq} the inequality holds, and we are done.
\end{proof}

\begin{corolario}\label{pathorien}Let $R$ be a planar region and $\gamma$ an (oriented) edge-path in $R$.
Then the following are equivalent:
\begin{itemize}
\item $\gamma$ always respects or always reverses edge orientation (as induced by the coloring of $R$);
\item $\gamma$'s edges alternate between horizontal and vertical.
\end{itemize}
\end{corolario}
\begin{proof}See Proof of Proposition \ref{hmaxfor}.
\end{proof}

For $v,w \in \Z^2$, let $\Psi(v,w)$\label{def:psivw} be the set of all edge-paths in $\Z^2$ joining $v$ to $w$ that reverse edge orientation.
The techniques used to obtain the characterization of $h_{\max}$ can be very similarly employed to obtain a characterization of the minimal height function $h_{\min}$ on $\mathbcal{H}_0(\Z^2)$, and derive analogous results.

\begin{corolario}[Characterization of $h_{\min}$]\label{hminplano}
Consider the infinite black-and-white square lattice $\Z^2$ (with $[0,1]^2$ black) and let $h_{\min} \in \mathbcal{H}_0(\Z^2)$ be its minimal height function.
Then $$h_{\min}(v) = - \left( \min\limits_{\gamma \in \Psi(0,v)} l(\gamma)\right),$$ where $l(\gamma)$ is the length of $\gamma$.
Furthermore, if $v=(x_1,x_2) \in \Z^2$ and $x_1 \equiv x_2 \textnormal{ (mod $2$)}$, then
\begin{equation*}
h_{\min}(v) = -2 \cdot \lVert v \rVert_\infty.
\end{equation*}

More generally, it holds that 
\begin{equation*}
\big| h_{\min}(v) + 2 \lVert v \rVert_\infty \big| \leq 1.
\end{equation*} 
\end{corolario}

\begin{proof}Similar to the proofs of Propositions \ref{hmax} and \ref{hmaxfor}.
\end{proof}

Before proving Proposition \ref{diamond}, we will need a quick lemma.

\begin{lema}\label{xx}Let $v,w$ be linearly independent vectors in $\Z^2$.
Then for each choice of signs in $(\pm x, \pm x)$, there is a nonzero integer linear combination of $v,w$ with that form.
\end{lema}
\begin{proof}Let $v = (a,b)$ and $w = (c,d)$.
It suffices to prove for $(x,x)$ and $(x,-x)$.

For the $(x,x)$ case, take $k=-c+d$ and $l=a-b$, so that $k\cdot v + l \cdot w = (ad-bc, ad-bc)$.
For the $(x,-x)$ case, take $k = c+d$ and $l=-a-b$, so that $k\cdot v + l \cdot w = (ad-bc, -(ad-bc))$.
In each case, the combination uses integer coefficients, and it is nonzero because $v$ and $w$ are linearly independent.
\end{proof}

We are now ready to prove Proposition \ref{diamond}.

\begin{proof}[Proof of Proposition \ref{diamond}]
Let $L$ be a valid lattice and $t$ a tiling of $\T_L$ with flux $\varphi_t$ and height function $h_t$.
It suffices to show that $\varphi_t \in Q$.

For any $(x,y) \in L$, we have that $\varphi_t(x,y)=\frac{1}{4}h_t(x,y)$.
Of course, this means $\varphi_t(x,y) \leq \frac{1}{4}h_{\max}(x,y)$ for any $(x,y) \in L$.
Because $L$ is a valid lattice, it is generated by two linearly independent vectors.
By Lemma \ref{xx}, for any choice of signs in $(\pm x, \pm x)$, there is a vector in $L$ with that form.
Proposition \ref{hmaxfor} then implies $\varphi_t(\pm x,\pm x) \leq \frac{1}{2}\lvert x \rvert$, so $\langle \varphi_t,(\pm 1, \pm 1) \rangle \leq \frac{1}{2}$.
Writing $\varphi_t = (x_t, y_t)$, it then holds that $\pm x_t \pm y_t \leq \frac{1}{2}$.

In particular, there is a choice of signs in the previous inequality that yields $\lVert \varphi_t \rVert_1 = \lvert x_t \rvert + \lvert y_t \rvert \leq \frac{1}{2}$, so $\varphi_t \in Q$ as desired.
\end{proof}

We now provide the remaining inclusion in Theorem \ref{fluxcarac}.

\begin{prop}\label{conversediamond}
For any valid lattice $L$, $\mathscr{F}(L) \supset L^\# \cap Q$.
\end{prop}

Before proving it, we will need a few lemmas.

\begin{lema}\label{hdelmod4}
Let $L$ be a valid lattice.
For all $v \in L$ and $\varphi \in L^\#$ it holds that$$4 \cdot \langle \varphi , v \rangle \equiv \Phi(v) \text{ mod 4,}$$where $\Phi$ is the mod 4 prescription function on the infinite square lattice\footnote{$\Phi$ is calculated just after Proposition \ref{hsquare}.}.
\end{lema}
\begin{proof}
Suppose $L$ is generated by $v_0=(x_0,y_0)$ and $v_1=(x_1,y_1)$.
Given $v \in L$, there are unique integers $a$ and $b$ with $v = a \cdot v_0 + b \cdot v_1$, so that $v=(ax_0 +bx_1, ay_0 + by_1)$.

Similarly, given $\varphi \in L^\#$, there are unique integers $z_0$ and $z_1$ with $\varphi = z_0 \cdot \varphi_0 + z_1 \cdot \varphi_1$.
We may then write
\begin{equation}\label{fluxmod4}
\begin{split}
4 \cdot \langle \varphi , v \rangle &= 4\big(az_0 \cdot \langle \varphi_0 , v_0 \rangle +  bz_1 \cdot \langle \varphi_1 , v_1 \rangle \big) \\
&= 2(az_0 + bz_1)
\end{split}
\end{equation}

Notice $x_0 \equiv y_0 \equiv z_0 \Mod{2}$, because $L$ is valid and $\varphi \in L^\#$.
By the same token, $x_1 \equiv y_1 \equiv z_1 \Mod{2}$.
Moreover, $L$ being valid implies $v$'s coordinates have the same parity.
We now analyze the mod 4 value of the expression in~\eqref{fluxmod4} for each case.

Suppose first that $v$'s coordinates are both even, that is, $ax_0 +bx_1 \equiv ay_0 + by_1 \equiv 0 \Mod{2}$.
We must show $2(az_0 + bz_1) \equiv 0 \Mod{4}$, or equivalently $az_0 + bz_1 \equiv 0 \Mod{2}$.
This is implied by the mod 2 equivalences between $x_0$, $y_0$ and $z_0$, and between $x_1$, $y_1$ and $z_1$, so we are done.

Suppose now that $v$'s coordinates are both odd, that is, $ax_0 +bx_1 \equiv ay_0 + by_1 \equiv 1 \Mod{2}$.
We must show $2(az_0 + bz_1) \equiv 2 \Mod{4}$, or equivalently $az_0 + bz_1 \equiv 1 \Mod{2}$.
Once again, this is implied by the mod 2 equivalences, and the proof is complete.
\end{proof}

\begin{lema}\label{mindiamond}
Let $L$ be a valid lattice and $\varphi \in L^\#$.
For each $w \in \Z^2$, consider the expression $$\min\limits_{v \in L} \Bigg( 4 \cdot \langle \varphi, v \rangle + \min\limits_{\gamma \in \Gamma(v,w)}l(\gamma) \Bigg).$$

The minimum exists if and only if $\varphi \in Q$.
\end{lema}
\begin{proof}
First, observe that Lemma \ref{hdelmod4} guarantees the minimum is taken over integer-valued expressions; in other words, the existence of the minimum is equivalent to the existence of a lower bound.

On the one hand, the following estimate holds whenever $v \in L$:
\begin{equation}\label{estvarphi}
-\lVert \varphi \rVert_1 \cdot \lVert v \rVert_{\infty} \leq \langle \varphi, v \rangle \leq \lVert \varphi \rVert_1 \cdot \lVert v \rVert_{\infty}
\end{equation}

On the other, $\min_{\gamma \in \Gamma(v,w)}l(\gamma) = \min_{\gamma \in \Gamma(0,w-v)}l(\gamma) = h_{\max}(w-v)$, because $v$'s coordinates have the same parity (it is in $L$), so a translation by $v$ preserves orientation.
Proposition \ref{hmaxfor} then implies that for all $v \in L$
\begin{equation}\label{estmin}
\begin{split}
\min\limits_{\gamma \in \Gamma(v,w)}l(\gamma) &\geq 2 \cdot \max\{|w_1-v_1|,|w_2-v_2|\} - 1 \\
&\geq 2 \cdot \max\big\{\big||w_1|-|v_1|\big|,\big||w_2|-|v_2|\big|\big\} - 1.
\end{split}
\end{equation}

Consider the set $R(w) = \{(x,y) \in \R^2 \text{; $|x| \leq |w_1|$ and $|y| \leq |w_2|$}\}$.
Because it is bounded and $L$ is discrete, $R(w) \cap L$ is finite.
Thus, we need only show $4 \cdot \langle \varphi, v \rangle + \min_{\gamma \in \Gamma(v,w)}l(\gamma)$ has a lower bound for $v \in L$ outside $R(w)$.
In this situation, inequality~\eqref{estmin} allows us to write
\begin{equation*}
\begin{split}
\min\limits_{\gamma \in \Gamma(v,w)}l(\gamma) &\geq 2 \cdot \max\{|v_1|-|w_1|,|v_2|-|w_2|\} - 1 \\
&\geq 2 \cdot \max\{|v_1|,|v_2|\} - 2\cdot \max\{|w_1|,|w_2|\} - 1 \\
& = 2 \cdot \lVert v \rVert_{\infty} - 2 \cdot \lVert w \rVert_{\infty} - 1.
\end{split}
\end{equation*}

Combining the two yields for all $v \in L$ the estimate
\begin{equation*}
4 \cdot \langle \varphi, v \rangle + \min\limits_{\gamma \in \Gamma(v,w)}l(\gamma) \geq \big(2-4 \cdot \lVert \varphi \rVert_1\big)\cdot \lVert v \rVert_{\infty} - 2 \cdot \lVert w \rVert_{\infty} - 1,
\end{equation*}
so that when $\varphi \in Q$ a lower bound outside $R(w)$ is given by $- (2 \cdot \lVert w \rVert_{\infty} + 1)$.
Similar manipulations show that when $\varphi \in Q$, $- (2 \cdot \lVert w \rVert_{\infty} + 1)$ is a lower bound everywhere.
In other words, when $\varphi \in Q$, the minimum exists.

To complete the proof, we show that when $\varphi \notin Q$, there is no lower bound.
Observe that Lemma \ref{xx} guarantees the existence of a vertex $\tilde{v} \in L$ of the form $(\pm x, \pm x)$ and such that for all $n \in \Z$ it holds that $\langle \varphi, n \cdot \tilde{v} \rangle = -n\cdot \lVert \varphi \rVert_1 \cdot \lvert x \rvert$.
When $\varphi \notin Q$, there is some $\epsilon > 0$ for which $\lVert \varphi \rVert_1 > \frac12 + \frac{\epsilon}{4}$;
in this case, we have $4\cdot \langle \varphi, n \cdot \tilde{v} \rangle < -2n\lvert x \rvert - n\epsilon \lvert x \rvert$ for all positive integers $n$.

As before, Proposition \ref{hmaxfor} can be used to show the following estimate must hold outside $R(w)$: $$\min\limits_{\gamma \in \Gamma(n\cdot \tilde{v},w)}l(\gamma) \leq 2n \cdot \lvert x \rvert - 2\cdot \lVert w \rVert_{\infty} + 1$$

Thus, when $\varphi \notin Q$, it holds that for all positive integers $n$ $$4\cdot \langle \varphi, n \cdot \tilde{v} \rangle + \min\limits_{\gamma \in \Gamma(n\cdot \tilde{v},w)}l(\gamma) < - n\epsilon \lvert x \rvert - 2\cdot \lVert w \rVert_{\infty} + 1.$$

For fixed $w \in \Z^2$, we see the expression has no lower bound as $n$ tends to infinity, so the proof is complete.
\end{proof}

We are now ready to prove Proposition \ref{conversediamond}.
\begin{proof}[Proof of Proposition \ref{conversediamond}]
Let $\varphi \in L^\# \cap Q$.
We will construct the height function $h_{\max}^{L,\varphi}$\label{def:hLphimax} that is maximal over toroidal height functions of $\T_L$ with flux $\varphi$ (and base value 0 at the origin).

First, observe that if $h$ is a toroidal height function of $\T_L$ with flux $\varphi$, then $h(v) = 4 \cdot \varphi(v) = 4 \cdot \langle \varphi, v \rangle$ for all $v \in L$.
Consider for each $v \in L$ the height function $h_{\max}^{v, \varphi}$\label{def:hvphimax}, maximal over height functions that take the value $4 \cdot \langle \varphi, v \rangle$ on $v$ (notice they need not have the value 0 on the origin).
An easy adaptation of Proposition \ref{hmax} yields $$h_{\max}^{v, \varphi}(w) = 4 \cdot \langle \varphi, v \rangle + \min\limits_{\gamma \in \Gamma(v,w)}l(\gamma).$$

As in Lemma \ref{mindiamond}, $\min_{\gamma \in \Gamma(v,w)}l(\gamma) = \min_{\gamma \in \Gamma(0,w-v)}l(\gamma) = h_{\max}(w-v)$, so we also have $$h_{\max}^{v, \varphi}(w) = 4 \cdot \langle \varphi, v \rangle + h_{\max}(w-v).$$

If $h$ is a toroidal height function of $\T_L$ with flux $\varphi$, it follows that $h(w) \leq h_{\max}^{v, \varphi}(w)$ for all $w \in \Z^2$ and $v \in L$.
By Lemma \ref{mindiamond}, the function given by $h_{\max}^{L,\varphi}(w) = \min_{v \in L}h_{\max}^{v, \varphi}(w)$ is well defined, which implies $h(w) \leq h_{\max}^{L,\varphi}(w)$ for all $w \in \Z^2$.
We claim $h_{\max}^{L,\varphi}$ is a toroidal height function of $\T_L$ with flux $\varphi$; in this case, clearly it is maximal over such height functions.

We first prove it is a height function that takes the value 0 on the origin, as characterized by Proposition \ref{hsquare}.

\paragraph{}\indent \indent \textbf{Item 1.} $h_{\max}^{L,\varphi}(0)=0$.

Since $0 \in L$ and by inspection $h_{\max}^{0, \varphi}(0) = 0$, we have the inequality  $h_{\max}^{L,\varphi}(0) \leq 0$.
We then need only show $h_{\max}^{v, \varphi}(0) \geq 0$ for all $v \in L$.
Since for each $v \in L$ we have $h_{\max}^{v, \varphi}(0) = 4 \cdot \langle \varphi, v \rangle + h_{\max}(-v)$, the equivalent inequality $4 \cdot \langle \varphi, v \rangle \geq -h_{\max}(-v)$ suffices.
Now, because $v \in L$ and $L$ is valid, Proposition \ref{hminplano} implies $-h_{\max}(-v) = -2 \cdot \lVert v \rVert_{\infty}$.
The inequality follows from applying estimate~\eqref{estvarphi} in Lemma \ref{mindiamond} (remember $\varphi \in Q$).

\paragraph{}\indent \indent \textbf{Item 2.} $h_{\max}^{L,\varphi}$ has the prescribed mod 4 values on all of $\Z^2$.

We show that $h_{\max}^{v, \varphi}$ satisfies this condition for all $v \in L$, from which the claim follows.
Indeed, Lemma \ref{hdelmod4} implies $h_{\max}^{v, \varphi}$ respects the condition on $L$.
To see this holds on all of $\Z^2$, it suffices to note the edge-paths in $\min_{\gamma \in \Gamma(v,w)}l(\gamma)$ respect edge orientation.

\paragraph{}\indent \indent \textbf{Item 3.} $h_{\max}^{L,\varphi}$ changes by at most 3 along an edge on $\Z^2$.

Let $e$ be an edge on $\Z^2$ joining $w_1$ to $w_2$  (in the orientation induced by the coloring of $\Z^2$).
We claim $h_{\max}^{L,\varphi}(w_2) \leq h_{\max}^{L,\varphi}(w_1) +1$.
Indeed, there is some $v \in L$ and $\gamma \in \Gamma(v,w_1)$ with $h_{\max}^{L,\varphi}(w_1) = 4 \cdot \langle \varphi, v \rangle + l(\gamma)$.
Consider the path $\widetilde{\gamma} = \gamma * e$, where $*$ is edge-path concatenation.
It is clear $\widetilde{\gamma} \in \Gamma(v,w_2)$.
Thus, it follows that $h_{\max}^{L,\varphi}(w_2) \leq h_{\max}^{v, \varphi}(w_2) \leq 4 \cdot \langle \varphi, v \rangle + l(\widetilde{\gamma})$.
Since $l(\widetilde{\gamma})=l(\gamma) + 1$, the claim holds.

Finally, we claim $h_{\max}^{L,\varphi}(w_1) \leq h_{\max}^{L,\varphi}(w_2) +3$.
Like before, there is some $\tilde{v} \in L$ and $\beta \in \Gamma(\tilde{v},w_2)$ with $h_{\max}^{L,\varphi}(w_2) = 4 \cdot \langle \varphi, \tilde{v} \rangle + l(\beta)$.
Consider the path $\widetilde{\beta} = \beta * \widetilde{e}$, where $\widetilde{e}$ is the edge-path joining $w_2$ to $w_1$ that goes round a square containing $e$.
Observe that $\widetilde{e}$ respects edge-orientation, so $\widetilde{\beta} \in \Gamma(\tilde{v},w_1)$.
It follows that $h_{\max}^{L,\varphi}(w_1) \leq h_{\max}^{\tilde{v}, \varphi}(w_1) \leq 4 \cdot \langle \varphi, \tilde{v} \rangle + l(\widetilde{\beta})$.
Since $\widetilde{e}$ has length 3, $l(\widetilde{\beta}) = l(\beta) + 3$ and the claim holds.

Together, both inequalities prove (3) above.

We have thus shown that $h_{\max}^{L,\varphi}$ is a height function; it remains to show it is $L$-quasiperiodic with flux $\varphi$.
For the $L$-quasiperiodicity, we will prove that for all $v \in L$ and $w_1,w_2 \in \Z^2$ $$h_{\max}^{L,\varphi}(w_1+v) - h_{\max}^{L,\varphi}(w_1) = h_{\max}^{L,\varphi}(w_2+v) - h_{\max}^{L,\varphi}(w_2)$$

To that end, observe that any $u \in L$ can be written as $\tilde{u} + v$, so
\begin{equation*}
\begin{split}
h_{\max}^{L,\varphi}(w_i+v) &= \min\limits_{(\tilde{u} + v) \in L} \Big( 4 \cdot \langle \varphi, \tilde{u} + v \rangle + h_{\max}(w_i - \tilde{u})\Big) \\
& = \min\limits_{(\tilde{u} + v) \in L} \Big( 4 \cdot \langle \varphi, \tilde{u} \rangle + h_{\max}(w_i - \tilde{u})\Big) + 4 \cdot \langle \varphi, v \rangle \\
& = \text{ }\min\limits_{\tilde{u} \in L} \text{ }\text{ }\text{ }\Big( 4 \cdot \langle \varphi, \tilde{u} \rangle + h_{\max}(w_i - \tilde{u})\Big) + 4 \cdot \langle \varphi, v \rangle \\
& = \text{ }\text{ } h_{\max}^{L,\varphi}(w_i) + 4 \cdot \langle \varphi, v \rangle
\end{split}
\end{equation*}

Now, because $h_{\max}^{L,\varphi}(0) = 0$, this also shows that $h_{\max}^{L,\varphi}(v) = 4 \cdot \langle \varphi, v \rangle$ for all $v \in L$, so $h_{\max}^{L,\varphi}$ has flux $\varphi$ and the proof is complete.
\end{proof}

Combining Propositions \ref{diamond} and \ref{conversediamond}, we obtain the full characterization provided by Theorem \ref{fluxcarac} at the beginning of this section.
\chapter{Flip-connectedness on the torus}
\label{chap:fliptorus}

We would like to prove a flux-analogue of Proposition \ref{hredux}.
This would allow us to use flip-connectedness in studying the properties of Kasteleyn matrices for the torus, in a manner similar to the planar case.
However, it turns out that for extremal values of the flux --- those lying on the boundary\footnote{Notice that Proposition \ref{hdelmeio} implies $L^\#$ always intersects $\partial Q$.} of $Q$ ---, domino tilings of the torus with that flux value are not flip-connected.
In fact, we will see that each tiling with an extremal flux value is a flip-isolated point in the space of domino tilings of the torus.

The image below features two tilings of the torus with identical flux values, but notice none of them admits any flip at all.
\begin{figure}[H]
		\centering
		\includegraphics[width=0.7\textwidth]{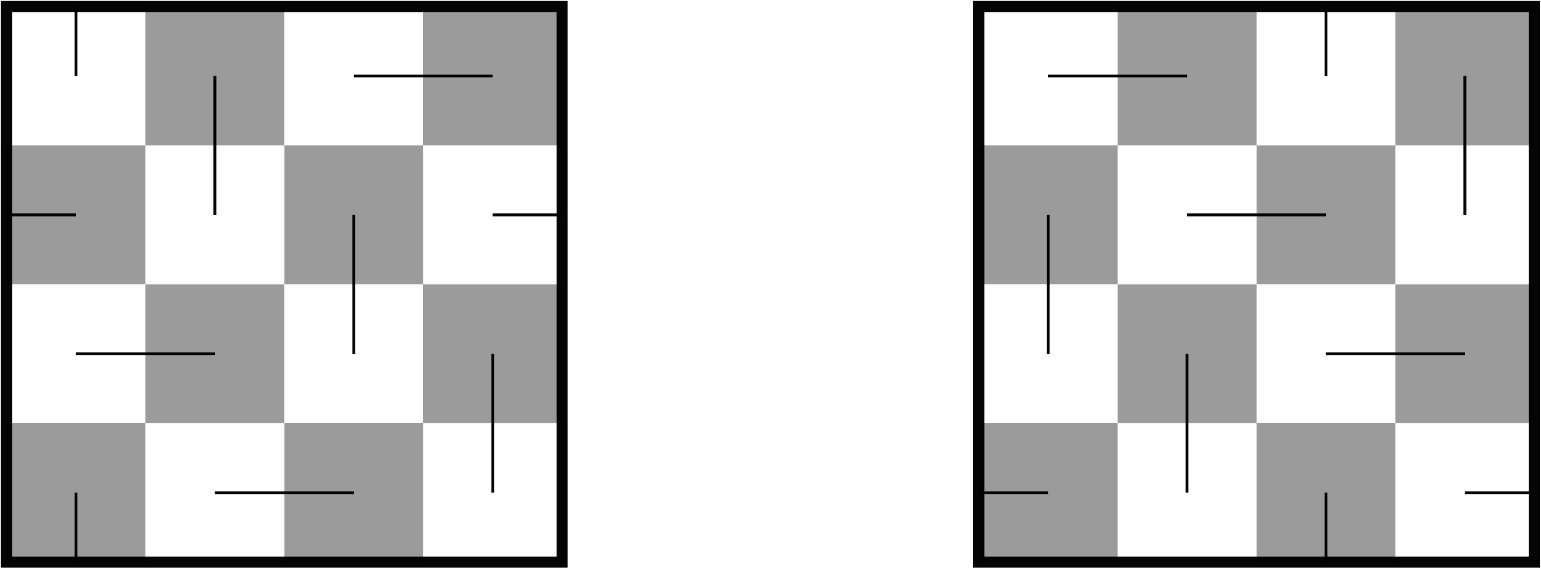}
		\caption{Two tilings of $\T_2$ that admit no flips and have identical flux values.}
\end{figure}

In attempting to reproduce the proof of Proposition \ref{hredux}, we can see how the situation is different.
Let $h,h_m$ be height functions of domino tilings of $\T_L$ with identical flux values, $h_m$ minimal and $h \neq h_m$.
Consider the difference  $g = h-h_m$.
By Proposition \ref{hsqtorogen}, for all $v \in \Z^2$ we have that $g(v) = g(v+u_x) = g(v+u_y)$, so that $g$ is $L$-periodic and in particular bounded.
Moreover, Proposition \ref{hsquare} means $g$ takes nonnegative values in $4\Z$.
Let $V$ be the set of vertices of $\Z^2$ on which $g$ is maximum.
Were we to proceed as before, we would now look for a vertex $v \in V$ that maximizes $h$, but there's generally no such vertex because of $h$'s quasiperiodicity.

What we truly seek, however, is a local maximum of $h$; if one such vertex exists, then $h$'s quasiperiodicity would give rise to a copy of the local maximum on each copy of the fundamental domain $D_L$.
Performing a flip on each of those would result in a new toroidal height function $\tilde{h}$ that is less than $h$ at these points, and identical at every other; remember flips preserve flux values.

Nonetheless, once again, there's nothing that guarantees the existence of a local maximum, and in fact for extremal values of the flux, it does not exist.
In order to understand these behaviors, we will investigate properties of tilings that do not admit flips.

\section{Tilings of the infinite square lattice}\label{sec:tilplancarac}

The next result will characterize domino tilings of the plane, but before stating it we need to introduce two new concepts.
A \textit{domino staircase}\label{def:dominostaircase}, or simply a staircase when the context is clear, is a sequence of neighbouring dominoes such that:
\begin{itemize}
\item All dominoes in the staircase are parallel, that is, either all of them are horizontal or all of them are vertical;
\item Neighbouring dominoes in the staircase always touch along one edge of the longer side;
\item Except for the first and last dominoes (if they exist), each domino in the staircase has exactly two neighbouring dominoes in the staircase;
\item For each domino in the staircase with exactly two neighbouring dominoes in the staircase, those two neighbours touch the domino at different squares.
\end{itemize}
\begin{figure}[H]
		\centering
		\includegraphics[width=0.95\textwidth]{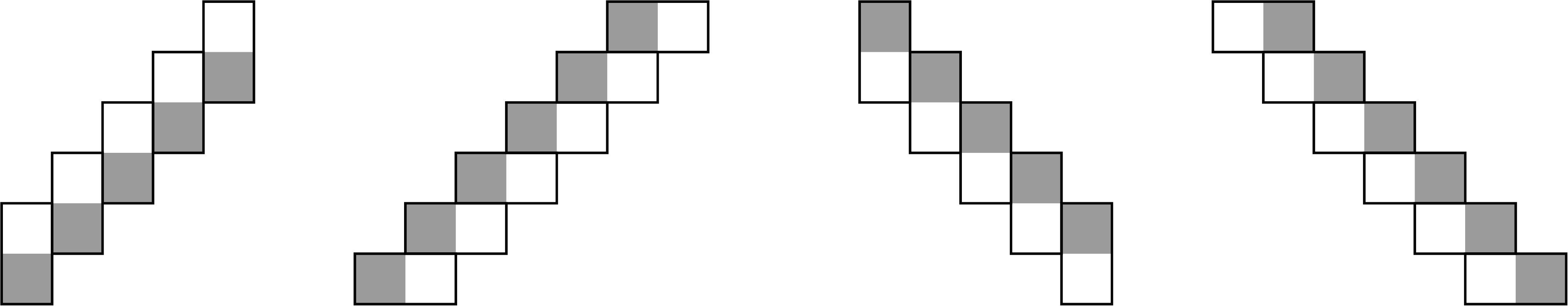}
		\caption{Examples of domino staircases.}
\end{figure}

If a staircase is finite, its length is the number of dominoes in it.
A staircase may be infinite in a single direction; if it is infinite in two directions, we say it is \textit{doubly-infinite}.

Notice that staircases have the very important property that for each domino in the staircase with exactly two neighbouring dominoes in the staircase, there are no flips involving that domino.
In particular, a doubly-infinite staircase admits no flips involving one of its dominoes.

A \textit{windmill tiling}\label{def:windmill} is a domino tiling of the plane that has one of the forms below:
\begin{figure}[H]
		\centering
		\includegraphics[width=0.85\textwidth]{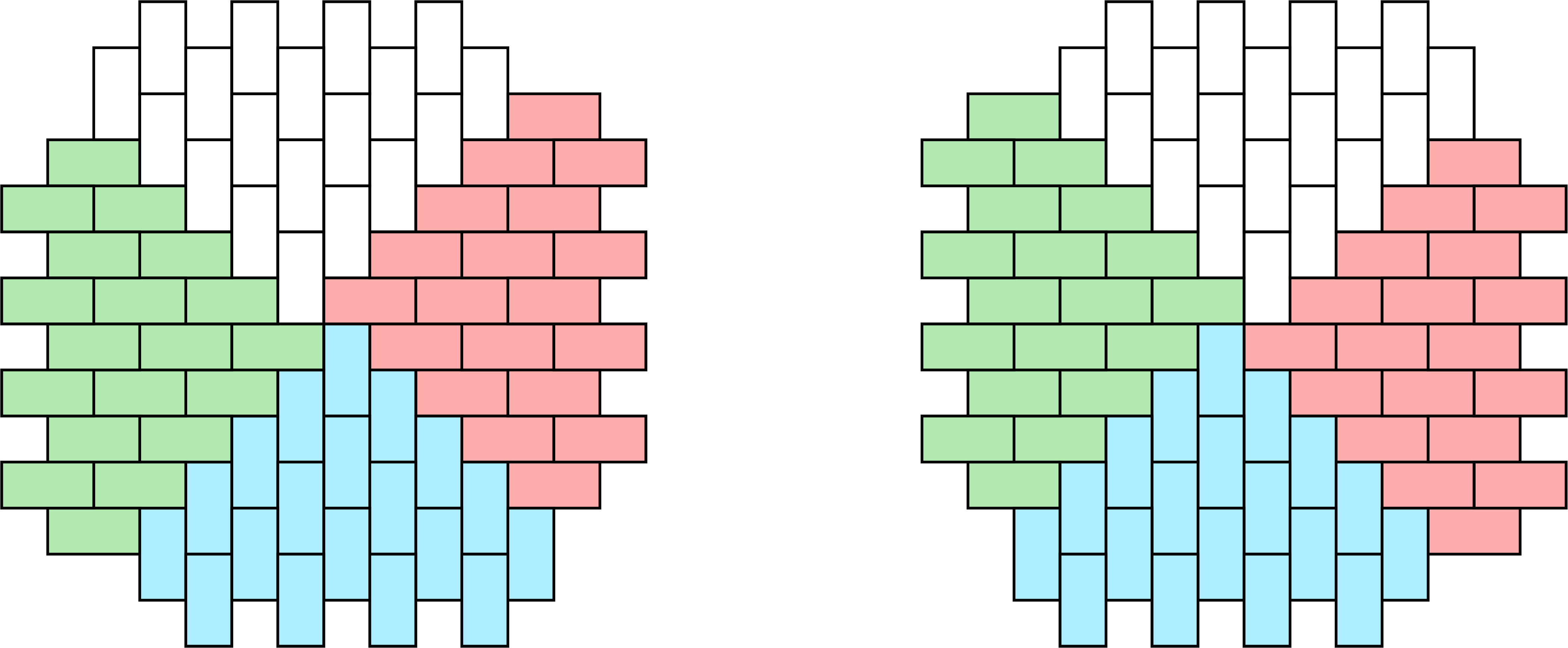}
		\caption{Windmill tilings of $\Z^2$.}
\end{figure}

Observe that a windmill tiling admits no flips and consists entirely of infinite staircases that are never doubly-infinite.
As Theorem \ref{hplanocarac} will show, they are the only tilings of $\Z^2$ with this property.

\begin{theo}[Characterization of tilings of the infinite square lattice]\label{hplanocarac}
Let $t$ be a tiling of $\Z^2$.
Then exactly one of the following applies:
\begin{enumerate}
\item $t$ admits a flip;
\item $t$ consists entirely of parallel, doubly-infinite domino staircases;
\item $t$ is a windmill tiling.
\end{enumerate}
\end{theo}

We will need a few lemmas for the proof, and thus it will be delayed.

A \textit{staircase edge-path}\label{def:staircaseedgepath} is an edge-path whose edges alternate between vertical and horizontal and such that all of its vertical edges have the same orientation and all of its horizontal edges have the same orientation.
Notice by Corollary \ref{pathorien} a staircase edge-path always respects or always reverses edge orientation.

There are essentially two kinds of staircase edge-paths: northeast-southwest or northwest-southeast.
If orientation (as induced by colors) is taken into consideration, there are four: we will refer to them by 1-3 and 3-1 for northeast-southwest, and by 2-4 and 4-2 for northwest-southeast (think quadrants in the plane).
There are thus four \textit{types}\label{def:staircaseedgepathtype} of staircase-edge paths.
Examples are provided in Figure \ref{fig:stairtypes}.
\begin{figure}[H]
		\vspace{0.1cm}
		\centering
		\includegraphics[width=0.95\textwidth]{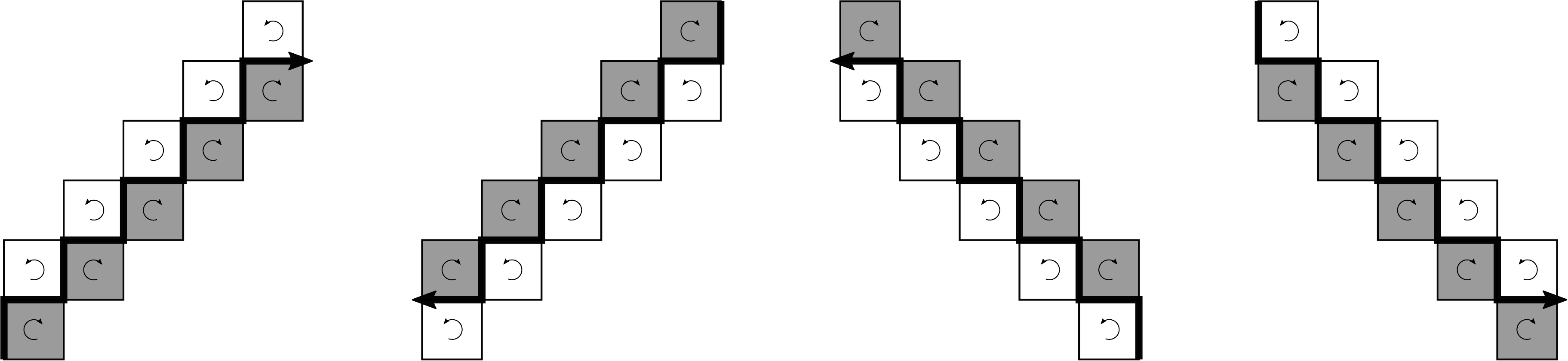}
		\caption{From left to right: 3-1, 1-3, 4-2 and 2-4 staircase edge-paths.}
		\label{fig:stairtypes}
\end{figure}

Observe that a domino staircase always admits two parallel staircase edge-paths that fit it, one on each side.
Furthermore, with orientation induced by the coloring of $\Z^2$, the staircase-edge paths are always the same type.
We can thus speak of types of domino staircases: it is the same type as that of the edge-paths that fit it (when those are oriented as per the coloring of $\Z^2$).

Four particularly interesting tilings of the plane related to this observation are the \textit{brick walls}\label{def:brickwall}.
Each of them uses only one type of domino (vertical or horizontal) and consists entirely of doubly-infinite domino staircases.
Another characterization is as follows: each of them can be seen as consisting entirely of northeast-southwest doubly-infinite domino staircases \textbf{and} entirely of northwest-southeast doubly-infinite domino staircases.
It's easy to see they are the only tilings of the plane with this property.

The following image shows the four different brick walls; the marked vertex is the origin.
\begin{figure}[H]
		\vspace{0.05cm}
		\centering
		\includegraphics[width=\textwidth]{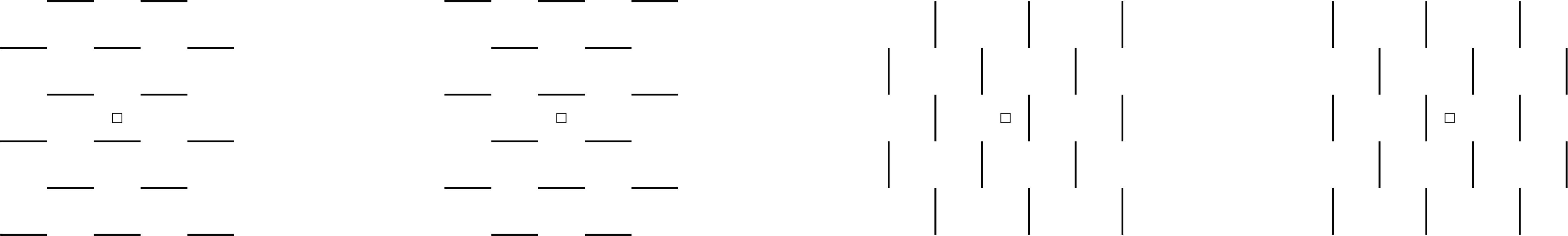}
		\caption{The four brick walls.}\label{brickimg}
\end{figure}

Later, Proposition \ref{brickflux} will expand on the importance of brick walls. 

\begin{lema}\label{stairpath}Let $t$ be a tiling of $\Z^2$.
Suppose there is a doubly-infinite staircase edge-path $\gamma$ in $t$, dividing $\R^2$ into two disjoint and quadriculated connected components $\mathcal{Z}_1$ and $\mathcal{Z}_2$.
Then for each of $\mathcal{Z}_1$ and $\mathcal{Z}_2$, its tiling by $t$ contains a doubly-infinite domino staircase that fits $\gamma$ or admits a flip.
\end{lema}
\begin{proof}
Let $\gamma$ be a doubly-infinite staircase edge-path in $t$ and choose any square $\mathcal{Q}$ in $\mathcal{Z}_1$ touching $\gamma$ along one (and thus two) of its edges.
Notice that any choice of type (either vertical or horizontal) for the domino tiling $\mathcal{Q}$ in $t$ propagates infinitely in one direction of $\gamma$, producing a domino staircase $S$ in $\mathcal{Z}_1$ that fits $\gamma$ and is infinite in that direction.
Observe the image below.
\begin{figure}[H]
    \centering
    \def\svgwidth{0.68\columnwidth}
    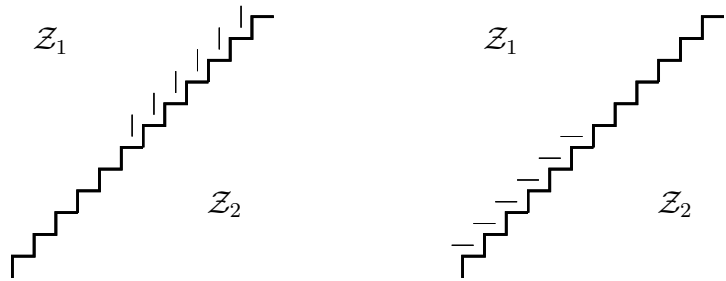
		\caption{Domino propagation across staircases.}\label{domprop}
\end{figure}

Let $S_1$ be the maximal domino staircase in $t$ containing the domino that tiles $\mathcal{Q}$; the preceding paragraph makes it clear that $S_1$ is in $\mathcal{Z}_1$ and is infinite in at least one direction.
If it is doubly-infinite, there is nothing to prove.
Otherwise, $S_1$ has exactly one extremal domino $\mathcal{D}_1$.
Consider the square $\mathcal{D}_2$ in $\mathcal{Z}_1 \setminus S_1$ that touches $\gamma$ along one of its edges and is closest to $S_1$'s extremal domino.
\begin{figure}[H]
    \centering
    \def\svgwidth{0.68\columnwidth}
    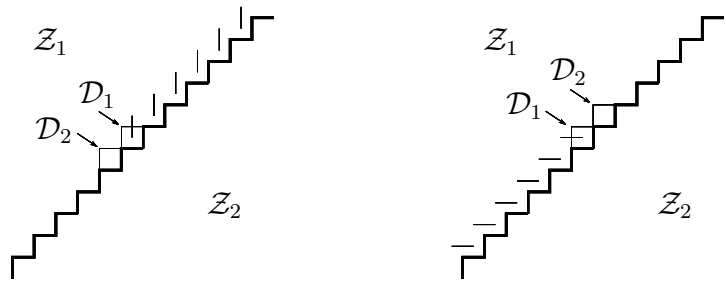
		\caption{The domino tiling $\mathcal{D}_2$ must be of the other type.}
\end{figure}

The domino tiling $\mathcal{D}_2$ cannot have the same type as the dominoes of $S_1$, for otherwise $\mathcal{D}_1$ would not be $S_1$'s extremal domino.
The choice of type for the domino tiling $\mathcal{D}_2$ is thus fixed, and like before it propagates infinitely, except this time in the other direction of $\gamma$.
This produces a new domino staircase $S_2$ in $\mathcal{Z}_1$ that fits $\gamma$ and is infinite in that direction.
Finally, consider the square in $\mathcal{Z}_1$ touching both $S_1$ and $S_2$.
\begin{figure}[H]
    \centering
    \def\svgwidth{0.285\columnwidth}
    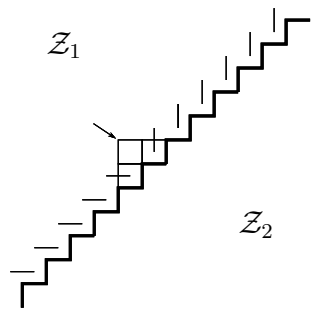
		\caption{A flip is inevitable.}
\end{figure}

Regardless of the choice of type for the domino tiling that square in $t$, a flip is enabled (either with $\mathcal{D}_1$ or with $\mathcal{D}_2$), so the claim on $\mathcal{Z}_1$ is proved.

The argument does not rely on any particular property of $\mathcal{Z}_1$, and using $\mathcal{Z}_2$ in its stead yields the complete proof.
\end{proof}

\begin{lema}\label{lemapath}
Let $t$ be a tiling of a planar region $R$ and $\gamma$ an edge-path in $t$ joining $v$ to $w$ that respects (respectively reverses) edge orientation.
Then
\begin{enumerate}
\item Any edge-path in $R$ joining $v$ to $w$ that respects (respectively reverses) edge orientation and has length $l(\gamma)$ is an edge-path in $t$;
\item $\gamma$ has minimal length over edge-paths in $R$ joining $v$ to $w$ that respect (respectively reverse) edge orientation;
\item At least one of the following properties is true:
\begin{enumerate}
\item[(a)] Every horizontal edge in $\gamma$ has the same orientation;
\item[(b)] Every vertical edge in $\gamma$ has the same orientation.
\end{enumerate}
\end{enumerate}
\end{lema}
\begin{proof}
(1) follows immediately from the constructive definition of height functions. 

For (2), either $\gamma \in \Gamma(v,w)$ or $\gamma \in \Psi(v,w)$.
Suppose we're in the first case; the other is analogous.
Let $h$ be $t$'s associated height function.
Consider the auxiliary height function $h_{\text{aux}}$, maximal over height functions $\tilde{h}$ on $\Z^2$ with $\tilde{h}(v)=h(v)$ (notice they need not have the value 0 on the origin).
An easy adaptation of Proposition \ref{hmax} yields $h_{\text{aux}}(u) = h(v) + \min_{\widetilde{\gamma} \in \Gamma(v,u)} l(\widetilde{\gamma})$.
Then $h \leq h_{\text{aux}}$ wherever $h$ is defined, and since the constructive definition of height functions implies $h(w) = h(v) + l(\gamma)$ (because $\gamma$ is an edge-path in $t$), we have
\begin{equation*}
h(v) + l(\gamma) = h(w) \leq h_{\text{aux}}(w) = h(v) +  \min\limits_{\widetilde{\gamma} \in \Gamma(v,w)} l(\widetilde{\gamma}).
\end{equation*}

In other words, $l(\gamma) \leq \min_{\widetilde{\gamma} \in \Gamma(v,w)} l(\widetilde{\gamma})$, so it must be an equality and (2) is proved.

Finally, consider the edge-path $\widetilde{\gamma}$ in $\Z^2$ obtained from $\gamma$ by a translation that takes $v$ to the origin, that is, $\widetilde{\gamma} = \gamma - v$.
It is clear either $\widetilde{\gamma} \in \Gamma(0,w-v)$ or $\widetilde{\gamma} \in \Psi(0,w-v)$, and by (2) it has minimal length over its corresponding set.
The Proof of Proposition \ref{hmaxfor} (or a trivial adaptation for $\Psi$ and $h_{\min}$) then implies (3), and we are done.
\end{proof}

Consider the black-and-white infinite square lattice $\Z^2$.
A finite planar region $R \subset \Z^2$ is a \textit{rugged rectangle}\label{def:rugrec} if its boundary consists of two pairs of staircase edge-paths $\{S_{00},S_{01}\}$ and $\{S_{10},S_{11}\}$ such that for each $i \in \{0,1\}$, $S_{i0}$ and $S_{i1}$ are the same type (\textbf{including} color-induced orientation).
In this case, we say the rugged rectangle has side lengths $l$ and $m$, where $l$ is the number of squares on $R$ that fit $S_{00}$ or $S_{01}$, and $m$ is the number of squares on $R$ that fit $S_{10}$ or $S_{11}$.
Observe that because $R$ is finite, there must be two northeast-southwest staircases and two northwest-southeast staircases, so side lengths are finite and do not depend on choice of staircase.

We provide examples of rugged rectangles below.
\begin{figure}[H]
		\centering
		\includegraphics[width=0.95\textwidth]{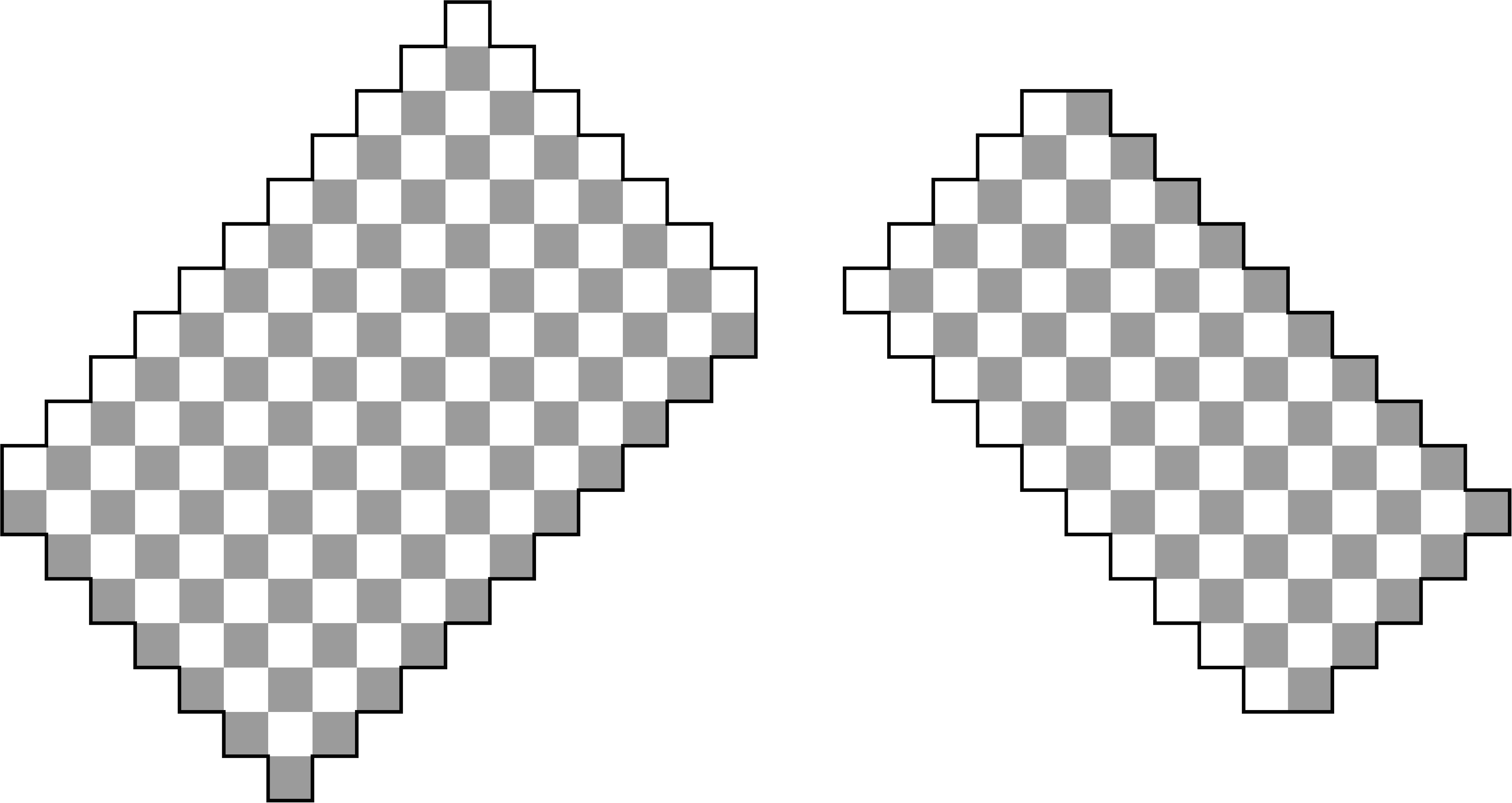}
		\caption{A rugged rectangle with side lengths 7 and 11, and one with side lengths 10 and 5.}
		\label{fig:rugrec0}
\end{figure}

Induction shows a rugged rectangle can be tiled in a single way, either entirely by vertical dominoes or entirely by horizontal dominoes.
In fact, rugged rectangles are pieces of brick walls; see Figure \ref{fig:rugrec1}.
\begin{figure}[ht]
		\centering
		\includegraphics[width=0.95\textwidth]{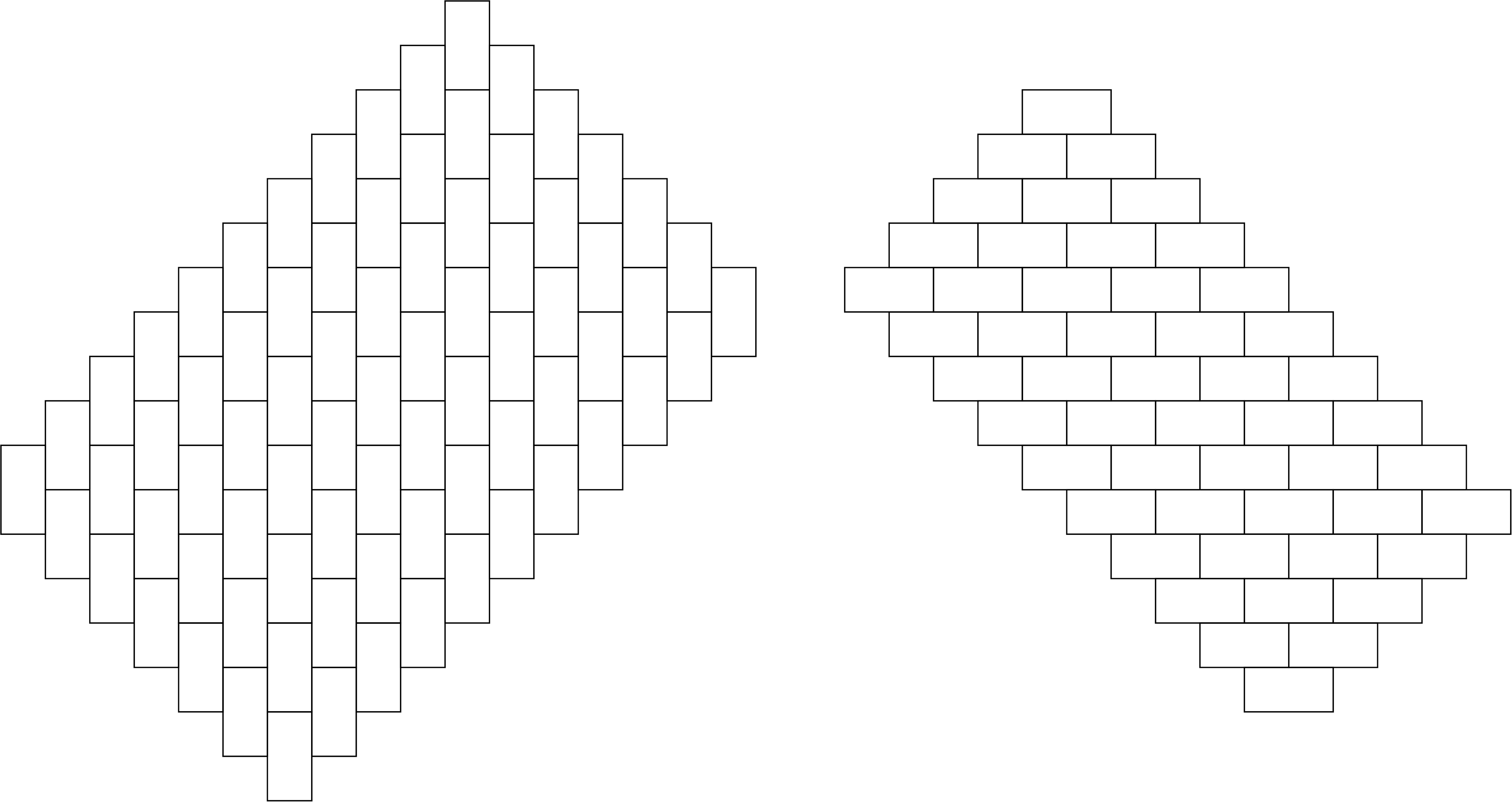}
		\caption{The unique tilings of the rugged rectangles from Figure \ref{fig:rugrec0}.}
		\label{fig:rugrec1}
\end{figure}

\begin{lema}\label{rugrec}Let $t$ be a tiling of $\Z^2$ and $\gamma$ a finite edge-path in $t$ that respects (respectively reverses) edge orientation.
At least one of the following holds\footnote{This is provided by Lemma \ref{lemapath}.}:
\begin{enumerate}
\item[(a)] Every horizontal edge in $\gamma$ has the same orientation;
\item[(b)] Every vertical edge in $\gamma$ has the same orientation.
\end{enumerate}

If exactly one of (a), (b) holds, then $\gamma$ defines a rugged rectangle $R$ in $t$.

Furthermore, if (a) holds but (b) doesn't, the side lengths of $R$ are given by the number of vertical edges in $\gamma$ that point upwards and the number of vertical edges in $\gamma$ that point downwards.
If (b) holds but (a) doesn't, the side lengths of $R$ are given by the number of horizontal edges in $\gamma$ that point left and the number of horizontal edges in $\gamma$ that point right.
\end{lema}
\begin{proof}
Let $v$ be $\gamma$'s starting point and $w$ be $\gamma$'s endpoint, so that $\gamma$ joins $v$ to $w$.
Suppose (a) holds but (b) doesn't; the opposite case is analogous.

Let $\gamma_{\text{up}}$ and $\gamma_{\text{down}}$ be edge-paths joining $v$ to $w$ that respect (respectively reverse) edge orientation and have length $l(\gamma)$;
$\gamma_{\text{up}}$ features all vertical edges pointing upwards before any vertical edge pointing downwards and $\gamma_{\text{down}}$ features all vertical edges pointing downwards before any vertical edge pointing upwards.
Notice $\gamma$, $\gamma_{\text{up}}$ and $\gamma_{\text{down}}$ have the same number of horizontal edges (all of which have the same orientation across all three edge-paths), the same number of vertical edges pointing upwards and the same number of vertical edges pointing downwards.
Moreover, because (b) does not hold, there is at least one vertical edge of each type.

Remember that whenever an edge-path respects or reverses edge orientation, vertical and horizontal edges appear in alternating fashion (see Corollary \ref{pathorien}).
This means $\gamma_{\text{up}}$ has a `rising' staircase edge-path $\gamma_{\text{up}}^+$ followed by a `descending' staircase edge-path $\gamma_{\text{up}}^-$, and $\gamma_{\text{down}}$ has a `descending' staircase edge-path $\gamma_{\text{down}}^-$ followed by a `rising' staircase edge-path $\gamma_{\text{down}}^+$.
Figure \ref{fig:exconstpath} provides an example of this construction.
\begin{figure}[ht]
    \centering
    \def\svgwidth{0.75\columnwidth}
    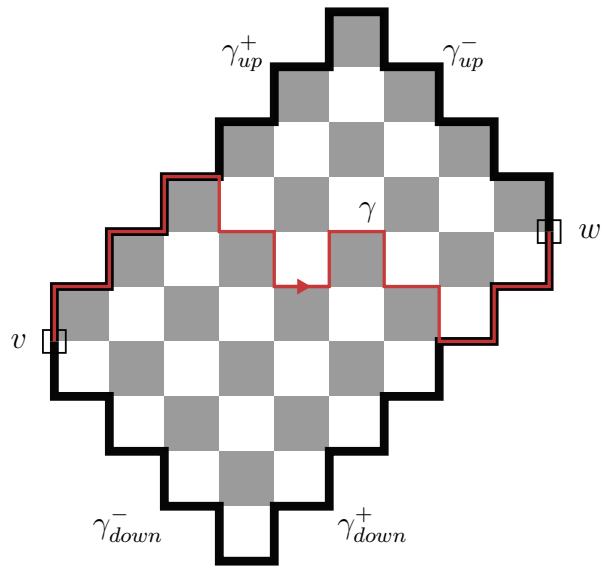
		\caption{An example construction of the paths $\gamma_{\text{up}}$ and $\gamma_{\text{down}}$.}
		\label{fig:exconstpath}
\end{figure}

By Lemma \ref{lemapath}, $\gamma_{\text{up}}$ and $\gamma_{\text{down}}$ are edge-paths in $t$.
We assert that the region $R$ they enclose is a rugged rectangle.
Indeed, $\{\gamma_{\text{up}}^+, \gamma_{\text{down}}^+ \}$ is a pair of parallel staircase edge-paths, for in both $\gamma_{\text{up}}^+$ and $\gamma_{\text{down}}^+$, vertical edge points upward and horizontal edges have the same orientation (because (a) holds).
Furthemore, because both respect (respectively reverse) edge orientation, this also means they are the same type.
Similarly, $\{\gamma_{\text{up}}^-, \gamma_{\text{down}}^- \}$ is a pair of staircase edge-paths that are the same type, so the assertion holds.

It's also easy to see that any square on $R$ that fits $\gamma_{\text{up}}^+$ lies beside a vertical edge on $\gamma_{\text{up}}^+$, and any square on $R$ that fits $\gamma_{\text{up}}^-$ lies beside a vertical edge on $\gamma_{\text{up}}^-$.
Since every vertical edge on $\gamma_{\text{up}}^+$ points upwards and every vertical edge on $\gamma_{\text{up}}^-$ points downwards, the claim on the side lengths of $R$ is proved.
\end{proof}

Two non-parallel doubly-infinite staircase edge-paths meet along a single edge and divide $\R^2$ into four disjoint and quadriculated connected components.
Each such connected component is a planar region we will call \textit{rugged quadrant}\label{def:rugquad}.

There are eight kinds of rugged quadrants, four of which can be tiled in a single way: either entirely by vertical dominoes, or entirely by horizontal dominoes.
These are the north, south, east and west rugged quadrants, and we will refer to them as \textit{cardinal} rugged quadrants.
Each of the other four kinds admits an infinite number of tilings, and will be of no interest to us.
Figure \ref{fig:rugquad} provides examples of rugged quadrants.
\begin{figure}[ht]
		\centering
		\def\svgwidth{0.65\columnwidth}
    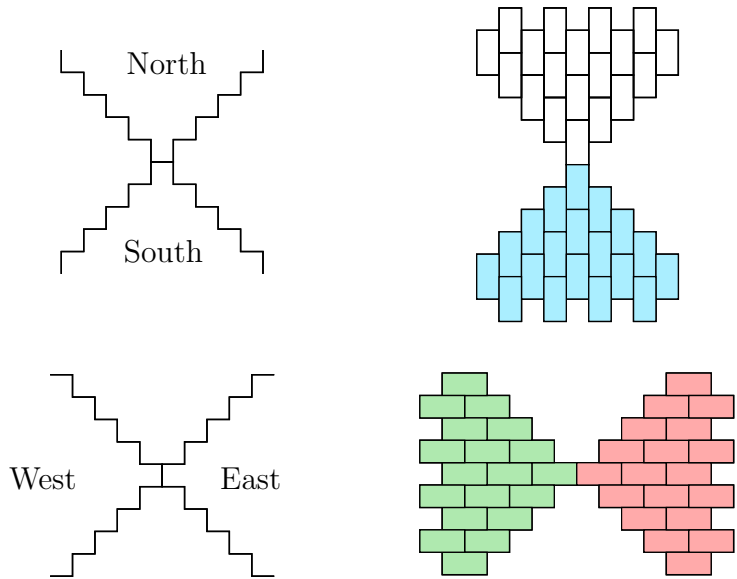
		\caption{The eight rugged quadrants, and the unique tilings of each cardinal quadrant.}
		\label{fig:rugquad}
\end{figure}

\begin{lema}\label{quad}
Let $t$ be a tiling of $\Z^2$ and $Q_0$, $Q_1$ be cardinal rugged quadrants in $t$.
If $Q_0$ and $Q_1$ are the same type (north, south, east or west), then there is a cardinal rugged quadrant $\widetilde{Q}$ in $t$ that is their type and contains them both.
\end{lema}
\begin{proof}
If $Q_0 \subseteq Q_1$ or $Q_1 \subseteq Q_0$, there is nothing to prove; suppose this is not the case.
Furthermore, suppose $Q_0$ and $Q_1$ are both north; the proof for other types is analagous to the one that follows.

Let $\gamma_i$ be the edge-path that fits the border of $Q_i$ ($i=0,1$).
Since its edges alternate between horizontal and vertical, an orientation of $\gamma_i$ always respects or always reverses edge orientation (see Corollary \ref{pathorien}); choose the orientation that always respects it.
We claim this choice is the same for both $\gamma_0$ and $\gamma_1$, that is, either both go from left to right or both go from right to left.

Indeed, because $Q_0$ neither contains nor is contained in $Q_1$, $\gamma_0$ and $\gamma_1$ must intersect on non-parallel segments along a single edge.
Furthermore, since a vertical domino lies above every horizontal edge of each $\gamma_i$ (because $Q_i$ is north), that edge must be horizontal; see the following image.
\begin{figure}[H]
		\centering
		\includegraphics[width=0.65\textwidth]{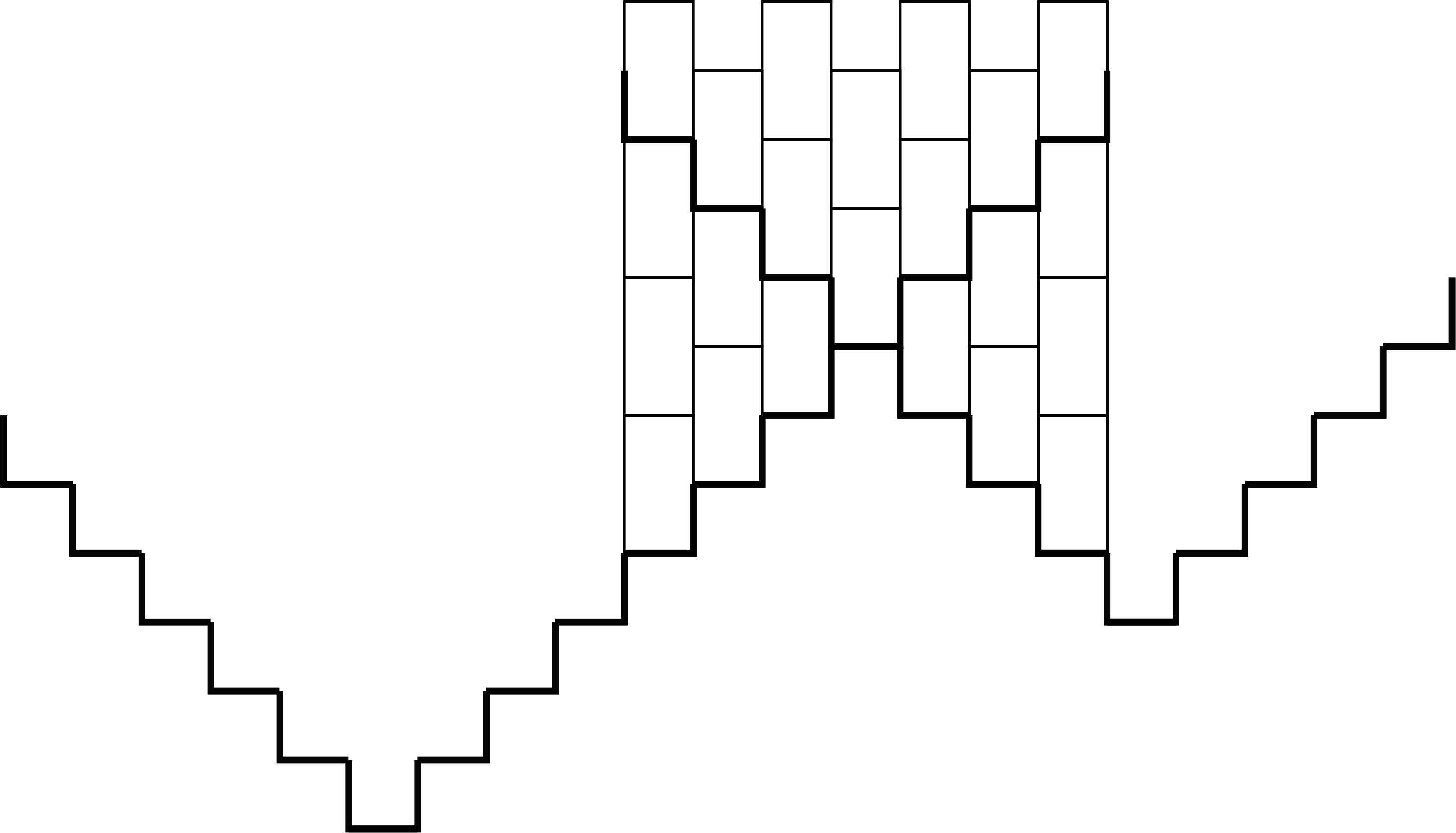}
		\caption{Intersecting north quadrants meet along a single horizontal edge.}
\end{figure}

The only possible edge orientations are shown below.
\begin{figure}[ht]
		\centering
		\includegraphics[width=0.75\textwidth]{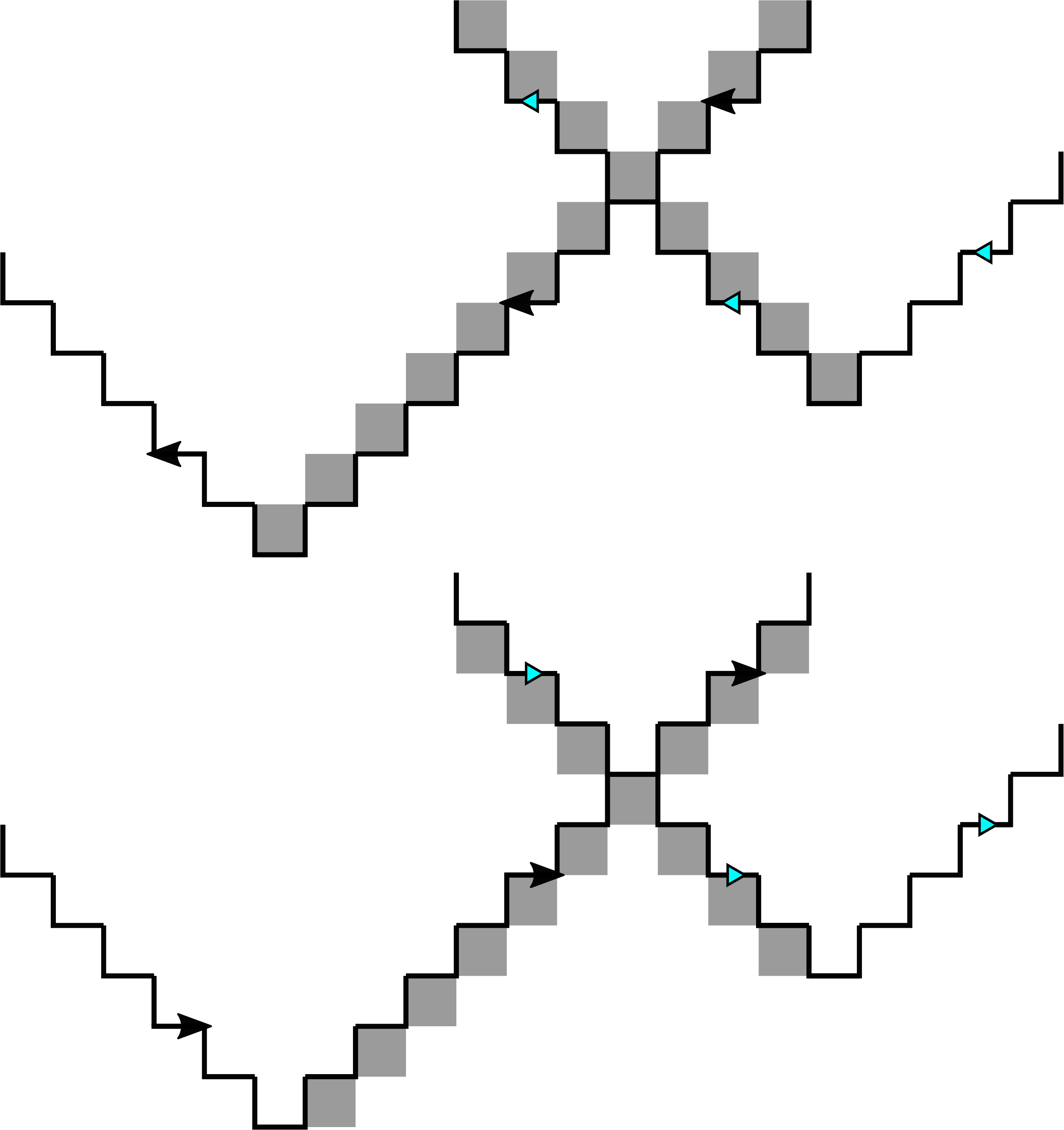}
		\caption{Color-induced orientations on the boundary of intersecting north quadrants.}
\end{figure}

Regardless of the situation, the claim holds.
Notice this implies every horizontal edge of $\gamma_0$ and every horizontal edge of $\gamma_1$ have the same orientation.

Decompose $\gamma_i$ into disjoint staircase edge-paths $\gamma_i^-$ and $\gamma_i^+$, where $\gamma_i^-$ has vertical edges pointing downwards and $\gamma_i^+$ has vertical edges pointing upwards.
Because $\gamma_0$ and $\gamma_1$ both have the same orientation, $\gamma_0^-$ and $\gamma_1^-$ are parallel, and $\gamma_0^+$ and $\gamma_1^+$ are parallel. This implies either $\gamma_0^+$ intersects $\gamma_1^-$ or $\gamma_1^+$ intersects $\gamma_0^-$.
Suppose without loss of generality that $\gamma_0^+$ intersects $\gamma_1^-$ and call $e$ the horizontal edge along which they intersect.

Let $\beta^+$ be the maximal segment of $\gamma_0^+$ that ends with $e$ and $\beta^-$ the maximal segment of $\gamma_1^-$ that starts with $e$.
Observe that $\gamma_0^- \cup \beta^+ \cup \beta^- \cup \gamma_1^+$ is the edge-path that fits the border of $Q_0 \cup Q_1$, and with the orientation inherited from $\gamma_0$ and $\gamma_1$ it always respects edge orientation.

Consider the edge-path $\beta = \beta^+ \cup \beta^-$.
Let $v$ be $\beta$'s starting point and $w$ be $\beta$'s endpoint.
Consider the edge-path $\alpha$ joining $v$ to $w$ that respects edge orientation, has length $l(\beta)$ and features all vertical edges pointing downwards before any vertical edge pointing upwards.
Notice $\beta$ and $\alpha$ have the same number of horizontal edges (all of which have the same orientation), the same number of vertical edges pointing upwards and the same number of vertical edges pointing downwards.

By Corollary \ref{pathorien}, edges in $\alpha$ alternate between vertical and horizontal, so that $\alpha$ has a `descending' staircase edge-path $\alpha^-$ followed by a `rising' staircase edge-path $\alpha^+$.
Furthermore, by Lemma \ref{lemapath}, $\alpha$ is an edge-path in $t$.

We claim $\gamma_0^- \cup \alpha^-$ is a `descending' staircase edge-path.
Indeed, since $\gamma_0^-$ and $\alpha^-$ both respect edge orientation, so does their union; this implies the edges on $\gamma_0^- \cup \alpha^-$ alternate between horizontal and vertical.
Furthermore, by construction all of its vertical edges point downwards, and all of its horizontal edges have the same orientation.
The claim thus holds.
Similarly, $\alpha^+ \cup \gamma_1^+$ is a `rising' staircase edge-path.
Moreover, the entire union $\zeta = \gamma_0^- \cup \alpha^- \cup \alpha^+ \cup \gamma_1^+$ features edges that alternate between horizontal and vertical, because $\alpha = \alpha^- \cup \alpha^+$ does.
\begin{figure}[H]
		\centering
		\def\svgwidth{0.8\columnwidth}
    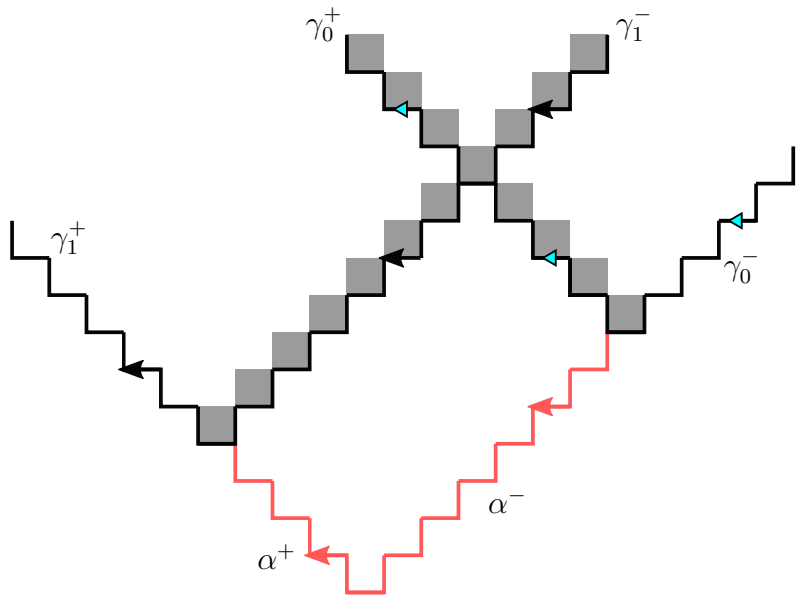
		\caption{The edge-path $\alpha = \alpha^- \cup \alpha^+$ is in $t$ and fits the border of a north quadrant.}
\end{figure}

This means $\zeta$ is an edge-path in $t$ that fits the border of a north rugged quadrant $\widetilde{Q}$.
Since $\gamma_0^-$ and $\gamma_1^+$ are contained in $\zeta$ (alternatively, since $\gamma_0^- \cup \beta^+ \cup \beta^- \cup \gamma_1^+$ is contained in $\widetilde{Q}$), it is clear $\widetilde{Q}$ contains both $Q_0$ and $Q_1$, and we are done.
\end{proof}

We are now ready to prove Theorem \ref{hplanocarac}.

\begin{proof}[Proof of Theorem \ref{hplanocarac}]
For the first part, we will show that if $t$ contains a doubly-infinite domino staircase but does not admit a flip, then it consists entirely of parallel, doubly-infinite staircases.
Indeed, suppose $t$ contains a doubly-infinite domino staircase $S$; then there are two distinct doubly-infinite staircase edge-paths in $t$ that fit $S$, one on either side of it.
By Lemma \ref{stairpath}, if $t$ admits no flips, then on either side of $S$ lies another doubly-infinite domino staircase that is parallel to $S$, so by induction $t$ consists entirely of those.

For the second part, suppose $t$ neither contains a doubly-infinite staircase nor admits a flips.
We will show that $t$ is a windmill tiling.
Let $h$ be $t$'s associated height function.
Because $t$ admits no flips at all, $h$ cannot have local extrema.

Take any $v_0 \in \Z^2$.
Since $v_0$ is not a local maximum of $h$, there must be a neighbouring vertex $v_1 \in \Z^2$ for which $h(v_1) > h(v_0)$.
Of course, $v_1$ is not a local maximum of $h$, so we can repeat the process.
This produces a list $(v_n)_{n \geq 0} \subset \Z^2$ with $h(v_{n+1})>h(v_n)$ for all $n \geq 0$ and in which each $v_n$ is neighbour to $v_{n+1}$.

Notice we may assume $h(v_{n+1})=h(v_n)+1$.
Indeed, when $h(v_{n+1})=h(v_n)+3$, the edge joining $v_n$ to $v_{n+1}$ crosses a domino in $t$, so going from $v_n$ to $v_{n+1}$ round that domino is allowed.
It's clear that each edge traversed this way increases $h$ by $+1$, so there is no loss of generality in the assumption.

Similarly, since $v_0$ is not a local minimum of $h$, there must be a neighbouring vertex $v_{-1} \in \Z^2$ for which $h(v_0)>h(v_{-1})$.
Repeating the process, we obtain a new list $(v_m)_{m \leq 0} \subset \Z^2$ with $h(v_m)>h(v_{m-1})$ for all $m \leq 0$ and in which each $v_m$ is neighbour to $v_{m-1}$.
Like before, we may assume $h(v_{m-1})=h(v_m)-1$.

The union of these two lists yields an indexed list $\left ( v_k\right )_{k \in \Z}$ with $h(v_{k+1}) = h(v_k) + 1$ for all $k \in \Z$ and in which each $v_k$ is neighbour to $v_{k+1}$.

For each $n \in \N$, consider the edge-paths $\gamma_n = (v_k)_{k=-n}^n$, $\gamma_n^+ = (v_k)_{k=0}^n$ and $\gamma_n^- = (v_k)_{k=-n}^0$.
Because $h$ always changes by $+1$ along an edge on each of these paths, they are by construction edge-paths in $t$ that respect edge orientation, so Lemma \ref{lemapath} applies to them.
Furthermore, because $\gamma_n$ is always contained in  $\gamma_{n+1}$, at least one of the statements below is true.

\begin{itemize}
\item[(a)] Each horizontal edge of $\cup_{n \in \N}\gamma_n$ has the same orientation;
\item[(b)] Each vertical edge of each $\cup_{n \in \N}\gamma_n$ has the same orientation.
\end{itemize}

We claim exactly one of these hold.
Indeed, if both (a) and (b) hold, $\cup_{n \in \N}\gamma_n$ is a doubly-infinite staircase edge-path, so Lemma \ref{stairpath} applies.
This contradicts our initial assumption that $t$ neither contains a doubly-infinite domino staircase nor admits a flip, and the claim is thus proved.

Suppose then that (a) holds but not (b).

Let $u_n^+$ be the number of vertical edges pointing upwards in $\gamma_n^+$, $d_n^+$ be the number of vertical edges pointing downwards in $\gamma_n^+$ and similarly for $u_n^-$ and $d_n^-$.
Notice $(u_n^+)_{n \in \N}$ is a nondecreasing sequence of nonnegative integers, and the same is true for the others.
Furthermore, because $\left(u_n^+ + d_n^+\right)$ is the number of vertical edges in $\gamma_n^+$, at least one of $(u_n^+)_{n \in \N}$ and $(d_n^+)_{n \in \N}$ is unbounded, and similarly for $(u_n^-)_{n \in \N}$ and $(d_n^-)_{n \in \N}$.
We assert that:
\begin{equation}\label{sdsu}
(d_n^-) \text{ is unbounded} \Longleftrightarrow \begin{array}{l} (u_n^+) \text{ is unbounded and }\\ (u_n^-)\text{, } (d_n^+) \text{ are bounded}\end{array}
\end{equation}

The $\Longleftarrow$ implication is obvious.
Suppose now that $(d_n^-)_{n \in \N}$ is unbounded; we will prove the $\Longrightarrow$ implication.
Consider the doubly-infinite staircase edge-path $S_d$ (respectively $S_u$) defined by:
\begin{itemize}
\item Its vertical edges all point downwards (respectively upwards);
\item Its horizontal edges have the same orientation as those in $\cup_{n \in \N}\gamma_n$;
\item It respects edge orientation;
\item It contains $v_0$.
\end{itemize}

Let $S_d^-$ be the infinite segment of $S_d$ that \textit{ends} in $v_0$ (as per the edge-path's own orientation), $S_d^+$ be the infinite segment that \textit{starts} at $v_0$, and similarly for $S_u^-$ and $S_u^+$.
Observe the image below.
\begin{figure}[H]
		\centering
		\def\svgwidth{0.925\columnwidth}
    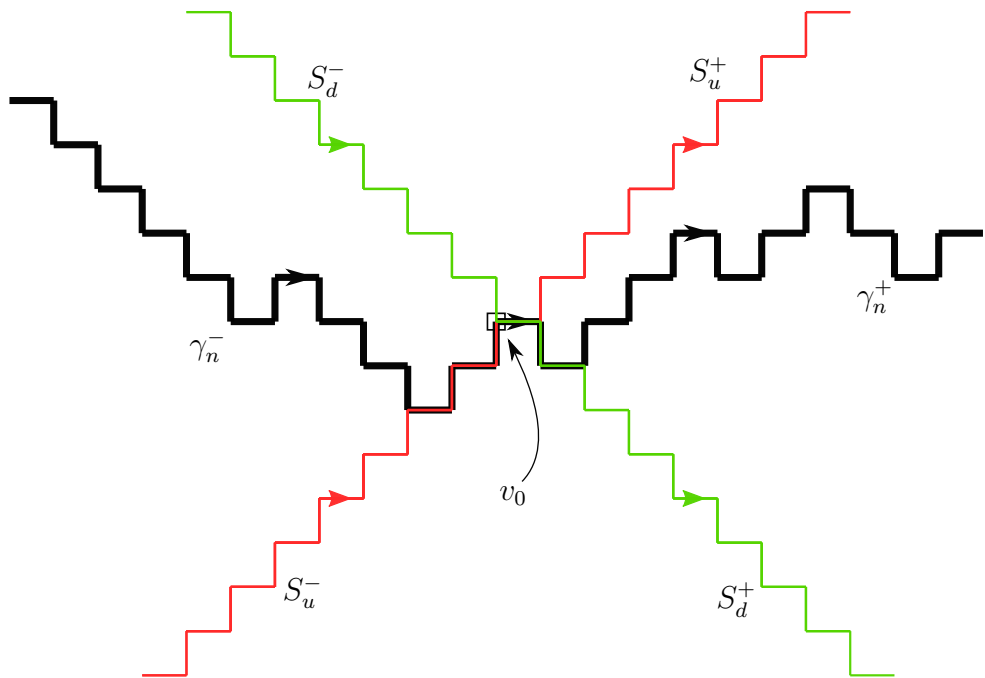
		\caption{The union $\cup_{n \in \N}\gamma_n$ defines doubly-infinite staircase edge-paths $S_d$ and $S_u$ about $v_0$.}
\end{figure}

We claim that no domino in $t$ crosses $S_d^-$, that is, $S_d^-$ is an edge-path in $t$.
Indeed, if $(u_n^-)_{n \in \N}$ is the 0 sequence $S_d^- = \cup_{n \in \N}\gamma_n^-$, so the claim is obviously true.
If $(u_n^-)_{n \in \N}$ is not the 0 sequence there is some $m \geq 0$ with the property that $u_n^- > 0$ for all $n \geq m$.
Then by Lemma \ref{rugrec}, each $\gamma_n^-$ with $n \geq m$ defines a rugged rectangle $R_n$ with side lengths $d_n^-$ along $S_d^-$ and $u_n^-$ along $S_u^-$.
Since $(d_n^-)_{n \in \N}$ is unbounded, the claim is proved.

We will now show that $(d_n^+)_{n \in \N}$ is bounded.
Indeed, if it were not, then $S_d^+$ would be an edge-path in $t$, as in the preceding paragraph.
Since we have shown that $S_d^-$ is an edge-path in $t$, this would mean $S_d$ is an edge-path in $t$, and because it is a doubly-infinite staircase edge-path, Lemma \ref{stairpath} applies.
This contradicts our initial assumptions, so $(d_n^+)_{n \in \N}$ must be bounded.

Now, because $(d_n^+)_{n \in \N}$ is bounded, $(u_n^+)_{n \in \N}$ must be unbounded.
Like before, this implies $S_u^+$ is an edge-path in $t$.
Finally, as in the previous paragraph, it follows that $(u_n^-)_{n \in \N}$ is bounded, thus proving the equivalence in~\eqref{sdsu}.
There are then two cases:
\begin{itemize}
\item $(d_n^-)_{n \in \N}$, $(u_n^+)_{n \in \N}$ are unbounded and  $(u_n^-)_{n \in \N}$, $(d_n^+)_{n \in \N}$ are bounded;
\item $(d_n^-)_{n \in \N}$, $(u_n^+)_{n \in \N}$ are bounded and  $(u_n^-)_{n \in \N}$, $(d_n^+)_{n \in \N}$ are unbounded.
\end{itemize}

In the first case, $S_d^-$ and $S_u^+$ are edge-paths in $t$, so their union is the border of a north rugged quadrant in $t$.
In the latter case, $S_u^-$ and $S_d^+$ are edge-paths in $t$, so their union is the border of a south rugged quadrant in $t$.

In other words, $v_0$ belongs to the `tip' of a north rugged quadrant in $t$ or to the `tip' of a south rugged quadrant in $t$.

We drew this conclusion under the supposition that (a) holds but not (b).
When (b) holds but not (a), the same techniques can be used to conclude that $v_0$ belongs to the `tip' of an east rugged quadrant in $t$ or to the `tip' of a west rugged quadrant in $t$.

In particular, since $v_0$ is free to assume any value in $\Z^2$, we discover that every vertex $v \in \Z^2$ belongs to the `tip' of a cardinal rugged quadrant in $t$.

Let $\mathcal{N}(t)$ be the set of all north rugged quadrants in $t$; suppose it is non-empty.
We claim the union $Q_{\mathcal{N}} = \cup_{Q \in \mathcal{N}(t)} Q$ is in $\mathcal{N}(t)$.
Indeed, Lemma \ref{quad} ensures $Q_{\mathcal{N}}$ is either a north rugged quadrant, a `rugged half plane', or the entire plane.
Since $t$ cannot contain doubly-infinite domino staircases, the last two possibilities are excluded and the claim holds.
Hence, if $\mathcal{N}(t)$ is non-empty, there is a maximal element $Q_{\mathcal{N}} \in \mathcal{N}(t)$ that contains every $\widetilde{Q} \in \mathcal{N}(t)$

This also applies to $\mathcal{S}(t)$, $\mathcal{E}(t)$ and $\mathcal{W}(t)$, respectively the set of all south, east and west rugged quadrants in $t$, whenever they're non-empty.

Now, notice none of $\mathcal{N}(t)$, $\mathcal{S}(t)$, $\mathcal{E}(t)$ or $\mathcal{W}(t)$ may be empty, for we have shown that every vertex of $\Z^2$ belongs to the `tip' of a cardinal rugged quadrant in $t$, and no fewer than four maximal cardinal rugged quadrants with different types can tile $\Z^2$.
It follows that $t$ can be decomposed into the four pieces $Q_{\mathcal{N}}$, $Q_{\mathcal{S}}$, $Q_{\mathcal{E}}$ and $Q_{\mathcal{W}}$.

Finally, it's easy to check the only ways to fit these pieces into a tiling of $\Z^2$ produce windmill tilings, so we are done.\end{proof}

\section{Back to the torus}\label{sec:plantor}

Because tilings of $\T_L$ are $L$-periodic when lifted to $\Z^2$, a windmill tiling can never be the tiling of a torus.
Theorem \ref{hplanocarac} then implies the following corollary:

\begin{corolario}[Characterization of tilings of the torus]\label{htorocarac}
Let $L$ be a valid lattice and $t$ a tiling of $\T_L$.
Then exactly one of the following applies:
\begin{enumerate}
\item $t$ admits a flip;
\item $t$ consists entirely of parallel, doubly-infinite domino staircases.
\end{enumerate}
\end{corolario}

The next proposition shows this characterization can be described in terms of the flux of a tiling.

\begin{prop}\label{diaescada}
Let $L$ be a valid lattice and $t$ a tiling of $\T_L$ with flux $\varphi_t \in \mathscr{F}(L)$.
Then $t$ admits no flips if and only if $\varphi_t \in \partial Q$.
\end{prop}

Before proving it, we need a lemma.

\begin{lema}\label{esctoro}
Let $L$ be a valid lattice, and $t$ a tiling of $\T_L$.
Suppose there is a staircase edge-path $\gamma$ in $t$ joining $w$ to $w+v$, where $w \in \Z^2$ and $v \in L$.
Then $t$ consists entirely of doubly-infinite domino staircases, each parallel to $\gamma$.
\end{lema}
\begin{proof}
Since $t$ is $L$-periodic and $v \in L$, $t$ is $v$-periodic.
Thus, for each $n \in \Z$ the translated edge-path $(\gamma + n\cdot v)$ is in $t$.
Now, $l(\gamma)$ is the sum of $v$'s coordinates, and because $L$ is valid, that number is even. 
This means $\gamma$'s first and last edge are different types (horizontal or vertical), which in turn implies the union $\cup_{n \in \Z} (\gamma + n\cdot v)$ is itself a staircase edge-path $\widetilde{\gamma}$ in $t$, except now doubly-infinite.

For any doubly-infinite staircase edge-path, the choice of a single domino fitting it propagates infinitely along the staircase in one direction; the direction is given by that domino's type (horizontal or vertical).
Figure \ref{domprop} in Lemma \ref{stairpath} illustrates this.

Using the $v$-periodicity of our tiling, this propagation can be extended infinitely to the other direction, so $\widetilde{\gamma}$ has a doubly-infinite domino staircase on each side.
For each of those, there is another doubly-infinite staircase edge-path that fits it and is parallel to $\widetilde{\gamma}$ (and thus also to $\gamma)$, so we may repeat the process.
The lemma follows.
\end{proof}

We now prove Proposition \ref{diaescada}.

\begin{proof}[Proof of Proposition \ref{diaescada}]
Suppose first that $t$ admits no flips.
In this case, Corollary \ref{htorocarac} implies $t$ consists entirely of parallel, doubly-infinite domino staircases.
Consider then a staircase edge-path $\gamma$ that respects edge orientation, starts at the origin, and fits a staircase in $t$.
Such an edge-path always exists; see the image below.
The marked vertex is the origin.
\begin{figure}[H]
		\centering
		\includegraphics[width=0.99\textwidth]{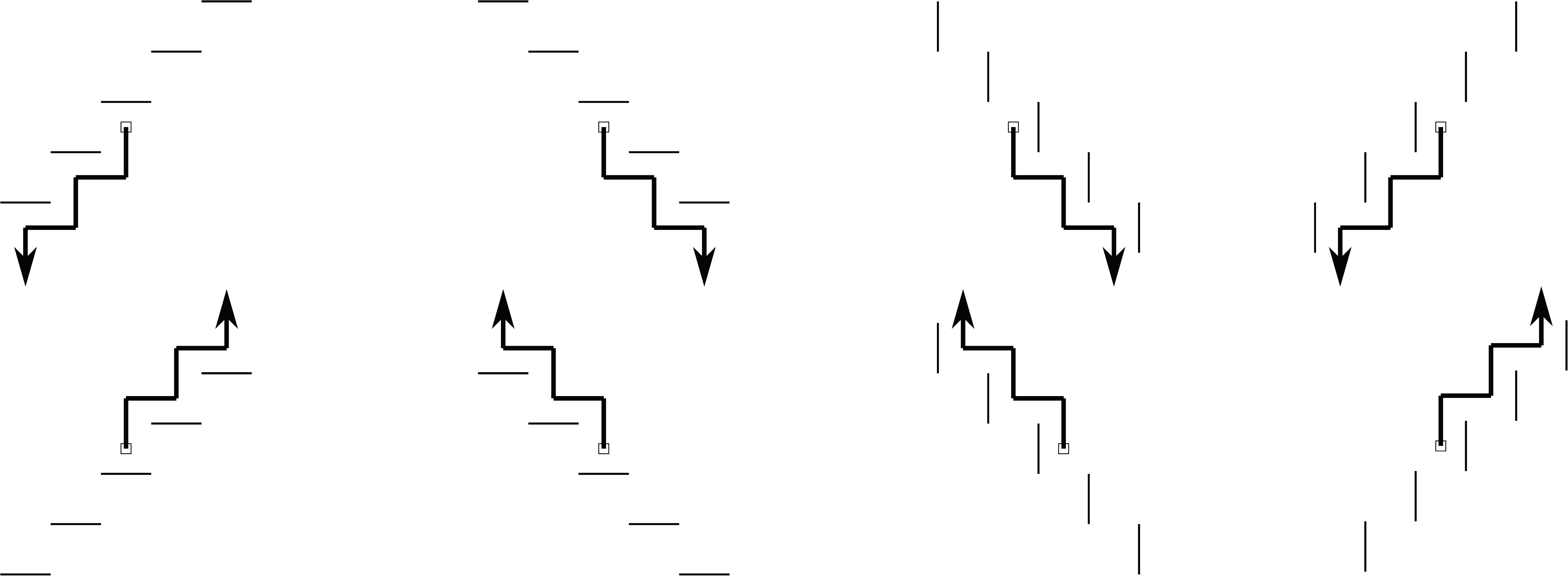}
		\caption{Possible domino staircases about the origin and choice of staircase edge-path $\gamma$.}
\end{figure}

Notice that for a suitable and fixed choice of signs, $\gamma$ contains all vertices of the form $(\pm x, \pm x)$, $x \in \N\setminus\{0\}$.
For these vertices, the constructive definition of height function implies $h(\pm x, \pm x) = 2x$.
By Lemma \ref{xx}, one of those vertices is in $L$.
We thus have$$\varphi_t(\pm x, \pm x) = \big\langle \varphi_t, (\pm x, \pm x) \big\rangle = \frac{1}{4}h(\pm x, \pm x) = \frac{1}{2} x.$$

On the other hand, it holds that$$\big\langle \varphi_t, (\pm x, \pm x) \big\rangle \leq \lVert \varphi_t \rVert_1 \cdot x \leq \frac{1}{2} x,$$where the last inequality follows from the fact that $\varphi_t \in Q$ (see Theorem \ref{fluxcarac}).
Combining the two yields $\lVert \varphi_t \rVert_1 = \tfrac12$, so $\varphi_t \in \partial Q$ as desired.
 
Suppose now that $\varphi_t \in \partial Q$.
Lemma \ref{xx} guarantees that for each choice of signs in $(\pm x, \pm x)$, there is a vertex in $L$ with that form.
For a suitable choice of signs then, there is a vertex in $L$ with that form and
\begin{equation*}
\big\langle \varphi_t, (\pm x, \pm x) \big\rangle = \lVert \varphi_t \rVert_1 \cdot x = \frac{1}{2} x,
\end{equation*}
which implies $h(\pm x, \pm x) = 2x$.
We may assume without loss of generality $x$ is positive (otherwise, take $-x$ instead).

Consider the staircase edge-paths that start at the origin and end in $(\pm x, \pm x)$; there are two: both have length $2x$, and one respects edge orientation while the other reverses it.
Because $h(\pm x, \pm x) = 2x$, the constructive definition of height functions implies the staircase edge-path that respects edge orientation is in $t$.
By Lemma \ref{esctoro}, $t$ consists entirely of parallel, doubly-infinite domino staircases, so it admits no flips and the proof is complete.
\end{proof}

We now know that if $\varphi \in \mathscr{F}(L) \cap \text{int}(Q)$, every tiling of $\T_L$ with flux $\varphi$ admits a flip; in other words, it has a local extremum.
It turns out, however, that it must have both a local minimum and a local maximum, that is, it must admit at least two flips.

\begin{prop}\label{fluxminmax}
Let $L$ be a valid lattice and $t$ a tiling of $\T_L$ with height function $h$ and flux $\varphi \in \mathscr{F}(L) \cap \text{int}(Q)$.
Then $h$ has a both a local minimum and a local maximum.
\end{prop}
\begin{proof}
The proof is by contradiction.
We will show that if $h$ does not have one kind of local extremum, there is an infinite staircase edge-path in $t$.
We claim in this case Lemma \ref{esctoro} applies.
Indeed, any edge-path $\gamma$ in $\Z^2$ can be projected onto an edge-path in $\quotient{\R^2}{L}$; since $\quotient{\R^2}{L}$ is finite, if $\gamma$ is long enough the projection must self-intersect, so the claim holds.
By Lemma \ref{esctoro}, $t$ consists entirely of parallel, doubly-infinite domino staircases, contradicting $\varphi \in \mathscr{F}(L) \cap \text{int}(Q)$.

Suppose $h$ does not have a local maximum.
The argument that follows goes similar to the proof of Theorem \ref{hplanocarac}, and is analogous when $h$ does not have a local minimum.

Since no $v_0 \in \Z^2$ is a local maximum of $h$, there is a list $(v_n)_{n \geq 0} \subset \Z^2$ with $h(v_{n+1})=h(v_n)+1$ for all $n \geq 0$ and in which each $v_n$ is neighbour to $v_{n+1}$.
For each $n \in \N$, consider the edge-path $\gamma_n^+ = (v_k)_{k=0}^n$.
Because $h$ always changes by $+1$ along an edge on each of these paths, they are by construction edge-paths in $t$ that respect edge orientation, so Lemma \ref{lemapath} applies to them.
Furthermore, because $\gamma_n^+$ is always contained in  $\gamma_{n+1}^+$, at least one of the statements below is true.
\begin{itemize}
\item[(a)] Each horizontal edge of $\cup_{n \in \N}\gamma_n^+$ has the same orientation;
\item[(b)] Each vertical edge of each $\cup_{n \in \N}\gamma_n^+$ has the same orientation.
\end{itemize}

If both (a) and (b) hold, $\cup_{n \in \N}\gamma_n^+$ is an infinite staircase edge-path in $t$, and we are done.

Suppose now (a) holds but not (b); the other case is analogous.
Consider the doubly-infinite staircase edge-path $S_d^+$ (respectively $S_u^+$) defined by:
\begin{itemize}
\item Its vertical edges all point downwards (respectively upwards);
\item Its horizontal edges have the same orientation as those in $\cup_{n \in \N}\gamma_n$;
\item It respects edge orientation;
\item It starting point is $v_0$.
\end{itemize}

Let $u_n^+$ be the number of vertical edges pointing upwards in $\gamma_n^+$ and $d_n^+$ be the number of vertical edges pointing downwards in $\gamma_n^+$.
Notice $(u_n^+)_{n \in \N}$ and $(d_n^+)_{n \in \N}$ are nondecreasing sequences of nonnegative integers.
Furthermore, because $\left(u_n^+ + d_n^+\right)$ is the number of vertical edges in $\gamma_n^+$, at least one of $(u_n^+)_{n \in \N}$ and $(d_n^+)_{n \in \N}$ is unbounded.

As in the proof of Theorem \ref{hplanocarac}, Lemma \ref{rugrec} guarantees that when $(u_n^+)_{n \in \N}$ is unbounded, $S_u^+$ is in $t$;
and when $(d_n^+)_{n \in \N}$ is unbounded, $S_d^+$ is in $t$.
In other words, at least one of $S_u^+$ and $S_d^+$ is in $t$.
Since they're both infinite staircase edge-paths, the proof is complete.
\end{proof}

Let $L$ be a valid lattice.
Remember $\T_L = \quotient{\R^2}{L}$, so any vertex of $\T_L$ has an $L$-equivalence class in $\Z^2$; we will denote $v$'s equivalence class by $[v]_L$\label{def:vbracketl}.

Let $t$ be a tiling of $\T_L$ with associated height function $h$.
Because $h$ is $L$-quasiperiodic, if $v \in \Z^2$ is a local extremum of $h$, each vertex in $[v]_L$ will also be a local extremum of the same kind.
In other words, we may perform a flip round each vertex in $[v]_L$.
We call this process an \textit{$L$-flip}\label{def:lflip} (round $v$): it is how a flip on a tiling of $\T_L$ manifests in the planar, $L$-periodic representation of $t$.

An $L$-flip round $v$ preserves $h$'s quasiperiodicity.
This is clear when $v$ is not in the equivalence class of the origin $[0]_L$; in this case, the height change on each vertex in $[v]_L$ will be the same, and no height change will occur on other vertices.

When $v$ \textbf{is} in $[0]_L$, the situation is different.
The toroidal height functions we consider take the base value 0 at the origin, so performing an $L$-flip round the origin does not change the value $h$ takes on it; instead, it changes the value on each vertex \textbf{not} in $[0]_L$.
Nonetheless, since that change is the same across all such vertices\footnote{When the origin is a local maximum, the change is $+4$; when the origin is a local minimum, the change is $-4$.}, $h$'s quasiperiodicity is preserved in this case too.

This also shows an $L$-flip preserves a tiling's flux value, because in either case the value $h$ takes on $L$ does not change --- observe that the equivalence class of the origin is $L$ itself.

If the reader had any thoughts about how Proposition \ref{fluxminmax} and Corollary \ref{htorminimal} were contradictory, the discussion above should have cleared those.
There is no conflict: for each flux $\varphi \in \mathscr{F}(L) \cap \text{int}(Q)$, $h_{\min}^{L, \varphi}$ must have all of its local maxima lying on $[0]_L$, so that performing an $L$-flip round those vertices does not contradict the minimality of $h_{\min}^{L, \varphi}$.
By the same token, $h_{\max}^{L, \varphi}$ must have all of its local minima lying on $[0]_L$.

We are now poised to prove the flux-analogue of Proposition \ref{hredux} for the torus.

\begin{prop}\label{fluxconecmin}
Let $L$ be a valid lattice and $\varphi \in \mathscr{F}(L) \cap \text{int}(Q)$.
Let $h_{\min}^{L, \varphi}$ be minimal over height functions of tilings of $\T_L$ with flux $\varphi$.
Let $h \neq h_{\min}^{L, \varphi}$ be a height function associated to a tiling $t$ of $\T_L$ with flux $\varphi$.
Then there is an $L$-flip on $t$ that produces a height function $\tilde{h} \leq h$ with $\tilde{h} < h$ on one equivalence class of vertices of $\Z^2$.
\end{prop}
\begin{proof}
By Proposition \ref{fluxminmax}, we know $h$ has a local maximum, but this is not enough.
Our previous consideration makes it clear we need to show $h$ has a local maximum on a vertex that is not in the equivalence class of the origin.

Consider the difference  $g = h-h_{\min}^{L, \varphi}$.
By Proposition \ref{hsqtorogen}, $g$ is $L$-periodic and in particular bounded.
Moreover, Proposition \ref{hsquare} means $g$ takes nonnegative values in $4\Z$.
Let $V$ be the set of vertices of $\Z^2$ on which $g$ is maximum.
We assert that $h$ has a local maximum lying on $V$.
Notice this is sufficient: since $h \neq h_{\min}^{L, \varphi}$, $g$ necessarily assumes positive values on $V$, so $[0]_L$ does not intersect $V$ (because $g$ is 0 at the origin).

We prove the assertion by contradiction; suppose $V$ contained no local maximum of $h$ and choose any $v_0 \in V$.
Since $v_0$ is not a local maximum of $h$ there must be a neighboring vertex $v_1 \in \Z^2$ for which $h(v_1) > h(v_0)$.
We claim $v_1 \in V$. Indeed, let $e$ be the edge joining $v_0$ to $v_1$.
The possible height changes along $e$ are either $+1$ and $-3$, or $-1$ and $+3$, depending on $e$'s orientation as induced by the coloring of $\Z^2$.
In either case, $h_{\min}^{L, \varphi}$ and $h$ must both increase along $e$, for if $h_{\min}^{L, \varphi}$ decreased along $e$, it would contradict the maximality of $g$ on $v_0$.

It follows that $v_1$ is also not a local maximum of $h$, so we may repeat the process.
This produces a list $(v_n)_{n \geq 0} \subset V$ with $h(v_{n+1})>h(v_n)$ for all $n \geq 0$ and in which each $v_n$ is neighbour to $v_{n+1}$.

Notice we may assume $h(v_{n+1})=h(v_n)+1$.
Indeed, let $t_{\min}^{L, \varphi}$ be the tiling associated to $h_{\min}^{L, \varphi}$.
When $h(v_{n+1})=h(v_n)+3$, the edge joining $v_n$ to $v_{n+1}$ crosses a domino in both $t$ and $t_{\min}^{L, \varphi}$, so going from $v_n$ to $v_{n+1}$ round that domino is allowed in both tilings.
It's clear that each edge traversed this way increases $h$ by $+1$, so there is no loss of generality in the assumption.

At this point, the proof of Proposition \ref{fluxminmax} can be applied verbatim here: the existence of one such list implies the existence of an infinite staircase edge-path in both $t$ and $t_{\min}^{L, \varphi}$.
By Lemma \ref{esctoro}, both $t$ and $t_{\min}^{L, \varphi}$ must consist entirely of parallel, doubly-infinite domino staircases, contradicting $\varphi \in \text{int}(Q)$.
\end{proof}
Like in the planar case, because the situation is finite, Proposition \ref{fluxconecmin} tells us any tiling of $\T_L$ with flux $\varphi$ can be taken by a sequence of $L$-flips to $t_{\min}^{L, \varphi}$. The following corollary is immediate.

\begin{corolario}[Flip-connectedness on the torus]\label{fluxconec}
Let $L$ be a valid lattice and $\varphi \in \mathscr{F}(L) \cap \text{int}(Q)$.
Any two distinct tilings of $\T_L$ with flux $\varphi$ can be joined by a sequence of flips.
\end{corolario}

\chapter{Kasteleyn matrices for the torus}
\label{chap:kastmattorus}

This chapter is devoted to adapting the construction of Kasteleyn matrices for tori.

Consider the dual graph $G(\Z^2)$\label{def:gz2} of $\Z^2$.
We will represent each unit square of $\Z^2$ by the vertex in its center, so each vertex of $G(\Z^2)$ lies in $\left(\Z+\frac12\right)^2$.
Let $L$ be a valid lattice.
As with vertices, edges on $G(\Z^2)$ have $L$-equivalence classes: two edges belong to the same class if they're related by a translation in $L$.
Given an edge $e$ on $G(\Z^2)$, its equivalence class will be denoted by $[e]_L$\label{def:ebracketl}.

We first tackle the problem of determining an \textit{$L$-Kasteleyn signing}\label{def:lkastsign} of $G(\Z^2)$, that is, an assignment of plus and minus signs to equivalence classes of edges on $G(\Z^2)$ with the following property:
for every four edges on $G(\Z^2)$ that make up a square, the product of their signs is $-1$ (where of course the sign of an edge is the sign of its equivalence class).

Similar to the planar case, these conditions guarantee that whenever we perform an $L$-flip on a tiling of $\T_L$, the total sign on the corresponding summands of the Kasteleyn determinant does not change, but this will become clear later.

An initial observation is that our usual assignment of minus signs to alternating lines of edges on $G(\Z^2)$ is generally not an $L$-Kasteleyn signing.
Indeed, if $e$ is an edge and $v \in \mathscr{O}\cap L$, then $e$ and $e+v$ have different signs.

Rather than show the existence of an $L$-Kasteleyn signing for a given valid lattice $L$, we will exhibit a universal Kasteleyn signing\label{def:lkastsignuni}, which applies to all valid lattices.

Recall the special brick wall tilings, defined just before Lemma \ref{stairpath} and shown in Figure \ref{brickimg}.
For any $v \in \mathscr{E}\sqcup\mathscr{O}$, translation by $v$ is color preserving and therefore a symmetry of each brick wall.
In particular, for each valid lattice $L$ and brick wall $b$, $b$ is $L$-periodic and thus a tiling of $\T_L$.
If all four brick walls are represented on $G(\Z^2)$, it's easy to see that for every four edges on $G(\Z^2)$ that make up a square, each of those edges lies in a different brick wall.
We may thus use each brick wall $b$ to define a universal Kasteleyn signing:
simply assign $-1$ to every edge on $G(\Z^2)$ that is also in $b$, and $+1$ to every other edge (or vice-versa, exchanging $-1$ with $+1$).

In light of this, we expand on the significance of Proposition \ref{hdelmeio}.

\begin{prop}\label{brickflux}
Let $L$ be a valid lattice.
Each of the points $\pm\left(\frac12,0\right)$ and $\pm\left(0,\frac12\right)$ is in $\mathscr{F}(L)$.
For each of those, there is only one tiling of $\T_L$ which realizes that flux, and it is a brick wall.
\end{prop}
\begin{proof}
That the points are in $\mathscr{F}(L)$ is provided by Proposition \ref{hdelmeio} and Theorem \ref{fluxcarac}.
We first show that the fluxes of the four brick walls are given by these four points.

Let $L$ be a valid lattice and $b$ a brick wall with flux $\varphi_b$ and associated height function $h_b$.
Each marked edge-path on the image below is a staircase that respects edge orientation and by Lemma \ref{xx} intercepts $L$.
The marked vertex is the origin.
\begin{figure}[H]
		\centering
		\begin{subfigure}[H]{0.3\textwidth}
		\includegraphics[width=\textwidth]{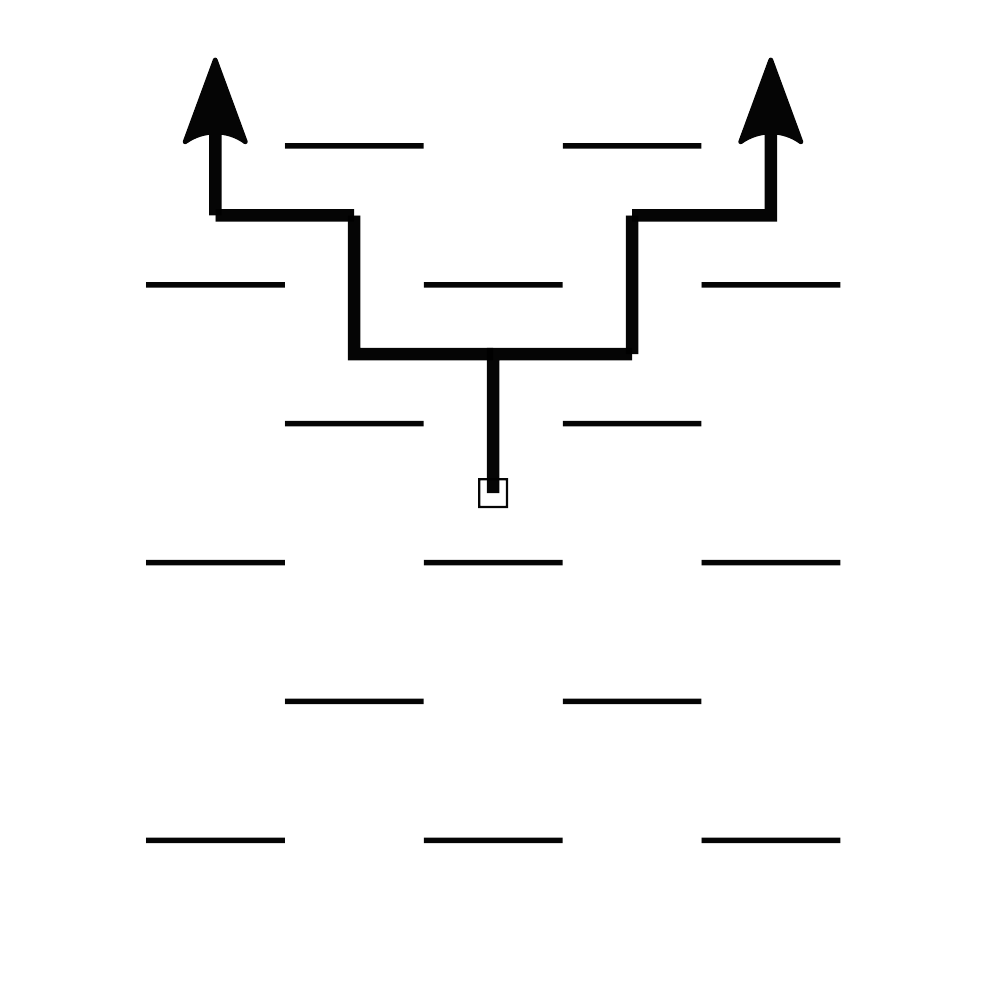}
		\caption{}\label{brickwallpathsa}
		\end{subfigure}
		\begin{subfigure}[H]{0.3\textwidth}
		\includegraphics[width=\textwidth]{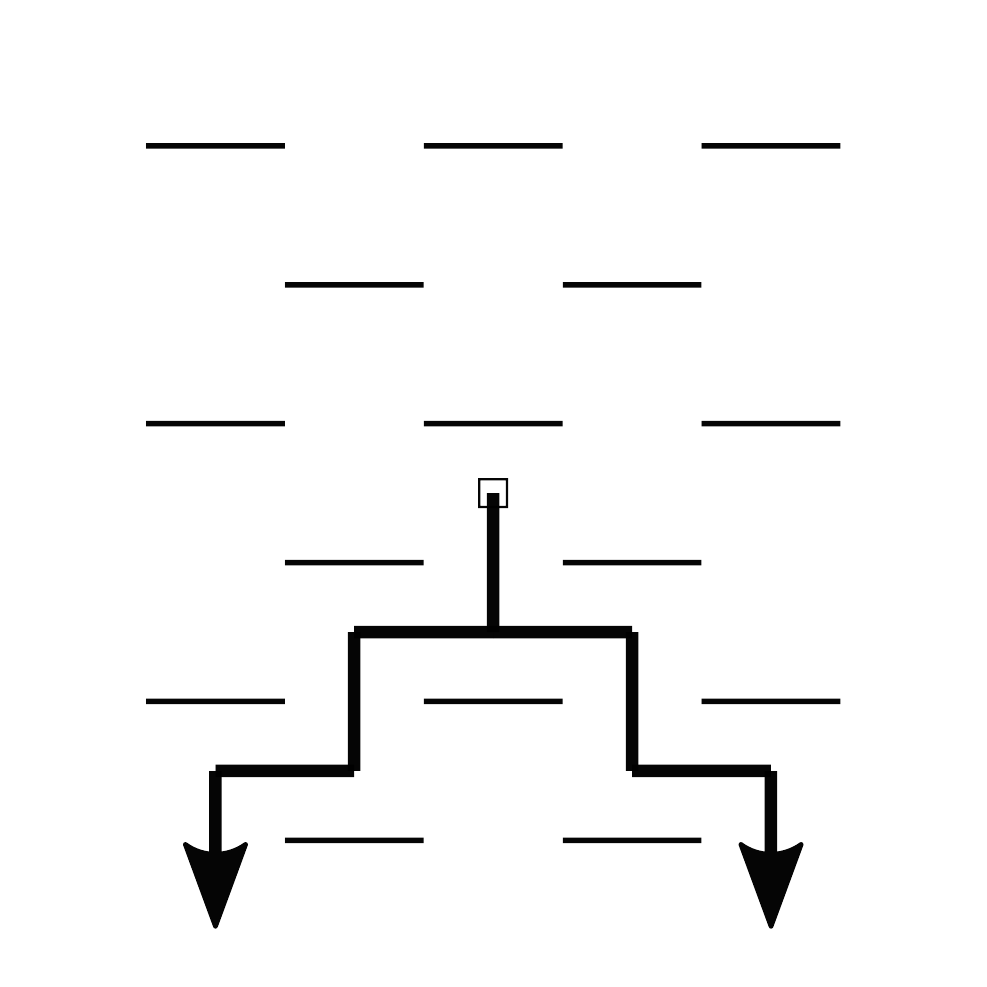}
		\caption{}
		\end{subfigure}\linebreak
		\begin{subfigure}[H]{0.3\textwidth}
		\includegraphics[width=\textwidth]{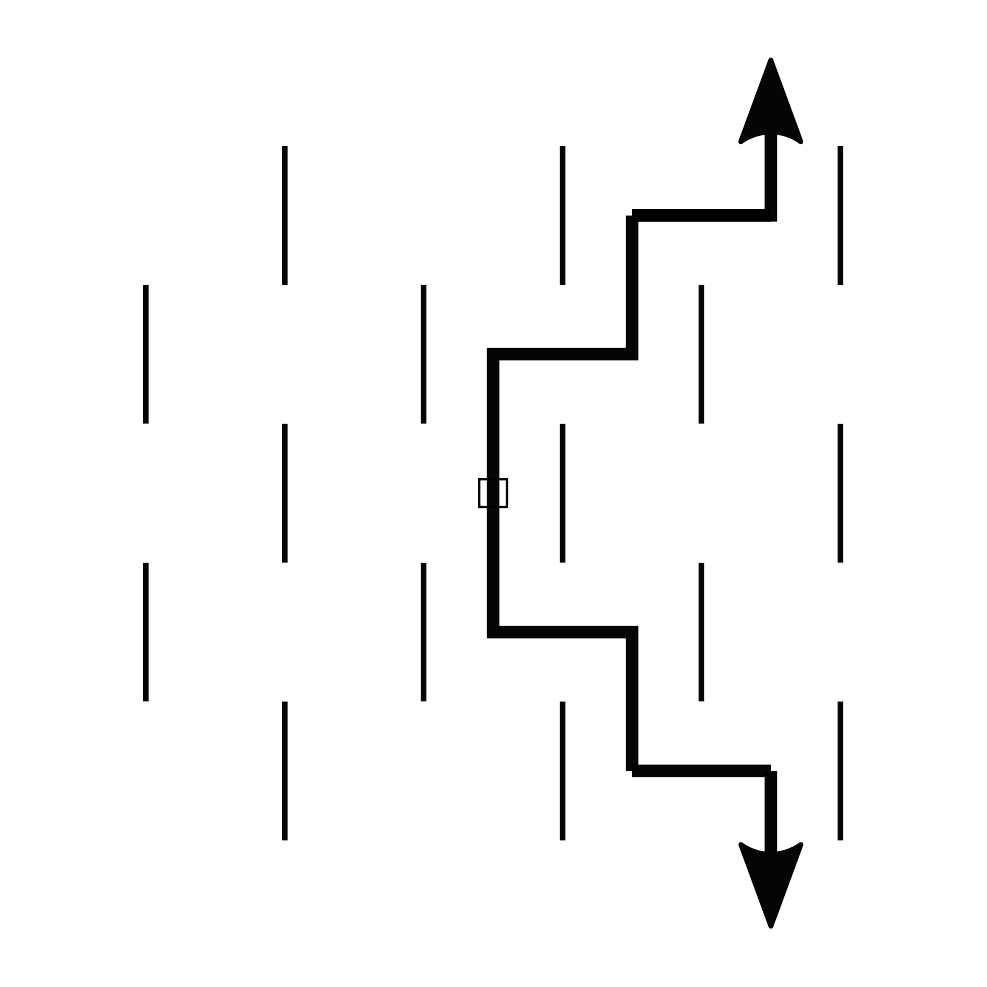}
		\caption{}
		\end{subfigure}
		\begin{subfigure}[H]{0.3\textwidth}
		\includegraphics[width=\textwidth]{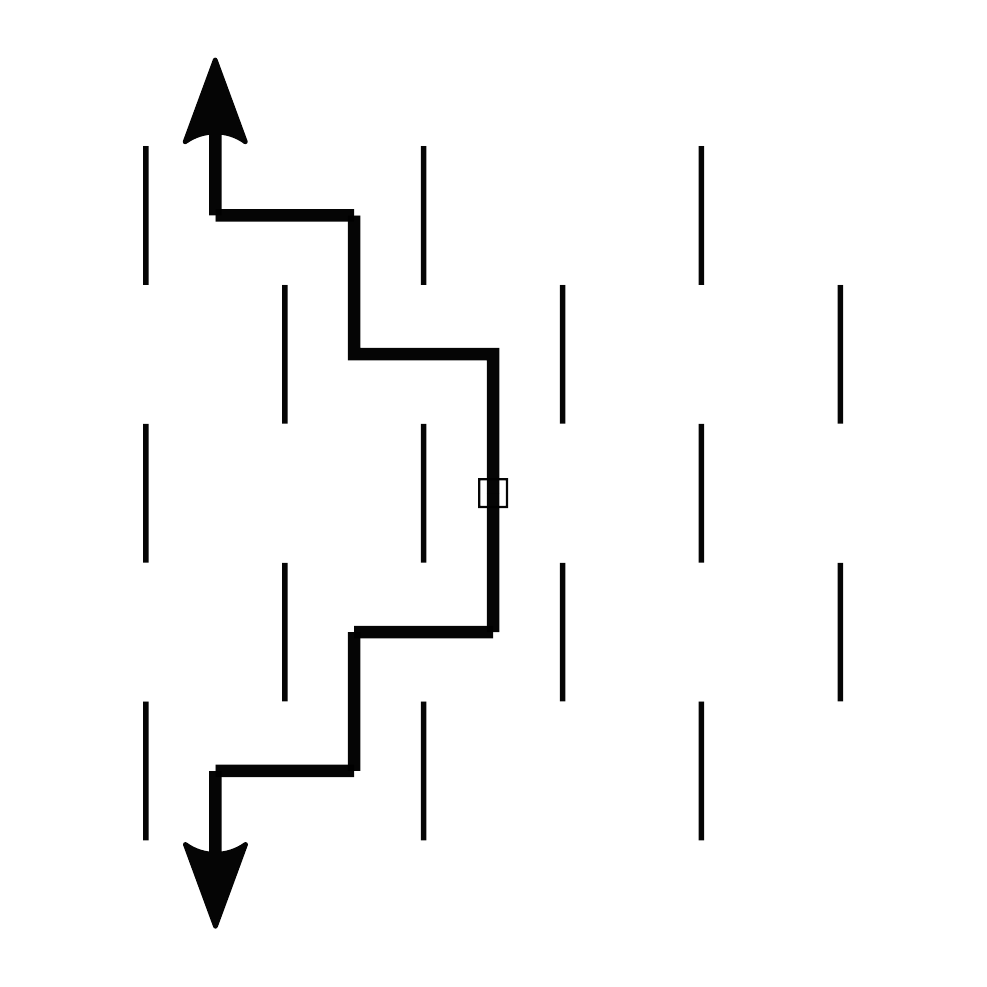}
		\caption{}
		\end{subfigure}
		\caption{Choice of staircase edge-paths for each brick wall.}
\end{figure}

This means there are nonzero $x, y \in \Z$ with $(x,x), (y,-y) \in L$ and
\begin{equation}\left\{ \begin{array}{l}
\begin{alignedat}{11}\label{brickfluxeq}
&h_b(x,x)& &=& \enspace &4& \enspace &\cdot& \enspace &\langle \varphi_b,(x,x) \rangle \enspace &=& \enspace 2 \lvert x \rvert \\
&h_b(y,-y)& &=& \enspace &4& \enspace &\cdot& \enspace &\langle \varphi_b,(y,-y) \rangle \enspace &=& \enspace 2 \lvert y \rvert
\end{alignedat}
\end{array}
\right.
\end{equation}

Each of the four brick walls corresponds to a choice of signs in $x,y \in \Z$ (positive or negative),
and the corresponding solution of system~\eqref{brickfluxeq} for the flux $\varphi_b$ yields one of $\pm\left(\frac12,0\right) , \pm\left(0, \frac12\right)$.
By inspection, $b_{\mathcal{E}}$ has flux $\left(\frac12,0\right)$, $b_{\mathcal{N}}$ has flux $\left(0,\frac12 \right)$, $b_{\mathcal{W}}$ has flux $\left(-\frac12,0\right)$ and $b_{\mathcal{S}}$ has flux $\left(0, -\frac12 \right)$.

It remains to show that a tiling with flux given by one of the four points is a brick wall.
We will do this for the flux $\left(\frac12, 0\right)$, but for other points the reasoning is analogous.
Let $L$ be a valid lattice and $t$ a tiling of $\T_L$ with flux $\left(\frac12, 0\right)$ and associated height function $h$.
By Lemma \ref{xx}, there are positive integers $x_0, x_1$ with $(x_0,x_0), (x_1,-x_1) \in L$.
Let $\gamma_0$ be a 3-1 staircase edge-path joining the origin to $(x_0,x_0)$ and $\gamma_1$ be a 2-4 staircase edge-path joining the origin to $(x_1,-x_1)$, like the edge-paths in Figure \ref{brickwallpathsa}; notice they respect edge orientation.

Clearly, $l(\gamma_i) = 2x_i$.
Since $h(x_0,x_0) = 4\cdot \langle \left(\frac12, 0\right), (x_0,x_0) \rangle = 2x_0$ and similarly $h(x_1,-x_1) =  2x_1$, the constructive definition of height functions implies $\gamma_0$ and $\gamma_1$ are both edge-paths in $t$.
Applying Lemma \ref{esctoro} to each of $\gamma_0$ and $\gamma_1$, it follows that $t$ consists entirely of 3-1 doubly-infinite domino staircases \textbf{and} entirely of 2-4 doubly-infinite domino staircases, so it must be a brick wall (in particular, since it has flux $\left(\frac12,0\right)$, it must be $b_{\mathcal{E}}$).
\end{proof}

Fix once and for all a choice of universal Kasteleyn signing: assign $-1$ to every edge on $G(\Z^2)$ that is also in $b_{\mathcal{N}}$.
We now describe how to construct a Kasteleyn matrix $K$ for a torus $\T_L$.

Let $x_0$ be the smallest positive integer with $v_0 = (x_0,0) \in L$.
Now let $y_1$ be the smallest positive integer for which the set $Y =\{(x,y_1)\enspace |\enspace x \in \R\}$ intersects $L$.
Choose the vertex $v_1 = (x_1,y_1) \in Y \cap L$ with $0 \leq x_1 < x_0$; notice there is always exactly one such $v_1$.
Indeed, if $v = (x,y_1) \in Y \cap L$, there is a unique integer $m$ with $m\cdot x_0 \leq x < (m+1)\cdot x_0$, so we may take $v_1 = v - m\cdot v_0$;
if there were more than one, their difference would contradict the minimality of $x_0$.
Arguments similar to this show that $\{v_0, v_1\}$ generates $L$.7

We now take the following fundamental domain: let $u(0)$ be the straight line edge-path in $\Z^2$ joining the origin to $v_0$.
Let $u(1)$ be the L-shaped edge-path in $\Z^2$ joining the origin to $v_1$ that never coincides with $u_0$ away from the origin.
Our fundamental domain $D_L \subset \R^2$\label{def:dl} is the rectangle with vertices $(0,0), (x_0,0), (0,y_1)$ and $(x_0,y_1)$.
Notice $u(0)$ and $u(1)$ always lie in the boundary of $D_L$, and these edge-paths are used in the flux definition that counts cross-over dominoes.

Enumerate each of $D_L$'s black squares (starting from 1) and do the same to white squares.
There is an obvious correspondence between squares on $D_L$ and equivalence classes of vertices on $G(\Z^2)$.
Also, observe that any equivalence class of edges on $G(\Z^2)$ is given by two equivalence classes of vertices on $G(\Z^2)$, its endpoints.
With this in mind, we now assign weights to each equivalence class of edges on $G(\Z^2)$.
Let $e_{ij}$ be the edge joining the $i$-th back vertex to the $j$-th white vertex: if no such edge exists, we assign the weight 0 to $[e_{ij}]_L$; otherwise, we assign its corresponding Kasteleyn sign (either $+1$ or $-1$) to it.

Next, if there is an edge in $[e_{ij}]_L$ that crosses $u(0)$, we will multiply the weight of $[e_{ij}]_L$ by either $q_0$ or $q_0^{-1}$; we now explain how the exponent is chosen.
Remember that edges on the dual graph represent dominoes, and whenever an edge-path crosses a domino on a tiling, the height function of that tiling changes by either $+3$ or $-3$ along that edge-path.
When the edge in $[e_{ij}]_L$ that crosses $u(0)$ corresponds to a height change of $+3$ along $u(0)$, we choose $q_0$; when it corresponds to a height change of $-3$, we choose $q_0^{-1}$.
Observe that because $u(0)$ joins the origin to $v_0$ (and $v_0$ is in a basis of $L$), there is at most one edge in $[e_{ij}]_L$ that crosses $u(0)$.
Of course, there may be none, and in that case this step does not change the previously assigned weight of $[e_{ij}]_L$.

Finally, we repeat the last step for $u(1)$ and $q_1$ or $q_1^{-1}$.
Notice the effect of Kasteleyn signs and each of the $q_k$'s is cumulative!

Now that all equivalence classes of edges on $G(\Z^2)$ are assigned their corresponding weights, the matrix entry $K(i,j)$ is simply the weight of $[e_{ij}]_L$.
We provide an example of this construction in the next page.
\begin{figure}[p]
		\centering
		\begin{subfigure}[H]{\textwidth}
				\def\svgwidth{\columnwidth}
				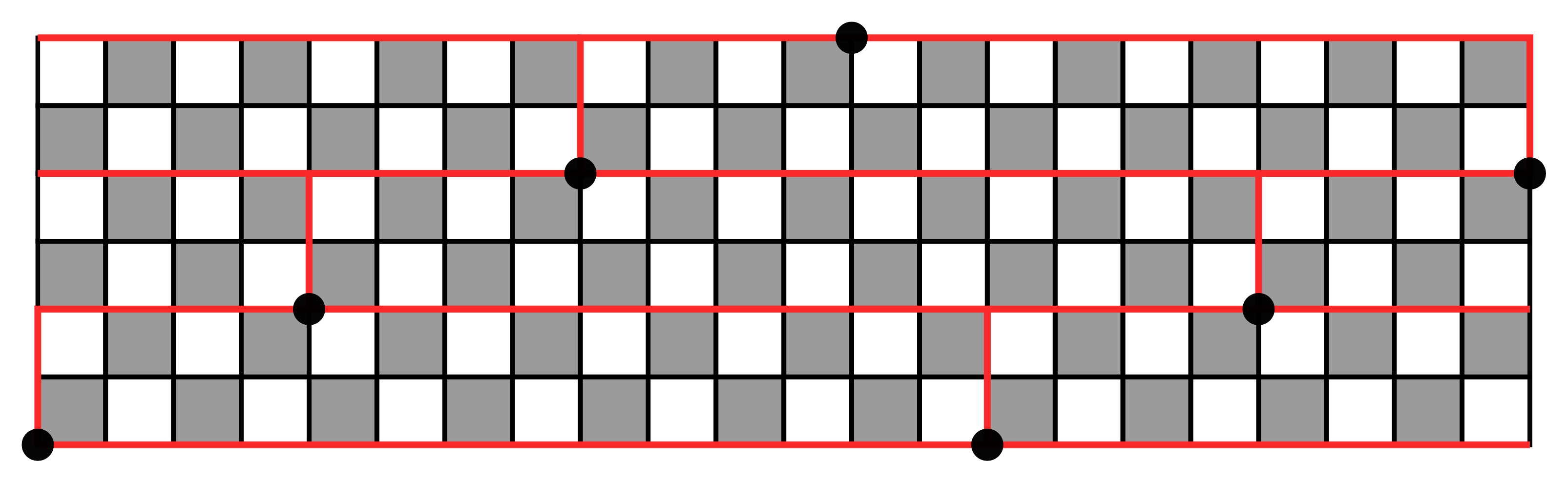
				\caption{A lattice $L$ with $v_0 = (14,0)$ and $v_1 = (4,2)$. The fundamental domain $D_L$ is represented by the red rectangle, and its squares are enumerated.}
				\vspace{30pt}
		\end{subfigure}
		\begin{subfigure}[H]{\textwidth}
				\def\svgwidth{\columnwidth}
				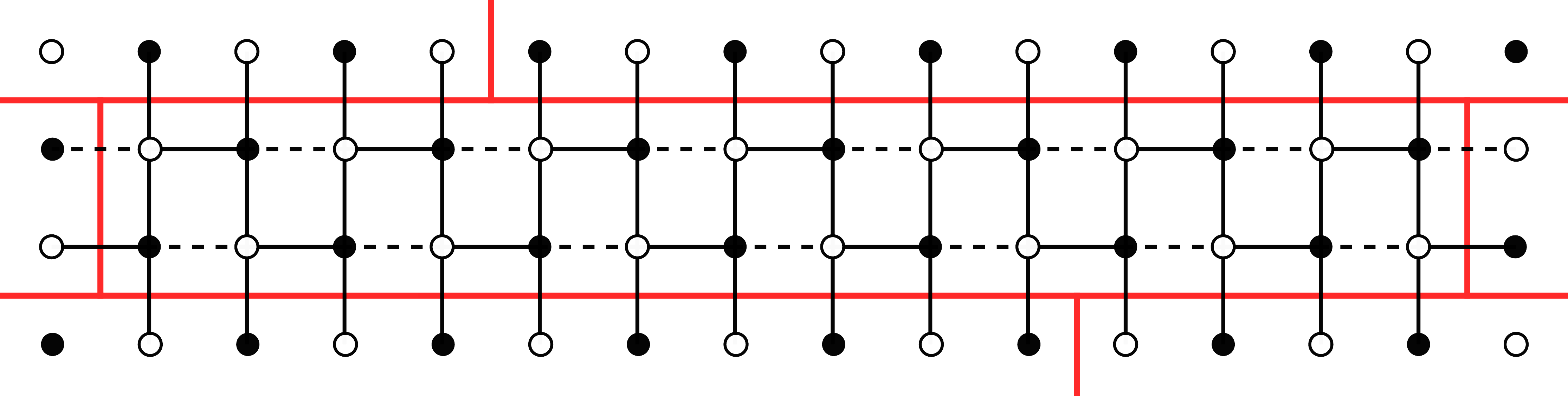
				\caption{The enumeration applied to vertices on $G(\Z^2)$. Negative edges are dashed.}
				\vspace{30pt}
		\end{subfigure}
		\begin{subfigure}[H]{\textwidth}
		\begin{equation*}
		\left(
		\begin{array}{cccccccccccccc}
		-1&0&0&0&0&0&q_1^{-1}&0&q_0&0&0&0&0&1   \\
		1&-1&0&0&0&0&0&1&0&q_0&0&0&0&0   \\
		0&1&-1&0&0&0&0&0&1&0&q_0&0&0&0   \\
		0&0&1&-1&0&0&0&0&0&1&0&q_0&0&0   \\
		0&0&0&1&-1&0&0&0&0&0&1&0&q_0&0   \\
		0&0&0&0&1&-1&0&0&0&0&0&1&0&q_0q_1   \\
		0&0&0&0&0&1&-1&q_0q_1&0&0&0&0&1&0   \\
		1&0&0&0&0&\frac{1}{q_0q_1}&0&-1&0&0&0&0&0&1   \\
		0&1&0&0&0&0&\frac{1}{q_0q_1}&1&-1&0&0&0&0&0   \\
		q_0^{-1}&0&1&0&0&0&0&0&1&-1&0&0&0&0   \\
		0&q_0^{-1}&0&1&0&0&0&0&0&1&-1&0&0&0   \\
		0&0&q_0^{-1}&0&1&0&0&0&0&0&1&-1&0&0   \\
		0&0&0&q_0^{-1}&0&1&0&0&0&0&0&1&-1&0   \\
		0&0&0&0&q_0^{-1}&0&1&0&0&0&0&0&1&-q_1   
		\end{array}
		\right)
		\end{equation*}
		\caption{The resulting Kasteleyn matrix $K$.}
		\vspace{30pt}
		\end{subfigure}
		\caption{An example construction of a Kasteleyn matrix for a torus.}\label{kastex}
\end{figure}

Similarly to the planar case, each nonzero term in the combinatorial expansion of $\det(K)$ can be seen as an $L$-periodic matching of $G(\Z^2)$.
In other words, it can be seen as an $L$-periodic tiling of $\Z^2$, or a tiling of $\T_L$. For each tiling $t$ of $\T_L$, let $K_t$ be its corresponding nonzero term in the combinatorial expansion of $\det(K)$.
From the construction of $K$ and the flux definition via counting cross-over dominoes, the following is clear:
$$\forall \varphi \in \mathscr{F}(L), \enspace \exists n_0,n_1 \in \Z, \enspace  \forall \text{ tiling $t$ of $\T_L$ with flux $\varphi$}, \enspace  K_t = \pm q_0^{n_0}q_1^{n_1}.$$
In fact, whenever $v_k \in \mathscr{E}$, the flux through $v_k$ is precisely $n_k$; whenever $v_k \in \mathscr{O}$, the flux through $v_k$ is $n_k+ \frac12$.
Furthermore, the use of a Kasteleyn signing in the construction of $K$ implies via Corollary \ref{fluxconec} that these terms are all identically signed whenever $\varphi \in \mathscr{F}(L) \cap \text{int}(Q)$.
Thus, for each monomial of the form $c_{ij}\cdot q_0^iq_1^j$ in the full expansion of $\det(K)$, if $q_0^iq_1^j$ corresponds to a flux value $\varphi \in \mathscr{F}(L) \cap \text{int}(Q)$, then $\lvert c_{ij} \rvert$ is the number of tilings of $\T_L$ with flux $\varphi$.

It turns out this is also true for fluxes in $\mathscr{F}(L) \cap \partial Q$.

\begin{prop}\label{signfrontq}
Let $L$ be a valid lattice and $K$ a Kasteleyn matrix for $\T_L$.
Let $\varphi \in \mathscr{F}(L)$ and $c_{ij}\cdot q_0^iq_1^j$ be the monomial in the full expansion of $\det(K)$ that corresponds to $\varphi$.
Then $\lvert c_{ij} \rvert$ is the number of tilings of $\T_L$ with flux $\varphi$.
\end{prop}

We know Proposition \ref{signfrontq} needs to be proved only for $\varphi \in \mathscr{F}(L)\cap \partial Q$, but before doing it we will study the structure of $\mathscr{F}(L) \cap \partial Q$.

\section{The structure of $\mathscr{F}(L) \cap \partial Q$}\label{sec:strucbound}

Observe that in Theorem \ref{hplanocarac}, when $t$ consists entirely of parallel, doubly-infinite domino staircases, these staircases are in fact the same type.
This is because every doubly-infinite staircase edge-path that fits one of these domino staircases actually fits two of them (one on each side), so they are both the same type.
By induction, this applies to them all.

Consider the boundary of $Q \subset \R^2$; it is a square.
Let $Q_1 \subset \partial Q$\label{def:qk} be the side of the square lying on the first quadrant of $\R^2$, and similarly for $Q_2$, $Q_3$ and $Q_4$.
We will use the observation above to classify the tilings in each $\mathscr{F}(L) \cap Q_k$.

\begin{prop}\label{escfrontq}
Let $L$ be a valid lattice.
For each $k \in \{1,2,3,4\}$, there is a type of domino staircase such that each tiling of $\T_L$ with flux in $\mathscr{F}(L) \cap Q_k$ consists entirely of doubly-infinite domino staircases which are all that type.
\end{prop}
\begin{proof}
We will prove this for $k=1$, but for other values of $k$ the proof is analogous.
Let $t$ be a tiling of $\T_L$ with associated height function $h$ and flux $\varphi \in \mathscr{F}(L) \cap Q_1$.
We will show $t$ consists entirely of 3-1 doubly-infinite domino staircases.

Write $\varphi = (a,b)$.
Because $\varphi \in Q_1$, $a+b=\frac12$.
By Lemma \ref{xx}, there is some positive integer $x_0$ with $(x_0,x_0) \in L$.
We thus have $h(x_0,x_0) = 4 \cdot \langle \varphi , (x_0,x_0) \rangle = 2x_0$.
Consider the edge-path $\gamma_0$ below.
As in the proof of Proposition \ref{brickflux}, it respects edge orientation, and is a 3-1 staircase edge-path joining the origin to $(x_0,x_0)$.
\begin{figure}[H]
		\centering
		\includegraphics[width=0.25\textwidth]{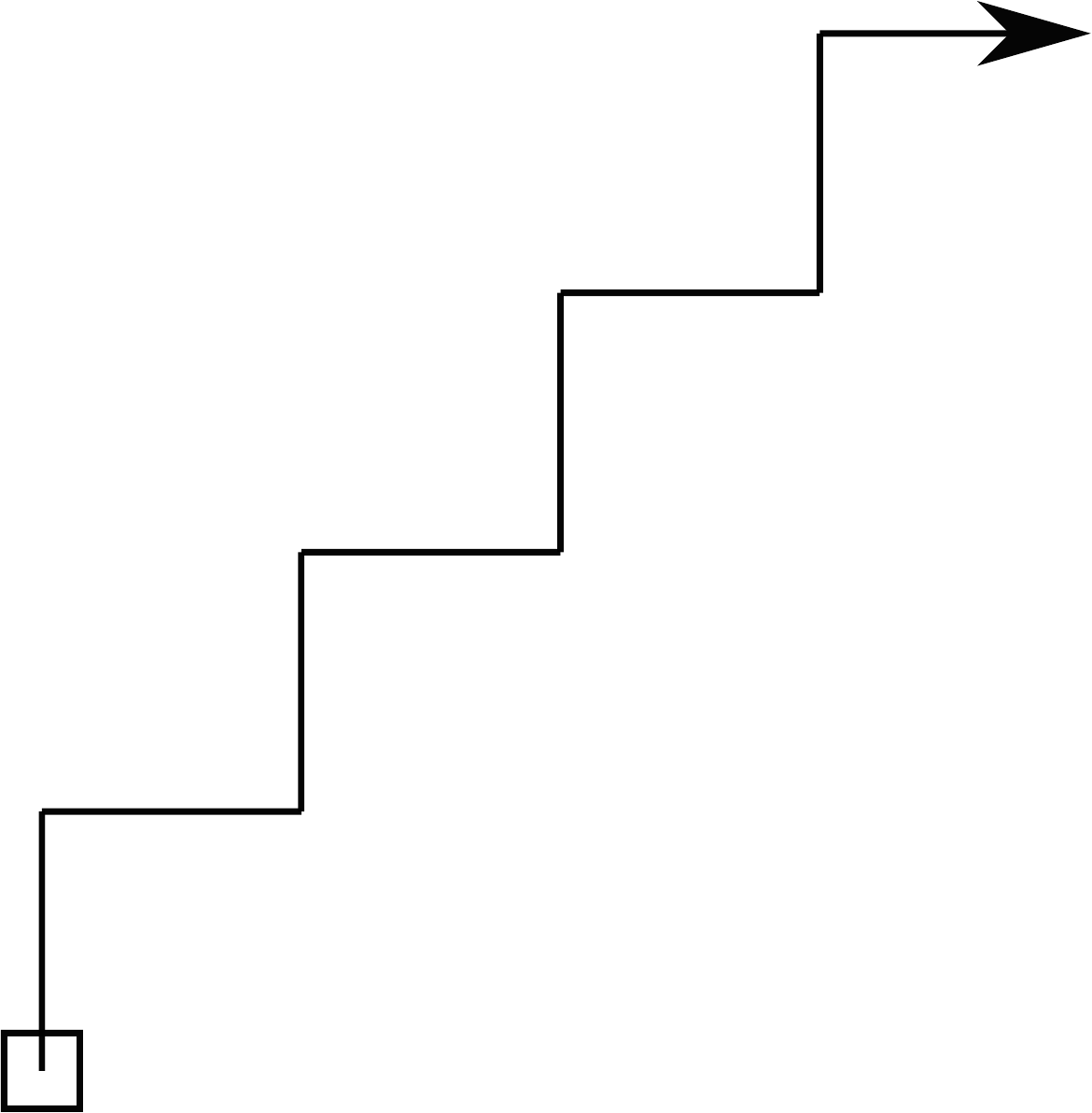}
		\caption{The edge-path $\gamma_0$; the marked vertex is the origin.}
\end{figure}

Like before, it's clear $l(\gamma_0) = 2x_0$, so by the constructive definition of height functions it must be an edge-path in $t$.
We may thus apply Lemma \ref{esctoro} to it, from which we conclude $t$ consists entirely of 3-1 doubly-infinite domino staircases, as desired.
\end{proof}

Let $L$ be a valid lattice and $t$ be a tiling of $\T_L$ with flux $\varphi \in \mathscr{F}(L)$.
By inspection, we have that:\label{def:qkprop}
\begin{alignat*}{1}
\varphi \in Q_1\enspace  &\Rightarrow \enspace \text{$t$ consists entirely of 3-1 doubly-infinite domino staircases} \\
\varphi \in Q_2\enspace  &\Rightarrow \enspace \text{$t$ consists entirely of 4-2 doubly-infinite domino staircases} \\
\varphi \in Q_3\enspace  &\Rightarrow \enspace \text{$t$ consists entirely of 1-3 doubly-infinite domino staircases} \\
\varphi \in Q_4\enspace  &\Rightarrow \enspace \text{$t$ consists entirely of 2-4 doubly-infinite domino staircases}
\end{alignat*}

Of course, the converse is also true.
This can be schematically represented by the diagram below:
\begin{figure}[H]
		\centering
		\def\svgwidth{0.875\columnwidth}
    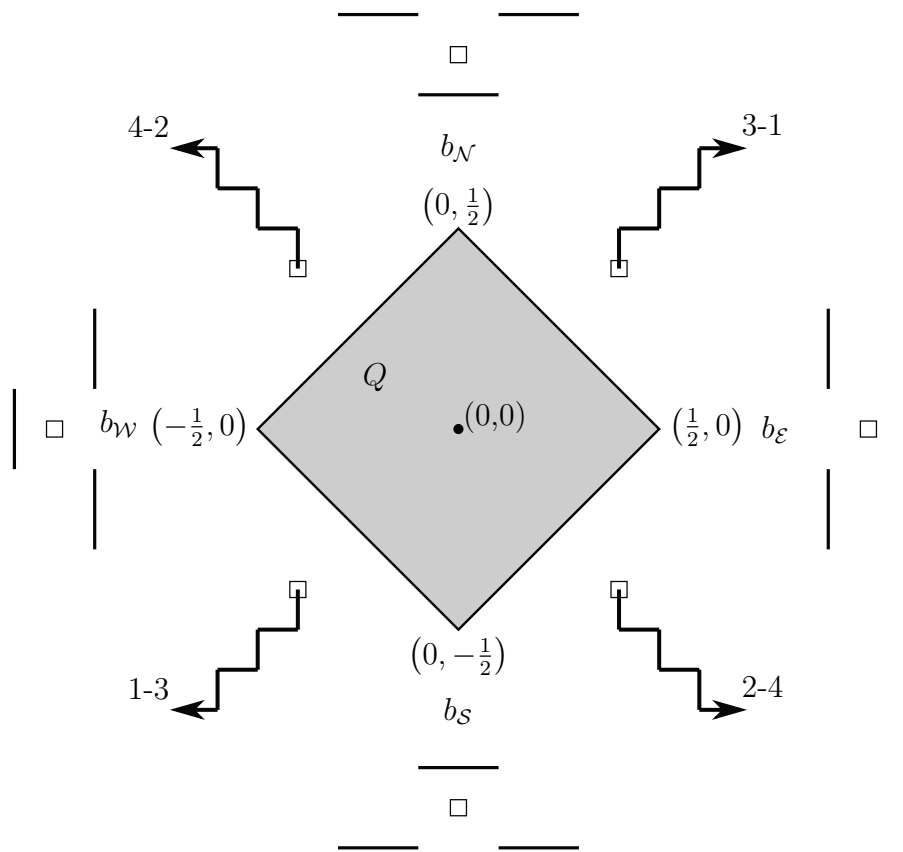
		\vspace{8pt}
		\caption{Schematic representation of the behavior of tilings with flux in $\partial Q$.}
\end{figure}

Notice how each brick wall belongs to two distinct $Q_k$'s: they are as transition tilings between their respective $Q_k$'s.
This will become clearer with the concept of stairflips.

Observe that for any doubly-infinite domino staircase, the dominoes in it are all the same type (horizontal or vertical).
A \textit{stairflip}\label{def:stairflip} on a doubly-infinite domino staircase $S$ is the process of exchanging all dominoes in $S$ by dominoes of the other type (horizontal or vertical).
It is clear a stairflip on $S$ produces a new doubly-infinite domino staircase $\widetilde{S}$;
furthermore, because the doubly-infinite staircase edge-paths that fit $S$ and $\widetilde{S}$ are the same, a stairflip preserves the type of a doubly-infinite domino staircase.
Of course, performing two successive stairflips produces no change.
\begin{figure}[H]
		\centering
		\includegraphics[width=0.65\textwidth]{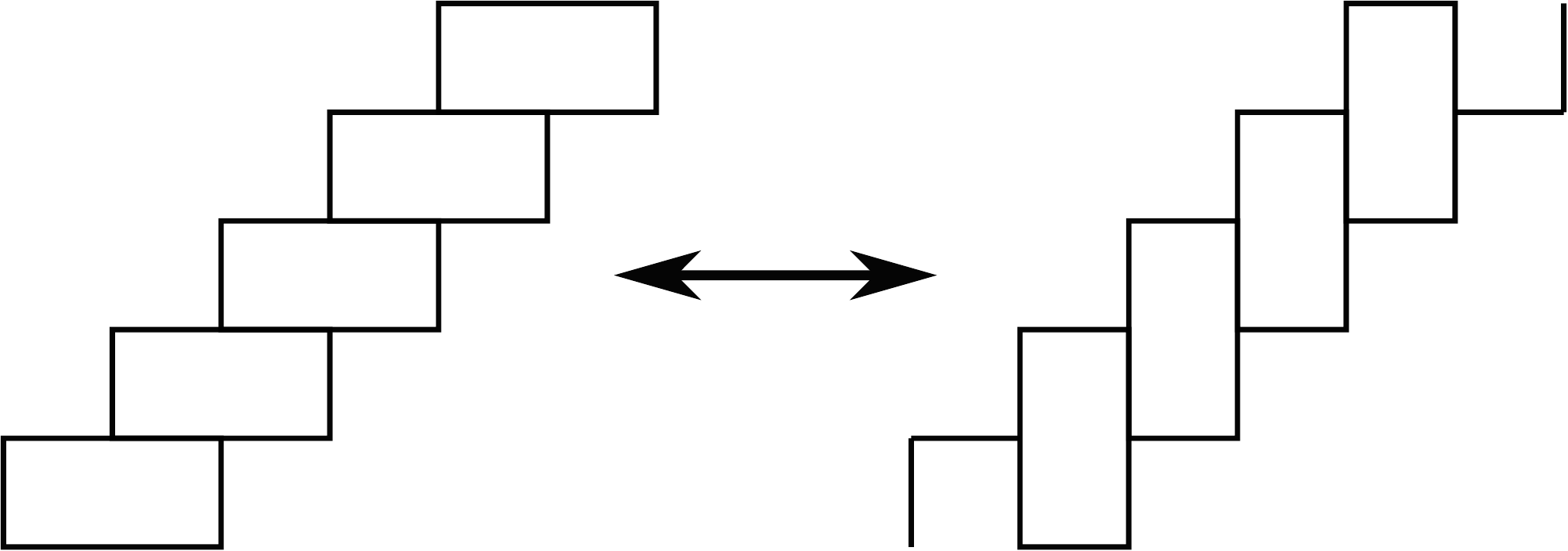}
		\caption{An example of stairflip. Remember the staircases are doubly-infinite!}
\end{figure}

Let $L$ be a valid lattice and $S$ a doubly-infinite domino staircase in $\Z^2$.
We define its $L$-equivalence class $[S]_L$\label{def:Lclassstaircase} as the set of all doubly-infinite domino staircases $\widetilde{S}$ in $\Z^2$ that are the same type as $S$ and satisfy $\Pi(\widetilde{S}) = \Pi(S)$, where $\Pi: \R^2 \longrightarrow \T_L$ is the projection map.
Notice all dominoes in all staircases of $[S]_L$ must be the same type (horizontal or vertical).

Like with flips, we may define an $L$-stairflip\label{def:Lstairflip} on a doubly-infinite domino staircase $S$: simply apply a stairflip to each staircase in $[S]_L$.

Henceforth, we will use interchangeably the terms type-\textbf{1}\label{def:stairtype} staircase and 3-\textbf{1} staircase, and similarly for types 2, 3 and 4.
Let $\text{Stair}(L)\label{def:stairL} =$ \{ $[S]_L$ $|$ $S$ is a doubly-infinite domino staircase in $\Z^2$\} and for $k\in \{ 1,2,3,4\}$ let $\text{Stair}(L;k)\label{def:stairLk} =$ \{ $[S]_L$ $|$ $S$ is a type-$k$ doubly-infinite domino staircase in $\Z^2$\}.
Clearly, $\text{Stair}(L)$ is finite.
A perhaps less obvious observation is that $\text{Stair}(L;1)$ and $\text{Stair}(L;3)$ have the same cardinality.
Indeed, translation by $e_1 = (1,0)$ is a bijection between 1-3 and 3-1 doubly infinite domino staircases, which extends into a bijection between corresponding equivalence classes.
By the same token, $\text{Stair}(L;2)$ and $\text{Stair}(L;4)$ have the same cardinality.

We may further decompose $\text{Stair}(L;k)$ into two disjoint subsets.
Let $\text{Stair}(L;k;\text{vert})$\label{def:stairLkverthor} be the set of equivalence classes in $\text{Stair}(L;k)$ whose domino staircases are all made up of vertical dominoes.
Define $\text{Stair}(L;k;\text{hor})$ similarly for horizontal dominoes.
The $L$-stairflip is an obvious bijection between them.

Define the $L$-stairflip operator\label{def:xiL} $\xi_L: \text{Stair}(L) \longrightarrow \text{Stair}(L)$.
One may restrict it to $\text{Stair}(L;k) \longrightarrow \text{Stair}(L;k)$, and furthermore to $\text{Stair}(L;k;\text{vert}) \longrightarrow \text{Stair}(L;k;\text{hor})$.
Notice that in each case, $\xi_L$ is an involution.

We may use these sets to describe tilings of $\T_L$ with flux in $\partial Q$.
Let $t$ be a tiling of $\T_L$ with flux in $Q_k$.
We know $t$ consists entirely type-$k$ doubly-infinite domino staircases, each of which corresponds to an element of $\text{Stair}(L;k)$.
We may thus identify $t$ with a subset $C(t)$ of $\text{Stair}(L;k)$.
What can we say about $C(t)$?
The crucial observation is that for each $[S]_L \in \text{Stair}(L;k)$ we have $[S]_L \in C(t)$ if and only if $\xi_L\big([S]_L\big) \notin C(t)$.
Clearly, $[S]_L$ and $\xi_L\big([S]_L\big)$ cannot both be in $C(t)$, for their lifts overlap\footnote{Recall the projection map $\Pi: \R^2 \longrightarrow \T_L$.}.
On the other hand, one of them must be in $C(t)$, for otherwise $t$ would not consist entirely of type-$k$ doubly-infinite domino staircases.
This property immediately implies $\text{card}\big(C(t)\big) = \text{card}\big(\text{Stair}(L;k;\text{vert})\big) = \text{card}\big(\text{Stair}(L;k;\text{hor})\big)$, but there's more.

We say a set $C \subset \text{Stair}(L;k)$ is $\xi_L$-$k$-exclusive\label{def:xiLkexclusive} if it satisfies
$$\forall [S]_L \in \text{Stair}(L;k), [S]_L \in C \Longleftrightarrow \xi_L\big([S]_L\big) \notin C.$$

By the preceding paragraph, every tiling $t$ of $\T_L$ with flux in $Q_k$ corresponds to a $\xi_L$-$k$-exclusive set $C(t)$.
Nonetheless, by Corollary \ref{htorocarac} and Propositions \ref{diaescada} and \ref{escfrontq}, every $\xi_L$-$k$-exclusive set $C$ also corresponds to a tiling $t_C$ of $\T_L$ with flux in $Q_k$.
The correspondence is obvious: simply lift the elements in $C$, producing an $L$-periodic tiling $t_C$ of $\Z^2$ that consists entirely of type-$k$ doubly-infinite domino staircases.

Notice $\text{Stair}(L;k;\text{vert})$ is a $\xi_L$-$k$-exclusive set, and it corresponds to a brick wall $b$.
It is clear any such set can be obtained from $\text{Stair}(L;k;\text{vert})$ by choosing a number of its elements and applying $\xi_L$ to them.
In other words, any tiling of $\T_L$ with flux in $Q_k$ can be obtained from $b$ by applying an $L$-stairflip to each of a number of equivalence classes of doubly-infinite domino staircases in $b$.
Moreover, this also means the number of tilings of $T_L$ with flux in $Q_k$ is $2^{c_k}$, where $c_k = \frac12 \text{card}\big(\text{Stair}(L;k)\big) =\text{card}\big(\text{Stair}(L;k;\text{vert})\big)$:
for each element of $\text{Stair}(L;k;\text{vert})$, choose whether or not to apply $\xi_L$ to it.

Observe that $c_k$ depends only on the parity of $k$, i.e. $c_1 = c_3$ and $c_2 = c_4$.
Moreover, they provide a way to count the number of tilings with flux in $\partial Q$:
$$2^{c_1} + 2^{c_2} + 2^{c_3} + 2^{c_4} - 4 = 2\cdot(2^{c_1} + 2^{c_2} - 2).$$

Here, we subtract $4$ because each brick wall is counted twice --- each belongs to two $Q_k$'s.

\begin{prop}\label{eschorvert}
Let $L$ be a valid lattice and $k \in \{1,2,3,4\}$.
As above, each tiling of $t$ of $\T_L$ with flux in $\mathscr{F}(L) \cap Q_k$ corresponds to a unique $\xi_L$-$k$-exclusive set; call it $C(t)$.
For each such tiling, let $n_k^{\textnormal{vert}}(t) = \textnormal{card}\big( C(t) \cap  \textnormal{Stair}(L;k;\textnormal{vert}) \big)$.
If $t$ is a tiling of $\T_L$ with flux in $\mathscr{F}(L) \cap Q_k$, the flux of $t$ depends only on $n_k^{\textnormal{vert}}(t)$.
\end{prop}
\begin{proof}
We will prove for $k = 1$, but for other values of $k$ the proof is analogous.

Let $t$ be a tiling of $\T_L$ with flux $\varphi \in \mathscr{F}(L) \cap Q_1$ and associated height function $h_t$.
By Proposition \ref{escfrontq}, $t$ consists entirely of 3-1 doubly-infinite domino staircases, so each 3-1 doubly-infinite staircase edge-path in $\Z^2$ is in $t$.
In particular, there is one such path through the origin, so for each $x \in \Z$, $h_t(x,x) = 2 x$.
Notice this implies the value $h_t$ takes on a vertex of the form $(x,x)$ in $L$ is the same for all tilings of $\T_L$ with flux in $\mathscr{F}(L) \cap Q_1$.

For $s \in \R$, let $\alpha(s) = \left(s, \frac12\right)$ be a continuous path in $\R^2$.
As $\alpha$ is traversed, it crosses each 3-1 doubly-infinite domino staircase in $t$.
Let $S_0$ be the staircase containing $\alpha\left(\frac12\right)$, and for each $i \in \Z$ let $S_{i+1}$ be the first staircase $\alpha$ crosses after $S_i$.
See the image below.
\begin{figure}[H]
		\centering
		\def\svgwidth{\columnwidth}
    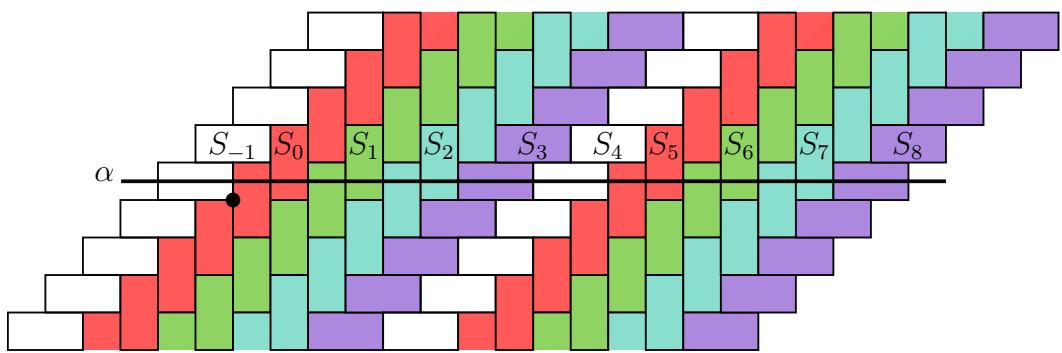
		\caption{Enumerating the $S_i$'s with $\alpha$. The marked vertex is the origin.}
\end{figure}

Observe that $[S]_L = [\widetilde{S}]_L$ if and only if there is some $v \in L$ with $S + v = \widetilde{S}$.
This means $[S_i]_L = [S_i + v]_L$ for all $i \in \Z$ and $v \in L$.
Let $c$ be the smallest positive integer for which $[S_0]_L = [S_c]_L$.
Letting $u \in L$ be such that $S_0 + u = S_c$, it's easy to see that $S_i + u = S_{i+c}$ for all $i \in \Z$.
In particular, $[S_i]_L = [S_{i+c}]_L$ for all $i \in \Z$, and no smaller positive integer may have this property.
Notice this also implies $[S_i]_L = \{S_{i+k\cdot c} \enspace | \enspace k \in \Z \}$ for all $i \in \Z$.
Indeed, suppose it were $\{S_{i+k\cdot c} \enspace | \enspace k \in \Z \} \subsetneq [S_i]_L$.
In this case, there is some integer $j$ not of the form $i +k\cdot c$ with $S_j \in [S_i]_L$.
There must be an integer $m$ such that $S_j$ lies between $S_{i+m\cdot c}$ and $S_{i+(m+1)\cdot c}$.
Let $w \in L$ have the property that $S_{i+m\cdot c} + w = S_j$.
Then $[S_0 + w]_L = [S_0]_L$, and letting $\tilde{c}$ be such that $S_0 + w = S_{\tilde{c}}$ it suffices to note that $0 < \tilde{c} < c$, contradicting the minimality of $c$.

We claim $C(t) = \{[S_0]_L, [S_1]_L, \dots, [S_{c-1}]_L \}$.
Indeed, since $[S_i]_L = [S_{i+c}]_L$ for all $i \in \Z$, $C(t) \subseteq \{[S_0]_L, [S_1]_L, \dots, [S_{c-1}]_L \}$.
On the other hand, because $[S_i]_L = \{S_{i+k\cdot c} \enspace | \enspace k \in \Z \}$ for all $i \in \Z$, the elements in $\{[S_0]_L, [S_1]_L, \dots, [S_{c-1}]_L \}$ are all distinct, so the inclusion must be an equality.
Notice this also implies $c = \text{card}\big(C(t)\big) = \text{card}\big(\text{Stair}(L;1;\text{vert})\big)$, which does not depend on $t$.

Let $\gamma$ be the doubly-infinite staircase edge-path through the origin that fits $S_0$.
Let $R \subset \R^2$ be the connected region enclosed by $\gamma$ and $\gamma + u$.
We claim $\text{int}(R) \cap L = \varnothing$.
Indeed, $\gamma + u$ fits $S_c$ on the same side that $\gamma$ fits $S_0$.
If there were some $v \in \text{int}(R) \cap L$, then there would be some integer $0 <j < c$ with $S_j = S_0 + v \neq S_c$, contradicting the minimality of $c$.

Observe that for any tiling of $\T_L$, its flux is entirely determined by the value its height function takes on two linearly independent vectors on $L$.
By Lemma \ref{xx}, $\gamma$ intersects $L$ away from the origin, say at $w$; and because $S_c = S_0 + u \neq S_0$, $u$ is not of the form $(x,x)$.
In other words, $u$ and $w$ are linearly independent, so $\varphi$ is entirely defined by the value $h_t$ takes on them.
Now, we've shown the value $h_t$ takes on vertices of the form $(x,x)$ in $L$ is the same for all tilings of $\T_L$ with flux in $\mathscr{F}(L) \cap Q_1$, so $\varphi$ is entirely defined by the value $h_t$ takes on $u$.

Let $\beta_0$ be the horizontal edge-path in $\Z^2$ of length $2c$ joining the origin to $\gamma + u$; its endpoint is $(2c,0)$.
Let $\beta_1$ be the edge-path in $\gamma + u$ joining $(2c,0)$ to $u$, so $\beta_0 * \beta_1$ is an edge-path in $\Z^2$ joining the origin to $u \in L$.
Notice $\gamma$ is a 3-1 staircase, so it is an edge-path in each tiling of $\T_L$ with flux in $\mathscr{F}(L) \cap Q_1$.
Of course, this means $\gamma + u$ and $\beta_1$ also have this property, so each such tiling has the same height change along $\beta_1$.
Thus, $h_t(u)$ depends only on the height change along $\beta_0$.

Each edge joining vertices $(2k-1,0)$ and $(2k,0)$ is an edge on a 3-1 staircase, so it is an edge on each tiling of $\T_L$ with flux in $\mathscr{F}(L) \cap Q_1$.
Each edge joining vertices $(2k,0)$ and $(2k+1,0)$ is an edge that crosses $S_k$.
It follows that the height change along $\beta_0$ depends only on these latter edges, and thus depends only on how many dominoes $\beta_0$ crosses along those.
Now, it's easy to check that for each horizontal edge crossing $S_k$, that edge crosses a domino if and only if $S_k$ is made up of vertical dominoes, that is, if and only if $[S_k]_L \in \text{Stair}(L;1;\text{vert})$.
The Proposition follows from the fact that $\beta_0$ crosses each equivalence class in $C(t)$.
\end{proof}

We saw that any $\xi_L$-$k$-exclusive set can be obtained from $\text{Stair}(L;k;\text{vert})$ by choosing a number of its elements and applying $\xi_L$ to them.
By Proposition \ref{eschorvert}, this number entirely determines the flux of the corresponding tiling of $\T_L$, regardless of the choice of elements themselves.
We also knew how to count the total tilings of $\T_L$ with flux in $\mathscr{F}(L) \cap Q_k$, but now we have a way to count the tilings of $\T_L$ with a given flux $\varphi \in \mathscr{F}(L) \cap Q_k$.
Indeed, let $c_k = \text{card}\big(\text{Stair}(L;k;\text{vert})\big)$.
By Proposition \ref{eschorvert}, $n_k^{\text{vert}}(t)$ is the same for each tiling $t$ of $\T_L$ with flux $\varphi$; we may thus speak of $n_k^{\text{vert}}(\varphi)$.
The number of tilings of $\T_L$ with flux $\varphi$ is then simply $\binom {c_k}{n_k^{\text{vert}}(\varphi)}$, corresponding to a choice of $n_k^{\text{vert}}(\varphi)$ elements in $\text{Stair}(L;k;\text{vert})$ to keep, and applying $\xi_L$ to the others.

Let $L$ be a valid lattice and fix $k \in \{1,2,3,4\}$.
Let $t$ be any tiling of $\T_L$ with flux $\varphi \in \mathscr{F}\cap Q_k$ and $\xi_L$-$k$-exclusive set $C(t)$.
If $C(t) \neq \text{Stair}(L;k;\text{hor})$, then $C(t) \cap \text{Stair}(L;k;\text{vert})$ is nonempty.
Let $\widetilde{C}$ be any set obtained from $C(t)$ by choosing an element in $C(t) \cap \text{Stair}(L;k;\text{vert})$ and applying $\xi_L$ to it.
By Proposition \ref{eschorvert}, the flux of the tiling that corresponds to $\widetilde{C}$ does not depend on the choice of element; let $\xi_{L;k}^\text{vert}(\varphi)$ be that flux.

\begin{prop}\label{frontqvec}
The vector $\varphi - \xi_{L;k}^\text{vert}(\varphi)$ is constant and nonzero across all $\varphi \in \big( \mathscr{F}(L)\cap Q_k \big) \setminus \left\{\big(0,\frac12\big), \big(0,-\frac12 \big) \right\}$.
\end{prop}
\begin{proof}
Notice $b_{\mathcal{N}}$ and $b_{\mathcal{S}}$ are the only tilings whose $\xi_L$-$k$-exclusive sets may be given by $\text{Stair}(L;k;\text{hor})$, and their fluxes are respectively $\big(0,\frac12\big)$ and $\big(0,-\frac12 \big)$.
Since these are excluded in the statement, $\xi_{L;k}^\text{vert}(\varphi)$ is always well-defined.

Let $t, \tilde{t}$ be tilings of $\T_L$ with fluxes respectively $\varphi$, $\xi_{L;k}^\text{vert}(\varphi)$ and associated height functions respectively $h$, $\tilde{h}$.
Observe that, given two linearly independent vectors $w_0,w_1 \in \R^2$, any vector $w \in \R^2$ is entirely determined by $\langle w, w_0 \rangle$ and $\langle w, w_1 \rangle$.
For vertices on $L$, $h$ and $\tilde{h}$ are given by inner product formulas with respectively $\varphi$ and $\xi_{L;k}^\text{vert}(\varphi)$, so $\varphi - \xi_{L;k}^\text{vert}(\varphi)$ is entirely defined by the value $h - \tilde{h}$ takes on two linearly independent vertices of $L$.
The idea of the proof is to show $h - \tilde{h}$ is constant on two linearly independent vertices of $L$ (across all $\varphi$ as in the statement), so $\varphi - \xi_{L;k}^\text{vert}(\varphi)$ is always the same, and nonzero if $h - \tilde{h}$ is nonzero on at least one of those vectors.

Recall the proof of Proposition \ref{eschorvert}; in what follows, we will use its notation and ideas.
Once again, we will show only the case $k=1$, but for other values of $k$ the proof is analogous.

Because $t, \tilde{t}$ have flux in $\mathscr{F}(L) \cap Q_1$, their height functions coincide on $(x,x)$ for all $x \in \Z^2$; in other words, $h - \tilde{h}$ is always 0 on those vertices.
Since by Lemma \ref{xx} there is a nonzero $v \in L$ with that form, we need only find a vertex in $L$ not of that form on which $h - \tilde{h}$ is constant and nonzero.

Let $C(t)$ be $t$'s $\xi_L$-$1$-exclusive set and similarly for $C(\tilde{t})$.
Choose $\tilde{t}$ so that $C(t)$ and $C(\tilde{t})$ differ in only one element; it's clear this is always possible (as in the paragraph just before the statement of Proposition \ref{frontqvec}).
Analyzing how $h$ and $\tilde{h}$ change along $\beta_0$, it is clear they differ only along the edge $e$ crossing the staircase $S$ whose equivalence class is different in $C(t)$ and $C(\tilde{t})$.
In $t$, $[S]_L \in \text{Stair}(L;1;\text{vert})$, so $e$ crosses a domino; in $\tilde{t}$, $[S]_L \in \text{Stair}(L;1;\text{hor})$, so $e$ does not cross a domino.
Therefore, the coloring of $\Z^2$ implies $h$ changes by $+3$ along $e$ while $\tilde{h}$ changes by $-1$ along $e$.
It follows that $h(u) - \tilde{h}(u)$ is always 4, and we are done.
\end{proof}

Proposition \ref{frontqvec} provides a visual way to interpret the counting of tilings of $\T_L$ with flux in $\mathscr{F}(L) \cap \partial Q$.
For each $k$, orient $Q_k$ (as a line segment in $\R^2$) from $\text{Stair}(L;k;\text{vert})$ to $\text{Stair}(L;k;\text{hor})$.
Let $\mathscr{F}(L) \cap Q_k = \{ p_0, p_1, \dots, p_{c_k}\}$, where the order respects $Q_k$'s orientation; in particular, $p_0$ corresponds to $\text{Stair}(L;k;\text{vert})$ and $p_{c_k}$ corresponds to $\text{Stair}(L;k;\text{hor})$.
Notice $c_k = \text{card}\big(\text{Stair}(L;k;\text{vert})\big)$.

For each $0 \leq j < c_k$,  Proposition \ref{eschorvert} implies $\xi_{L;k}^\text{vert}(p_j) \in Q_k\setminus \{p_j\}$.
In particular, if $p_i = \xi_{L;k}^\text{vert}(p_0)$, then $i>0$.
By Proposition \ref{frontqvec}, if we define $p_{i_j} = \xi_{L;k}^\text{vert}(p_j)$ it follows that $i_j > j$, for $p_{i_j} - p_j$ is constant.
In particular, it must be $i_{c_k-1} = c_k$, so by induction we have $i_j = j+1$.
The following formula thus holds: $\big(\xi_{L;k}^\text{vert}\big)^{i}(p_0) = p_i$.

This means that for all $0 \leq i \leq c_k$, a tiling of $\T_L$ with flux $p_i$ is obtained from the only tiling of $\T_L$ with flux $p_0$ --- the brick wall that corresponds $\text{Stair}(L;k;\text{vert})$ --- by choosing $i$ elements in its $\xi_L$-$k$-exclusive set that are also in $\text{Stair}(L;k;\text{vert})$ and applying $\xi_L$ to them.
In other words, $n_k^{\text{vert}}(p_i)=c_k-i$, so the number of tilings of $\T_L$ with flux $p_i$ is simply $\binom{c_k}{c_k-i} = \binom{c_k}{i}$.

Notice this means that regardless of $L$ or $\mathscr{F}(L)$'s behaviour, if we know $\mathscr{F}(L) \cap Q_k$, then for each $\varphi \in \mathscr{F}(L) \cap Q_k$ we know the number of tilings of $\T_L$ with flux $\varphi$.

Moreover, the effect of $L$-stairflips on a tilings with flux in $\mathscr{F}(L)\cap \partial Q$ can now be better understood: it navigates between adjacent fluxes.
We explain it: let $t$ be one such tiling and suppose its flux $\varphi$ lies in the interior of $Q_k$.
Then it consists entirely of type-$k$ doubly-infinite domino staircases, some of which are made up of vertical dominoes and some of which are made of horizontal ones.
Each extremal point of $Q_k$ corresponds to a different tiling that consists entirely of type-$k$ doubly-infinite staircases and uses dominoes of only one type (vertical or horizontal).
Applying an $L$-stairflip to a staircase in $t$ takes us to a tiling of $\T_L$ whose flux is closest to $\varphi$ in $Q_k$: if that stairflip is applied to a vertical staircase, the new flux is closest to the extremal point of horizontal staircases, and vice-versa. 
Now, if the flux $\varphi$ lies in the boundary of $Q_k$, it is one of the brick walls.
In this case, $t$ can be seen as consisting entirely of type-$k_0$ domino staircases and entirely of type-$k_1$ domino staircases: it is in $\partial Q_{k_0} \cap \partial Q_{k_1}$.
Applying an $L$-stairflip to a type-$k_i$ staircase in $t$ takes us to a tiling of $\T_L$ whose flux is closest to $\varphi$ in $Q_{k_i}$ --- and now there's only one such flux!
This also makes it clear how brick walls are as transition tilings between different $Q_k$'s.

We are now ready to prove Proposition \ref{signfrontq}.

\begin{proof}[Proof of Proposition \ref{signfrontq}]
The discussion before the statement of Proposition \ref{signfrontq} makes it clear we need only prove the case $\varphi \in \mathscr{F}(L) \cap \partial Q$.
Furthermore, it suffices to show that, for each tiling $t$ of $\T_L$ with flux $\varphi \in \mathscr{F}(L) \cap \partial Q$, the sign in $K_t = \pm q_0^iq_1^j$ depends only on the flux (and not on $t$).

By Proposition \ref{escfrontq}, given $\varphi \in \mathscr{F}(L) \cap \partial Q$, there is some $k \in \{1,2,3,4\}$ such that each tiling of $\T_L$ with flux $\varphi$ consists entirely of type-$k$ doubly-infinite domino staircases.
In particular, $\varphi \in Q_k$.
Let $c_k = \text{Stair}(L;k;\text{vert})$.
For each tiling of $\T_L$ with flux in $Q_k$, let $C(t)$ be its $\xi_L$-$k$-exclusive set.
By Proposition \ref{eschorvert}, each tiling $t$ of $\T_L$ with flux $\varphi$ satisfies $\text{card}\big(C(t) \cap \text{Stair}(L;k;\text{vert}) \big) = n_k^{\text{vert}}(\varphi)$, and $\text{card}\big(C(t) \cap \text{Stair}(L;k;\text{hor}) \big) = c_k  -n_k^{\text{vert}}(\varphi)$; in particular, neither depends on $t$.

Remember our universal Kasteleyn signing, based on $b_{\mathcal{N}}$.
An equivalence class of edges on $G(\Z^2)$ is on $b_{\mathcal{N}}$, and thus is negatively signed, if and only if it lies in an element of $\text{Stair}(L;1;\text{hor})$ or in an element of $\text{Stair}(L;2;\text{hor})$.
Now,
\begin{alignat*}{1}
&\text{card}\bigg(C(t) \cap \Big( \text{Stair}(L;1;\text{hor}) \sqcup \text{Stair}(L;2;\text{hor}) \Big) \bigg) = \\
&\text{card}\Big( C(t) \cap \text{Stair}(L;1;\text{hor}) \Big) + \text{card}\Big( C(t) \cap \text{Stair}(L;2;\text{hor}) \Big),
\end{alignat*}
and the preceding paragraph then implies that for each tiling $t$ of $\T_L$ with flux $\varphi$, the expression above depends only on $\varphi$.

Notice that the number of equivalence classes of edges on $G(\Z^2)$ in any two elements of $\text{Stair}(L;1) \sqcup \text{Stair}(L;3)$ is the same, and similarly for any two elements of $\text{Stair}(L;2) \sqcup \text{Stair}(L;4)$:
the former is given by the smallest $x>0$ for which $(x,x) \in L$, and the latter is given by the smallest $x>0$ for which $(x,-x) \in L$.
Together with the previous paragraph, this means that for each tiling of $\T_L$ with flux $\varphi$, the number of equivalence classes of edges on $b_{\mathcal{N}}$ is the same.
Moreover, this also shows that for each such tiling $t$ the permutation associated to $K_t$ is given by $c_k$ cycles of equal length; in particular, the signs of these permutation are all the same.

It follows that the sign of $K_t$ depends only on $\varphi$, as desired.
\end{proof}
\chapter{Sign distribution over $\mathscr{F}(L)$}
\label{chap:signflux}

We now know that each tiling of $\T_L$ with flux $\varphi \in \mathscr{F}(L)$ is assigned the same sign in the combinatorial expansion of the Kasteleyn determinant, so we may speak of the sign of $\varphi$.
Our next results will work to describe how these signs are distributed.

\section{Cycles and cycle flips}
\label{sec:cyc}

Let $t_0, t_1$ be two tilings of $\T_L$.
Represent both simultaneously on a fundamental domain $D_L$ and orient dominoes of $t_0$ from black to white and dominoes of $t_1$ from white to black.
With this orientation, $D_L$ is decomposed into disjoint domino cycles whose dominoes belong alternatingly to $t_0$ and $t_1$; call the collection of these cycles $C(t_0,t_1)$.
This set provides a way to go from $t_0$ to $t_1$: for each cycle $c$ in $C(t_0, t_1)$\label{def:ct0t1}, simply replace each domino in $c$ that is also in $t_0$ with the domino that follows it in $c$.
We call this process a \textit{cycle flip}\label{def:cycflip}.

The image below provides an example of this construction. 
Notice each cycle in $C(t_0, t_1)$ defines edge-paths that are in both $t_0$ and $t_1$ (the edge-paths that fit that cycle).
\begin{figure}[H]
		\centering
		\includegraphics[width=0.65\textwidth]{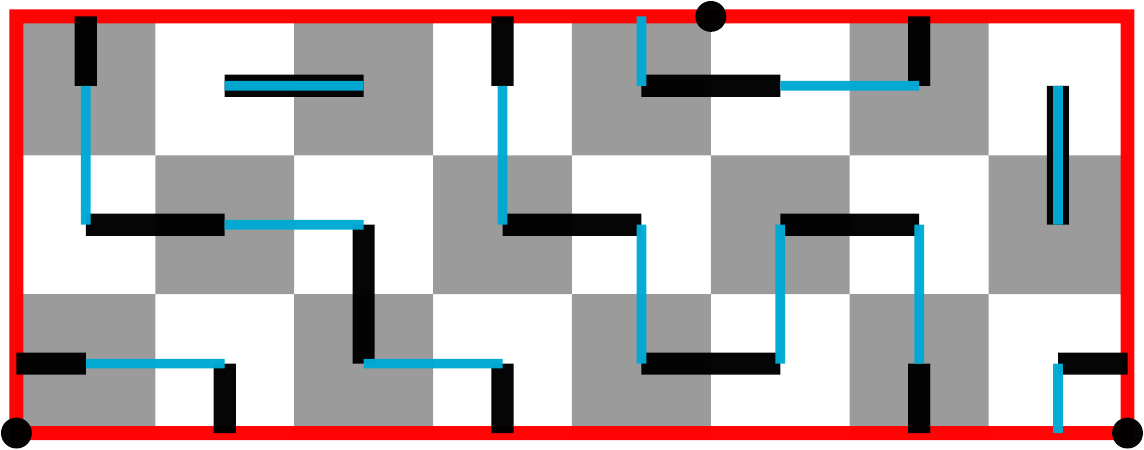}
		\caption{An example of cycle construction. The dominoes of $t_0$ are represented by the black, thicker edges while the dominoes of $t_1$ are represented by the thinner, blue edges. Marked vertices are in the lattice $L$, and the red rectangle is the fundamental domain $D_L$.}
\end{figure}

Of course, we may lift this representation to $\Z^2$, decomposing it into disjoint, $L$-periodic domino paths whose dominoes belong alternatingly to $t_0$ and $t_1$.
Under this representation, each cycle is a collection of either finite, closed domino paths or infinite domino paths.
We will refer to the former by \textit{closed cycles} and to the latter by \textit{open cycles}.
\begin{figure}[ht]
		\centering
		\includegraphics[width=0.99\textwidth]{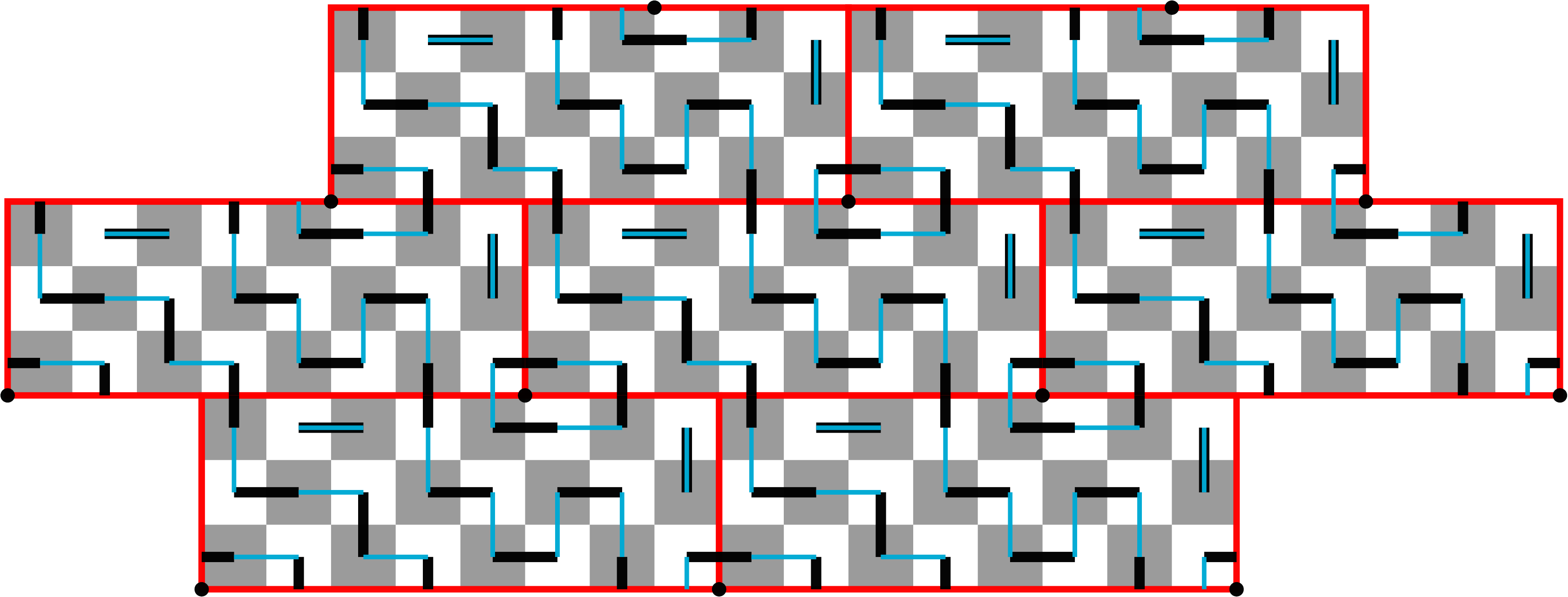}
		\caption{Lifting our previous example to $\Z^2$. We can see $C(t_0,t_1)$ has two trivial cycles, one closed cycle, and two open cycles.}
\end{figure}

Remember dominoes may be seen as edges on $G(\Z^2)$, and vertices of $G(\Z^2)$ lie on $\big(\Z+\frac12\big)^2$.
Thus, each domino path of a cycle can be seen as an edge-path in $\big(\Z+\frac12\big)^2$, which decomposes $\R^2$ into two disjoint, connected components.
For open cycles these components are both unbounded, while for closed cycles one is unbounded and one is bounded.
In the latter case, we call the unbounded component the exterior of the path, and the bounded component the interior of the path.

Define the interior $\text{int}(c)$\label{def:intc} of a closed cycle $c$ to be the union of the interior of domino paths in $c$, and the exterior $\text{ext}(c)$\label{def:extc} of $c$ to be the intersection of the exterior of domino paths in $c$.
Notice $\text{int}(c)$ is never connected, and $\text{ext}(c)$ is always connected; moreover, $\text{ext}(c) \cap \Z^2$ is always connected by edge-paths.
We provide an example of these constructions below.
\begin{figure}[ht]
		\centering
		\includegraphics[width=0.99\textwidth]{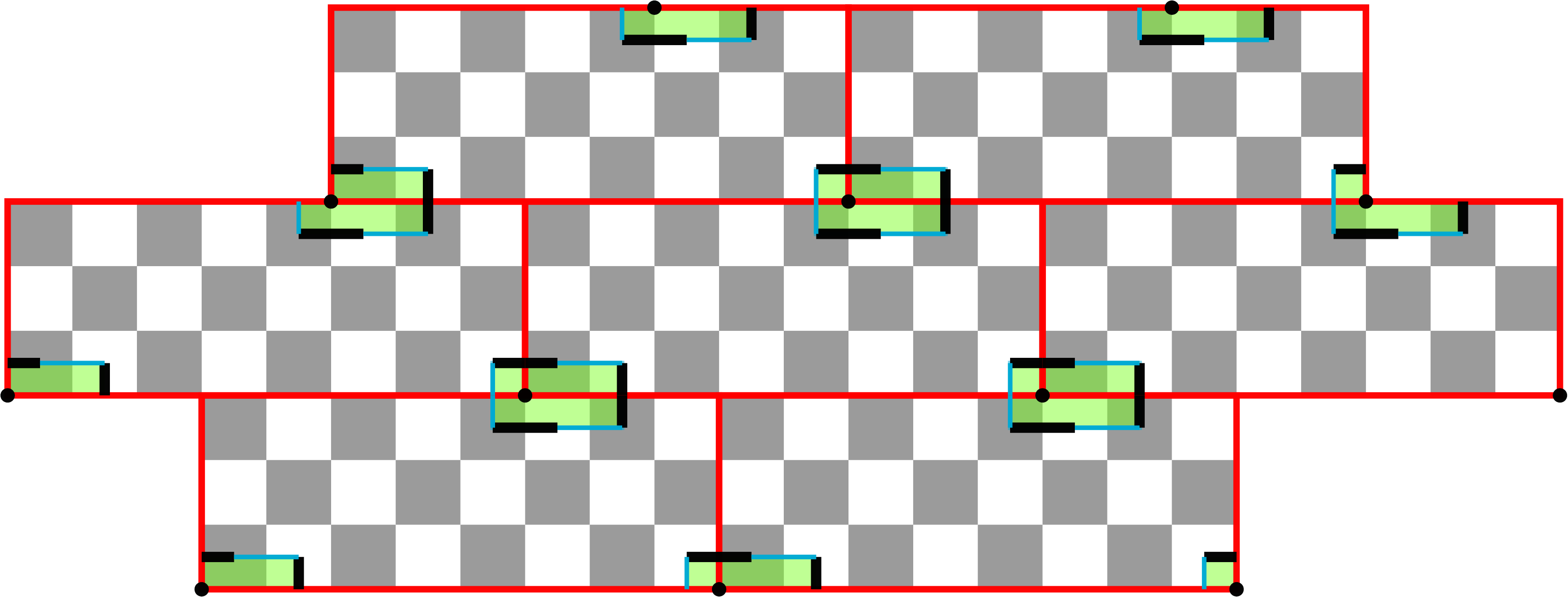}
		\caption{The closed cycle in our previous example. Its interior is tinted green.}
		\vspace{5pt}
\end{figure}

\begin{prop}
Let $L$ be a valid lattice and $t_0, t_1$ be two tilings of $\T_L$ with fluxes respectively $\varphi_0, \varphi_1 \in \mathscr{F}(L)$.
If $t_0$ and $t_1$ differ by a single closed cycle $c$, then $\varphi_0=\varphi_1$.
\end{prop}
\begin{proof}
Let $h_i$ be $t_i$'s associated height function.
We will show $h_0$ and $h_1$ agree on $L$, from which the proposition follows.
There are two cases:
\begin{enumerate}
	\item $\text{int}(c)$ does not intersect $L$;
	\item $\text{int}(c)$ does intersect $L$.
\end{enumerate}

In the first case, $L \subset \text{ext}(c)$, so any two points of $L$ can be joined by an edge-path contained entirely in $\text{ext}(c)$.
Since $t_0$ and $t_1$ coincide along these edge-paths, the height change along them is the same for both $h_0$ and $h_1$, and the they agree on $L$.

In the second case, because domino paths of $c$ are $L$-periodic, each $v \in L$ must belong to the interior of a single domino path of $c$, and each domino path of $c$ must contain a single $v \in L$. Let $v_a, v_b \in L$ and $\delta_a, \delta_b$ be their respective domino paths of $c$.
Let $u_a \in \Z^2$ be the first point to $v_a$'s right in the exterior of $\delta_a$, and similarly for $u_b$.
Let $\gamma_a$ be the horizontal edge-path in $\Z^2$ joining $v_a$ to $u_a$ and $\gamma_b$ be the horizontal edge-path in $\Z^2$ joining $u_b$ to $v_b$.
Finally, let $\gamma_{ab}$ be any edge-path in $\text{ext}(c) \cap \Z^2$ joining $u_a$ to $u_b$, so $\gamma = \gamma_a * \gamma_{ab} * \gamma_b$ is an edge-path in $\Z^2$ joining $v_a$ to $v_b$.
Notice $\gamma$ has a single edge $e_a$ that crosses $\delta_a$, and a single edge $e_b$ that crosses $\delta_b$.

Since $t_0$ and $t_1$ agree except on $c$, $h_0$ and $h_1$ have the same change along $\gamma$ except along $e_a$ and $e_b$.
Now, because of $c$'s $L$-periodicity, $e_a$ and $e_b$ have the same color-induced orientation, but are traversed by $\gamma$ in opposite directions.
Moreover, the $L$-periodicity also implies the domino $e_a$ crosses in $\delta_a$ and the domino $e_b$ crosses in $\delta_b$ belong to the same $t_i$.
In other words, the changes of $h_i$ along $e_a$ and along $e_b$ have equal magnitude but opposite signs, so they cancel each other out.
This means the total change of $h_0$ and $h_1$ along $\gamma$ is in fact the same.
Since they agree on the origin (and the origin is in $L$), they must agree everywhere on $L$, and we are done.
\end{proof}

\begin{corolario}\label{closcyc}
Let $L$ be a valid lattice and $t_0, t_1$ be two tilings of $\T_L$ with fluxes respectively $\varphi_0, \varphi_1 \in \mathscr{F}(L)$.
If each cycle in $C(t_0, t_1)$ is closed, then $\varphi_0 = \varphi_1$.
\end{corolario}

\section{The effect of a cycle flip on the sign of a flux}

We know any two tilings of $\T_L$ can be joined by cycle flips, but by Corollary \ref{closcyc} only flips on open cycles can affect the flux.
Our attention now turns to studying how flips on open cycles affect the Kasteleyn sign of a tiling, and thus of their respective fluxes.

Let $L$ be a valid lattice.
We say $v \in L$ is \emph{short}\label{def:short} if $s\cdot v \notin L$ for all $s \in [0,1)$.

Let $t_0$, $t_1$ be two tilings of $\T_L$.
Let $c$ be an open cycle in $C(t_0, t_1)$.
Let $\gamma_c$ be any infinite domino path of $c$.
Since these paths are $L$-periodic, for each $v \in L$ $\gamma_c + v$ is a domino path of $c$.
Moreover, there is a short $u\in L$ for which $\gamma_c = \gamma_c + u$, and it is clear $u$ is unique up to multiplication by $-1$.
Now, if $\widetilde{\gamma}_c$ is another infinite domino path of $c$, $L$-periodicity implies this unique vector is the same.
We will then say $u=u(c)$ is $c$'s \textit{parameter}\label{def:cycparam}.

Let $\widetilde{c}$ be any other open cycle in $C(t_0,t_1)$.
We claim $u(c) = u(\widetilde{c})$.
Indeed, if there were some $k \in \R\setminus\{-1,1\}$ with $u(c) = k \cdot u(\widetilde{c})$, it would contradict the shortness of $u(c)$ or of $u(\widetilde{c})$.
On the other hand, if $u(c)$ and $u(\widetilde{c})$ were linearly independent, $c$ and $\tilde{c}$ would intersect, contradicting their disjointness.
It follows that whenever $C(t_0,t_1)$ contains an open cycle, it has a well-defined parameter $u = u(t_0,t_1)$, a short vector in $L$ unique up to multiplication by $-1$.

For any short $v \in L$, a \textit{$v$-quasicyle}\label{def:quasicyc} is a function $\gamma: \Z \longrightarrow \big(\Z+\frac12\big)^2$ with $\lVert \gamma(t+1) - \gamma(t) \rVert = 1$ for all $t \in \Z$ and such that there is a positive integer $T$ with $\gamma(t+T) = \gamma(t)+v$ for all $t \in \Z$.
We say $T$ is $\gamma$'s quasiperiod, and notice it is always even (because $v$ is in a valid lattice).
We say $\gamma: \Z \longrightarrow \big(\Z+\frac12\big)^2$ is \textit{simple} if it is injective; in other words, if it does not self-intersect in the plane.

Observe that any quasicyle $\gamma$ can be interpreted as a domino-path in the infinite square lattice.
With this in mind, let $C(t_0,t_1)$ contain an open cycle and $v$ be its parameter.
It's easy to see that for any open cycle $c \in C(t_0,t_1)$ and any infinite domino path $\gamma_c$ of $c$, $\gamma_c$ is a simple $v$-quasicycle.

For any $v$-quasicycle $\gamma$, define its sign by
$$\sgn(\gamma)=(-1)^{^{\textstyle \big(\frac{T}{2}+1\big)}} \cdot \prod\limits_{0 \leq t < T}{\sgn\Big(\big[\gamma(t),\gamma(t+1)\big] \Big)},$$
where $T$ is $\gamma$'s quasiperiod and $\sgn(e)$ is the Kasteleyn sign of the edge $e$ in $G(\Z^2) = \big(\Z + \frac12\big)^2$.

Notice that for any $u$ in $L$ and any edge $e$ in $G(\Z^2)$, $\sgn(e) = \sgn(e+u)$.
In particular, because $\gamma(t+T) = \gamma(t) + v$ and $v \in L$, we always have that
$$\sgn\Big(\big[\gamma(t),\gamma(t+1)\big] \Big) = \sgn\Big(\big[\gamma(t),\gamma(t+1)\big] + v\Big) = \sgn\Big(\big[\gamma(t + T),\gamma(t+1 + T)\big] \Big).$$

This means the sign of a quasicycle $\gamma$ obtained from any infinite domino path of an open cycle in $C(t_0,t_1)$ does not depend on a particular choice of parametrization (a choice of edge to be $\gamma(0)$).
Moreover, any two infinite domino paths of one same open cycle in $C(t_0,t_1)$ are related by a translation in $L$, which preserves the quasiperiod and each Kasteleyn sign.
We may thus define the sign of an open cycle $c \in C(t_0,t_1)$ to be the sign of any quasicycle obtained from any infinite domino path of $c$.

Let $t_0, t_1$ be tilings of $\T_L$ with fluxes respectively $\varphi_0$ and $\varphi_1$.
Suppose $\varphi_0 \neq \varphi_1$. By Corollary \ref{closcyc}, $C(t_0,t_1)$ contains an open cycle $c$.
Let $\tilde{t}$ be obtained from $t_0$ by a cycle flip on $c$, and let $\widetilde{\varphi}$ be its flux.
Notice $\varphi_0 \neq \widetilde{\varphi}$, and by construction $\sgn(c)$ is the sign change produced by the cycle flip on $c$, so
$$\frac{\sgn(\varphi_0)}{\sgn(\widetilde{\varphi})}=\sgn(c).$$

Indeed, $\frac{T}{2}$ is the length of the permutation cycle that represents the cycle flip, so its sign is $(-1)^{\big(\frac{T}{2}+1\big)}$.
The product accounts for the sign changes from the Kasteleyn signing.

When $\gamma$ is simple, it divides $\R^2$ into two unbounded connected components: $\R^2_+(\gamma)$ to the left of $\gamma$, and $\R^2_-(\gamma)$ to the right of $\gamma$.
For any simple $v$-quasicycle $\gamma$, let $\gamma_+$ be the edge-path in $\R^2_+(\gamma)\cap\Z^2$ that fits it, and similarly for $\gamma_-$.
The height changes along $\gamma_+$ and $\gamma_-$ are well-defined.
For any vertex $w$ in $\gamma_+$, there is an edge-path along $\gamma_+$ joining $w$ to $w+v$; call it $\gamma_+^w$.
We claim the height change along $\gamma_+^w$ does not depend on $w$.
Indeed, if $w_1,w_2$ are vertices in $\gamma_+$, there is an edge-path $\gamma_+^{w_1,w_2}$ along $\gamma_+$ joining $w_1$ to $w_2$.
Because $\gamma$ is a $v$-quasicycle, $\gamma_+^{w_1,w_2} + v$ is also an edge-path in $\gamma_+$; it joins $w_1+v$ to $w_2 + v$.
It's easy to see the height change along $\gamma_+^{w_1,w_2}$ and $\gamma_+^{w_1,w_2} + v$ is the same, from which the claim follows.
Define then $h(\gamma_+)$ to be the height change along any $\gamma_+^w$.
Of course, the same applies to $\gamma_-$, and $h(\gamma_-)$ is well-defined.
We claim $h(\gamma_+) = h(\gamma_-)$.

Indeed, let $u_+$ and $u_-$ be adjacent vertices with $u_+$ in $\gamma_+$ and $u_-$ in $\gamma_-$.
Let $e$ be the edge joining $u_+$ to $u_-$.
Of course, $u_+ + v$ and $u_- + v$ are also adjacent vertices in their respective edge-paths, and $e + v$ joins them.
Consider then the edge-path $\beta = \gamma_+^{u_+}*(e+v)*\left(\gamma_-^{u_-}\right)^{-1}*e^{-1}$, where $e$ is oriented from $u_+$ to $u_-$ and ${\Bigcdot}^{-1}$ indicates a reversal of orientation.
\begin{figure}[H]
		\centering
		\def\svgwidth{0.99\columnwidth}
    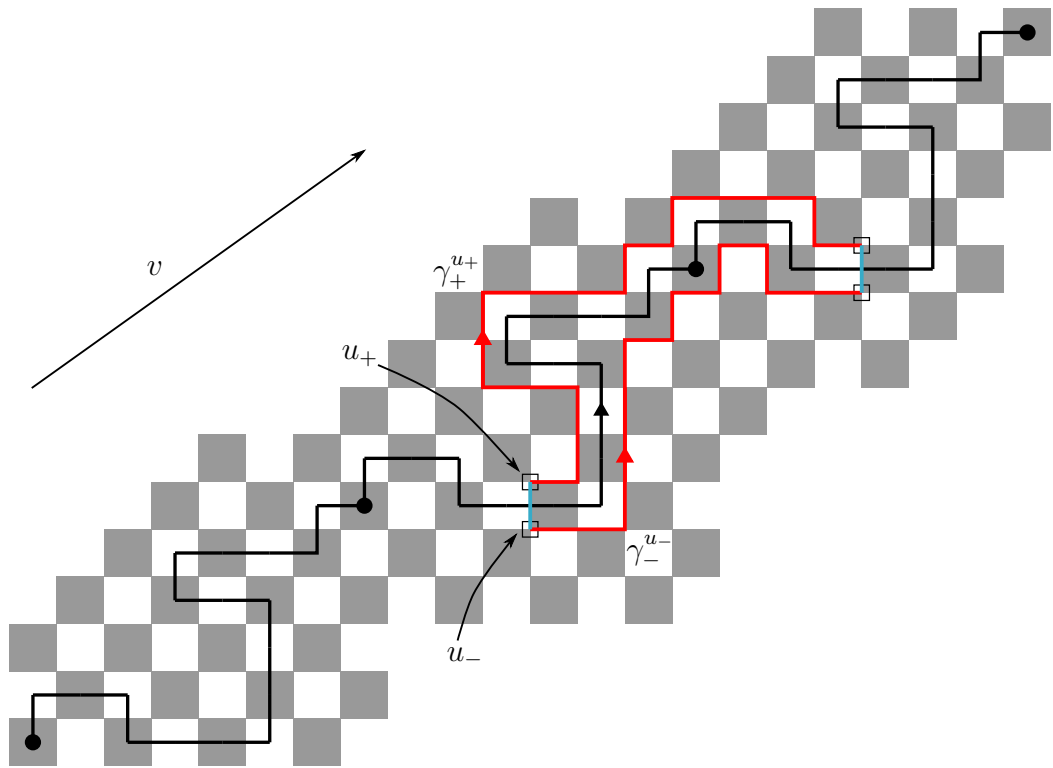
		\caption{The simple $v$-quasicycle $\gamma$ in black, the marked vertices $u_+$, $u_+ + v$, $u_-$ and $u_- + v$, and the edge-paths $\gamma_+^{u_+}$ and $\gamma_-^{u_-}$ in red.}
\end{figure}

Since it is closed, the height change along $\beta$ is 0, so $h(\gamma_+) + h(e+v) - h(\gamma_-) - h(e) = 0$, where $h(e)$ is the height change along $e$.
The claim follows from noting that $h(e+v) = h(e)$, because $v \in L$.

We say a tiling of $\T_L$ is \emph{compatible with $\gamma$}\label{def:compatible} if $t$ contains every other domino in $\gamma$.
In this case, it's clear $\gamma_+$ and $\gamma_-$ are edge-paths in $t$, so by definition
\begin{equation}\label{varphigammaflux}
h(\gamma_+) = 4 \cdot \langle \varphi_t , v \rangle = h(\gamma_-),
\end{equation}
where $\varphi_t$ is the flux of $t$.

Define then $\gamma$'s \textit{pseudo-flux}\label{def:pseudoflux} by $\phi(\gamma) = \frac14 h(\gamma_+) = \frac14 h(\gamma_-)$, so $\phi(\gamma) = \langle \varphi_t , v \rangle$ whenever $t$ is a tiling that's compatible with $\gamma$.
Here, $\varphi_t$ is $t$'s flux and $v$ is $\gamma$'s parameter.

We are now ready to state Proposition \ref{sgnvcyc}.

\begin{prop}\label{sgnvcyc}
Let $L$ be a valid lattice and $v \in L$ be short.
If $\gamma_0, \gamma_1$ are simple $v$-quasicycles, then
$$\frac{\sgn(\gamma_0)}{\sgn(\gamma_1)} = (-1)^{^{\textstyle \phi(\gamma_0) - \phi(\gamma_1)}}.$$
\end{prop}

For the proof of Proposition \ref{sgnvcyc}, we will need two lemmas.

For any simple quasicycle $\gamma$, define $V_+(\gamma) = \R^2_+(\gamma) \cap \Z^2$, the set of all vertices of $\Z^2$ that lie in $\R^2_+(\gamma)$, and similarly for $V_-(\gamma)$.
Notice that when $\gamma$ is a $v$-quasicycle, $V_+(\gamma)$ and $V_-(\gamma)$ are invariant under translation by $v$.

Now, suppose $\gamma_0, \gamma_1$ are both simple $v$-quasicyles\footnote{Note that the same $v$ applies to both $\gamma$'s, and this also implies they are oriented the same way.}.
Consider the set $V(\gamma_0,\gamma_1)$ of vertices in $\Z^2$ that are to the left of one $\gamma$ but to the right of the other; in other words, the set $V(\gamma_0,\gamma_1) = V_+(\gamma_0) \Delta V_+(\gamma_1) = V_-(\gamma_0) \Delta V_-(\gamma_1)$.
Since the $V_{\pm}(\gamma_i)$ are invariant under translation by $v$, so too is $V(\gamma_0,\gamma_1)$; this means that for each $u \in V(\gamma_0,\gamma_1)$, there are infinitely many copies of $u$ in $V(\gamma_0,\gamma_1)$.
We may thus take the quotient $\widetilde{V}(\gamma_0,\gamma_1) = V(\gamma_0,\gamma_1)/\langle v\rangle$.

\begin{lema}\label{sgnvcycdiff}
Let $L$ be a valid lattice and $v \in L$ be short.
If $\gamma_0, \gamma_1$ are simple $v$-quasicycles and $\card\big(\widetilde{V}(\gamma_0,\gamma_1)\big) = 1$, then
\begin{equation*}
\frac{\sgn(\gamma_0)}{\sgn(\gamma_1)} = (-1)^{^{\textstyle \varphi(\gamma_0) - \varphi(\gamma_1)}}.
\end{equation*}

In other words, when $\card\big(\widetilde{V}(\gamma_0,\gamma_1)\big) = 1$, the statement of Proposition \ref{sgnvcyc} holds.
\end{lema}
Notice $\card\big(\widetilde{V}(\gamma_0,\gamma_1)\big)=0$ if and only if $\gamma_0 = \gamma_1$.
\begin{proof}
Choose a vertex $w \in \big(\Z+\frac12\big)^2$ that belongs to both $\gamma_0$ and $\gamma_1$.
For each $i\in\{0,1\}$, consider the segment $\gamma_i^w$ of $\gamma_i$ that joins $w$ to $w+v$.
They coincide except round the boundary of a square $S$ with vertices in $\big(\Z+\frac12\big)^2$, which represents the equivalence class in $\widetilde{V}(\gamma_0,\gamma_1)$.
There are two cases:
\begin{enumerate}
	\item Three of the edges of $S$ belong to a $\gamma_i^w$, the other edge belongs to the other $\gamma_j^w$;
	\item Two of the edges of $S$ belong to a $\gamma_i^w$, the other two edges belong to the other $\gamma_j^w$.
\end{enumerate}

\begin{figure}[H]
		\centering
		\def\svgwidth{0.525\columnwidth}
    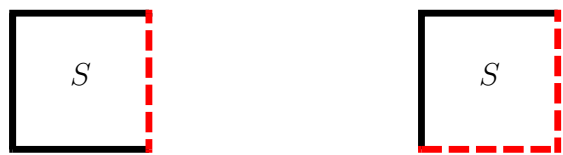
		\caption{An example of each case. $\gamma_i^w$ is represented by black edges and $\gamma_j^w$ by dashed, red edges.}
\end{figure}

Recall that either one or three of the edges of $S$ are negatively signed, so the product of Kasteleyn signs in each $\sgn(\gamma_i)$ is always different.
It follows that $\tfrac{\sgn(\gamma_0)}{\sgn(\gamma_1)}$ is defined entirely by $T_0$ and $T_1$.

In case (1), suppose without loss of generality three of the edges of $S$ belong to $\gamma_0^w$.
Then in the obvious notation, the quasiperiods satisfy $T_0 = T_1 + 2$.
This means the signs $(-1)^{\frac{1}{2}\cdot T_i + 1}$ in each $\sgn(\gamma_i)$ are also different, so $\sgn(\gamma_0)=\sgn(\gamma_1)$.

On the other hand, in this case there is a sign $x \in \{+,-\}$ such that ${(\gamma_0)}_x$ and ${(\gamma_1)}_x$ coincide.
For instance, in our previous example ${(\gamma_0)}_- = {(\gamma_1)}_-$.
It follows that $\phi(\gamma_0) = \frac14 h\big({(\gamma_0)}_x\big) = \frac14 h\big({(\gamma_1)}_x\big)\vphantom{^{^{^{^{}}}}} = \phi(\gamma_1)$, so $(-1)^{\phi(\gamma_0) - \phi(\gamma_1)} = 1$ and the proof for this case is complete.

Suppose now we are in case (2).
Clearly $T_0 = T_1$, implying the signs $(-1)^{\frac{1}{2}\cdot T_i + 1}$ in each $\sgn(\gamma_i)$ are the same, so $\sgn(\gamma_0)=-\sgn(\gamma_1)$.

On the other hand, in this case for each sign $x \in \{+,-\}$, ${(\gamma_0)}_x$ and ${(\gamma_1)}_x$ coincide except round the boundary of a square $S_x$ and its translations by $v$.
Along the segments on $S_x$, the height change for a ${(\gamma_i)}_x$ is $+2$ while the height change for the other ${(\gamma_j)}_x$ is $-2$, so $\phi(\gamma_0) - \phi(\gamma_1) = \frac14 \cdot \big(h({\gamma_0}_x)-h({\gamma_1}_x)\big)$ is either $+1$ or $-1$.
Regardless of the situation, $(-1)^{\phi(\gamma_0) - \phi(\gamma_1)} = -1$ as desired.$\vphantom{^{^{^{^{}}}}}$
\end{proof}

Let $\gamma$ be an oriented edge-path, and suppose no two of its consecutive edges coincide (except for orientation).
For such a path $\gamma$, we now describe the construction of an \textit{argument function}\label{def:argfunction} defined over its edges.

We choose $\gamma$'s initial edge $e_0$ as a base edge, and assign the choice of base value 0 to it; in other words, $\arg(e_0) = 0$.
For other edges, $\arg$ is given recursively by $\arg(e_{j+1}) = \arg(e_j) + \alpha_j$, where $\alpha_j \in \R$ is the angle with smallest modulus such that rotation by $\alpha_j$ round $e_j$'s starting point results in an edge that is parallel and identically oriented to $e_{j+1}$.
Notice the condition we imposed guarantees $\alpha_j$ is always well-defined (there is never a choice between $\pi$ or $-\pi$), so $\arg$ is too.

\begin{lema}\label{cycconec}
Let $L$ be a valid lattice and $v \in L$ be short.
Let $\gamma_0, \gamma_1$ be distinct simple $v$-quasicycles.
There is a finite sequence of simple $v$-quasicyles $(\beta_k)_{k=0}^n$ with $\beta_0 = \gamma_0$, $\beta_n = \gamma_1$, and such that for all $0 \leq k < n$ it holds that $\card\big(\widetilde{V}(\beta_k,\beta_{k+1})\big)=1$.
\end{lema}
\begin{proof}
Observe that $\gamma$'s quasiperiod is the length of any segment in $\gamma$ joining a vertex to its translation by $v$.
The reader may find it easier to follow the proof with this interpretation, and over the course of this proof, we will refer to $\gamma$'s quasiperiod as $\gamma$'s length.

We need only prove for $\gamma_1$ with minimal length.\footnote{The minimal length is $\lVert v \rVert_1$.}
Indeed, let $\gamma_0$ and $\gamma_1$ be any two simple $v$-quasicycles, and suppose $\widetilde{\gamma}$ is a $v$-quasicycle with minimal length; clearly it is simple.
If Lemma \ref{cycconec} holds for $\gamma_0, \widetilde{\gamma}$ and for $\widetilde{\gamma}, \gamma_1$, we may combine both sequences obtained this way, so Lemma \ref{cycconec} applies to $\gamma_0$ and $\gamma_1$.
We thus assume without loss of generality that $\gamma_1$ has minimal length.

Let $\gamma$ be any simple $v$-quasicycle.
Because it is simple, no two of its consecutive edges coincide, so we may define an argument function $\arg_{\gamma}$ over its edges.
Now, it's easy to see that the following are equivalent:
\begin{itemize}
	\item $\gamma$ has minimal length;
	\item $\arg_{\gamma}$ assumes at most two values.
\end{itemize}

We divide the proof in two cases.

\paragraph{}\indent \indent \textbf{Case 1.} $\gamma_0$ also has minimal length.

Observe that $\arg_{\gamma_0}$ assumes a single value if and only if one of $v$'s coordinates is 0, that is, if and only if and $\gamma_0$ and $\gamma_1$ are parallel straight lines. 
In this situation, it is obvious the lemma holds.

Suppose then that $\arg_{\gamma_0}$ assumes two values, let $\beta_0 = \gamma_0$ and consider $\widetilde{V}(\beta_0,\gamma_1)$.
Make correspond to each vertex in $\Z^2$ the square in $\R^2$ with side length 1 centered on that vertex, so each edge of $\beta_0$ is the side of one such square.
Choose a class $w \in \widetilde{V}(\beta_0,\gamma_1)$ such that $\beta_0$ fits each of $w$'s corresponding squares on two of its sides.
Notice $\widetilde{V}(\beta_0,\gamma_1)$ is non-empty (since $\beta_0 = \gamma_0$ and $\gamma_1$ are distinct), and there's always one such $w$ because $\arg_{\beta_0}$ assumes two values.

Now, because $\arg_{\beta_0}$ does not assume three or more values, all horizontal edges of $\beta_0$ have the same orientation and all of its vertical edges also do.
This means no edge of $\beta_0$ touches any of the other two sides of each of $w$'s squares.
We may thus consider the path $\beta_1$ that coincides with $\beta_0$ except on $w$'s squares; it fits each of these squares on the other two sides.

By construction, $\beta_1$ is a simple $v$-quasicycle, it preserves the minimal length, and its argument function also assumes two values.
Moreover, it is clear $\widetilde{V}(\beta_0,\beta_1) = \{w\}$, so $\card\big(\widetilde{V}(\beta_0,\beta_1)\big)=1$.
Finally, because $w$ lies on different sides of $\beta_0$ and of $\beta_1$, and also on different sides of $\beta_0$ and of $\gamma_1$, $w$ lies on the same side of $\beta_1$ and of $\gamma_1$.
Since $\card\big(\widetilde{V}(\beta_0,\gamma_1)\big) = \card\big(\widetilde{V}(\gamma_0,\gamma_1)\big)$ is finite, this means $\card\big(\widetilde{V}(\beta_1,\gamma_1)\big)= \card\big(\widetilde{V}(\beta_0,\gamma_1)\big)-1$.

If $\card\big(\widetilde{V}(\beta_1,\gamma_1)\big) > 0$ we may repeat the process, and in fact as long as $\card\big(\widetilde{V}(\beta_k,\gamma_1)\big) > 0$ we may do so.
It's easy to verify that $\card\big(\widetilde{V}(\beta_k,\gamma_1)\big) = \card\big(\widetilde{V}(\beta_0,\gamma_1)\big) - k,$ so after a finite number $n = \card\big(\widetilde{V}(\beta_0,\gamma_1)\big)$ of steps, we have $\card\big(\widetilde{V}(\beta_n,\gamma_1)\big) = 0$, that is, $\beta_n = \gamma_1$ as desired.

\paragraph{}\indent \indent \textbf{Case 2.} $\gamma_0$ does not have minimal length.

If $\gamma_0$ does not have minimal length, $\arg_{\gamma_0}$ assumes at least three values.
This means $\gamma_0$ contains at least one segment as in the figure below:
\begin{figure}[H]
		\centering
		\def\svgwidth{0.18\columnwidth}
    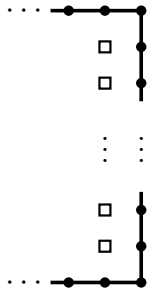
		\caption{A segment with three consecutive argument values.}\label{figarg1}
\end{figure}

Each of these segments is a sequence of consecutive edges $(e_k)_{k=1}^n$ satisfying one of the following:
\begin{equation*}
\begin{split}
	\arg_{\gamma_0}(e_1) < \arg_{\gamma_0}(e_2) = \cdots = \arg_{\gamma_0}(e_{n-1}) < \arg_{\gamma_0}(e_n)\\
	\arg_{\gamma_0}(e_1) > \arg_{\gamma_0}(e_2) = \cdots = \arg_{\gamma_0}(e_{n-1}) > \arg_{\gamma_0}(e_n)
\end{split}
\end{equation*}

For one such segment, we say its length is the number $n\geq3$.
Now, let $\delta$ be one such segment with minimal length.
Because its length is minimal, $\gamma_0$ does not touch any of the `inner vertices' near $\delta$, indicated by a square in Figure \ref{figarg1}.
We may thus consider the simple $v$-quasicycles $\beta_1, \beta_2, \cdots, \beta_{n-2}$ obtained from $\beta_0 = \gamma_0$ by changing the edges (and its translations by $v$) as shown in the following image:
\begin{figure}[H]
		\hspace{0.275cm}\def\svgwidth{\columnwidth}
    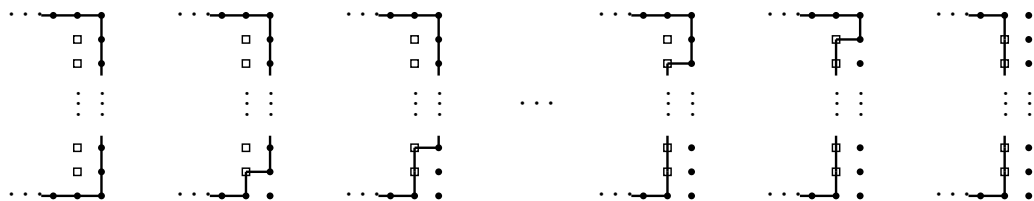
		\caption{The simple $v$-quasicycles $\beta_k$ obtained from $\beta_0 = \gamma_0$.}
\end{figure}

It is clear $\card\big(\widetilde{V}(\beta_k,\beta_{k+1})\big)=1$.
Moreover, in the obvious notation the lengths satisfy $l(\beta_{0}) = l(\beta_{1}) = \cdots = l(\beta_{n-3}) = l(\beta_{n-2}) + 2$, so the length has decreased.
If $l(\beta_{n-2})$ is not minimal, then $\arg_{\beta_{n-2}}$ assumes at least three values, and we may repeat the process.
Since $l(\beta_{0})$ is finite, this must end in a finite number of steps, producing a simple $v$-quasicycle $\beta_k$ with minimal length.
We have thus reduced it to the previous case, and the proof is complete.
\end{proof}

We are now ready to prove Proposition \ref{sgnvcyc}.

\begin{proof}[Proof of Proposition \ref{sgnvcyc}]
By Lemma \ref{sgnvcycdiff}, the Proposition holds when $\card\big(\widetilde{V}(\gamma_0,\gamma_1)\big)\leq 1$.
When $\card\big(\widetilde{V}(\gamma_0,\gamma_1)\big)> 1$, we may use Lemma \ref{cycconec} to obtain a sequence of simple $v$-quasicyles $(\beta_k)_{k=0}^n$ with $\beta_0 = \gamma_0$, $\beta_n = \gamma_1$, and such that for all $0 \leq k < n$ it holds that $\card\big(\widetilde{V}(\beta_k,\beta_{k+1})\big)=1$.
Then
$$\frac{\sgn(\gamma_0)}{\sgn(\gamma_1)} = \frac{\sgn(\beta_0)}{\sgn(\beta_n)} = \prod\limits_{0\leq k < n} \frac{\sgn(\beta_k)}{\sgn(\beta_{k+1})}$$

Applying Lemma \ref{sgnvcycdiff} to each $\beta_k, \beta_{k+1}$ yields
\begin{alignat*}{3}
\frac{\sgn(\gamma_0)}{\sgn(\gamma_1)} =& \enspace \prod\limits_{k=0}^n{(-1)^{^{\textstyle \phi (\beta_k) - \phi (\beta_{k+1})}}} \enspace &=& \enspace (-1)^{^{\textstyle \sum_{k=0}^n{\phi (\beta_k) - \phi (\beta_{k+1})}}} \\
=& \enspace \enspace (-1)^{^{\textstyle \phi (\beta_0)-\phi (\beta_n)}} &=& \enspace (-1)^{^{\textstyle \phi (\gamma_0)-\phi (\gamma_1)}},
\end{alignat*}
so we are done.
\end{proof}

The following corollary is automatic.

\begin{corolario}[Sign formula for quasicycles]\label{sgncyc}
Let $v \in L$ be short.
There is a sign $C_v \in \{-1,+1\}$ such that for any $v$-quasicycle $\gamma$ $$\sgn(\gamma) = C_v \cdot (-1)^{^{\textstyle\phi(\gamma)}}.$$
\end{corolario}

\section{The effect of a cycle flip on the flux itself}\label{sec:cycfluxsgn}

We know each open cycle in $C(t,t_0)$ has the same effect on $\sgn(\varphi_t)$:
as per Corollary \ref{sgncyc}, it is given by $C_v \cdot (-1)^{\langle \varphi_t, v \rangle}$, where $v$ is $C(t,t_0)$'s parameter.
But what about the effect of on the flux itself?

\begin{prop}\label{varphicyc}
Let $L$ be a valid lattice.
Consider two different fluxes $\varphi_0, \varphi_1 \in \mathscr{F}(L)$ and let $t_i$ be any tiling of $\T_L$ with flux $\varphi_i$.
Let $v$ be $C(t_0,t_1)$'s parameter.
Then $\varphi_1 - \varphi_0 \perp v$.

Choose any open cycle $c \in C(t_0,t_1)$ and let $t_c$ be obtained from $t_0$ by a cycle flip on $c$.
Let $\varphi_c$ be its flux.
Then, in addition to $(\varphi_c - \varphi_0) \perp v$, $\varphi_c - \varphi_0$ is short in $L^*$.
In particular, $\varphi_c - \varphi_0$ is uniquely defined up to sign.
\end{prop}
\begin{proof}
That $C(t_0,t_1)$ contains an open cycle (and thus its parameter is well-defined) is provided by Corollary \ref{closcyc}.

For any infinite domino path of $c$, of course $t_0$ and $t_c$ are both compatible with that path, so $\langle \varphi_0 , v \rangle = \langle \varphi_c, v \rangle$, as in equation~\eqref{varphigammaflux}. 
This implies $\varphi_c - \varphi_0 \perp v$; notice this argument also applies to $t_0$ and $t_1$.
It remains to show that $\varphi_c - \varphi_0$ is short in $L^*$.

Consider the set $\Gamma(c) = \{ \text{$\gamma$ $|$ $\gamma$ is an infinite domino path of $c$}\}$.
Let $u \in L$ be such that $\{u,v\}$ is a basis for $L$; the existence of such a vertex is guaranteed by Lemma \ref{vshortbase} below.
Choose any $\gamma \in \Gamma(c)$, and for each $k \in \Z$ let $\gamma_k = \gamma + k \cdot u$.
Notice that because $c$ has parameter $v$ and $\{u,v\}$ is linearly independent, $\gamma_j = \gamma_k$ if and only if $j = k$.
Clearly, each $\gamma_k \in \Gamma(c)$.
We claim $\Gamma(c) = \{\gamma_k \text{ }|\text{ } k \in \Z\}$.
Indeed, $L$ acts transitively on $\Gamma(c)$ by translation, so for any $\widetilde{\gamma} \in \Gamma(c)$ there is some $w = a \cdot u + b \cdot v \in L$ with $\widetilde{\gamma} = \gamma + w$.
It follows that $\widetilde{\gamma} = \gamma_a + b \cdot v = \gamma_a$, because elements of $\Gamma(c)$ are invariant under translation by $v$ (they are $v$-quasicycles).

Now, as edge-paths in $G(\Z^2) = \big(\Z+\frac12\big)^2$, two distinct $\gamma_k$'s are disjoint and `parallel' with respect to $v$.
Consider then $R = \R^2 \setminus \Gamma(c)$.
Each connected component of $R$ has boundary given by exactly two consecutive $\gamma_k$'s, and each $\gamma_k$ is in the boundary of exactly two adjacent connected components of $R$.
Moreover, it is clear $\Z^2 \subset R$.

Let $R_0$ be the connected component of $R$ that contains the origin.
Let $\partial R_0 = \gamma_j \sqcup \gamma_{j+1}$.
Relabel the paths $\gamma_k := \gamma_{k-j}$, so the boundary of $R_0$ is given by $\gamma_0 \sqcup \gamma_1$.
More generally, let $R_k$ be the connected component of $R$ whose boundary is given by $\gamma_k \sqcup \gamma_{k+1}$; clearly, $R_k$ and $R_{k+1}$ are adjacent.

We now compare the height functions $h_0$ of $t_0$ and $h_c$ of $t_c$ on each $R_k \cap \Z^2$.
Remember $t_0$ and $t_c$ differ only by the $\gamma_k$'s, so they coincide on each $R_k$. 
Now, because $h_0$ and $h_c$ are both 0 at the origin, they coincide everywhere on $R_0$; in particular, $h_c - h_0$ is always $0$ on $\langle v \rangle$.

We now compute $h_c - h_0$ on $u \in L \cap R_1$.
Choose any edge-path $\beta$ in $\Z^2$ joining the origin to $u$ and that crosses $\gamma_1 = \partial R_0 \cap \partial R_1$ only once.
It is clear such a path exists, because each $R_k\cap\Z^2$ is connected by edge-paths in $\Z^2$.
The height change along $\beta$ for each of $h_0$ and $h_c$ is the same except on the edge that crosses $\gamma_1$.
Either for one that change $+3$ and for the other that change is $-1$, or these changes are $-3$ and $+1$,
depending on whether the domino of $\gamma_1$ that's crossed by $\beta$ lies in $t_0$ or in $t_c$, and on the orientation (as induced by the coloring of $\Z^2$) of the crossing edge on $\beta$.
This means $h_c - h_0$ is either $+4$ or $-4$ on $u$ (and thus on all of $R_1\cap \Z^2$).
By the same token, it's easy to see that when it is $+4$, $h_c - h_0$ is $4k$ on all of $R_k\cap \Z^2$; when it is $-4$, $h_c - h_0$ is $-4k$ on all of $R_k\cap \Z^2$.

Recall that evaluating a height function on $L$ yields information about the corresponding flux via the inner product identification.
In particular, we now know $\langle \varphi_c - \varphi_0, v \rangle = 0$ and $\langle \varphi_c - \varphi_0, u \rangle = \pm 1$.
This entirely defines $\varphi_c - \varphi_0 \in L^*$.
Moreover, for any $\varphi \in L^*$ we have $\langle \varphi , u \rangle \in \Z$, so $\varphi_c - \varphi_0$ must be short, completing the proof.
\end{proof}

The following lemma was used in the proof of Proposition \ref{varphicyc} above.

\begin{lema}\label{vshortbase}
Let $L$ be a lattice generated by linearly independent vectors $v_0, v_1 \in \Z^2$.
Let $v \in L$ be short.
Then there is some $u \in L$ such that $\{u, v\}$ is a basis for $L$.
\end{lema}
\begin{proof}
Observe that $v = a \cdot v_0 + b \cdot v_1 \in L$ is short if and only if $\gcd(a,b) = 1$.
In this case, there are integers $k_a,k_b \in \Z$ with $k_a \cdot a + k_b \cdot b = 1$.
Let $u \in L$ be the vector $-k_b \cdot v_0 + k_a \cdot v_1$.
It's easy to check that $k_a \cdot v - b \cdot u = v_0$ and $k_b \cdot v + a\cdot u = v_1$, so $\{u, v\}$ generates $L$ and is thus a basis for it.
\end{proof}

We now show there is a set of two `short moves' that connects $\mathscr{F}(L)$.
We will use it to describe a `sign pattern' for fluxes in $\mathscr{F}(L)$ via Proposition \ref{varphicyc}.

Recall that for any two fluxes in $\mathscr{F}(L)$, their difference is an element of $L^*$ (indeed, as per Proposition \ref{hdelmeio}, $L^\#$ is a translation of $L^*$).
We say a basis $\{v_0^*, v_1^*\}$ for $L^*$ is \textit{flux-connecting}\label{def:fluxconnecbasis} if given any two fluxes $\varphi, \widetilde{\varphi} \in \mathscr{F}(L)$ there is a sequence of fluxes $(\varphi_k)_{k=0}^n$ with $\varphi_k \in \mathscr{F}(L)$ for all $0 \leq k \leq n$, $\varphi_0 = \varphi$, $\varphi_n = \widetilde{\varphi}$ and such that $\varphi_{k+1} - \varphi_k = \pm v_i^*$ for all $0 \leq k < n$.
In other words, the moves $\pm v_0^*$ and $\pm v_1^*$ connect $\mathscr{F}(L)$.

\begin{lema}\label{fluxonecbasis}
For any valid lattice $L$, $L^*$ admits a flux-connecting basis.
\end{lema}
\begin{proof}
Consider $Q_2 \subset Q$, the side of $\partial Q$ contained in the second quadrant.
We know $b_{\mathcal{N}} = \big(0,\frac12\big)$ and $b_{\mathcal{W}} = \big(-\frac12,0\big)$ are in $\mathscr{F}(L) \cap Q_2$, so every flux in $\mathscr{F}(L)$ belongs to a unique line that is parallel to $Q_2$.
We say $l_0 \supset Q_2$ is the first such line, and $l_{k+1}$ is the line just below $l_k$.
See Figure \ref{fig:qlines}.
\begin{figure}[ht]
		\centering
		\def\svgwidth{0.75\columnwidth}
    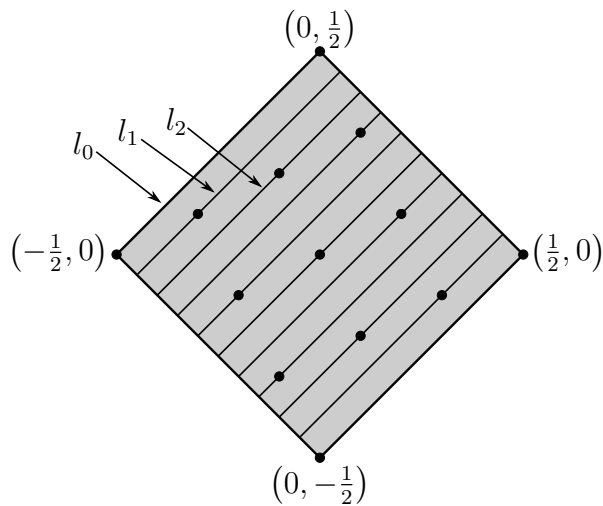
		\caption{Enumerating the lines $l_k$. Marked vertices are elements of $\mathscr{F}(L)$.}
		\label{fig:qlines}
\end{figure}

Consider the brick wall $b_{\mathcal{W}}$, and let $f_0$ be the flux in $l_0$ that is closest to it (but different from it).
Consider the line $l_1$, and let $f_1$ be the flux in it that is closest to $b_{\mathcal{W}}$.
Let $v_0^* = f_0 - f_1$ and $v_1^* = b_{\mathcal{W}} - f_1$. 
Figure \ref{fig:qbase} illustrates this construction.
We claim $\{v_0^*, v_1^*\}$ is a flux-connecting basis of $L^*$.
\begin{figure}[ht]
		\vspace{0.75cm}
		\centering
		\def\svgwidth{0.75\columnwidth}
    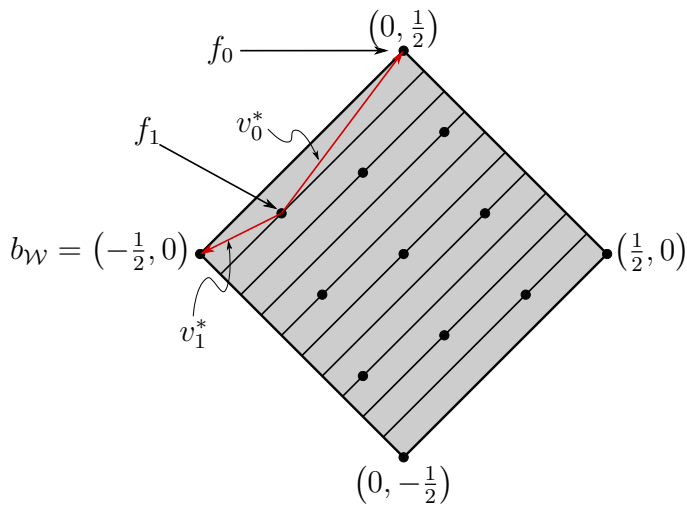
		\caption{The vectors $v_0^*$ and $v_1^*$ form a flux-connecting basis of $L^*$.}
		\label{fig:qbase}
\end{figure}

First, we show the moves $\pm v_i^*$ connect any two fluxes in one same line $l_k$.
Let $\varphi, \widetilde{\varphi} \in \mathscr{F}(L) \cap l_k$ and suppose without loss of generality $\widetilde{\varphi}$ is to the right of $\varphi$.
There must be a line $l_{k-1}$ above $l_k$ or a line $l_{k+1}$ below it (possibly both, but at least one).
In the first case, $v_0^*$ takes $\varphi_0 = \varphi$ to a flux $\varphi_1$ in $l_{k-1}$; and $-v_1^*$ takes $\varphi_1$ to a flux $\varphi_2$ back in $l_k$.
Notice $\varphi_2$ is to the right of $\varphi$, and because of how $v_0^*,v_1^*$ were chosen, there can be no flux in $l_k$ between $\varphi$ and $\varphi_2$. 
If $\varphi_2 \neq \widetilde{\varphi}$, we may repeat the process, and because $l_k \cap \mathscr{F}(L)$ is finite, it must end after a finite number of steps.

The latter case is similar: $-v_1^*$ takes $\varphi_0 = \varphi$ to a flux $\varphi_1$ in $l_{k+1}$; and $v_0^*$ takes $\varphi_1$ to a flux $\varphi_2$ back in $l_k$.
Once again, induction shows $\varphi$ and $\widetilde{\varphi}$ are connected by the moves $\pm v_i^*$.

To complete the proof, we show the moves $\pm v_i^*$ connect any two adjacent lines.
Consider the lines $l_k, l_{k+1}$.
Let $\varphi_k$ be the first flux in $l_k$, that is, the flux in $l_k$ that is closest to $Q_3$; and similarly for $\varphi_{k+1}$.
When $\varphi_k$ is to the right of $\varphi_{k+1}$, $v_0^*$ takes $\varphi_k$ to $\varphi_{k+1}$; when $\varphi_k$ is to the left of $\varphi_{k+1}$, $v_1^*$ takes $\varphi_{k+1}$ to $\varphi_k$.
The image below exemplifies the two cases.
\begin{figure}[H]
		\centering
		\def\svgwidth{0.65\columnwidth}
    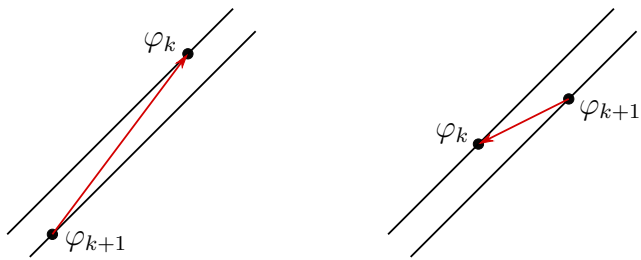
		\caption{Each $v_i^*$ connects adjacent lines in each case.}
\end{figure}

Regardless of the situation, the moves $\pm v_i^*$ connect two adjacent lines.
Similar line arguments also show $\{v_0^*,v_1^*\}$ generates $L^*$, so we are done.
\end{proof}

We are now ready to describe the aforementioned sign pattern.

\begin{prop}[Sign patterns in $\mathscr{F}(L)$]\label{fluxsignpat}
Let $L$ be a valid lattice and $\{v_0^*, v_1^*\}$ be a flux-connecting basis of $L^*$.
Decompose $\mathscr{F}(L)$ in lines parallel to $v_0^*$ and in lines parallel to $v_1^*$.
Then, along any given one of those lines, the sign change between two adjacent fluxes is always the same.
Moreover, for parallel and adjacent lines, the sign change along each line is different.
\end{prop}
\begin{proof}
Let $i,j \in \{0,1\}$ be distinct.
We will show that if $\varphi \in \mathscr{F}(L)$ is such that $\varphi + v_i^*$ and $\varphi - v_i^*$ are in $\mathscr{F}(L)$, then
\begin{equation}\label{samesign}
\frac{\sgn(\varphi)}{\sgn(\varphi + v_i^*)} = \frac{\sgn(\varphi - v_i^*)}{\sgn(\varphi)}.
\end{equation}

This proves the claim on each given line.
Additionally, we will show that if $\varphi + v_j^*$ and $\varphi+v_i^* + v_j^*$ are in $\mathscr{F}(L)$, then
\begin{equation}\label{diffsign}
\frac{\sgn(\varphi)}{\sgn(\varphi + v_i^*)} = - \frac{\sgn(\varphi+v_j^*)}{\sgn(\varphi + v_j^*+v_i^*)}.
\end{equation}

This proves the claim on parallel and adjacent lines.

Let $t$ be a tiling of $\T_L$ with flux $\varphi$, and similarly for $t^{\pm}$ with $\varphi \pm v_i^*$, for $t_j$ with $\varphi + v_j^*$ and for $t_j^+$ with $\varphi+v_j^* + v_i^*$.
Suppose all relevant vectors are in $\mathscr{F}(L)$.
Let $v^+$ and $v^-$ be respectively the parameters of $C(t, t^+)$ and $C(t^-,t)$.
Observe that $(\varphi + v_i^*) - \varphi = \varphi - (\varphi - v_i^*) = v_i^*$, so by Proposition \ref{varphicyc} $v_i^*\perp v^+,v^-$.
Since $v^+$ and $v^-$ are both short, this implies $v^+ = v^-$ or $v^+ = - v^-$; in other words, they're equal up to multiplication by $-1$.

Inspecting the proof of Proposition \ref{varphicyc}, we see that any cycle flip on a cycle with parameter $u$ changes the flux by a vector in $L^*$ that is short and perpendicular to $u$.
In particular, for any open cycle $c$ in $C(t, t^+) \cup C(t^-,t)$, a cycle flip on $c$ changes the flux by either $v_i^*$ or $-v_i^*$ (recall that $v_i^*$ is in a basis for $L^*$, so it must be short).
This implies the number of open cycles in $C(t,t^+)$ is odd.
Indeed, the total change (from performing a cycle flip on each open cycle) is $v_i^*$, so the number $n^-$ of cycles with a $-v_i^*$ change and the number $n^+$ of cycles with a $v_i^*$ change must satisfy $n^+ = n^- + 1$.
Of course, the same holds for $C(t^-,t)$.

Now, any open cycle in $C(t, t^+) \cup C(t^-,t)$ is compatible with $t$, so by Corollary \ref{sgncyc} there is a sign $C_{v^+}$ such that for each open cycle $c \in C(t, t^+) \cup C(t^-,t)$
$$\sgn(c) = C_{v^+} \cdot (-1)^{^{\textstyle\langle \varphi, v^+ \rangle}}.$$

Since we've shown each of $C(t, t^+)$ and $C(t^-,t)$ has an odd number of open cycles, the total sign change in each case is precisely $C_{v^+} \cdot (-1)^{\langle \varphi, v^+ \rangle}$.
In particular, in each case the total sign change is the same, so we've proved equation~\eqref{samesign}.

For equation~\eqref{diffsign}, observe that $(\varphi + v_j^* + v_i^*) - (\varphi + v_j^*) = v_i^*$,
so once again the parameter $w^+$ of $C(t_j,t_j^+)$ is perpendicular to $v_i^*$ and thus equal to $v^+$ up to multiplication by $-1$ (because they're both short).
The same argument used above also shows $C(t_j,t_j^+)$ has an odd number of open cycles, and since each such cycle is compatible with $t_j$, its total sign change is given by
$$C_{v^+} \cdot (-1)^{^{\textstyle\langle \varphi + v_j^*, v^+ \rangle}} = C_{v^+} \cdot (-1)^{^{\textstyle\langle \varphi, v^+ \rangle}} \cdot (-1)^{^{\textstyle\langle v_j^*, v^+ \rangle}}.$$

We will show $\langle v_j^*, v^+ \rangle$ is either $+1$ or $-1$, from which equation~\eqref{diffsign} follows.
To that end, notice $v_j^*$ is a short element of $L^*$ (because it is in a basis).
Since $v^+$ is short in $L$, the only possible integer values for $\langle v_j^*, v^+ \rangle$ that respect $v_j^*$'s shortness are $-1$, $0$, and $+1$.
Now, $\langle v_i^*, v^+ \rangle = 0$ (because $v_i^* \perp v^+$), so if $\langle v_j^*, v^+ \rangle$ were also 0 it would contradict $\{v_i^*,v_j^*\}$ being a basis for $L^*$, for all its elements would be 0 on the sublattice generated by $v^+$.
It follows that $\langle v_j^*, v^+ \rangle$ is either $+1$ or $-1$, as desired.
\end{proof}

Notice in the proof above that since $v^+$ is unique up to multiplication by $-1$, we may choose it with $\langle v_j^*, v^+ \rangle = 1$.
In particular, since $\langle v_i^*, v^+ \rangle = 0$, we have  $\langle v_k^*, v^+ \rangle = \delta_{ik}$.
This implies the following: let $\{v_0,v_1\}$ be the basis of $L$ that is dual to the flux-connecting basis of $L^*$ $\{v_0^*,v_1^*\}$, that is, they satisfy $\langle v_k^*, v_l \rangle = \delta_{kl}$.
Let $t$, $t_0$ and $t_1$ be tilings of $\T_L$ with fluxes respectively $\varphi$, $\varphi+v_0^*$ and $\varphi + v_1^*$. Then $v_1$ is $C(t,t_0)$'s parameter and $v_0$ is $C(t,t_1)$'s parameter.

The sign pattern in Proposition \ref{fluxsignpat} may also be described as a pattern in which one odd-one-out sign is always surrounded by different signs.
\begin{figure}[H]
		\centering
		\def\svgwidth{0.95\columnwidth}
    \input{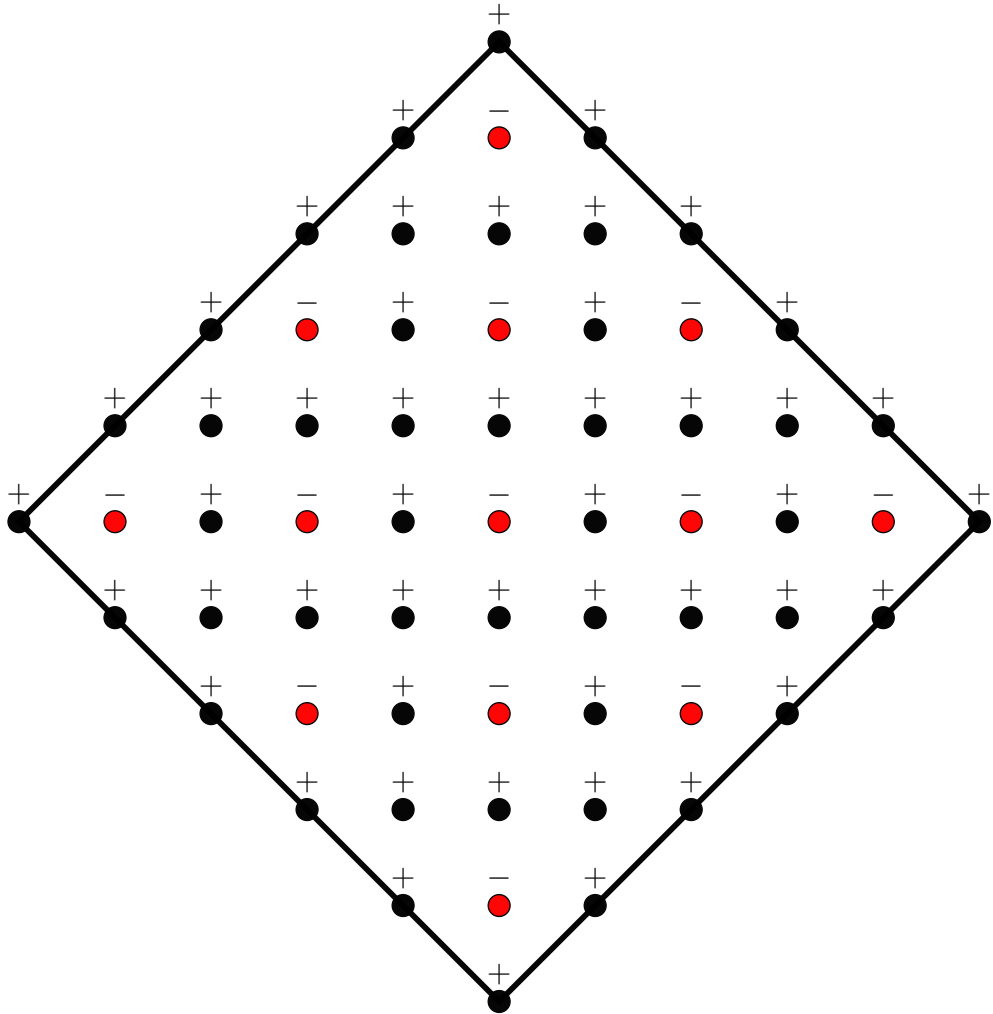}
		\caption{In this example signing, the minus sign is the odd-one-out.}
\end{figure}

Of course, which precise sign ($+1$ or $-1$) is the odd-one-out depends on the choice of Kasteleyn signing for edges of $G(\Z^2)$ and the enumeration of $D_L$'s squares, but the pattern is always the same.

Proposition \ref{fluxsignpat} is also instrumental in showing we can always obtain the total number of tilings of $\T_L$ with a linear combination of $\det\big(K(\pm 1, \pm 1)\big)$.

\begin{prop}\label{proplinearcomb}
Let $L$ be a valid lattice and $K$ a Kasteleyn matrix for $\T_L$.
Let $p_K(q_0,q_1)$ be the Laurent polynomial given by $\det(K)$.
Then there is a choice of constants $s_{00},s_{01},s_{10},s_{11} = \pm\frac{1}{2}$ such that the total number of tilings of $\T_L$ is
\begin{equation*}
s_{00}\cdot p_K(1,1) + s_{01}\cdot p_K(1,-1) + s_{10}\cdot p_K(-1,1) + s_{11}\cdot p_K(-1,-1).
\end{equation*}
\end{prop}
\begin{proof}
For the first part, we will show that for each monomial of the form $c_{ij} \cdot q_0^i q_1^j$ in the full expansion of $\det(K)$, the sign of $c_{ij}$ depends only on the parity of $i,j$.
To that end, observe that by Proposition \ref{fluxsignpat} whenever $\varphi_0, \varphi_1 \in \mathscr{F}(L)$ are such that $\varphi_0 - \varphi_1 \in 2L^*$, then $\sgn(\varphi_0) = \sgn(\varphi_1)$.

Recall our construction of the Kasteleyn matrix, using the rectangular fundamental domain $D_L$ defined by $v_0 = (x_0,0)$ and $v_1=(x_1,y_1)$.
Here, $v_0$ and $v_1$ generate $L$, $x_0,y_1 > 0$ and $0 \leq x_1 < x_0$.
We know that by construction
$$\forall \varphi \in \mathscr{F}(L), \enspace \exists i_0,i_1 \in \Z, \enspace  \forall \text{ tiling $t$ of $\T_L$ with flux $\varphi$}, \enspace  K_t = \pm q_0^{i_0}q_1^{i_1},$$
so we may speak of the $q$-exponents $(i_0,i_1)$ of $\varphi$.
Also by construction, we know that if $y_1$ is even and $\varphi$ has $q$-exponents $(i_0,i_1)$, then $\langle \varphi, v_0 \rangle = i_0$ and $\langle \varphi, v_1 \rangle = i_1$.
When $y_1$ is odd, the latter changes to $\langle \varphi, i_1 \rangle = i_1 + \frac12$.

Let $q_0^{i_0}q_1^{i_1}$ be a monomial in $p_K$ and consider any monomial $q_0^{j_0}q_1^{j_1}$ in $p_K$ such that $i_0 \equiv j_0 \Mod{2}$ and $i_1 \equiv j_1 \Mod{2}$.
Let $\varphi(i_0,i_1) \in \mathscr{F}(L)$ correspond to $q_0^{i_0}q_1^{i_1}$ and similarly for $\varphi(j_0,j_1)$.
Consider the difference $\varphi(i_0,i_1) - \varphi(j_0,j_1) \in L^*$.
Regardless of the parity of $y_1$, it satisfies
\begin{alignat*}{1}
&\langle \varphi(i_0,i_1) - \varphi(j_0,j_1) , v_0 \rangle = i_0 - j_0 \\
&\langle \varphi(i_0,i_1) - \varphi(j_0,j_1) , v_1 \rangle = i_1 - j_1
\end{alignat*}

Because $i_k$ and $j_k$ have the same parity, these numbers are both even integers, so $\varphi(i_0,i_1) - \varphi(j_0,j_1) \in 2L^*$.
It follows that $\varphi(i_0,i_1)$ and $\varphi(j_0,j_1)$ have the same sign, that is, the coefficients of $q_0^{i_0}q_1^{i_1}$ and $q_0^{j_0}q_1^{j_1}$ in $p_K$ have the same sign, so the first part is complete.

Now, for the second part, write
\begin{equation*}
\begin{split}
s_{00}\cdot p_K(1,1) + s_{01}\cdot p_K(1,-1) + s_{10}\cdot p_K(-1,1) + s_{11}\cdot p_K(-1,-1)& \\
\hfill = \sum\limits_{i,j} \lvert c_{ij} \rvert \cdot \sgn(c_{ij}) \cdot \left(s_{00} + (-1)^j\cdot s_{01} + (-1)^i\cdot s_{10} + (-1)^{i+j} \cdot s_{11}\right)&,
\end{split}
\end{equation*}
where from the first part we know $\sgn(c_{ij})$ depends only on the parity of $i,j$.
We thus need to check the system below admits a solution:
\begin{equation*}\begin{array}{c}\left\{\begin{aligned}
s_{00} + s_{01} + s_{10} + s_{11} &= \text{sign of $i,j$ even} \\
s_{00} - s_{01} + s_{10} - s_{11} &= \text{sign of $i$ even, $j$ odd} \\
s_{00} + s_{01} - s_{10} - s_{11} &= \text{sign of $i$ odd, $j$ even} \\
s_{00} - s_{01} - s_{10} + s_{11} &= \text{sign of $i,j$ odd}
\end{aligned}\right.\end{array}\end{equation*}

Notice Proposition \ref{fluxsignpat} ensures exactly three of the signs are equal.
The claim follows from observing that the matrix
\begin{equation*}A = \left(\begin{array}{rrrr}
1&1&1&1 \\
1&-1&1&-1 \\
1&1&-1&-1 \\
1&-1&-1&1 
\end{array}\right)
\end{equation*}
is invertible\footnote{We have $\det(A) = 16$.}, and that applying $A^{-1}$ to any of the eight possible sign configurations yields $s_{ij}$'s as in the statement.
\end{proof}
\chapter{Kasteleyn determinants for the torus}
\label{chap:detkast}

In Section \ref{sec:exret}, we calculated the Kasteleyn determinant for rectangles with integral sides by finding a suitable basis of eigenvectors, from which we derived the eigenvalues of the Kasteleyn matrix itself.
We will now employ similar techniques for calculating Kasteleyn determinants of tori.

Let $L$ be a valid lattice and $K$ be a Kasteleyn matrix for $\T_L$.
For all $i,j \in \Z$, let $v_{i,j} = (i,j) + (\frac12,\frac12)$, so the $v_{i,j}$ enumerate the vertices of $G(\Z^2)$.
$v_{i,j}$ is black whenever $i \equiv j \Mod{2}$, and it is white otherwise.
As we defined it, $K$ is a linear map from the space of equivalence classes of black vertices in $G(\Z^2)$ to the space of white ones, and it acts as follows:
\begin{alignat*}{3}
Kv_{i,j} =& \enspace &\text{Kast}(v_{i,j}v_{i,j+1})& \cdot v_{i,j+1} + \text{Kast}(v_{i,j}v_{i,j-1}) \cdot v_{i,j-1}\\
+& &\text{Kast}(v_{i,j}v_{i+1,j})& \cdot v_{i+1,j} + \text{Kast}(v_{i,j}v_{i-1,j}) \cdot v_{i-1,j}
\end{alignat*}

Here, $\text{Kast}(e)$ is the Kasteleyn weighting of the edge $e$, and if $p_0,p_1$ are two adjacent vertices, $p_0p_1$ indicates the edge joining them.
Observe that $\text{Kast}(e)$ does not depend on a particular choice of representative of $[e]_L$, so our dropping the braces round each $v_{i,j}$ is justified.

\section{The case of $M = KK^* \oplus K^*K$}\label{sec:casom}

Like before, we will consider the matrix $M = KK^* \oplus K^*K$ rather than $K$.
Notice $K^*$ goes from the space of equivalence classes of white vertices in $G(\Z^2)$ to the space of black ones, via:
\begin{alignat*}{3}
K^*v_{i,j} =& \enspace &\overline{\text{Kast}(v_{i,j}v_{i,j+1})}& \cdot v_{i,j+1} + \overline{\text{Kast}(v_{i,j}v_{i,j-1})} \cdot v_{i,j-1}\\
+& &\overline{\text{Kast}(v_{i,j}v_{i+1,j})}& \cdot v_{i+1,j} + \overline{\text{Kast}(v_{i,j}v_{i-1,j})} \cdot v_{i-1,j}
\end{alignat*}
Here, $\overline{\Bigcdot}$ is the complex conjugation.

The entries of $K$ are Laurent monomials in $q_0$ and $q_1$ (their exponents indicate flux value), so we may think of $K$ as a matrix $K(q_0,q_1)$.
Of course, this also means $\det(K)$ is a Laurent polynomial $p_K$ in $q_0$ and $q_1$.
Rather than tackle the problem of finding a formula for $p_K(q_0,q_1) = \det(K)$, we will find the values of $p_K$ when $q_0,q_1 \in \Sp^1 \subset \C$.
Observe this is enough to describe $p_K$.

Notice $\overline{z} = z^{-1}$ whenever $z \in \Sp^1$, so for $q_0,q_1 \in \Sp^1$ we have $K^*_{ij}=(K_{ji})^{-1}$.
In other words, for $q_0,q_1 \in \Sp^1$ we have that
\begin{alignat*}{3}
K^*v_{i,j} =& \enspace &\text{Kast}(v_{i,j}v_{i,j+1})^{^{\scriptstyle -1}}& \cdot v_{i,j+1} + \text{Kast}(v_{i,j}v_{i,j-1})^{^{\scriptstyle -1}} \cdot v_{i,j-1}\\
+& &\text{Kast}(v_{i,j}v_{i+1,j})^{^{\scriptstyle -1}}& \cdot v_{i+1,j} + \text{Kast}(v_{i,j}v_{i-1,j})^{^{\scriptstyle -1}} \cdot v_{i-1,j}
\end{alignat*}

Recall the rectangular fundamental domain $D_L \subset \R^2$ we used in the construction of $K$; its vertices are $(0,0), (x_0,0), (0,y_1)$ and $(x_0,y_1)$, where $v_0 = (x_0,0)$ and $v_1=(x_1,y_1)$ generate $L$ (and $0 \leq x_1 < x_0$).
Because translations of $D_L$ by $v_0$ and $v_1$ tile $\R^2$, we may use it to partition $\big(\Z+\frac12\big)^2$, the set of vertices on $G(\Z^2)$.
For $a,b \in \Z$, consider the sets $D(a,b)\label{def:dab} = \big(D_L + a\cdot v_0 + b \cdot v_1\big)\cap \big(\Z+\tfrac12\big)^2$.
It's easy to see that $P = \{ D(a,b) \text{ } | \text{ } a,b \in \Z\}$ is a partition of $\big(\Z+\frac12\big)^2$.

We now study transitions between adjacent $D(a,b)$'s via edge-paths in $\big(\Z+\tfrac12\big)^2$.
The diagram below represents this behavior schematically; remember the orientation induced by the coloring of $\Z^2$ on edges of its unit squares when determining the Kasteleyn weights of crossing dominoes.
\begin{figure}[H]
		\centering
		\def\svgwidth{0.975\columnwidth}
    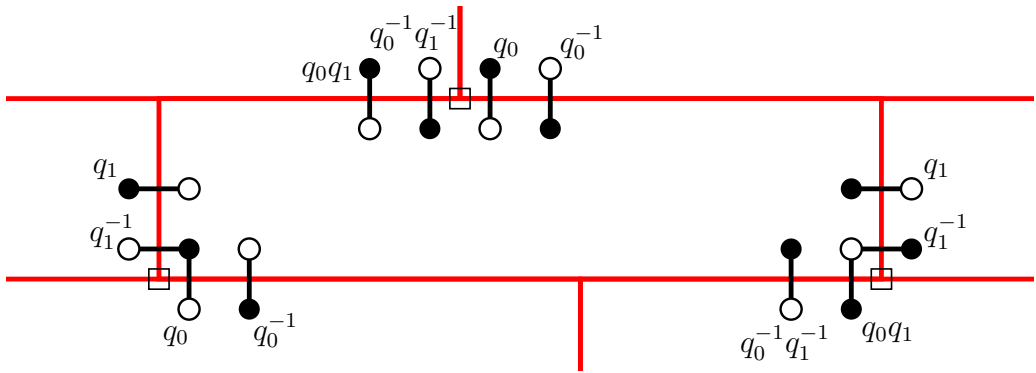
		\caption{Schematic representation of $q$-weights for crossing edges. Red rectangles are fundamental domains, square vertices are elements of $L$, and round black or white vertices are elements of $G(\Z^2)$.}
\end{figure}

More precisely, we mean that:
\begin{itemize}
	\item Edges joining a black vertex in $D(a,b)$ (respectively white) to a white vertex in $D(a,b+1)$ (respectively black) have Kasteleyn weight $\pm q_0^{-1}$ (respectively $\pm q_0$);
	\item Edges joining a black vertex in $D(a,b)$ (respectively white) to a white vertex in $D(a+1,b)$ (respectively black) have Kasteleyn weight $\pm q_1$ (respectively $\pm q_1^{-1}$);
	\item Edges joining a black vertex in $D(a,b)$ (respectively white) to a white vertex in $D(a-1,b+1)$ (respectively black) have Kasteleyn weight $\pm q_0^{-1} q_1^{-1}$ (respectively $\pm q_0 q_1$).
\end{itemize}

Remember our choice of positive orientation for dominoes: from their black square to their white square.
When we assigned Kasteleyn weights to dominoes it was done irrespective of the domino's own orientation, but $K$ naturally maps black vertices to white vertices, so we may think of it as the weight being assigned to dominoes with positive orientation.
For dominoes with negative orientation, we assign it the inverse of its weight with positive orientation.
Thus, each oriented domino has an \emph{oriented weight}\label{def:orientedweight}:
its Kasteleyn weight for dominoes oriented positively, and the inverse of its Kasteleyn weight for dominoes oriented negatively.

For adjacent vertices $p_0,p_1 \in \big(\Z+\frac12)^2$, let $p_0p_1$ the edge that joins them and is oriented from $p_0$ to $p_1$.
For an oriented edge $e$, let $o(e)$ denote its oriented weight.
We thus have that:
\begin{itemize}
	\item $\text{Kast}(p_0p_1) = \text{Kast}(p_1p_0)$, that is, Kasteleyn weight does not depend on orientation;
	\item $o(p_0p_1) = o(p_1p_0)^{-1}$, that is, reversing edge orientation inverts oriented weight;
	\item If $p_0$ is black, then $o(p_0p_1) = \text{Kast}(p_0p_1)$.
\end{itemize}

Furthermore, our previous observation may be simplified.
Let $u,v$ be adjacent vertices with $u \in D(a,b)$.
\begin{itemize}	
	\item If $v \in D(a,b+1)$, then $o(uv) = \pm q_0^{-1}$;
	\item If $v \in D(a+1,b)$, then $o(uv) = \pm q_1$;
	\item If $v \in D(a-1,b+1)$, then $o (uv) = \pm q_0^{-1} q_1^{-1}$.
\end{itemize}

Let $b_i$ be $D_L$'s $i$-th black vertex and $w_j$ be its $j$-th white vertex, as we enumerated them.
With these conventions, $K_{ij} = o(b_iw_j)$.
Moreover, observe that when $q_0,q_1 \in \Sp^1$, $K^*_{ij}=(K_{ji})^{-1}$, so $K^*_{ij} = o(b_jw_i)^{-1} = o(w_ib_j)$.
In this case, the actions of $K$ and $K^*$ can be described by essentially the same formula, below.
The `$=$' symbol draws attention to the fact that each of $K, K^*$ acts on vertices of different colors, so for any given $v_{i,j}$ only one of $K(v_{i,j}), K^*(v_{i,j})$ actually makes sense.
\begin{equation}\begin{alignedat}{3}\label{formulak}
K(v_{i,j}) \enspace \text{`$=$'} \enspace K^*(v_{i,j}) \enspace \text{`$=$'} &\vphantom{+}& \enspace &o(v_{i,j}v_{i,j+1}) \cdot v_{i,j+1} \enspace + \enspace o(v_{i,j}v_{i,j-1}) \cdot v_{i,j-1} \\
&+& &o(v_{i,j}v_{i+1,j}) \cdot v_{i+1,j} \enspace + \enspace o(v_{i,j}v_{i-1,j}) \cdot v_{i-1,j}
\end{alignedat}\end{equation}

For $a,b \in \Z$, let $\mathcal{Q}(a,b) = q_0^{-b}q_1^a$.
For $v \in \big(\Z + \frac12\big)^2$, let $\mathcal{Q}(v) = \mathcal{Q}(a_v,b_v)$, where $v \in D(a_v,b_v)$.
Recall that because $\big(\Z + \frac12\big)^2$ is partitioned by the $D(a,b)$'s, $\mathcal{Q}(v)$ is always well-defined.

\begin{prop}\label{qfund}
Let $(p_k)_{k=0}^n$ be a sequence of adjacent vertices in $\big(\Z+\frac12\big)^2$.
Then
$$\prod\limits_{k=0}^{n-1}{\lvert o(p_kp_{k+1})\rvert} = \frac{\mathcal{Q}(p_n)}{\mathcal{Q}(p_0)}.$$
In particular, the product depends only on the initial and final $D(a,b)$'s.
\end{prop}
\begin{proof}
Notice that whenever $u$ and $v$ are adjacent vertices in the same $D(a,b)$, then $\lvert o(uv)\rvert = 1$.
Combined with our previous (simplified) observation, it's easy to check that for any two adjacent vertices $u, v \in \big(\Z+\frac12\big)^2$ we have
$$\lvert o(uv)\rvert = \frac{\mathcal{Q}(v)}{\mathcal{Q}(u)}.$$

From this, the proposition follows.
\end{proof}

With Proposition \ref{qfund} and formula~\eqref{formulak}, we can make a good description of the action of $M = KK^* \oplus K^*K$. Notice~\eqref{formulaq} applies irrespective of $v_{i,j}$'s color.

\vspace{-0.2cm}\begin{equation}\begin{alignedat}{4}\label{formulaq}
M v_{i,j} = 4\cdot v_{i,j} \enspace &-& \enspace \frac{\mathcal{Q}(v_{i+2,j})}{\mathcal{Q}(v_{i,j})} \cdot v_{i+2,j} \enspace &-& \enspace \frac{\mathcal{Q}(v_{i-2,j})}{\mathcal{Q}(v_{i,j})} \cdot v_{i-2,j} \\
 &+& \enspace \frac{\mathcal{Q}(v_{i,j+2})}{\mathcal{Q}(v_{i,j})} \cdot v_{i,j+2} \enspace &+& \enspace \frac{\mathcal{Q}(v_{i,j-2})}{\mathcal{Q}(v_{i,j})} \cdot v_{i,j-2}
\end{alignedat}\end{equation}

The coefficient in $v_{i,j}$ comes from moving forward then backwards in each cardinal direction, so any negative edge traversed this way will account for two minus signs, and the end result is always $1$, as per Proposition \ref{qfund}.
Vertices of the form $v_{i\pm 1, j \pm 1}$ do not feature:
each such vertex can be reached from $v_{i,j}$ via exactly two distinct edge-paths, which enclose a square with vertices in $\big(\Z+\frac12\big)^2$.
By Proposition \ref{qfund}, these paths contribute with coefficients that are equal in absolute value, but the definition of Kasteleyn signing ensures they have opposite signs, so they cancel out.
The vertices $v_{i, j\pm 2}$ are reached from $v_{i,j}$ via vertical edge-paths, so they never contain a negative edge for our choice of Kasteleyn signing.
On the other hand, the vertices $v_{i\pm 2, j}$ are reached from $v_{i,j}$ via horizontal edge-paths, so they always contain exactly one negative edge for our fixed Kasteleyn signing.

We will use formula~\eqref{formulaq} to find $M$'s eigenvectors.
$M$ acts on the space of equivalence classes of vertices of $G(\Z^2)$, so a vector in that space may be thought of as a weight on each such equivalence class.
Alternatively, we may think of it as a function on the vertices of $G(\Z^2)$ that is $L$-periodic, so it coincides on each vertex of a given equivalence class.

For each $x \in \big(\R^2\big)^*$, consider the function $\zeta_x: \R^2 \longrightarrow \Sp^1$ given by $\zeta_x(v) = \exp\big(2\pi\bi \cdot x(v)\big)$.
Now, let $q_0,q_1 \in \Sp^1$ be fixed, and take any $z \in \big(\R^2\big)^*$ with $\zeta_z(v_0) = q_1^{-1}$ and $\zeta_z(v_1) = q_0$, where $v_0,v_1 \in L$ are the vectors used in the definition of $D_L$.
Notice this implies $\zeta_z(v+a\cdot v_0 + b\cdot v_1) = q_0^b \cdot q_1^{-a}\cdot\zeta_z(v)$ for all $v \in \R^2$.
Finally, consider the function $f_z$ defined on $\big(\Z+\frac12\big)^2$ by $f_z(v) = \zeta_z(v) \cdot \mathcal{Q}(v)$; observe it is $L$-periodic.
We claim $f_z$ is an eigenvector of $M = M(q_0,q_1)$, i.e., $$v_z = \sum\limits_{i,j} f_z(v_{i,j})\cdot v_{i,j}$$is an eigenvector of $M$.
Indeed, notice the $v_{i,j}$-coordinate of $Mv_z$ is given by
\begin{alignat*}{3}
[Mv_z\big]_{i,j} = 4 \cdot f_z(v_{i,j}) \enspace &-& \enspace f_z(v_{i+2,j}) \cdot \frac{\mathcal{Q}(v_{i,j})}{\mathcal{Q}(v_{i+2,j})} \enspace - \enspace f_z(v_{i-2,j}) \cdot \frac{\mathcal{Q}(v_{i,j})}{\mathcal{Q}(v_{i-2,j})}& \\
&+& \enspace f_z(v_{i,j+2}) \cdot \frac{\mathcal{Q}(v_{i,j})}{\mathcal{Q}(v_{i,j+2})} \enspace + \enspace f_z(v_{i,j-2}) \cdot \frac{\mathcal{Q}(v_{i,j})}{\mathcal{Q}(v_{i,j-2})}&,
\end{alignat*}
which we may rearrange to
\begin{equation*}
f_z(v_{i,j}) \cdot \left(4 \enspace - \enspace \frac{\zeta_z(v_{i+2,j})}{\zeta_z(v_{i,j})} \enspace - \enspace \frac{\zeta_z(v_{i-2,j})}{\zeta_z(v_{i,j})} \enspace + \enspace \frac{\zeta_z(v_{i,j+2})}{\zeta_z(v_{i,j})} \enspace + \enspace \frac{\zeta_z(v_{i,j-2})}{\zeta_z(v_{i,j})}\right).
\end{equation*}

If $x = (x_0,x_1)$, let $\zeta_{x,0} = \exp(2\pi\bi \cdot x_0)$ and $\zeta_{x,1} = \exp(2\pi\bi \cdot x_1)$.
Then we may write
\begin{alignat*}{7}
[Mv_z\big]_{i,j} &=& &f_z(v_{i,j})& \enspace &\cdot& \enspace &\left(4 - {\zeta_{z,0}}^2 - {\zeta_{z,0}}^{-2} + {\zeta_{z,1}}^2 + {\zeta_{z,1}}^{-2}\right) \\
&=& \enspace &f_z(v_{i,j})& \enspace &\cdot& \enspace &\left( - \left(\zeta_{z,0} - {\zeta_{z,0}}^{-1}\right)^2 + \left(\zeta_{z,1} + {\zeta_{z,1}}^{-1}\right)^2 \right).
\end{alignat*}

In other words, for $z=(z_0,z_1)$, $v_z$ is an eigenvector of $M = M(q_0,q_1)$ with associated eigenvalue
\begin{align*}
\lambda_z =& \left( - \left(\zeta_{z,0} - {\zeta_{z,0}}^{-1}\right)^2 + \left(\zeta_{z,1} + {\zeta_{z,1}}^{-1}\right)^2 \right)\\
=& \enspace 4 \cdot \left( \big(\sin(2\pi z_0)\big)^2 + \big(\cos(2\pi z_1)\big)^2 \right).
\end{align*}

We will think of each pair $q_0,q_1 \in \Sp^1$ as a homomorphism $q: L \longrightarrow \Sp^1$ with $q(v_0) = q_1^{-1}$ and $q(v_1) = q_0$, so the condition that $\zeta_z(v_0) = q_1^{-1}$ and $\zeta_z(v_1) = q_0$ reduces to $\zeta_z\raisebox{-.2em}{$\big|_L$} = q$.
Thus, for each $z \in \big(\R^2\big)^*$ with $\zeta_z\raisebox{-.2em}{$\big|_L$} = q$, $f_z$ is an eigenvector of $M = M(q)$.
Are all its eigenvectors of that form?

First note that $\zeta_z\raisebox{-.2em}{$\big|_L$} = \zeta_{\widetilde{z}}\raisebox{-.2em}{$\big|_L$} \Longleftrightarrow \forall v \in L, \big(z- \widetilde{z}\big)(v) \in \Z \Longleftrightarrow z - \widetilde{z} \in L^*$.
We may thus identify $\text{Hom}\big(L, \Sp^1\big)$ with $\quotient{\big(\R^2\big)^*}{L^*}$.$\vphantom{^{^{^{^{^{^{}}}}}}}$
On the other hand, $z - \widetilde{z} \in \big(\Z^2\big)^* \Longleftrightarrow \exists c \in \{\pm 1\}, \forall v \in \big(\Z + \frac12\big)^2, \zeta_z(v) = c \cdot \zeta_{\widetilde{z}}(v)$.

In particular, if $z$ and $\widetilde{z}$ are in the same equivalence class $q \in \quotient{\big(\R^2\big)^*}{L^*}$ and $z - \widetilde{z}$ is in $\big(\Z^2\big)^*$
--- i.e., $z$ and $\widetilde{z}$ represent the same element in $\quotient{q}{\big(\Z^2\big)^*}$ ---,
then $f_z$ and $f_{\widetilde{z}}$ are essentially the same eigenvector of $M(q)$.

\begin{prop}\label{eigenort}
Let $z, \widetilde{z}$ be in the same equivalence class $q \in \quotient{\big(\R^2\big)^*}{L^*}$, but in different equivalence classes of \vphantom{.} $\quotient{q}{\big(\Z^2\big)^*}$.
Then $v_z$ and $v_{\widetilde{z}}$ are orthogonal.
\end{prop}
\begin{proof}
We will show that $\langle v_z, v_{\widetilde{z}} \rangle = 0$.
Notice $f_z$ and $f_{\widetilde{z}}$ are both $L$-periodic, so we need only verify $\sum_{i,j} f_z(v_{i,j}) \cdot \overline{f_{\widetilde{z}}(v_{i,j})} = 0$, where the sum is carried over the $v_{i,j}$ in a single $D(a,b)$, say $D(0,0)$. 

Since $\overline{\zeta_x} = \zeta_{-x}$ for all $x \in \big(\R^2\big)^*$, it holds that
\begin{alignat*}{1}
&\sum\limits_{i,j} f_z(v_{i,j}) \cdot \overline{f_{\widetilde{z}}(v_{i,j})} = \sum\limits_{i,j} \zeta_z(v_{i,j}) \cdot \overline{\zeta_{\widetilde{z}}(v_{i,j})} = \sum\limits_{i,j} \exp\Big(2\pi\bi \cdot \big(z-\widetilde{z}\big)(v_{i,j})\Big).
\end{alignat*}

The $v_{i,j}$ in $D(0,0)$ are vertices $(i,j) + \big(\frac12, \frac12\big)$ in $\big(\Z+\frac12\big)^2$ with $0 \leq i < x_0$ and $0 \leq j < y_1$, where $v_0 = (x_0,0)$ and $v_1 = (x_1,y_1)$ are the generators of $L$ used in the construction of $D_L$.
Letting $C = \exp\left(2\pi\bi \cdot \left\langle z-\widetilde{z}, \left(\frac12,\frac12\right) \right\rangle \right)$, we have that $\sum_{i,j} f_z(v_{i,j})\cdot \overline{f_{\widetilde{z}}(v_{i,j})}$ is
\begin{alignat*}{1}
 &C \cdot\left(\sum\limits_{0 \leq i < x_0} \exp\Big(2\pi\bi \cdot \Big\langle z - \widetilde{z}, (i,0) \Big\rangle\Big)\right) \cdot \left( \sum\limits_{0 \leq j < y_1} \exp\Big(2\pi\bi \cdot \Big\langle z - \widetilde{z}, (0,j) \Big\rangle\Big)\right)
\end{alignat*}

Rewrite the expression above as:
\begin{equation}\begin{alignedat}{1}\label{fourm}
&\sum\limits_{i,j} f_z(v_{i,j}) \cdot \overline{f_{\widetilde{z}}(v_{i,j})} = \enspace C \cdot \left(\sum\limits_{0 \leq i < x_0} \big(\zeta_{z-\widetilde{z},0}\big)^i\right) \cdot \left( \sum\limits_{0 \leq j < y_1} \big(\zeta_{z-\widetilde{z},1}\big)^j \right)
\end{alignedat}\end{equation}

Observe that $(\zeta_{z-\widetilde{z},0})^{x_0} = \zeta_{z-\widetilde{z}}(v_0) = 1$, because $z$ and $\widetilde{z}$ are both in the same equivalence class $q \in \quotient{\big(\R^2\big)^*}{L^*}$ $\left(\text{they satisfy $\zeta_z\raisebox{-.2em}{$\big|_L$} = \zeta_{\widetilde{z}}\raisebox{-.2em}{$\big|_L$}$}\right)$.
In similar fashion, $\zeta_{z-\widetilde{z}}(v_1) = (\zeta_{z-\widetilde{z},0})^{x_1} \cdot (\zeta_{z-\widetilde{z},1})^{y_1} = 1$, so $(\zeta_{z-\widetilde{z},1})^{y_1} = (\zeta_{z-\widetilde{z},0})^{-x_1}$.
Now, $\zeta_{z-\widetilde{z},0}$ and $\zeta_{z-\widetilde{z},1}$ may not both be $1$, for $z$ and $\widetilde{z}$ are in different equivalences classes of $\quotient{q}{\big(\Z^2\big)^*}$, so their difference does not lie in $\big(\Z^2\big)^*$.

If $\zeta_{z-\widetilde{z},0} \neq 1$, then $(\zeta_{z-\widetilde{z},0})^{x_0} = 1$ implies that $\sum_{0 \leq i < x_0} \big(\zeta_{z-\widetilde{z},0}\big)^i$ is a symmetrical sum in $\Sp^1$, and hence 0.
If $\zeta_{z-\widetilde{z},0} = 1$, then $(\zeta_{z-\widetilde{z},1})^{y_1} = 1$, so similarly $\sum_{0 \leq j < y_1} \big(\zeta_{z-\widetilde{z},1}\big)^j = 0$.
Regardless of the situation, we see from~\eqref{fourm} that $\sum_{i,j} f_z(v_{i,j}) \cdot \overline{f_{\widetilde{z}}(v_{i,j})} = 0$, as desired.
\end{proof}

Proposition \ref{eigenort} tells us each equivalence class of $\quotient{q}{\big(\Z^2\big)^*}$ corresponds to a different eigenvector of $M(q)$.
The next proposition shows these are `all' of its eigenvectors.

\begin{prop}\label{isom}
The groups $\quotient{L^*}{\big(\Z^2\big)^*}$, $\left(\quotient{\Z^2}{L}\right)^*$ and $\quotient{\Z^2}{L}$ are all isomorphic.
\end{prop}
\begin{proof}
We first show $\quotient{L^*}{\big(\Z^2\big)^*}$ and $\left(\quotient{\Z^2}{L}\right)^*$ are isomorphic (naturally so, in fact).
Any $f \in L^*$ admits a unique extension by linearity to a functional in $\big(\R^2\big)^*$, which we still call $f$.
Given any such $f$, define $\zeta_f: \R^2 \longrightarrow \Sp^1$ by $\zeta_f(v) = \exp\big(2\pi\bi\cdot f(v)\big)$; notice it is a homomorphism.
Consider $\zeta:\big(\R^2\big)^* \longrightarrow \text{Hom}(\R^2, \Sp^1)$ given by $\zeta(f) = \zeta_f$.
We claim $\zeta$ defines an isomorphism between $\quotient{L^*}{\big(\Z^2\big)^*}$ and $\left(\quotient{\Z^2}{L}\right)^*$.
It's easy to check that it is an homomorphism; we prove it is a bijection. Observe that the following are equivalent:
\begin{itemize}
	\item $f$ and $\widetilde{f}$ are in $L^*$ and belong to the same equivalence class of $\quotient{L^*}{\big(\Z^2\big)^*}$;
	\item $\zeta_f \raisebox{-.2em}{$\big|_L$} = \zeta_{\widetilde{f}}\raisebox{-.2em}{$\big|_L$} = \mathbbm{1}$ and ${\displaystyle \frac{\zeta_f}{\zeta_{\widetilde{f}}}}\raisebox{-.35em}{$\bigg|_{\Z^2}$} = \mathbbm{1}$ (i.e., $\zeta_f$ and $\zeta_{\widetilde{f}}$ coincide on $\Z^2$).
\end{itemize}
 
In addition, each $g \in \left(\quotient{\Z^2}{L}\right)^* = \text{Hom}\left(\quotient{\Z^2}{L}\makebox[0.1ex]{}, \Sp^1\right)$ corresponds to a unique $\widetilde{g} \in \text{Hom}\left(\Z^2,\Sp^1\right)$ that is $L$-periodic (it satisfies $\widetilde{g}\raisebox{-.2em}{$\big|_L$} = \mathbbm{1}$).
It follows that $\zeta$ injectively maps $\quotient{L^*}{\big(\Z^2\big)^*}$ to $\left(\quotient{\Z^2}{L}\right)^*$.

To see $\zeta$ also does so surjectively, observe that any $\widetilde{g} \in \text{Hom}\left(\Z^2,\Sp^1\right)$ satisfies
$$\widetilde{g}(a_1,a_2) = {\widetilde{g}(1,0)}^{a_1} \cdot {\widetilde{g}(0,1)}^{a_2}.$$

Choose any $\alpha_1,\alpha_2 \in \R$ with $\widetilde{g}(1,0) = \exp(2\pi\bi \cdot \alpha_1)$ and $\widetilde{g}(0,1) = \exp(2\pi\bi \cdot \alpha_2)$. Then$$\widetilde{g}(a_1,a_2) = \exp\Big(2\pi\bi \cdot (\alpha_1 \cdot a_1 + \alpha_2 \cdot a_2)\Big),$$so $\widetilde{g}$ is realized by $\zeta_f$, where $f(a_1,a_2) = \alpha_1 \cdot a_1 + \alpha_2 \cdot a_2$. Now it's easy to check that $\widetilde{g}\raisebox{-.2em}{$\big|_L$} = \mathbbm{1}$ if and only if $\forall v \in L$, $f(v) \in \Z$ --- that is, if and only if $f \in L^*$.

We now show $\left(\quotient{\Z^2}{L}\right)^*$ and $\quotient{\Z^2}{L}$ are isomorphic.
Since $\quotient{\Z^2}{L}$ is abelian and finite, there are prime powers $m_1, \cdots, m_k$ with $\quotient{\Z^2}{L} \simeq \quotient{\Z}{(m_1)} \oplus \cdots \oplus \quotient{\Z}{(m_k)}$.
For each $1 \leq i \leq k$, let $e_i$ be a generator of $\quotient{\Z}{(m_i)}$, so each $f \in \left(\quotient{\Z^2}{L}\right)^*$ is uniquely defined by the value it takes on the $e_i$'s.
Each $1 \leq i \leq k$ satisfies $f(e_i)^{m_i} = 1$, so $f(e_i)$ must be a $m_i$-th root of unity.
Under multiplication, the $n$-th roots of unity form a cyclic group of order $n$, which is isomorphic to $\quotient{\Z}{(n)}$.
It follows that $\left(\quotient{\Z^2}{L}\right)^*$ is also isomorphic to $\quotient{\Z}{(m_1)} \oplus \cdots \oplus \quotient{\Z}{(m_k)}$, which completes the proof.
\end{proof}

Each class $q \in \quotient{\big(\R^2\big)^*}{L^*}$ is isomorphic to $L^*$, so $\quotient{q}{\big(\Z^2\big)^*} \simeq \quotient{L^*}{\big(\Z^2\big)^*}$.
Since $n = \card\left(\quotient{\Z^2}{L}\right)$ is the number of squares on the fundamental domain $D_L$, $M$ is $n \times n$.
It follows from Proposition \ref{isom} that $\quotient{q}{\big(\Z^2\big)^*}$ has $n$ elements, so Proposition \ref{eigenort} implies they provide all $n$ of $M(q)$'s eigenvectors.

\begin{corolario}
For each $q \in \textnormal{Hom}\left(L, \Sp^1\right) \simeq \quotient{\big(\R^2\big)^*}{L^*}$ we have
\begin{equation*}
\det\Big(M(q)\Big) = \enspace \prod\limits_{\hphantom{.......}\mathclap{\displaystyle z \in \quotient{q}{\big(\Z^2\big)^*}}} \hphantom{..}{\lambda_z} \enspace\hphantom{...},
\end{equation*}
where for $z=(z_0,z_1)$ it holds that $\lambda_z = 4 \cdot \left( \big(\sin(2\pi z_0)\big)^2 + \big(\cos(2\pi z_1)\big)^2 \right)$.
\end{corolario}
Notice $\lambda_z$ does not depend on choice of representative for $z \in \quotient{q}{\big(\Z^2\big)^*}$.
Moreover, $M = KK^* \oplus K^*K$ implies $\det(M) = \lvert\det(K)\rvert^4$; in particular, $\det(M)$ is always a nonnegative real and $\lvert \det(K) \rvert = \sqrt[4]{\det(M)}$.

\section{Spaces of $L$-quasiperiodic functions and the case of $K, K^*$}\label{sec:espfunc}

We saw that $M$ acted on the space of $L$-periodic functions on $\big(\Z+\frac12\big)^2$ --- the weights on squares of $D_L$ ---, but we now present another interpretation of the situation.
For each $q \in \text{Hom}\big(L, \Sp^1\big)$, consider the space $\mathcal{F}(L,q)$\label{def:flq} of complex functions on $\big(\Z+\frac12\big)^2$ that are $L$-quasiperiodic with parameter $q$, meaning
$$f \in \mathcal{F}(L,q) \Longleftrightarrow \forall u \in \left(\Z+\tfrac12\right)^2, \forall v \in L, f(u + v) = q(v)\cdot f(u).$$

Of course, $q$ is defined by the values it takes on a basis of $L$, $v_0$ and $v_1$ for instance.
If $q(v_0) = q_1^{-1}$ and $q(v_1) = q_0$, the figure in the next page is a representation of some generic $f \in \mathcal{F}(L,q)$, where the number next to a vertex indicates the value $f$ takes on it.

\begin{sidewaysfigure}[p]
		\centering
		\def\svgwidth{0.775\columnwidth}
    \vspace{0.375cm}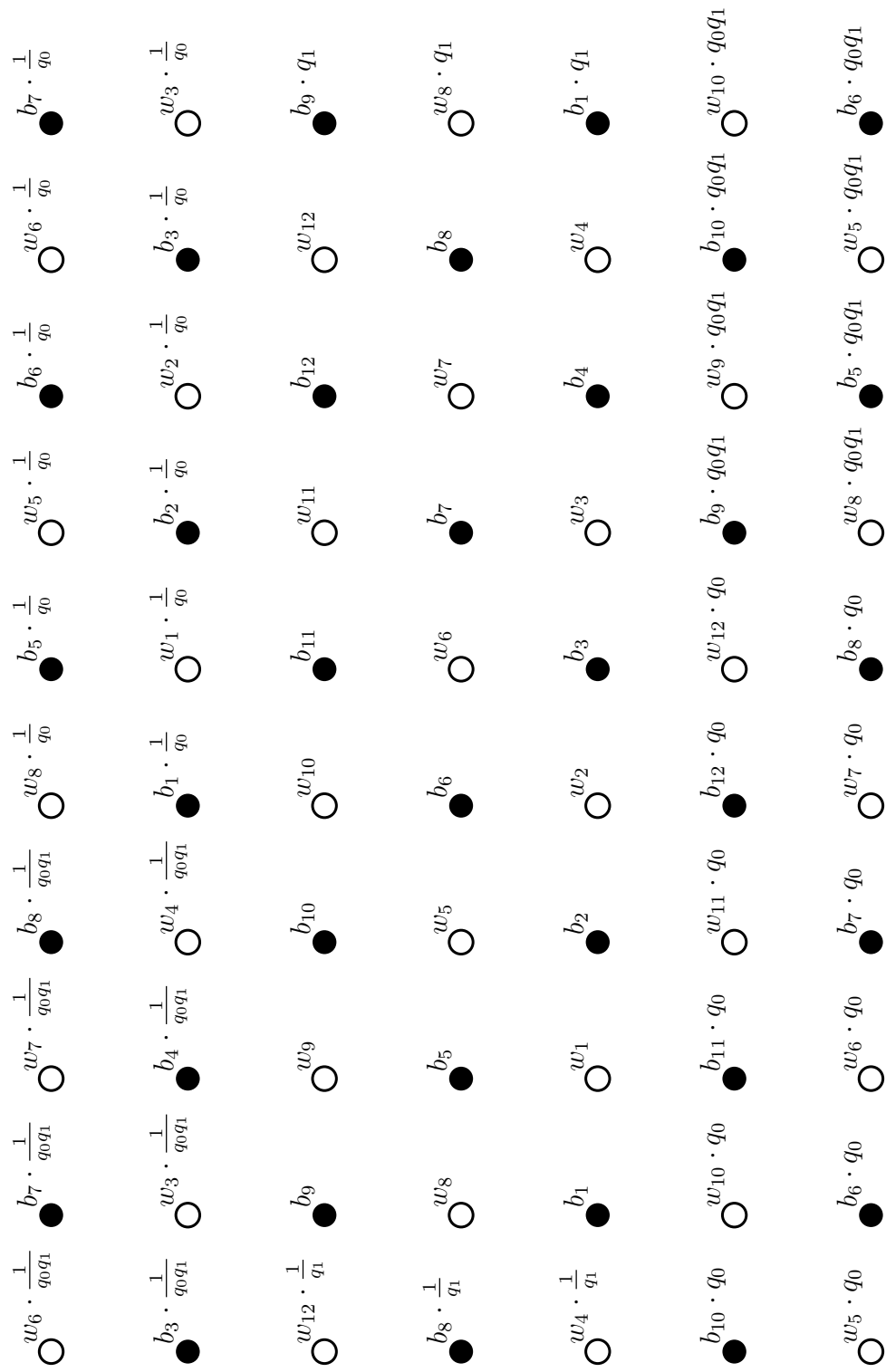
		\caption{Representation of some generic $f \in \mathcal{F}(L,q)$, where $L$ is generated by $v_0 = (8,0)$ and $v_1 = (3,3)$. Vertices are in $G(\Z^2)$.}
\end{sidewaysfigure}

There is a host of obvious isomorphisms between the space of $L$-periodic functions on $\big(\Z+\frac12\big)^2$ and $\mathcal{F}(L,q)$:
simply make two functions correspond if they agree on any particular fundamental domain of $L$.
We will use the isomorphism $\psi$\label{def:psi} that checks for functions agreeing on $D_L = D(0,0)$;
under $\psi$, each $v_{i,j} \in D_L$ corresponds to the function $g_{i,j}$ defined by taking the value 1 on $v_{i,j}$ and 0 on every other vertex of $D_L$ (and by extension through $L$-quasiperiodicity).

This definition serves to take the $\mathcal{Q}(v)$'s away from formula~\eqref{formulaq} and into the space itself, which in turn simplifies the expression of $M(q)$.
Indeed, let $\widetilde{M}$ be given by
\begin{equation}\begin{alignedat}{1}\label{formulameq}
\big(\widetilde{M}g\big)(x,y) = 4\cdot g(x,y) &- g(x+2,y) - g(x-2,y)\\
 &+ g(x,y+2) + g(x,y-2)
\end{alignedat}\end{equation}

With our choice isomorphism and convention for values of $q(v_0)$ and $q(v_1)$ above, it's easily checked that $\widetilde{M}\psi(v_{i,j}) = \psi(Mv_{i,j})$, so $\widetilde{M}(q)$ emulates $M(q)$'s action on $\mathcal{F}(L,q)$.
We will thus refer to $\widetilde{M}$ also by $M$.
Moreover, arguments similar to those used before will show that for formula~\eqref{formulameq}, functions of the form $\zeta_z$ are eigenvectors of $M$ (rather than $f_z = \zeta_z \cdot \mathcal{Q}$), and they are associated to the same eigenvalues $\lambda_z$.
Perhaps a yet better benefit of this approach is that it defines the linear map $M(q)$ without resorting to a particular choice or construction of fundamental domain for $L$.

We may use this interpretation to describe $K$ and $K^*$ similarly.
Let $L_0 \subset \Z^2$\label{def:l0} be the lattice spanned by $\{ (2,0), (1,1) \}$;
observe that any valid lattice is a sublattice of $L_0$. Consider the affine lattices $L_b = L_0 + \big(\frac12, \frac12 \big)$\label{def:l0b} and $L_w = L_0 + \big(\frac12, -\frac12\big)$\label{def:l0w};
notice $L_b$ is the set of black vertices of $G(\Z^2)$ and $L_w$ is the set of white ones.
Now let $\mathcal{B}(L,q)$\label{def:blq} be the space of complex functions on $L_b$ that are $L$-quasiperiodic with parameter $q$, and similarly for $\mathcal{W}(L,q)$\label{def:wlq}.

Like with $M(q)$, we may interpret $K(q)$ as a linear map $\mathcal{B}(L,q) \longrightarrow \mathcal{W}(L,q)$ and ${K(q)}^*$ as a linear map $\mathcal{W}\left(L,q\right) \longrightarrow \mathcal{B}\left(L,q\right)$ (note that for $z \in \Sp^1$, $\vphantom{\Big(}\overline{z^{-1}} = z$).
We provide the analogues of formula~\eqref{formulameq}:
\begin{subequations}
\begin{alignat}{6}
&\big(Kg\big)&(x,y)& \enspace &=&  &g(x+1,y) - g(x-1,y) + g(x,y+1) + g(x,y-1) \label{formulakeq}\\
&\big(K^*g\big)&(x,y)& \enspace &=& \enspace - &g(x+1,y) + g(x-1,y) + g(x,y+1) + g(x,y-1) \label{formulak*eq}
\end{alignat}
\end{subequations}

In each case, which term is negatively signed is justified by our fixed choice of Kasteleyn signing: each domino on the brick wall $b_{\mathcal{N}}$ has a black square on the left and a white square on the right.
\begin{figure}[H]
		\centering
		\includegraphics[width=0.9\textwidth]{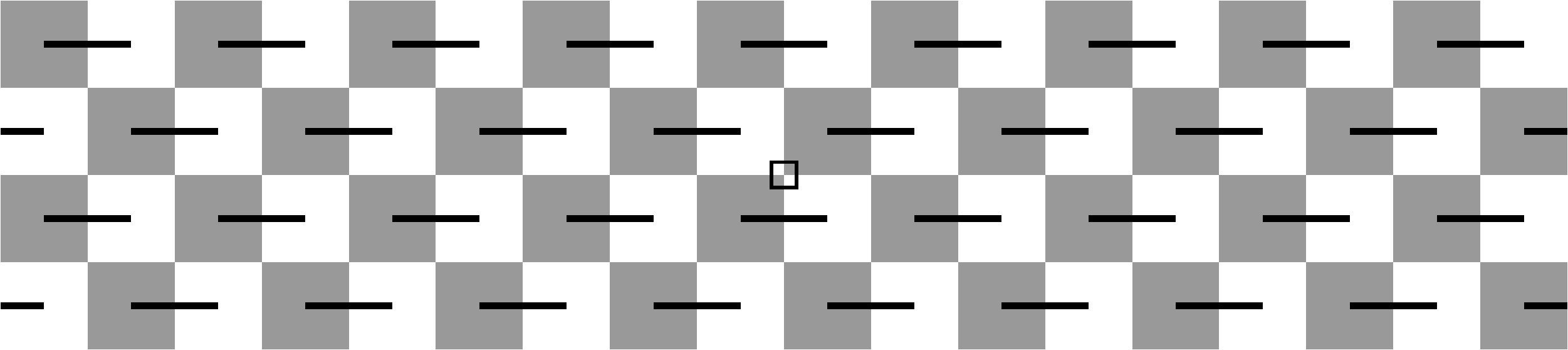}
		\caption{The brick wall $b_{\mathcal{N}}$. The marked vertex is the origin.}
\end{figure}

The \textit{matrices} $K$ and $K^*$ we previously constructed represent the linear maps above in bases we now describe\label{def:dbdw}.
Let $b_i$ be $D_L$'s $i$-th black vertex and $w_j$ its $j$-th white vertex, as we enumerated them.
Let $g[b_i] \in \mathcal{B}(L,q)$ be defined by taking the value $1$ on $b_i$ and $0$ on each other black vertex of $D_L$, and let $g[w_j] \in \mathcal{W}(L,q)$ take the value $1$ on $w_j$ and $0$ on each other white vertex of $D_L$.
Then the $g[b_i]$ form an ordered basis $D_b$ of $\mathcal{B}(L,q)$, and the $g[w_j]$ form an ordered basis $D_w$ of $\mathcal{W}(L,q)$.
The matrix $K = K_D$\label{def:kd} represents the linear map of~\eqref{formulakeq} from $D_b$ to $D_w$, and in similar fashion $K^*_D$ goes from $D_w$ to $D_b$.

We will now choose different bases for $\mathcal{B}(L,q)$ and $\mathcal{W}(L,q)$, given by $\zeta_z$'s.
Observe that one such function satisfies $\zeta_z(u+v) = \zeta_z(u) \cdot \zeta_z(v)$, so it\footnote{Its restriction to the relevant subset of $\big(\Z+\frac12\big)^2$} is an element of $\mathcal{B}(L,q)$ or of $\mathcal{W}(L,q)$ if and only if $\zeta_z\raisebox{-.2em}{$\big|_L$} = q$.
Recall our identification of $\text{Hom}\big(L, \Sp^1\big)$ with $\quotient{\big(\R^2\big)^*}{L^*}$: $\zeta_z\raisebox{-.2em}{$\big|_L$} = \zeta_{\widetilde{z}}\raisebox{-.2em}{$\big|_L$}$ if and only if $z$ and $\widetilde{z}$ are in the same equivalence class.

Granted, for $z$ and $\widetilde{z}$ in one same equivalence class $q \in \quotient{\big(\R^2\big)^*}{L^*}$, $\zeta_z$ and $\zeta_{\widetilde{z}}$ need not agree on $L_b$.
We assert that $\exists c \in \C, \enspace \zeta_z\raisebox{-.2em}{$\big|_{L_b}$} = c \cdot \zeta_{\widetilde{z}}\raisebox{-.2em}{$\big|_{L_b}$} \Longleftrightarrow z - \widetilde{z} \in {L_0}^*$. Indeed, it's easy to check that
$$\zeta_z\raisebox{-.2em}{$\big|_{L_b}$} = c \cdot \zeta_{\widetilde{z}}\raisebox{-.2em}{$\big|_{L_b}$} \Longleftrightarrow \forall u \in L_0, \exp\Big(2\pi\bi\cdot \big(z-\widetilde{z}\big)(u)\Big) \cdot \exp\Big(2\pi\bi\cdot \big(z-\widetilde{z}\big)\big(\tfrac12, \tfrac12\big)\Big) = c,$$
so when there is one such $c$ we have $\exp\left(2\pi\bi\cdot \big(z-\widetilde{z}\big)\big(\tfrac12, \tfrac12\big)\right) = c$, because $u = 0 \in L_0$.
It follows that $\exp\left(2\pi\bi\cdot \big(z-\widetilde{z}\big)(u)\right) = 1$ whenever $u \in L_0$, so $z-\widetilde{z} \in {L_0}^*$.
On the other hand, when $z-\widetilde{z} \in {L_0}^*$ we may take $c = \exp\left(2\pi\bi\cdot \big(z-\widetilde{z}\big)\big(\tfrac12, \tfrac12\big)\right)$, and the assertion is proved.

By the same token, $\exists c \in \C, \enspace \zeta_z\raisebox{-.2em}{$\big|_{L_w}$} = c \cdot \zeta_{\widetilde{z}}\raisebox{-.2em}{$\big|_{L_w}$} \Longleftrightarrow z - \widetilde{z} \in {L_0}^*$.$\vphantom{^{^{^{}}}}$

\begin{prop}\label{eigenkort}
Let $z, \widetilde{z}$ be in the same equivalence class $q \in \quotient{\big(\R^2\big)^*}{L^*}$, but in two different equivalence classes of \vphantom{.} $\quotient{q}{{L_0}^*}$.
Then $\zeta_z$ and $\zeta_{\widetilde{z}}$ are orthogonal in each of $\mathcal{B}(L,q)$ and $\mathcal{W}(L,q)$.
\end{prop}
\begin{proof}
Observe that if $L = {L_0}^*$, there are no such $z, \widetilde{z}$ as in the statement: $\quotient{q}{{L_0}^*}$ has a single class.
We thus suppose without loss of generality that $L \subsetneqq L_0$.

The idea of the proof goes similar to that of Proposition \ref{eigenort}.
We will show the expression below is always 0 whenever $v = z - \widetilde{z} \in L^*\setminus{{L_0}^*}$.
\begin{subequations}
\begin{alignat}{1}
&C \cdot \left[\sum\limits_{0 \leq i < x_0} \big(\zeta_{v,0}\big)^i \cdot \left( \sum\limits_{\substack{0 \leq j < y_1 \\\hphantom{...} j \equiv i \Mod{2}}} \big(\zeta_{v,1}\big)^j \right)\right] \label{indicej}\\
= \enspace &C \cdot \left[\sum\limits_{0 \leq j < y_1} \big(\zeta_{v,1}\big)^j \cdot \left( \sum\limits_{\substack{0 \leq i < x_0 \\\hphantom{...} i \equiv j \Mod{2}}} \big(\zeta_{v,0}\big)^i \right)\right], \label{indicei}
\end{alignat}
\end{subequations}
where $C = \exp\big(2\pi\bi \cdot v\big(\frac12,\frac12\big)\big)$ for $\mathcal{B}(L,q)$ and $C = \exp\big(2\pi\bi \cdot v\big(\frac12,-\frac12\big)\big)$ for $\mathcal{W}(L,q)$.
Notice that because $v \in L^*$ we have
\begin{subequations}
\begin{alignat}{3}
\zeta_v(v_0) \enspace &= \enspace (\zeta_{v,0})^{x_0} \enspace &= \enspace 1& \label{lvo}\\
\zeta_v(v_1) \enspace &= \enspace (\zeta_{v,0})^{x_1} \cdot (\zeta_{v,1})^{y_1} \enspace &= \enspace 1&.\label{lv1}
\end{alignat}
\end{subequations}

On the other hand, $v \notin {L_0}^*$ implies $\left(\zeta_{v,0}, \zeta_{v,1}\right)$ is not in the span of $\big\{\left(\frac12, -\frac12\right), (0, 1)\big\}$.
We will divide the proof in cases.

\paragraph{}\indent \indent \textbf{Case 1.} $\zeta_{v,0} \neq 1$ and $x_0 > 2$.

Remember $x_0$ is always positive and even, so $x_0 = 2k$ for some integer $k>1$.
Equation~\eqref{lvo} implies $\left(\left(\zeta_{v,0}\right)^2\right)^k = 1$, so that
\begin{equation*}
\sum\limits_{\substack{0 \leq i < x_0 \\ i \text{ even}}} \big(\zeta_{v,0}\big)^i = \sum\limits_{0 \leq i < k} \left(\left(\zeta_{v,0}\right)^2\right)^i
\end{equation*}
is a symmetrical sum in $\Sp^1$, and hence 0.
Of course, 
\begin{equation*}
\sum\limits_{\substack{0 \leq i < x_0 \\ i \text{ odd}}} \big(\zeta_{v,0}\big)^i \enspace = \enspace \zeta_{v,0} \cdot \sum\limits_{\substack{0 \leq i < x_0 \\ i \text{ even}}} \big(\zeta_{v,0}\big)^i,
\end{equation*}
and thus it is also 0. It follows that expression~\eqref{indicei} is 0.

\paragraph{}\indent \indent \textbf{Case 2.} $\zeta_{v,0} \neq 1$ and $x_0 = 2$.

In this case, equation~\eqref{lvo} implies $\zeta_{v,0} = -1$, so we may rewrite expression~\eqref{indicej} as
\begin{alignat*}{1}
&C \cdot \left[\left( \sum\limits_{\substack{0 \leq j < y_1 \\ j \text{ even}}} \big(\zeta_{v,1}\big)^j \right) - \left( \sum\limits_{\substack{0 \leq j < y_1 \\ j \text{ odd}}} \big(\zeta_{v,1}\big)^j \right) \right] = C \cdot \left[\sum\limits_{0 \leq j < y_1}\big(-\zeta_{v,1}\big)^j\right].
\end{alignat*}

Additionally, because $v \notin L_0^*$ it may not be $\zeta_{v,1} = \pm 1$.

Suppose first that $x_1,y_1$ are both even.
Then $\left(\zeta_{v,0}\right)^{x_1} = 1$, and equation~\eqref{lv1} implies $\left(\zeta_{v,1}\right)^{y_1} = 1$.
Since $\zeta_{v,1} \neq \pm 1$, we must have $y_1 \geq 4$, so $\sum_{0 \leq j < y_1}\left(-\zeta_{v,1}\right)^j$ is a symmetrical sum in $\Sp^1$ and hence 0.

Suppose now that $x_1,y_1$ are both odd.
Then $\left(\zeta_{v,0}\right)^{x_1} = -1$, and equation~\eqref{lv1} implies $\left(\zeta_{v,1}\right)^{y_1} = -1$, or $\left(-\zeta_{v,1}\right)^{y_1} = 1$.
Since $\zeta_{v,1} \neq \pm 1$, we must have $y_1 \geq 3$, so $\sum_{0 \leq j < y_1}\left(-\zeta_{v,1}\right)^j$ is a symmetrical sum in $\Sp^1$ and hence 0.

\paragraph{}\indent \indent \textbf{Case 3.} $\zeta_{v,0} = 1$.

In this case, equation~\eqref{lv1} implies $\left(\zeta_{v,1}\right)^{y_1} = 1$.
Because $v \notin L_0^*$, it may not be $\zeta_{v,1} = \pm 1$, so $y_1 \geq 3$.
Additionally, we may rewrite expression~\eqref{indicej} as
\begin{alignat*}{1}
&C \cdot \frac{x_0}{2} \cdot \left[\sum\limits_{0 \leq j < y_1}\big(\zeta_{v,1}\big)^j\right],
\end{alignat*}and $\sum_{0 \leq j < y_1}\left(\zeta_{v,1}\right)^j$ is a symmetrical sum in $\Sp^1$.
\end{proof}

Proposition \ref{eigenkort} guarantees that if we choose a $z$ out of each class in $\quotient{q}{{L_0}^*}$, the $\zeta_z$'s are linearly independent in each of $\mathcal{B}(L,q)$ and $\mathcal{W}(L,q)$.
Moreover, as in Proposition \ref{isom}, $\quotient{L^*}{{L_0}^*}$ and $\quotient{L_0}{L}$ are isomorphic, so these vectors generate their respective spaces --- they are bases for them.

Now, observe that applying formulas~\eqref{formulakeq} and~\eqref{formulak*eq} to $\zeta_z$ yield very simple results. 
Indeed, we have that $\left( K \zeta_z\raisebox{-.2em}{$\big|_{L_b}$} \right)(x,y)$ is
\begin{alignat*}{1}
&\zeta_z\raisebox{-.2em}{$\big|_{L_b}$}(x+1,y) - \zeta_z\raisebox{-.2em}{$\big|_{L_b}$}(x-1,y) + \zeta_z\raisebox{-.2em}{$\big|_{L_b}$}(x,y+1) + \zeta_z\raisebox{-.2em}{$\big|_{L_b}$}(x,y-1)\\
=\enspace &\Big(\left( \zeta_{z,0} - {\zeta_{z,0}}^{-1}\right)  +  \left(\zeta_{z,1} + {\zeta_{z,1}}^{-1}\right)\Big) \cdot \zeta_z\raisebox{-.2em}{$\big|_{L_w}$}(x,y).
\end{alignat*}

In other words, $\left( K \zeta_z\raisebox{-.2em}{$\big|_{L_b}$} \right) = \lambda(K,z) \cdot \zeta_z\raisebox{-.2em}{$\big|_{L_w}$}$, and by the same token $\left( K^* \zeta_z\raisebox{-.2em}{$\big|_{L_w}$} \right) = \lambda(K^*,z) \cdot \zeta_z\raisebox{-.2em}{$\big|_{L_b}$}$, where for $z = (z_0,z_1)$
\begin{equation}\label{lambdak}\begin{alignedat}{3}
\lambda(K,z) \hphantom{^*} &= \left(\zeta_{z,1} + {\zeta_{z,1}}^{-1}\right) &+& \left(\zeta_{z,0} - {\zeta_{z,0}}^{-1}\right)\\
&= \hphantom{..} 2\cos(2\pi z_1) &+& \hphantom{..} 2\bi\sin(2\pi z_0)
\end{alignedat}\end{equation}
\begin{equation}\label{lambdak*}\begin{alignedat}{3}
\lambda(K^*,z) &= \left(\zeta_{z,1} + {\zeta_{z,1}}^{-1}\right) &-& \left(\zeta_{z,0} - {\zeta_{z,0}}^{-1}\right)\\
&= \hphantom{..} 2\cos(2\pi z_1) &-& \hphantom{..} 2\bi\sin(2\pi z_0)
\end{alignedat}\end{equation}

Let $n = 2m$ be $\card\left(\quotient{\Z^2}{L}\right)$, so $\mathcal{B}(L,q)$ and $\mathcal{W}(L,q)$ are both $m$ dimensional.
Choose $z_1, \cdots, z_m$ in different classes of $\quotient{q}{{L_0}^*}$.
For each $1 \leq i \leq m$, let $\zeta_i = \zeta_{z_i}$, $\zeta_i[b] = \zeta_i\raisebox{-.2em}{$\big|_{L_b}$}$ and $\zeta_i[w] = \zeta_i\raisebox{-.2em}{$\big|_{L_w}$}$.
We will denote by $E_b$  the ordered basis for $\mathcal{B}(L,q)$ given by the $\zeta_i[b]$, and by $E_w$ the ordered basis for $\mathcal{W}(L,q)$ given by the $\zeta_i[w]$.
We say these are \textit{exponential} bases\label{def:ebew}.

The discussion leading up to this point should make it clear that the matrix $K_E$\label{def:ke} representing $K$ from $E_b$ to $E_w$ is diagonal with entries $\lambda(K,z_i)$, and in similar fashion $K^*_E$ is diagonal with entries $\lambda(K^*,z_i)$.
We can thus calculate $\det(K_E)$ and $\det(K^*_E)$ with a simple product.
How do these determinants relate to that of the Kasteleyn matrix $K_D$ we had previously constructed?

Let $X(D_b,E_b)$\label{def:XAB} be the matrix that changes basis from $D_b$ to $E_b$, and similarly define $X(D_w,E_w)$.
Then it's clear that $K_D = X(D_w,E_w)^{-1} \cdot K_E \cdot X(D_b,E_b)$.
We will study the matrices $X$ in order to understand the relation between $\det(K_E)$ and $\det(K_D)$.

The $j$-th column of $X(D_b,E_b)$ is $\zeta_j[b]$ written in the ordered basis $D_b$.
It is easy to do so:
$$\zeta_j[b] = \sum\limits_{1 \leq i \leq m} \zeta_j[b](b_i) \cdot g[b_i].$$

Notice $\zeta_j[b](b_i)$ is simply $\zeta_j(b_i)$.
Thus, $X(D_b,E_b)_{i,j} = \zeta_j(b_i)$ and similarly $X(D_w,E_w)_{i,j} = \zeta_j(w_i)$.

Is there a relation between $\zeta_j(b_i)$ and $\zeta_j(w_i)$?
Not in principle --- they depend on our enumeration of colored vertices.
In Figure \ref{fig:kastenum}, we present a choice of vertex enumeration on $D_L$ that makes one such relation apparent.
\begin{figure}[ht]
		\centering
		\def\svgwidth{0.975\columnwidth}
    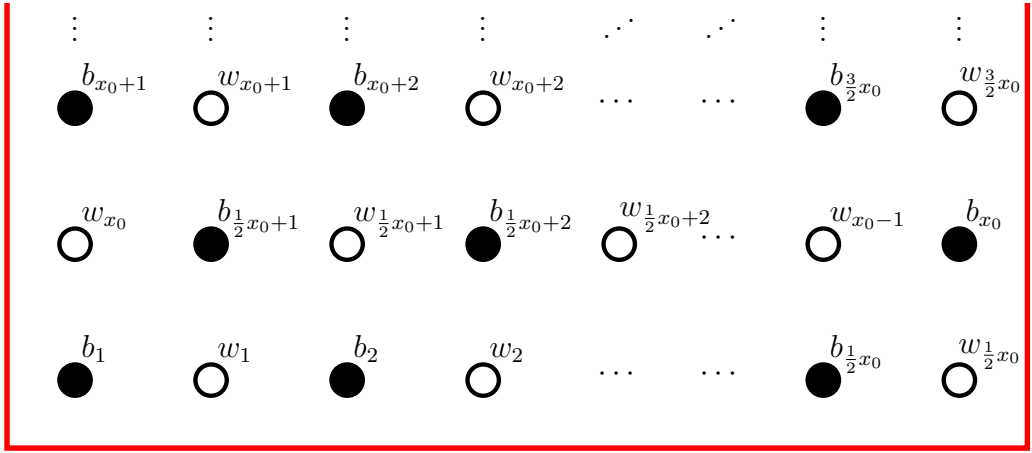
		\caption{Enumeration of $D_L$'s vertices. $b_1$ is the black square $[0,1]^2$. Notice this was used on the example of Figure \ref{kastex}.}
		\label{fig:kastenum}
\end{figure}

We explain it in words.
$D_L$ has $y_1$ lines with $x_0$ squares each.
There are two types of lines: black lines are lines whose first square (from left to right) is black, and similarly for white lines.
Notice line types alternate, and the first line (whose squares touch the horizontal line through the origin) is always black.
We assign 1 to the leftmost black square in that line, and the same number to the white square on its right.
If there's a black square to the right of that white square, we assign the next number, and so on until the line's squares are all labeled; we then proceed to the next line.
We repeat the process above, except this line is white, so we skip its first white square.
Because $x_0$ is always even, a white line 	always ends with a black square, and as we enumerated them, no white square lies to the right of it.
We then assign the first white square on that line (the one we skipped) the same number as its last, black square.
Repeating this procedure enumerates all squares on $D_L$.

With this enumeration, each white vertex $w_i$ satisfies $w_i = b_i + (1,0)$ \textbf{except} for white vertices in the beginning of a white line: these satisfy $w_i + v_0 = b_i + (1,0)$.
It's easy to see there are $\lfloor \frac{y_1}{2} \rfloor$ such vertices.
For vertices of the first kind, we have
\begin{equation*}
\zeta_j(w_i) = \zeta_j(b_i + (1,0)) = \zeta_j(b_i) \cdot \zeta_{j,0},
\end{equation*}
while for vertices of the second kind we have
\begin{equation*}
\zeta_j(w_i) = \zeta_j(b_i + (1,0) - v_0) = \zeta_j(b_i) \cdot \zeta_{j,0} \cdot q_1,
\end{equation*}
where the equality $\zeta_j(-v_0) = q_1$ comes from $\zeta_j\raisebox{-.2em}{$\big|_L$} = q$.

Thus, $X(D_w,E_w)$ is obtained from $X(D_b,E_b)$ by multiplying each column $j$ by $\zeta_{j,0}$ and $\lfloor \frac{y_1}{2} \rfloor$ of its lines --- the ones that correspond to indices $i$ for which $w_i$ is in the beginning of a white line --- by $q_1$.
It follows that
\begin{equation*}
\det\Big(X(D_w,E_w)\Big) = q_1^{^{\displaystyle \floor*{\tfrac{y_1}{2}}}} \cdot \left(\prod\limits_{j=1}^m{\zeta_{j,0}}\right) \cdot \det\Big(X(D_b,E_b)\Big),
\end{equation*}
so $\det(K_D) = \rho \cdot \det(K_E)$, where $\rho = q_1^{-\floor*{\frac{y_1}{2}}} \cdot \prod_{j=1}^m{(\zeta_{j,0})^{-1}}$\label{def:rho}. Notice $\rho \in \Sp^1$.

\section{Formulas for $\det(K_E)$, $\rho$ and uniform scaling}\label{sec:contas}

We will now make choices for our exponential bases from which we'll derive explicit formulas for $\det(K_E)$ and $\rho$.
Observe $L_0^*$ may be alternatively described (under inner product identification) as the lattice $\Z^2 \cup \big(\Z+\frac12\big)^2$, so by equation~\eqref{lambdak} each diagonal entry of $K_E$ is unique up to sign.

Because they're $L$-quasiperiodic with parameter $q$, the $\zeta_z$'s must satisfy:
\begin{equation*}\left\{ \begin{array}{l}
\begin{alignedat}{4}
\zeta_z(v_0) \enspace &= \enspace (\zeta_{z,0})^{x_0} \enspace &=& \enspace q_1^{-1}& \\
\zeta_z(v_1) \enspace &= \enspace (\zeta_{z,0})^{x_1} \cdot (\zeta_{z,1})^{y_1} \enspace &=& \enspace q_0&
\end{alignedat}\end{array} \right. \end{equation*}

There are $2m = x_0 \cdot y_1$ solutions --- twice the number of elements in a basis for $\mathcal{B}(L,q)$ or $\mathcal{W}(L,q)$.
If $q_0 = \exp(2\pi\bi\cdot u_0)$ and $q_1 = \exp(2\pi\bi\cdot u_1)$, these can be written as
\begin{equation*}\left\{ \begin{array}{l}
\begin{alignedat}{5}
\zeta_{z,0} \enspace &=& \enspace {\zeta[k_0,k_1]}_0 \enspace &=& \enspace &\exp\left(2\pi\bi\cdot\frac{1}{x_0}\cdot(k_0-u1)\right) \\
\zeta_{z,1} \enspace &=& \enspace {\zeta[k_0,k_1]}_1 \enspace &=& \enspace &\exp\left(2\pi\bi\cdot\frac{1}{x_0\cdot y_1}\cdot\Big((u_0+k_1)\cdot x_0 + (u_1-k_0)\cdot x_1\Big)\right)
\end{alignedat}\end{array} \right. \end{equation*}
where $0 \leq k_0 < x_0$ and $0 \leq k_1 < y_1$.

Let $E = \left\{\zeta[k_0,k_1]\text{ $|$ $0 \leq k_0 < \frac{1}{2}x_0$ and $0 \leq k_1 < y_1$}\right\}$.

\begin{prop}\label{basis}
$E$ defines a basis for each of $\mathcal{B}(L,q)$ and $\mathcal{W}(L,q)$.
\end{prop}
\begin{proof}
By Proposition \ref{eigenkort} and the discussion preceding it, if for all $0 \leq k_0,l_0 < \frac{1}{2}x_0$ and for all $0 \leq k_1, l_1 < y_1$ it holds that
\begin{equation*}
\left(\exists c\in \C, \enspace \zeta[k_0,k_1]\raisebox{-.2em}{$\big|_{L_b}$} = c \cdot \zeta[l_0,l_1]\raisebox{-.2em}{$\big|_{L_b}$}\right) \longrightarrow k_0=l_0 \text{ and }  k_1 = l_1,
\end{equation*}
then $E$ defines a basis for $\mathcal{B}(L,q)$, and similarly for $\mathcal{W}(L,q)$.
As before, in both cases there is one such $c$ if and only if $\zeta[k_0,k_1] \cdot \big(\zeta[l_0,l_1]\big)^{-1} = \mathbbm{1}$ on $L_0 = \text{span}\{(2,0),(1,1)\}$.
On the other hand, we have
\begin{subequations}
\begin{flalign}\label{basis0}
\left(\zeta[k_0,k_1] \cdot \big(\zeta[l_0,l_1]\big)^{-1}\right)(2,0) = \exp\left(2\pi\bi\cdot\frac{2}{x_0}\cdot(k_0-l_0)\right)&&
\end{flalign}
\begin{equation}\label{basis1}
\begin{split}
\begin{flalign*}
\left(\zeta[k_0,k_1] \cdot \big(\zeta[l_0,l_1]\big)^{-1}\right)(1,1) = &&
\end{flalign*}
\\
\begin{flalign}
&&\exp\left(2\pi\bi\cdot\frac{1}{x_0\cdot y_1}\cdot\Big((k_0-l_0)\cdot(y_1-x_1) + (k_1-l_1)\cdot x_0\Big)\right)
\end{flalign}
\end{split}
\end{equation}
\end{subequations}

Now, $-\frac{1}{2}x_0 < k_0 - l_0 < \frac{1}{2}x_0$, so if~\eqref{basis0} is 1 then $k_0 - l_0 = 0$.
In this case,~\eqref{basis1} reduces to $\exp\left(2\pi\bi \cdot\tfrac{1}{y_1}\cdot(k_1 - l_1)\right)$.
Since $-y_1 < k_1 - l_1 < y_1$, if~\eqref{basis0} and~\eqref{basis1} are both 1, then $k_0 = l_0$ and $k_1 = l_1$, as desired.
\end{proof}

With these bases and in the obvious notation, we have that
\begin{align*}\lambda(k_0,k_1) = \enspace &2\cos\left(2\pi\cdot\frac{1}{x_0\cdot y_1}\cdot\Big( (u_0 + k_1) \cdot x_0 + (u_1 - k_0) \cdot x_1 \Big) \right)\\
 + \enspace &2\bi\sin\left(2\pi\cdot\frac{1}{x_0} \cdot (k_0 - u_1) \right),\end{align*}
so $\det(K_E) = \prod_{0 \leq k_0 < \frac12 x_0}{\prod_{0 \leq k_1 < y_1}{\lambda(k_0,k_1)}}$. The term $\prod_{j=1}^m{(\zeta_{j,0})^{-1}}$ in the complex phase $\rho$ also admits a simple formula:$\vphantom{^{^{^{^{}}}}}$
$$\prod\limits_{0 \leq k_0 < \frac{1}{2}x_0}{\prod\limits_{0 \leq k_1 <y_1}{\big({\zeta[k_0,k_1]}_0\big)^{-1}}} = \exp\left(2\pi\bi\cdot\frac{y_1}{2}\cdot\left(u_1 + \frac12 - \frac{x_0}{4}\right)\right)$$

Using this, we may write\label{def:rho12}
$$\rho = \underbrace{\exp\left(2\pi\bi\cdot\frac{y_1}{2}\cdot\left(\frac12 - \frac{x_0}{4}\right)\right)}_{\displaystyle \rho_1\vphantom{^{^{}}}} \cdot \underbrace{\vphantom{\bigg(}\exp\left(2\pi\bi \cdot u_1 \cdot \left(\frac{y_1}{2}-\floor*{\frac{y_1}{2}}\right) \right)}_{\displaystyle \rho_2\vphantom{^{^{}}}}$$

Notice $\rho_2 = 1$ whenever $y_1$ is even.
When $y_1$ is odd, $\rho_2$ is a square root of $q_1 = \exp(2\pi\bi\cdot u_1)$.
If we restrict $u_1$ to lie on the interval $[0,1)$ --- that is, if $\arg(q_1) \in [0,2\pi)$ ---, then $\rho_2$ is the square root of $q_1$ in the upper half-plane that is not $-1$.

When $y_1$ is even, $\rho_1 = 1$ except when $x_0 \equiv 0 \Mod{4}$ and $y_1 \equiv 2 \Mod{4}$; in this case $\rho_1 = -1$. In particular, $\rho = \pm 1$ whenever $y_1$ is even.

When $y_1$ is odd, there are more cases for $\rho_1$:
\begin{equation*}\left\{
\begin{array}{lcl}
x_0 \equiv 2 \Mod{8} &\Longrightarrow &\rho_1 = +1 \\
x_0 \equiv 6 \Mod{8} &\Longrightarrow &\rho_1 = -1 \\
x_0 \equiv 0 \Mod{8}\text{, } y_1 \equiv 1 \Mod{4} &\Longrightarrow &\rho_1 = +\bi \\
x_0 \equiv 0 \Mod{8}\text{, } y_1 \equiv 3 \Mod{4} &\Longrightarrow &\rho_1 = -\bi \\
x_0 \equiv 4 \Mod{8}\text{, } y_1 \equiv 1 \Mod{4} &\Longrightarrow &\rho_1 = -\bi \\
x_0 \equiv 4 \Mod{8}\text{, } y_1 \equiv 3 \Mod{4} &\Longrightarrow &\rho_1 = +\bi
\end{array}\right.
\end{equation*}

An interesting fact is that when $y_1$ is even, $\det(K_E)$ is always real, regardless of the values of $q_0, q_1 \in \Sp^1$.
Indeed, it's easy to check that for each $0 \leq k_1 < \frac12 y_1$ it holds that
$$\lambda(k_0,k_1) = - \overline{\lambda\left(k_0,k_1+\tfrac12 y_1\right)},$$
so in this case $\det(K_E) = \prod_{0 \leq k_0 < \frac12 x_0}{\prod_{0 \leq k_1 < \frac12 y_1}{-{\lvert\lambda(k_0,k_1)\rvert}^2}}$, with sign given by $(-1)^{\frac14x_0\cdot y_1}$.
In particular, taking $\rho$ into account we conclude $\det(K_D) \leq 0$ whenever $y_1 \equiv 0 \Mod{4}$ and $\det(K_D) \geq 0$ whenever $y_1 \equiv 2 \Mod{4}$.

These formulas allow us to better understand what happens as some uniform scaling dilates the lattice $L$, and this will be the content of our next result.
Observe that for any valid lattice $L$, the numbers $x_0,x_1$ and $y_1$ are uniquely defined, and vice-versa.
Let $p_{[L,E]}: \R^2 \longrightarrow \C$ be defined by $p_{[L,E]}(u_0,u_1) = \det\big(K_E(u_0,u_1)\big)$, where $K_E$ is the Kasteleyn matrix for $L$ represented in our choice exponential bases and $u_0, u_1$ are the arguments of $q_0,q_1$ as above.

\begin{prop}\label{plperiod}
For any valid lattice $L$, $p_{[L,E]}$ satisfies the following periodicity relations:
\begin{alignat*}{3}
p_{[L,E]}(u_0+1,u_1) &= &p_{[L,E]}(u_0,u_1)& \\
p_{[L,E]}(u_0,u_1+1) &= (-1)^{y_1} \cdot \enspace &p_{[L,E]}(u_0,u_1)&
\end{alignat*}
In particular, we always have $p_{[L,E]}(u_0,u_1+2) = p_{[L,E]}(u_0,u_1)$.
\end{prop}
\begin{proof}
In the obvious notation, we have that $p_{[L,E]}(u_0,u_1) = \det\big(K_D(u_0,u_1)\big) \cdot \rho^{-1}$, and it's clear $K_D(u_0+a,u_1+b) = K_D(u_0,u_1)$ whenever $a,b \in \Z$.
We thus need only study how $\rho$ varies with $u_0,u_1$.

Inspecting the formula $\rho = \rho_1 \cdot \rho_2$ above, we see that $\rho_1$ depends only on $L$ (and not on $u_0,u_1$) and $\rho_2$ depends only on $u_1$, so the periodicity in $u_0$ is proved.
Now $\rho_2(u_1)$ is always 1 when $y_1$ is even, so the relation holds in this case.
When $y_1$ is odd, $\rho_2 = \exp\big(2\pi\bi \cdot \frac{1}{2}u_1\big)$ and the relation also is true.
\end{proof}

\begin{theo}\label{pkn}
Let $L$ be a valid lattice.
For any positive integer $n$ and reals $u_0, u_1$ $$p_{[nL,E]}(n\cdot u_0, n\cdot u_1) = \prod\limits_{0 \leq i,j < n} p_{[L,E]}\left(u_0+\frac{i}{n},u_1 - \frac{j}{n}\right).$$
\end{theo}
\begin{proof}
Let $L$ be generated by $v_0 = (x_0,0)$ and $v_1=(x_1,y_1)$, with $x_0, y_1 > 0$ and $0 \leq x_1 < x_0$.
Of course, in this case $nL$ is generated by $n\cdot v_0$ and $n\cdot v_1$.
Applying our formulas to $nL$ yields
\begin{equation*}
p_{[nL,E]}(n\cdot u_0, n\cdot u_1) = \prod_{\mathclap{\substack{\setlength{\jot}{-0.8\baselineskip}\begin{aligned}\scriptstyle\hphantom{.....} &\scriptstyle 0\leq k_0 < \frac{1}{2} nx_0 \\ &\scriptstyle0 \leq k_1 < n y_1\end{aligned}}}}{\hphantom{..}\lambda_n(k_0,k_1),}
\end{equation*}
where $\lambda_n(k_0,k_1)$ is given by
\begin{alignat*}{3}
2\cos &\left( 2\pi\cdot\frac{1}{nx_0\cdot ny_1}\cdot\Big( (nu_0 + k_1) \cdot nx_0 + (nu_1 - k_0) \cdot nx_1 \Big) \right)\enspace \\
+ \enspace 2\bi\sin &\left( 2\pi\cdot\frac{1}{nx_0} \cdot (k_0 - nu_1) \right) \\
= \hphantom{+} \enspace 2\cos &\left( 2\pi\cdot\frac{1}{x_0\cdot y_1} \cdot \left( \left(u_0 + \frac{k_1}{n}\right) \cdot x_0 + \left(u_1 - \frac{k_0}{n}\right) \cdot x_1 \right) \right) \enspace \\
+ \enspace 2\bi\sin &\left( 2\pi\cdot\frac{1}{x_0} \cdot \left(\frac{k_0}{n} - u_1 \right)\right).
\end{alignat*}

Now, for each $0\leq k_0 < \frac{1}{2} nx_0$ there are unique integers $0 \leq l_0 < \frac12 x_0$ and $0 \leq j < n$ with $k_0 = n\cdot l_0 + j$ (division with remainder).
In similar fashion, $k_1 = n \cdot l_1 + i$, where $0 \leq l_1 < y_1$ and $0 \leq i < n$ are unique integers.
We may then rewrite $\lambda_n(k_0,k_1)$ as $\lambda(l_0,l_1,i,j)$:
\begin{align*}
2\cos &\left(\frac{2\pi}{x_0\cdot y_1} \cdot \left( \left[\left(u_0 + \frac{i}{n}\right) + l_1\right] \cdot x_0 + \left[\left(u_1 -\frac{j}{n}\right) - l_0\right] \cdot x_1 \right) \right) \\
+ \enspace 2\bi\sin &\left(\frac{2\pi}{x_0} \cdot \left[l_0 - \left( u_1 - \frac{j}{n}\right)\right]\right).
\end{align*}

It's easy to see the $l_0$'s and $j$'s are in bijection with the $k_0$'s, and similarly for the $l_1$'s and $i$'s with the $k_1$'s.
It follows that:
\begin{equation*}
p_{[nL,E]}(n\cdot u_0, n\cdot u_1) = \prod\limits_{\vphantom{^{^{^{.}}}}0 \leq i, j < n}{\hphantom{..}\prod_{\mathclap{\substack{\setlength{\jot}{-0.8\baselineskip}\begin{aligned}\scriptstyle\hphantom{....} &\scriptstyle 0\leq l_0 < \frac{1}{2} x_0 \\ &\scriptstyle0 \leq l_1 < y_1\end{aligned}}}}{\hphantom{...}\lambda(l_0,l_1,i,j).}}
\end{equation*}

The theorem follows from observing that
$$\prod_{\mathclap{\substack{\setlength{\jot}{-0.8\baselineskip}\begin{aligned}\scriptstyle\hphantom{....} &\scriptstyle 0\leq l_0 < \frac{1}{2} x_0 \\ &\scriptstyle0 \leq l_1 < y_1\end{aligned}}}}{\hphantom{...}\lambda(l_0,l_1,i,j),} = p_{[L,E]}\left(u_0+\frac{i}{n},u_1 - \frac{j}{n}\right).$$
\end{proof}

Intuitively, Theorem \ref{pkn} says $\det\big({K[nL]}_E(q_0,q_1)\big)$ can be obtained from determinants of ${K[L]}_E$ by considering all $n$-th roots of $q_0$ and of $q_1$.
We can make this more precise.

We will say $L$\label{def:latticeoddeven} is odd if $y_1$ is odd, and $L$ is even if $y_1$ is even.
Let $P_{[L,E]}: \Sp^1 \times \Sp^1 \longrightarrow \C$ be given by $P_{[L,E]}(q_0,q_1) = p_{[L,E]}(u_0,u_1) \cdot \rho_2(L,u_1)$, where $q_i = \exp(2\pi\bi \cdot u_i)$.
In other words:
\begin{equation*}
P_{[L,E]}(q_0,q_1) = \left\{ \begin{array}{ll}
p_{[L,E]}(u_0,u_1)&\text{if $L$ is even;}\\
p_{[L,E]}(u_0,u_1)\cdot \exp(\pi\bi \cdot u_1)&\text{if $L$ is odd.}
\end{array}\right.
\end{equation*}

Notice that Proposition \ref{plperiod} ensures $P_{[L,E]}$ is well-defined.
Because $p_{[L,E]}(u_0,u_1) = \det(K_D(u_0,u_1))\cdot \rho^{-1}$, we have that:
\begin{equation}\label{pdetk}
\begin{alignedat}{2}
P_{[L,E]}(q_0,q_1) =  p_{[L,E]}(u_0,u_1) \cdot \rho_2(L,u_1) &= \det\big(K_D(u_0,u_1)\big) \cdot &{\rho_1(L)}^{-1} \\
& = \det\big(K_D(q_0,q_1)\big) \cdot &{\rho_1(L)}^{-1}
\end{alignedat}
\end{equation}

This means that, except for the complex phase ${\rho_1(L)}^{-1}$ --- which does not depend on $q_0$ or $q_1$ ---, $P_{[L,E]}$ is in fact the initial Laurent polynomial $P_K$ we calculated from our matrix $K_D$. 
In other words, the coefficients of $P_{[L,E]}$ are the Fourier coefficients of $p_{[L,E]}$, so that in particular $p_{[L,E]}$ has finitely many nonzero Fourier coefficients.
We hope to further consider this point of view in future work.

Of course, $P_{[L,E]}$ also admits its own version of Theorem \ref{pkn}.
Observe that $P_{[nL,E]}({q_0}^n,{q_1}^n)$ equals:
\begin{alignat*}{3}
 \enspace &\rho_2(nL,n\cdot u_1) \enspace  &\cdot& \enspace p_{[nL,E]}(n\cdot u_0,n \cdot u_1) \\
= \enspace &\rho_2(nL,n\cdot u_1) \enspace &\cdot& \prod\limits_{0 \leq i,j < n}{ p_{[L,E]}\left(u_0+\frac{i}{n},u_1 - \frac{j}{n}\right)} \\
= \enspace &\rho_2(nL,n\cdot u_1) \enspace  &\cdot& \prod\limits_{0 \leq i,j < n}{ P_{[L,E]}\left(q_0 \cdot \zeta^i,q_1 \cdot \zeta^{-j}\right) \cdot {\rho_2\left(L, u_1-\frac{j}{n}\right)}^{-1}},
\end{alignat*}
where $\zeta = \exp\left(2\pi\bi\frac{1}{n}\right)$.

When $L$ is even, $\rho_2(nL,n\cdot u_1)$ and the product $\prod_{0 \leq i,j < n}{{\rho_2\left(L, u_1-\frac{j}{n}\right)}^{-1}}$ are both trivially 1.
When $L$ is odd, it holds that
\begin{alignat*}{1}
\prod\limits_{0 \leq i,j < n}{{\rho_2\left(L, u_1-\frac{j}{n}\right)}^{-1}} &= \prod\limits_{0 \leq i,j < n}{\exp(-\pi\bi\cdot u_1)\cdot \exp\left(\pi\bi\cdot\tfrac{j}{n}\right)} \\
&= \exp\left(-n^2 \pi\bi\cdot u_1\right) \cdot \exp\left(\pi\bi \cdot \tfrac12 n (n-1)\right).
\end{alignat*}

If $n$ is even, $\rho_2(nL,n\cdot u_1) = 1$.
When both $n$ and $L$ are odd, $\rho_2(nL,n\cdot u_1) = \exp(n\pi\bi \cdot u_1)$ and we have that
\begin{equation*}
\begin{split}
\rho_2(nL,n\cdot u_1) \cdot \prod\limits_{0 \leq i,j < n}{{\rho_2\left(L, u_1-\frac{j}{n}\right)}^{-1}} \\
= \exp(-\pi\bi \cdot n(n-1) \cdot u_1) \cdot \exp\left(\pi\bi \cdot \tfrac12 n (n-1)\right)
\end{split}
\end{equation*}

Observe how in each case the value of the product $\rho_2(nL,n\cdot u_1) \cdot \prod_{0 \leq i,j < n}{{\rho_2\left(L, u_1-\frac{j}{n}\right)}^{-1}}$ does not depend on the choice of $u_1$ for which ${q_1}^n = \exp(2\pi\bi \cdot nu_1)$.

We can summarize these findings with
$$P_{[nL,E]}({q_0}^n,{q_1}^n) = \mu_0(q_1,n,L) \cdot \prod\limits_{0 \leq i,j < n}{ P_{[L,E]}\left(q_0 \cdot \zeta^i,q_1 \cdot \zeta^{-j}\right)},$$
\begin{flalign*}
\text{where $\mu_0(q_1,n,L) =$ }\left\{ \begin{array}{ll}
\begin{alignedat}{1}
&1 \\
&{q_1}^{-\tfrac12 n^2}  \\
-&{q_1}^{-\tfrac12 n^2} \\
&{q_1}^{-\tfrac12 n(n-1)} \\
-&{q_1}^{-\tfrac12 n(n-1)}
\end{alignedat}
&
\begin{alignedat}{1}
&\text{if $L$ is even;}\vphantom{1} \\
&\text{if $L$ is odd and $n \equiv 0 \Mod{4}$;}\vphantom{{q_1}^{-\tfrac12 n^2}} \\
&\text{if $L$ is odd and $n \equiv 2 \Mod{4}$;}\vphantom{{q_1}^{-\tfrac12 n^2}}  \\
&\text{if $L$ is odd and $n \equiv 1 \Mod{4}$;}\vphantom{{q_1}^{-\tfrac12 n(n-1)}} \\
&\text{if $L$ is odd and $n \equiv 3 \Mod{4}$.}\vphantom{{q_1}^{-\tfrac12 n(n-1)}}
\end{alignedat}
\end{array}\right.&&
\end{flalign*}

A more elegant expression can be given.
Using equation~\eqref{pdetk} we wite:
\begin{flalign*}
\mathrlap{\det\Big(K_D[nL]\left({q_0}^n,{q_1}^n\right)\Big) =}&&\\
&&\mu_0(q_1,n,L) \cdot \rho_1(nL) \cdot {\rho_1(L)}^{-n^2} \cdot \prod\limits_{0 \leq i,j < n}{ \det\Big(K_D[L]\left(q_0 \cdot \zeta^i,q_1 \cdot \zeta^{-j}\right)\Big)}
\end{flalign*}

Notice we always have ${\rho_1(L)}^4 = 1$, so ${\rho_1(L)}^{-n^2}$ is 1 whenever $n$ is even, and it is ${\rho_1(L)}^{-1}$ whenever $n$ is odd.
Letting $\mu_1(q_1,n,L) = \mu_0(q_1,n,L) \cdot \rho_1(nL) \cdot {\rho_1(L)}^{-n^2}$, it is not hard to check that
\begin{equation*}
\mu_1(q_1,n,L) = \left\{ \begin{array}{ll}
\begin{alignedat}{1}
&1 \\
&{q_1}^{-\tfrac12 n^2}  \\
&{q_1}^{-\tfrac12 n(n-1)} \\
\end{alignedat}
&
\begin{alignedat}{1}
&\text{if $L$ is even;}\vphantom{1} \\
&\text{if $L$ is odd and $n$ is even;} \vphantom{{q_1}^{-\tfrac12 n^2}}\\
&\text{if $L$ is odd and $n$ is odd.} \vphantom{{q_1}^{-\tfrac12 n(n-1)}}
\end{alignedat}
\end{array}\right.
\end{equation*}

Note how in each case the exponent of $q_1$ is an integer --- there is no ambiguity with choosing square roots.

\begin{corolario}Let $\mu_1$ be as above.
Then for each $q_0,q_1 \in \Sp^1$
\begin{equation*}
\det\Big(K_D[nL]\left({q_0}^n,{q_1}^n\right)\Big) = \mu_1(q_1,n,L) \cdot \prod\limits_{0 \leq i,j < n}{ \det\Big(K_D[L]\left(q_0 \cdot \zeta^i,q_1 \cdot \zeta^{-j}\right)\Big)},
\end{equation*}where $\zeta = \exp\left(2\pi\bi \frac{1}{n}\right)$ is an $n$-th root of unity.
\end{corolario}

\glsaddall

\printnoidxglossaries

\bibliography{biblio}{}
\bibliographystyle{plain}
\end{document}